\documentclass[11pt]{amsart}
\usepackage{amsmath,amsfonts,amssymb}

\usepackage{latexsym}
\usepackage{graphics,epsf,psfrag,epsfig}
\usepackage{eufrak}
\usepackage{xcolor}
\usepackage{hyperref}

\theoremstyle{plain}
\newtheorem{lemma}{Lemma}[section]
\newtheorem{theorem}[lemma]{Theorem}

\newtheorem{proposition}[lemma]{Proposition}

\theoremstyle{definition}
\newtheorem{remark}[lemma]{Remark}

\numberwithin{equation}{section}

\def\var{\text{var}}

\newcommand{\GG}{\mathbb{G}}
\newcommand{\ZZ}{\mathbb{Z}}

\newcommand{\RR}{\mathbb{R}}

\newcommand{\EE}{\mathbb{E}}
\newcommand{\LL}{\mathbb{L}}
\newcommand{\VV}{\mathbb{V}}

\newcommand{\mA}{\mathcal{A}}

\newcommand{\mkB}{\mathfrak{B}}
\newcommand{\mC}{\mathcal{C}}

\newcommand{\mE}{\mathcal{E}}
\newcommand{\mF}{\mathcal{F}}
\newcommand{\mG}{\mathcal{G}}

\newcommand{\mH}{\mathcal{H}}
\newcommand{\mI}{\mathcal{I}}
\newcommand{\mJ}{\mathcal{J}}
\newcommand{\mL}{\mathcal{L}}
\newcommand{\mM}{\mathcal{M}}

\newcommand{\mT}{\mathcal{T}}

\newcommand{\bdm}{\mathbf{m}}
\newcommand{\mkm}{\mathfrak{m}}
\newcommand{\mkM}{\mathfrak{M}}

\newcommand{\tp}{\hat{p}}
\newcommand{\tb}{\hat{b}}

\newcommand{\hb}{\hat{b}}

\newcommand{\hT}{\hat{T}}

\newcommand{\hw}{\hat{w}}
\newcommand{\hz}{\hat{z}}

\newcommand{\trho}{\tilde{\rho}}

\newcommand{\ep}{\epsilon}

\newcommand{\pa}{\partial}
\newcommand{\bs}{\backslash}

\newcommand{\Del}{{\rm Del}}

\newcommand{\cur}{{\rm cur}}

\newcommand{\ol}{\overline}

\newcommand{\fatdot}{\boldsymbol{\cdot}}

\newcommand{\tred}{\textcolor{red}}

\begin{document}

\title[Uniform bounds in FPP]{Uniform fluctuation and wandering bounds in first passage percolation}
\author{Kenneth S. Alexander}
\address{Department of Mathematics \\
University of Southern California\\
Los Angeles, CA  90089-2532 USA}
\email{alexandr@usc.edu}


\begin{abstract} 
We consider first passage percolation on certain isotropic random graphs  in $\RR^d$.  We assume exponential concentration of passage times $T(x,y)$, on some scale $\sigma_r$ whenever $|y-x|$ is of order $r$, with $\sigma_r$ ``growning like $r^\chi$'' for some $0<\chi<1$.  Heuristically this means transverse wandering of geodesics should be at most of order $\Delta_r = (r\sigma_r)^{1/2}$. We show that in fact uniform versions of exponential concentration and wandering bounds hold:  except with probability exponentially small in $t$, there are no $x,y$ in a natural cylinder of length $r$ and radius $K\Delta_r$ for which either (i) $|T(x,y) - ET(x,y)|\geq t\sigma_r$, or (ii) the geodesic from $x$ to $y$ wanders more than distance $\sqrt{t}\Delta_r$ from the cylinder axis.  We also establish that for the time constant $\mu = \lim_n ET(0,ne_1)/n$, the ``nonrandom error'' $|\mu|x| - ET(0,x)|$ is at most a constant multiple of $\sigma(|x|)$.
\end{abstract}

\maketitle

\tableofcontents

\section{Introduction.}
In i.i.d.~first passage percolation (FPP) on a graph $\GG=(\VV,\EE)$, i.i.d.~\emph{(edge) passage times} $\tred{t_e}$ are attached to the edges $e\in\EE$, and for a path $\Gamma $ in $\GG$, the \emph{(path) passage time} $\tred{T(\Gamma)}$ is the sum of the times $t_e$ over $e\in\Gamma$.  For $x,y\in\VV$, the \emph{passage time} from $x$ to $y$ is 
\begin{equation}\label{Tdef}
  \tred{T(x,y)} = \inf\{T(\Gamma): \Gamma \text{ is a path from $x$ to $y$ in $\GG$}\}.
\end{equation}
The \emph{geodesic} from $x$ to $y$ is the path, denoted $\tred{\Gamma_{xy}}$, which minimizes the path passage time; when $t_e$ is a continuous random variable (as we always assume), a unique geodesic exists a.s.~\cite{Ke86}.

There are two exponents of primary interest in the study of FPP.  First, the fluctuations (i.e.~standard deviation) of passage times $T(x,y)$ for $|y-x|$ of scale $r$ in $\ZZ^d$ are believed to be of order $r^\chi$ for some $\tred{\chi}=\chi_d<1/2$, with $\chi_2=1/3$.  Second, the typical transverse wandering of a geodesic, meaning the maximum distance from any point on $\Gamma_{xy}$ to the straight line (denoted \tred{$\Pi_{xy}$}) from $x$ to $y$, is believed to be of order $r^\xi$ for some $\tred{\xi}=\xi_d$.  For $|y-x|$ of order $r$, if $\Gamma_{xy}$ contains a vertex $z$ at distance of order $r^\xi$ from $\Pi_{xy}$ (not too near $x$ or $y$), then the associated extra distance $|z-x| + |y-z| - |y-x|$ traveled by the geodesic in order to pass through $z$ is of order $r^{2\xi-1}$.  For such wandering to have non-negligible probability, the passage time fluctuations $r^\chi$ should be at least as large as the extra distance; heuristically this leads to the relation $\chi=2\xi-1$.  There are various ways to formally define the exponents $\chi,\xi$; these must allow for the fact that the true scales of fluctuations and wandering are not known to be pure powers of $r$.  Chatterjee \cite{C13} gave a rigorous version of the relationship $\chi=2\xi-1$, under the assumption that multiple possible definitions of each exponent actually agree.  

Looking more finely than just at the level of exponents, the heuristic for $\chi=2\xi-1$ says that if the fluctuation scale is $\tred{\sigma_r}$ for $|y-x|$ of order $r$, then the scale of transverse wandering should be
\[
  \tred{\Delta_r} = \Delta(r) = (r\sigma_r)^{1/2}.
\]
In \cite{Al20} and (for $d=2$) in \cite{Ga19} it was shown that under natural assumptions, the transverse wandering with high probability does not exceed $(r\sigma_r\log r)^{1/2}$.  One of our main results here is an upper bound on wandering: under somewhat weaker assumptions, for all $d\geq 2$, the probability of wandering greater than $s\Delta_r$ decays as $e^{-cs^2}$ or faster, for $s\leq r/\Delta_r$.  Previously such a result has only been known for integrable cases of last passage percolation (LPP) in $d=2$, from \cite{BSS16} (with $e^{-cs}$ in place of $e^{-cs^2}$) and \cite{BSS19}.  The bound is optimal in the sense of being on the scale $\Delta_r$, though the second power of $s$ in the exponent may not be optimal, as suggested by LPP results in $d=2$ in \cite{GH20}.

In fact we have this bound uniformly over many geodesics simultaneously, in the following sense:  Consider a cylinder of length $r$ and radius $K\Delta_r$, and let $\ep>0$ and $s>2K$.  Then under the assumptions we will make, the probability that there exists \emph{any} geodesic $\Gamma_{xy}$, with $x,y$ in the cylinder and $|y-x|\geq\ep r$, which wanders farther than $s\Delta_r$ from the cylinder axis decays as $e^{-cs^2}$ or faster, for $s\leq r/\Delta_r$. The scale $\Delta_r$ here is optimal, though the decay $e^{-cs^2}$ may not be, based on LPP results for $d=2$ in \cite{GH20}.

By comparison, in \cite{Al20} it was shown roughly that if there is exponential concentration on some scale $\sigma_r$ which ``grows like a power of $r$,'' uniformly for passage times over distance $r$, then the probability of a transverse fluctuation of size $t(\log r)^{1/2}\Delta_r$ for a \emph{single} geodesic $\Gamma_{xy}$ is bounded by $C_{1}e^{-C_{2}t^2\log t}$ for all $t>0$. This tells us nothing, though, about transverse fluctuations of size $t\Delta_r$ with $1 \ll t \ll (\log r)^{1/2}$, which should also be subject to exponential concentration, as in our present result.

It should be emphasized that in our transverse wandering (and other) results, $\sigma_r$ is not necessarily the actual scale of the standard deviation---it need only be an upper bound in the sense that exponential concentration holds for passage times $T(x,y)$ on scale $\sigma(|y-x|)$.  Then the corresponding value $\Delta_r$ is the scale that appears in the upper bound for transverse wandering.

Our uniform wandering bound will be a byproduct of another uniform--bound result for passage times; to describe it we first discuss exponential bounds. 
For the lattice $\ZZ^d$, Kesten \cite{Ke93} proved that, assuming
\begin{equation}\label{Kestenhyp}
  Ee^{\lambda t_e}<\infty\ \text{for some $\lambda>0$, \hskip .5 cm and } \quad P(t_e=0)<p_c(\ZZ^d)
\end{equation}
(where $p_c(\ZZ^d)$ is the bond percolation threshold for $\ZZ^d$), there is exponential concentration of $T(x,y)$ on scale $r^{1/2}$ for $|x-y|\leq r$:
\[
  P(|T(x,y) - ET(x,y)|\geq	 tr^{1/2}) \leq C_{3}e^{-C_{4}t} \quad\text{for all } t\leq C_{5}r.
\]
Talagrand \cite{Ta95} improved this:  assuming just an exponential moment for $t_e$,
\[
  P(|T(x,y) - ET(x,y)|\geq	 tr^{1/2}) \leq C_{6}e^{-C_{7}\min(t^2,tr^{1/2}) } \quad\text{for all } t>0.
\]
Damron, Hanson, and Sosoe \cite{DHS14} improved the bound to a subgaussian scale: under \eqref{Kestenhyp},
\[
  P\left(|T(x,y) - ET(x,y)|\geq t\left(\frac{r}{\log r}\right)^{1/2}\right) \leq C_{8}e^{-C_{9}t} \quad\text{for all } t>0.
\]
None of these bounds are near--optimal, though---an optimal bound would be on the scale of the standard deviation of $T(x,y)$.  What we will prove here is roughly as follows.  Suppose passage times satisfy exponential concentration on a scale $\sigma(\cdot)$, uniformly: 
\begin{equation}\label{expbound}
  P\Big( |T(x,y) - ET(x,y)| \geq t\sigma(|y-x|) \Big) \leq C_{10}e^{-C_{11}t} \quad\text{for all } x,y,
\end{equation}
for some $\sigma(r)$ which ``grows like $r^\chi$'' for some $\chi\in(0,1)$, in a sense we will make precise. Then for $\tred{G_r(K)}$ a cylinder of length $r$ and radius $K\Delta_r$ for some fixed $K$, we have concentration on the same scale, uniformly over $x,y\in G_r(K)$:
\begin{equation}\label{Qrunif}
  P\Big( \big| T(x,y) - ET(x,y) \big| \geq  t\sigma_r \text{ for some $x,y \in G_K(r)$ with } |y-x|\geq \ep r \Big) 
    \leq C_{12}e^{-C_{13}t}
\end{equation}
for all $r$ large and $t\geq cK^2$.  This has previously been proved for integrable models of LPP in $d=2$ (\cite{BSS16}, \cite{BHS18}), but even the non-integrable part of the proof there does not carry over to FPP---see Remark \ref{basu}.

For $d\geq 3$ there is no generally-agreed-upon value of $\chi$ in the physics literature.  Heuristics and simulations suggest that $\chi$ should decrease with dimension; simulations in \cite{ROM15} for a model believed to be in the same (KPZ) universality class as FPP show a decrease from $\chi=.33$ to $\chi=.054$ as $d$ increases from 2 to 7.  Some have predicted the existence of a finite upper critical dimension, possibly as low as 3.5, above which $\chi=0$ (\cite{Fo08},\cite{LW05}); others predict that $\chi$ is positive for all $d$ (\cite{AOF14},\cite{MPPR02}), with simulations in \cite{KK14} showing $\chi>0$ all the way to $d=12$, decaying approximately as $1/(d+1)$. Our results here require $\chi>0$ so they only have content below the upper critical dimension, should it be finite. 

In the preceding and throughout the paper, $\tred{c_1,c_2,\dots}$ and $\tred{C_1,C_2,\dots}$, and $\tred{\ep_0,\ep_1,\dots}$ represent unspecified constants which depend only on the graph $\GG$ (or its distribution, if it is random) and the distribution of the passage times $t_e$ (or speeds $\eta_e$, to be given below.) We use $C_i$ for constants which occur outside of proofs and may be referenced later; any given $C_i$ has the same value at all occurrences.  We use $c_i$ for those which do not recur and are only needed inside one proof.  For the $c_i$'s we restart the numbering with $c_0$ in each proof, and the values are different in different proofs.

As is standard, since passage times $T(x,y)$ are subadditive, assumptions much weaker than \eqref{Kestenhyp} guarantee the a.s.~existence (positive and finite for $x\neq 0$) of the limit
\begin{equation}\label{gdef}
  \tred{g(x)} = \lim_n \frac{T(0,nx)}{n} = \lim_n \frac{ET(0,nx)}{n} = \inf_n \frac{ET(0,nx)}{n} \quad\text{a.s.~and in } L^1
\end{equation}
for $x\in\ZZ^d$; $g$ extends to $x$ with rational coordinates by considering only $n$ with $nx\in\ZZ^d$, and then to a norm on $\RR^d$ by uniform continuity.  We write \tred{$\mkB_g$} for the unit ball of this norm. To obtain the optimal uniform results for wandering and \eqref{Qrunif} for fluctuations, we need to understand both parts of the discrepancy
\begin{equation}\label{randomnon}
  T(0,x) - g(x) = \Big( T(0,x) - ET(0,x) \Big) + \Big( ET(0,x) - g(x) \Big).
\end{equation}
Here in the parentheses on the right are the \emph{random part} and \emph{nonrandom part} of the discrepancy.  In \cite{Al97} it was shown that under \eqref{Kestenhyp},
\begin{equation}\label{KAnonrand}
  g(x) \leq ET(0,x) \leq g(x) + C_{14}|x|^{1/2}\log |x| \quad\text{ for all } |x|>1.
\end{equation}
In \cite{Te18} the error term was improved to $C_{14}(|x|\log |x|)^{1/2}$, and in \cite{Ga19} to $c_\eta |x|^{1/2}(\log |x|)^\eta$ for all $\eta>0$.  For the Euclidean first passage percolation of \cite{HN97}, the analog of \eqref{KAnonrand} was proved in \cite{DW16} with an error term of $C_{14}\Psi(|x|)\log^{(k)}|x|$ for arbitrary $k\geq 1$, where $\Psi(|x|)$ is a scale on which an exponential bound is known (analogous to $\sigma(|x|)$ in \eqref{expbound}) and $\log^{(k)}|x|$ is the $k$--times--iterated logarithm.  Here we will obtain an essentially optimal bound for the nonrandom part: if $\sigma(\cdot)$ satisfies certain regularly and \eqref{expbound} holds, then a log factor as in the earlier bounds is unnecessary in our context: we have 
\begin{equation}\label{hmg}
  0 \leq ET(0,x) - g(x) \leq C_{15}\sigma(|x|).
\end{equation}
There is a strong interdependence among this result, our uniform wandering bounds, and \eqref{Qrunif}, as discussed in Remark \ref{strategy}.

Analogs of \eqref{hmg}, of \eqref{Qrunif}, and of of our uniform wandering result, with optimal scale $\sigma(|x|) = \var(T(0,x))^{1/2} \asymp |x|^{1/3}$, are known for certain integrable models of directed last passage percolation (LPP).  We note that an exponential bound like \eqref{expbound}, but with centering at the analog of $g(x)$ instead of at $ET(0,x)$, shows that (i) \eqref{hmg} must hold, and (ii) an exponential bound like \eqref{expbound} must also hold with centering at $g(y-x)$.  Such a recentered bound
appears in \cite{BSS16} (extracted from \cite{BFP14}) and in \cite{LR10} for LPP on $\ZZ^2$ with exponential passage times, in \cite{LM01}, \cite{LMR02} for LPP based on a Poisson process in the unit square, and in \cite{CLW16} for LPP on $\ZZ^2$ with geometric passage times.  An analog of our transverse wandering bound for LPP on $\ZZ^2$ with exponential passage times appears in \cite{BSS19}.  All of these require integral probability methods, which are not available for FPP.

Rather than work on the lattice $\ZZ^d$, we will consider isotropic models, built on a random graph $\tred{\GG = (\VV,\EE)}$ embedded in $\RR^d$.  The  \emph{dilation} of such an embedded graph is the least $C$ such that, for every $x,y\in\VV$ there is a path from $x$ to $y$ in $\GG$ for which the total (Euclidean) length of the edges is at most $C|y-x|$.  
We say that such a graph has \emph{bounded dilation} if there exists $C$ such that with probability one the dilation of $\GG$ is at most $C$.  For $A\subset\RR^d$, the \emph{restriction of $\GG$ to} $A$ is the graph with vertex set $\tred{\VV_A} = \{x\in\VV: \langle x,y \rangle \in \EE$ for some $y\in\VV\cap A\}$ and edge set $\tred{\EE_A} = \{\langle x,y \rangle\in \EE: x\in A\}$.

We require that the graph $\GG$ satisfy the assumptions A1 below, which are somewhat stringent and include bounded dilation, but we will construct an example that works. (We see no need to make the graph as general as possible; we simply need one to know we can work with one that has certain desirable properties.) That example is built roughly as follows.  We first construct a point process $\VV$ to serve as vertices, with $\VV$ satisfying those parts of A1 which involve only the vertices.  To make a graph from $\VV$ we use the Voronoi diagram, which divides $\RR^d$ into closed polyhedrons $\{\tred{Q_x}:x\in \VV\}$ (called \emph{Voronoi cells}), the interior of the polygon $Q_x$ consisting of those points which are strictly closer to $x$ than to any other point of $\VV$. We refer to $x$ as the \emph{center point} of $Q_x$, and define $\varphi$ by $\tred{\varphi(y)}=x$ for $y\in Q_x$. To produce the \emph{Delaunay graph} (or \emph{Delaunay triangulation} in $d=2$) one places an edge between each pair $x,y\in \VV$ for which $Q_x$ and $Q_y$ have a face of positive $(d-1)$--volume in common.  For $d=2$, it is known that the dilation of the Delaunay graph of any locally finite subset of $\RR^2$ is at most 1.998 \cite{Xi98}, ensuring A1 is fully satisfied, but such bounded dilation for $d\geq 3$ is not known.  We therefore modify the Delaunay graph by adding certain non-nearest-neighbor edges of uniformly bounded length, by a deterministic local rule, and show that bounded dilation then holds.  Here and in what follows, by the \emph{length} of an edge $e=\langle x,y \rangle$ we mean the Euclidean distance $|y-x|$, which we denote \tred{$|e|$}.

Our FPP proofs for isotropic random graphs should adapt to $\ZZ^d$, but (we expect) only by assuming two unproven properties of FPP on $\ZZ^d$: first, uniform curvature of $\pa\mkB_g$, and second, a kind of smoothness of the mean as the direction changes:
\[
  \sup\Big\{ |ET(0,x)-ET(0,y)|: \max(|g(x)-r|,|g(y)-r|) \leq C_{16}, |y-x| \leq C_{17}\Delta_r \Big\} = O(\sigma_r)\quad
    \text{as } r\to\infty.
\]
Results in \cite{Al97} show that for lattice FPP, the left side is $O(\sigma_r\log r)$. This is (implicitly) improved to $O(\sigma_r(\log r)^{1/2})$ in \cite{Te18} and to $O(\sigma_r(\log r)^\kappa)$ for all $\kappa>0$, in \cite{Ga19}.

Here then are the assumptions $\GG = (\VV,\EE)$ must satisfy.  \\

\noindent {\bf A1. Acceptability of random graphs.}
\begin{itemize}
\item[(i)] $\GG = (\VV,\EE)$ is isotropic, stationary, and ergodic;
\item[(ii)] Bounded hole size: every open ball in $\RR^d$ of radius 1 contains at least one vertex of $\VV$;
\item[(iii)] Finite range of dependence: there exists \tred{$\beta$} such that if $A,B$ are Lebesgue--measurable subsets of $\RR^d$ separated by distance $d(A,B)\geq \beta$, then the restrictions of $\GG$ to $A$ and to $B$ are independent;
\item[(iv)] Bounded dilation: the dilation of $\GG = (\VV,\EE)$ is bounded a.s.~(and hence equal to some nonrandom $C_{18}$ a.s., by (i));
\item[(v)] Exponential bound for the local density: given $r_0>0$ there exist $C_{19},C_{20}$ such that for all $r>r_0$ and $a\geq 1$, $P(|\VV\cap B_r(0)| \geq ar^d) \leq C_{19}e^{-C_{20}a}$;
\end{itemize}
We say a random graph with these properties is \emph{acceptable}.  We will show that acceptable random graphs exist.  By rescaling, we may replace radius 1 in (ii) by any other positive value.  Condition (v) can be weakened from exponential to stretched exponential; we use exponential to simplify the exposition.

Conditionally on $\GG$ we define a collection of i.i.d.~nonnegative continuous random variables $\eta=\{\tred{\eta_e},e\in\EE\}$.  Formally the pair $\tred{\omega}=(\VV,\eta)$ is defined on a probability space $\tred{(\Omega,\mF,P)}$, with $\GG$ determined by $\VV$.  

In contrast to the usual FPP on a true lattice, here we view $\eta_e$ not as a time but as a speed.  We thus define the passage time of a bond $e$ to be $\eta_e|e|$, and proceed ``as usual'': for $x,y\in \VV$, a \emph{path} $\Gamma$ from $x$ to $y$ is a finite sequence of alternating vertices and edges of $\GG$, of the form $\Gamma=(x=x_0, \langle x_0,x_1 \rangle, x_1,\dots,x_{n-1},\langle x_{n-1},x_n \rangle, x_n=y)$. We may designate a path by specifying only the vertices. The \emph{(path) passage time} of $\Gamma$ is 
\[
  \tred{T(\Gamma)} := \sum_{e\in\Gamma} \eta_e|e|,
\]
and the \emph{passage time} from $x$ to $y$ is
\begin{equation}\label{Tinf}
  \tred{T(x,y)} := \inf\{T(\Gamma): \Gamma \text{ is a path from $x$ to $y$ in } \GG\}.
\end{equation}
More generally, for $x,y\in\RR^d$ we define 
\[
 T(x,y) := T(\varphi(x),\varphi(y)).
\]
For technical convenience we do not require that paths be self-avoiding, but for the moment this is irrelevant because geodesics are always self-avoiding.
For general $x,y$ not necessarily in $\VV$, we take ``a geodesic from $x$ to $y$'' to mean a geodesic from the center point $\varphi(x)$ to $\varphi(y)$.  Let $\tred{\zeta(\lambda)} = Ee^{\lambda\eta_e}$. We assume the following.\\

\noindent{\bf A2. $\eta_e$ properties.}
\begin{itemize}
\item[(i)] $\eta_e$ is a continuous random variable.
\item[(ii)] There exists $\lambda>0$ such that $\zeta(\lambda)<\infty$.
\end{itemize}
Here (i) guarantees that there is at most one geodesic from $x$ to $y$ a.s., for each $x,y$.

We do not assume that $\sigma(r)$ in \eqref{expbound} is a power of $r$, but we do require that it ``grows like $r^\chi$''  for some $0<\chi<1$ in a sense similar to \cite{Al20}, as follows.
We call a nonnegative function $\{\sigma(r)=\sigma_r, r>0\}$ \emph{powerlike (with exponent $\chi$)} if there exist $0<\tred{\chi_1}<\chi<\tred{\chi_2}$ and constants $C_i$ such that 
\begin{equation}\label{powerlike}
  \lim_{r\to\infty} \frac{\log \sigma_r}{\log r} = \chi\quad \text{ and for all } s\geq r\geq C_{21}, \quad 
    C_{22}\left( \frac sr \right)^{\chi_1} \leq \frac{\sigma_s}{\sigma_r}
     \leq C_{23} \left( \frac sr \right)^{\chi_2}.
\end{equation}
If \eqref{powerlike} holds with $\chi_2<1$ we say $\sigma(\cdot)$ is \emph{sublinearly powerlike}. 

The requirement that $\sigma(\cdot)$ be sublinearly powerlike is perhaps not as stringent as it first appears, due to Lemma \ref{sublin} below.   It implies that if one is only interested in fluctuations at the level of exponents $\chi,\xi$, and one knows uniform the exponential tightness \eqref{expbound} but not necessarily the powerlike property for $\sigma(\cdot)$, then there is a sublinearly powerlike function with the same $\chi$ for which \eqref{expbound} holds.

\begin{remark}\label{sigmaprop}
If $\sigma(\cdot)$ is powerlike, then so is the increasing function $\tred{\hat\sigma(r)} = \sup_{s\leq r} \sigma(s)$; by further increasing $\hat\sigma$ (though by at most a constant factor) we may make it strictly increasing and continuous while preserving the powerlike property.  Therefore we may and do without loss of generality always assume $\sigma(\cdot)$ is strictly increasing and continuous.  The inverse function $\Delta^{-1}$ is well-defined, and
for $\xi=(1+\chi)/2 \in (\frac 12,1)$ we have $\Delta(r) \asymp r^\xi$ and $\Delta^{-1}(a)\asymp a^{1/\xi}$ in the sense that
\[
  \lim_{r\to\infty} \frac{\log \Delta(r)}{\log r} = \xi, \qquad \lim_{a\to\infty} \frac{\log \Delta^{-1}(a)}{\log a} = \frac 1\xi.
\]
\end{remark}

The following is proved in section \ref{last}.

\begin{lemma}\label{sublin}
Let $\tred{\rho}: (1,\infty)\to(0,\infty)$ satisfy
\[
  \lim_{r\to\infty} \frac{\log \rho_r}{\log r} = \chi\in (0,\infty).
\]
Then there exist $\tred{\trho}\geq\rho$ which is sublinearly powerlike with the same exponent $\chi$, and given $\ep>0$ we may take $\trho$ to satisfy \eqref{powerlike} with $|\chi_i-\chi|<\ep,\ i=1,2$.
\end{lemma}

For general $x,y\in\RR^d$ not necessarily in $\VV$, we write $\Gamma_{xy}$ for $\Gamma_{\varphi(x),\varphi(y)}$.
In general we view $\Gamma_{xy}$ as an undirected path, but at times we will refer to, for example, the first point of $\Gamma_{xy}$ with some property.  Hence when appropriate, and clear from the context, we view $\Gamma_{xy}$ as a path from $\varphi(x)$ to $\varphi(y)$.

Our final standard assumption is the following.\\

\noindent{\bf A3. Uniform exponential tightness.}\\
For some $\sigma(\cdot)$ which is sublinearly powerlike with exponent $\chi\in(0,1)$,
\begin{equation}\label{expbound2}
  P\Big( |T(x,y) - ET(x,y)| \geq t\sigma(|y-x|) \Big) \leq C_{24}e^{-C_{25}t} \quad\text{for all } x,y\in\RR^d.
\end{equation}

The isotropic property assumed for $\VV$ in A1 means that $g(x)=\mu|x|$ for all $x$, where $\tred{\mu}=g(e_1)$, and $ET(0,x)$ depends only on $|x|$, so we define
\[
  \tred{h(r)} = ET(0,re_1).
\]
Let $\tred{\mkB_{d-1}}$ be the Euclidean unit ball of $\RR^{d-1}$ and define the cylinders
\[
  \tred{G_r(K)} = [0,r]\times K\Delta_r\mkB_{d-1}.
\]
Here is our first main result.

\begin{theorem}\label{nofast}
Suppose $\GG=(\VV,\EE)$ and $\{\eta_e,e\in\EE\}$ satisfy A1, A2, and A3.  Then given $\ep>0$ there exist constants $C_i=C_i(\ep)$ such that for all $r\geq 1, K\geq 1, t\geq C_{26}K^2$,
\begin{align}\label{Qrunif1}
  P\Big( \Big| T(x,y) - ET(x,y) \Big| \geq t\sigma_r \text{ for some $x,y \in G_r(K)$ with } |y-x|\geq \ep r \Big) 
    \leq C_{27}e^{-C_{28}t},
\end{align}
and 
\begin{equation}\label{Qrunif2}
  P\Big( T(x,y) \leq h(|(y-x)_1|) - t\sigma_r \text{ for some } x,y \in G_r(K) \Big) \leq C_{29}e^{-C_{30}t}.
\end{equation}
\end{theorem}

Equation \eqref{Qrunif2} is weaker than \eqref{Qrunif1} in the sense that $h(|(y-x)_1|)$ is (up to a constant) smaller than $ET(x,y)$ (see Lemma \ref{monotoneE}), but stronger in that it isn't limited to $|y-x|\geq \ep r$.

We can split \eqref{Qrunif1} into upward and downward deviations:
\[
  T(x,y) \geq ET(x,y) + t\sigma_r \quad\text{and}\quad T(x,y) \leq ET(x,y) - t\sigma_r.
\]
Then the downward--deviations part of \eqref{Qrunif1} is a consequence of \eqref{Qrunif2} and Proposition \ref{hmu} below, because 
\[
  x,y\in G_r(K), |y-x|\geq\ep r \implies \Big| |y-x| - |(y-x)_1| \Big| \leq C_{31}K^2\sigma_r
\]
for some $C_{31}(\ep)$.  In general, if we think of $x$ to $y$ as an increment a path might make within $G_r(K)$ in going from the one end to the other in the $e_1$ direction, then $(y-x)_1$ measures progress made by that increment in the $e_1$ direction, so it is a natural normalization of $T(x,y)$ in the context of such paths. It is also sufficient for application to the next two theorems.

\begin{remark}\label{smalltubes}
For \eqref{Qrunif1} in Theorem \ref{nofast} we can replace the conditions $K\geq 1,t\geq C_{26}K^2$ with $K\geq C_{32},t\geq \ep K^2$.  This is because there exists $m$ such that $G_r(K)$ is contained in $C_{32}\ep^{-m}$ ``thin cylinders'' (not necessarily oriented parallel to $e_1$) of length between $\ep r$ and $r$ and radius $K_{\ep}\Delta_r=\sqrt{\ep/C_{26}}K\Delta_r$ such that every pair $x,y$ as in \eqref{Qrunif1} is contained in one of these cylinders.  Here $C_{32}$ does not depend on $K$ or $r$.  We can apply the theorem to each thin cylinder, since $t \geq C_{26}K_{\ep}^2$, then sum over the thin cylinders.
\end{remark}

We will use Theorem \ref{nofast} together with a coarse-graining scheme in establishing the following.
 
\begin{theorem}\label{ghbound}
Suppose $\GG=(\VV,\EE)$ and $\{\eta_e,e\in\EE\}$ satisfy A1, A2, and A3.  There exists $C_{33}$ such that
\begin{equation}\label{bestnon}
  \mu |x| \leq h(|x|) \leq \mu |x| + C_{33}\sigma(|x|) \quad \text{for all } x\in\RR^d.
\end{equation}
\end{theorem}

As we have noted, for $|y-x|$ of order $r$, if $\Gamma_{xy}$ contains a vertex $z$ at distance of order $\gg\Delta_r$ from $\Pi_{xy}$ (not too near $x$ or $y$), then the associated extra distance $g(z-x)+g(y-z)-g(y-x)$ traveled by the geodesic is of order $\gg\Delta_r^2/r=\sigma_r$, and by Theorem \ref{ghbound} the same is true for $h$ in place of $g$.  Since the corresponding passage times satisfy $T(x,z)+T(z,y)-T(x,y)=0$, this means that either 
\[
  T(x,y) - ET(x,y) \gg \sigma_r, \quad T(x,z) - ET(x,z) \ll -\sigma_r, \quad\text{or}\quad 
    T(z,y) - ET(z,y) \ll -\sigma_r.
\]
The assumption \eqref{expbound2} says the first of these is unlikely, and Theorems \ref{nofast} and \ref{ghbound} can be used to show it is unlikely that there exists a $z$ for which the second or third occurs.  (Not without complications, though, as we cannot assume $z\in G_r$.)  This is the idea behind the following. For $r,s>0$ define intervals enlarging $[0,r]$:
\begin{equation}\label{Irs}
  \tred{\mI_{r,s}} = \begin{cases} 
    [-s^2\sigma_r\log r,r+s^2\sigma_r\log r] &\text{if } s\leq (C_{34}\log r)^{1/2};\\
    [-s^2\sigma_r,r+s^2\sigma_r] &\text{if } (C_{34}\log r)^{1/2}<s\leq r/\Delta_r;\\
    [-s\Delta_r,r+s\Delta_r] &\text{if } s> r/\Delta_r, \end{cases}
\end{equation}
where $C_{34}$ ``sufficiently large'' will be specified later, and
\[
  \tred{G_{r,s}} = \mI_{r,s} \times s\Delta_r\mkB_{d-1}.
\]
For $s>K$ we have $G_r(K)\subset G_{r,s}$; in this case we may view $G_{r,s}$ as being the cylinder $G_r(K)$ fattened transversally to width $s\Delta_r$, and lengthened by an amount which varies with the size of $s$ relative to $r$, chosen to be ``enough to make wandering of geodesics out the cylinder end at least as unlikely as out the sides.''

For $x\in\RR^d$ we write $\tred{x^*}$ for $(x_2,\dots,x_d)$ so $x=(x_1,x^*)$. 

\begin{theorem}\label{tversethm}
Suppose $\GG=(\VV,\EE)$ and $\{\eta_e,e\in\EE\}$ satisfy A1, A2, and A3.  There exist $C_i$ such that for all $K\geq C_{35}$,
\begin{align}\label{tverse3}
  P\Big( \Gamma_{xy} &\not\subset G_{r,s} \text{ for some $x,y\in G_r(K)$ with } |(y-x)^*|\leq (y-x)_1  \Big) \notag\\
  &\leq \begin{cases} C_{36}e^{-C_{37}s^2} &\text{for all } C_{38}K \leq s \leq r/\Delta_r,\\
     C_{36}e^{-C_{37}s\Delta_r/\sigma(s\Delta_r)} &\text{for all } s>r/\Delta_r.  \end{cases}
\end{align}
\end{theorem}
We include the condition $|(y-x)^*|\leq (y-x)_1$ because we are primarily interested in transverse fluctuations of geodesics out the side of $G_{r,s}$, so we wish to avoid $y-x$ oriented in a too--transverse direction.

\begin{remark}\label{strategy}
The strategy for proving Theorems \ref{nofast}--\ref{tversethm} is as follows:
\begin{itemize}
\item[(1)] prove Theorem \ref{nofast} for downward deviations---this is the most difficult part;
\item[(2)] use Theorem \ref{nofast} for downward deviations to prove Theorem \ref{tversethm} restricted to a fixed $(x,y)$;
\item[(3)] use Theorem \ref{nofast} for downward deviations and the restricted Theorem \ref{tversethm} to prove Theorem \ref{ghbound};
\item[(4)] use Theorem \ref{ghbound} to prove Theorem \ref{nofast} for upward deviations;
\item[(5)] use the full Theorem \ref{nofast} to prove the unrestricted Theorem \ref{tversethm}.
\end{itemize}
\end{remark}

\begin{remark}\label{basu}
In \cite{BSS16} and \cite{BHS18}, an alternate strategy was used to prove LPP analogs (in integrable cases) of Theorems \ref{nofast} and \ref{tversethm} in $d=2$. Theorem \ref{ghbound} was already known for that context---see the comments following \eqref{hmg}.  Essentially the strategy for the Theorem \ref{nofast} analog in \cite{BSS16} 
is this, when translated to FPP:  first the easier upward--deviations half of Theorem \ref{nofast} is proved.  For downward deviations, consider the points 0 and $3re_1$, and a cylinder $\mG_r$ of radius $\Delta_r$ with axis from $re_1$ to $2re_1$.  Suppose there are (random) vertices $u,v\in\mG_r$ with $u_1<v_1$ for which the passage time is fast: $T(u,v) \leq h(|v-u|) -3t\sigma_r$ for some large $t$.  From the upward--deviations half of Theorem \ref{nofast}, with high probability we have also $T(x,u) \leq h(|u-x|) + t\sigma_r$ and $T(v,y) \leq h(|y-v|) + t\sigma_r$. From this, using that $\sigma_r$ is proportional to $r^{1/3}$ and $h(r) = \mu r + O(\sigma_r)$, assuming $t$ is large enough,
\begin{equation}\label{basu2}
  T(x,y) \leq	h(|u-x|) + h(|v-u|) + h(|y-v|) - t \sigma_r \leq h(|y-x|) - \frac t2 \sigma_r.
\end{equation}
This has probability exponentially small in $t$, by \eqref{expbound2}, hence so does the probability of such $u,v$ existing.  This strategy does not work for FPP, however, as it requires one to already know Theorem \ref{ghbound} to obtain the second inequality in \eqref{basu2}.  
\end{remark}

\section{Existence of acceptable random graphs}

We construct a random graph $\GG=(\VV,\EE)$ satisfying A1.  We begin by constructing the point process $\VV$ of vertices.
We start with a ``space--time'' Poisson process $\tred{\VV_0}$ (which we view as a random countable set) of density 1 with respect to Lebesgue measure in $\RR^d\times(0,\infty)$.
We say a point $v\in\RR^d$ \emph{appears at time} $t$ in $\VV_0$ if $(v,t)\in\VV_0$, and we say $A\subset\RR^d$ is \emph{empty at time} $t-$ if no point in $A$ appears in $\VV_0$ during $(0,t)$.  We then define
\begin{align*}
  \tred{\VV} = \Big\{ v\in\RR^d: &\text{ for some $t>0$ and $x\in\RR^d$ with $v\in B_1(x)$, 
    $v$ appears at time $t$ in $\VV_0$ and }\\
  &\qquad\text{$B_1(x)$ is empty at time $t-$} \Big\}.
\end{align*}
In other words, we keep in $\VV$ the $\RR^d$ coordinate $v$ of a point of $\VV_0$ if $v$ is the first point to appear in some ball of radius $1$.  Then almost surely, for each $v\in\VV$ there is a unique point $(v,t_v^*) \in \VV_0$, and we view $t_v^*$ as the time at which $v$ appeared in $\VV_0$.  With probability one, every unit ball in $\RR^d$ contains a point of $\VV$.  We call $\VV$ the \emph{available--space point process}.

For the set $\VV$ recall that $\{Q_v,v\in\VV\}$ denotes the corresponding Voronoi cells.  
More generally we write \tred{$Q_x$} for the Voronoi cell containing any $x$ (with some arbitrary convention if $x$ is on the boundary of multiple cells), and $\varphi(x)$ for the unique point of $\VV$ in $Q_x$.  When convenient we view $e=\langle x,y \rangle$ as the line segment joining $x$ and $y$.
Let $\tred{B_r(x)}$ denote the open Euclidean ball of radius $r$ about $x$. For $d=2$, the Delaunay graph of $\VV$ is our graph $\GG$.  For $d\geq 3$ we fix $0<\tred{\delta_{\GG}}<1$ and use
\[
  \tred{\EE} = \{\langle x,y \rangle: d(Q_x,Q_y)\leq\delta_{\GG} \}.
\]
We call $\GG=(\VV,\EE)$ the \emph{augmented Delaunay graph} of $\VV$.  The edges in $\EE$ which are not in the Delaunay graph are called \emph{augmentation edges}. We write \tred{$x\sim y$} to denote that $x,y$ are adjacent vertices in $\GG$, and \tred{$x\sim_\Del y$} to denote adjacency in the Delaunay graph of $\VV$.  If $y\sim_\Del z$, then for all $u\in Q_y\cap Q_z, B_{|y-z|/2}(u)\cap \VV=\emptyset$.  Hence by A1(ii), 
\begin{equation}\label{separation}
  |z-y| < 2 \text{  whenever $y\sim_\Del z$}.
  \end{equation}
Similarly if $z\in Q_y$ then $B_{|z-y|}(z)$ contains no point of $\VV$, so $|z-y|<1$; thus 
\begin{equation}\label{Qxsize}
  Q_y \subset B_1(y).
\end{equation}

\begin{remark}\label{augment}
The purpose of augmentation is roughly the following.  Consider the line segment $\Pi_{xy}$ between $x,y\in\VV$.  It passes through a sequence of Voronoi cells $Q_x=Q_{x_0},Q_{x_1},\dots,Q_{x_m}=Q_y$, and there is a corresponding path $x=x_0 \to x_1 \to\cdots\to x_m=y$ in the Delaunay graph.  If too many of the cells $Q_{x_j}$ are ``thin,'' then the Delaunay path length $\sum_{j=1}^m |x_j-x_{j-1}|$ may be much greater than $|y-x|$, making it difficult to bound the dilation.  The augmentation effectively allows paths in $\GG$ that skip over such problematic sequences of cells, at least for a small distance, enabling us to prove bounded dilation while preserving other properties of the Delaunay triangulation. We can reduce the occurrence of augmentation to involve an arbitrarily small proportion of vertices by using a small enough $\delta>0$.  We will not give details here.
\end{remark}

For $A\subset \RR^d$ and $r>0$ let $\tred{A^r} = \{x: d(x,A)<r\}$.

\begin{proposition}\label{controlled}
The augmented Delaunay graph of the available--space point process satisfies A1.
\end{proposition}

\begin{proof}
A1(i) for $\VV$ follows from the same properties for $\VV_0$; since $\GG$ is constructed from $\VV$ via isotropic and translation-invariant local rules, A1(i) also holds for $\GG$.  A1(ii) follows from the fact that the first point of $\VV_0$ to appear in any radius-1 ball is always a point of $\VV$. 

To prove A1(iii), observe first that by \eqref{separation}, given $x\in\VV$ the cell $Q_x$ and all Delaunay edges $\langle x,y \rangle$ are determined by $\VV\cap B_2(x)$.  Therefore the collection of Voronoi cells intersecting $B_3(x)$ is determined by $\VV\cap B_5(x)$, and hence so are all augmentation edges $\langle x,y \rangle$.  $\VV\cap B_5(x)$, in turn, is determined by $\VV_0\cap (B_7(x)\times(0,\infty))$, so for $A\subset\RR^d$, the restriction $(\VV_A,\EE_A)$ is determined by $\VV_0\cap (A^7\times(0,\infty))$.  Since $\VV_0$ is independent in disjoint regions, A1(iii) follows with $\beta=14$.

Turning to A1(v), let $q=1/(1+\lfloor 3\sqrt{d} \rfloor)$ and $r>r_0>0$, and let $r_1=r_0\wedge 1$. For $x\in qr_1\ZZ^d$ let $\tred{J_x}$ denote the cube $\prod_{i=1}^d [x_i,x_i+qr_1)$ and $\tred{\mJ_x} = \{J_y: y\in qr_1\ZZ^d, |y-x|_\infty = 2qr_1\}$.  This means the $5^d-3^d$ cubes in $\mJ_x$ form a shell around $J_x$, with a smaller shell of cubes in between, and the diameter of $\mJ_x$ is less than $2r_1$.  Then any radius-$1$ ball intersecting $J_x$ must contain a cube in $\mJ_x$.  Letting $\tred{t^*(J_x)} = \max\{t_v^*:v\in\VV\cap J_x\}$ it follows that at time $t^*(J_x)$, for some $J_y\in\mJ_x$, at least $|\VV\cap J_x|$ points of $\VV_0$ have appeared in $J_x$ but none in $J_y$.  Letting \tred{$N_{xy}$} be the number of points of $\VV_0$ appearing in $J_x$ before the first point of $\VV_0$ appears in $J_y$, this says that $\tred{N_x} :=\max_{J_y\in \mJ_x} N_{xy}\geq |\VV\cap J_x|$.  Since $|\mJ_x|\leq 5^d$ it follows that 
\begin{align}\label{allJx}
  P(|\VV\cap J_x| \geq n) &\leq P(N_x\geq n) \leq \sum_{J_y\in\mJ_x} P(N_{xy}\geq n) \leq 5^d2^{-n},
\end{align}
and hence for $\lambda<\log 2$,
\begin{equation}\label{emom}
  Ee^{\lambda N_x} \leq \frac{5^d}{1-\frac{e^\lambda}{2}}.
\end{equation}

Let $\tred{\mI_0} = \{J_x: J_x\cap [-r,r)^d\neq\emptyset\}$, so $|\mI_0|\leq (2(1+\frac{r}{qr_1}))^d$ and $B_r(0)\subset \cup_{J_x\in \mI_0}\, J_x$.  Now $\mI_0$ is a lattice of cubes, and we divide it into $5^d$ sublattices, each consisting of a cube $J_x$ together with all its translates by vectors in $5qr_1\ZZ^d$ which intersect $[-r,r)^d$.  We label these sublattices of cubes as $\tred{\mI_{0,1},\dots,\mI_{0,5^d}}$, and the cardinality satisfies 
\begin{equation}\label{Isize}
  |\mI_{0,j}|\leq \left(2\left(1+\frac{r}{5qr_1}\right)\right)^d
\end{equation}
for each $j$.  We denote the corresponding union of cubes as $\tred{I_{0,j}} = \cup_{J_x\in\mI_{0,j}} J_x$.  The spacing of the cubes in $\mI_{0,j}$ means that that the shells $\{\mJ_x: J_x\in \mI_{0,j}\}$ are disjoint, so the variables $\{N_x: J_x\in \mI_{0,j}\}$ are i.i.d.

For $a\geq 1$, taking $\lambda=1/5$ in \eqref{emom} we obtain $Ee^{\lambda N_x} \leq 3\cdot 5^d$ and hence using \eqref{Isize}, provided $a$ is large (depending on $r_1$),
\begin{align}\label{manypts}
  P(|\VV\cap B_r(0)| \geq ar^d) &\leq \sum_{j=1}^{5^d} P\left(|\VV\cap I_{0,j}|\geq \frac{ar^d}{5^d}\right) \notag\\
  &\leq \sum_{j=1}^{5^d} P\left(\sum_{J_x\in\mI_{0,j}} N_x \geq \frac{ar^d}{5^d}\right) \notag\\
  &\leq 5^d\left( 3\cdot 5^d \right)^{|\mI_{0,j}|}  e^{-ar^d/5^{d+1}} \notag\\
  &\leq e^{-ar^d/2\cdot 5^{d+1}},
\end{align}
proving A1(v).

Finally we consider A1(iv). Let $x,y\in\VV$. As in Remark \ref{augment}, $\Pi_{xy}$ passes through a sequence of Voronoi cells $Q_x=Q_{x_0},Q_{x_1},\dots,Q_{x_m}=Q_y$, and there is a corresponding path $x=x_0 \to x_1 \to\cdots\to x_m=y$ in the Delaunay graph.  (There is probability 0 that $\Pi_{xy}$ intersects some $Q_u$ in just a single point, so we will ignore this possibility, meaning that ``passes through'' here is unambiguous.) For $j<m$ let $a_j$ be the first point of $Q_{x_j}$ in $\Pi_{xy}$ and let $a_m=y$, so by convexity of cells, $Q_{x_j}\cap\Pi_{xy} = [a_j,a_{j+1}]$ for all $0\leq j<m$.  We select indices $0=j(0)<j(1)<\cdots<j(\ell)=m$ iteratively, taking $j(k+1)$ as the least index $j>j(k)$ for which either $|a_{j+1} - a_{j(k)+1}|> \delta_{\GG}$ or $j=m$.  Then $\langle x_{j(k)},x_{j(k+1)} \rangle$ is always a Delaunay or augmentation edge, so we consider the path $x=x_{j(0)}\to x_{j(1)}\to\cdots\to x_{j(\ell)}=y$ in $\GG$; in particular we want to bound $|x_{j(k+1)} - x_{j(k)}|$ relative to $|a_{j(k+1)+1} - a_{j(k)+1}|$ for $0\leq k\leq\ell-2$.

For $k\leq\ell-1$ we have using \eqref{Qxsize}
\[
  |x_{j(k+1)} - x_{j(k)}| \leq |x_{j(k+1)} - a_{j(k+1)}| + |a_{j(k+1)} - a_{j(k)+1}| + |a_{j(k)+1} - x_{j(k)}| \leq 2+\delta_{\GG}
\]
so for $k\leq \ell-2$,
\[
  |a_{j(k+1)+1} - a_{j(k)+1}| > \delta_{\GG} \geq \frac{\delta_{\GG}}{2+\delta_{\GG}} |x_{j(k+1)} - x_{j(k)}|.
\]
Therefore if $\ell\geq 2$,
\begin{equation}\label{ellnot1}
  \sum_{k=0}^{\ell-2} |x_{j(k+1)} - x_{j(k)}| \leq \frac{2+\delta_{\GG}}{\delta_{\GG}} \sum_{k=0}^{\ell-2} |a_{j(k+1)+1} - a_{j(k)+1}|
    = \frac{2+\delta_{\GG}}{\delta_{\GG}} |a_{j(\ell-1)-1} - a_1| \leq \frac{2+\delta_{\GG}}{\delta_{\GG}}|y-x|.
\end{equation}
Having $\ell\geq 2$ also ensures $|y-x| \geq |a_{j(1)} - a_{j(0)}| >\delta_{\GG}$ so 
\[
  |x_{j(\ell)} - x_{j(\ell-1)}| = |y - x_{j(\ell-1)}| \leq |y - a_{j(\ell-1)+1}| + |a_{j(\ell-1)+1} - x_{j(\ell-1)}| \leq |y-x|+1
    \leq \frac{1+\delta_{\GG}}{\delta_{\GG}}|y-x|
\]
which with \eqref{ellnot1} yields
\begin{equation}\label{dilate}
  \sum_{k=0}^{\ell-1} |x_{j(k+1)} - x_{j(k)}| \leq \frac{3+2\delta_{\GG}}{\delta_{\GG}}|y-x|.
\end{equation}
On the other hand, if $\ell=1$ then
\[
  \sum_{k=0}^{\ell-1} |x_{j(k+1)} - x_{j(k)}| = |y-x|
\]
so \eqref{dilate} still holds. This proves bounded dilation.
\end{proof}

\section{Straightness of geodesics and and regularity of means}
For $q>0$ and $x\in\RR^d$, let 
\begin{equation}\label{psiqF}
  \tred{\psi_q(x)} = \text{ the point of $q\ZZ^d$ closest to $x$ (with ties broken arbitrarily)}, \quad 
    \tred{F_y} = \psi_q^{-1}(y), y\in q\ZZ^d,
\end{equation}
the latter being a cube (ignoring the boundary.)  More generally for $u\in\RR^d$ we define $F_u$ to be the cube $F_y$ containing $u$, with some arbitrary rule for cube--boundary points.

The bound \eqref{expbound} applies to deterministic $x,y$; we cannot for example take $x,y\in \VV$.  Instead for random $x,y$ we can apply \eqref{expbound} to nearby points of $q\ZZ^d$ for some $q$, and use the following.  It is the only place the assumption A2(iv) of bounded dilation is used.

\begin{lemma}\label{neighbortimes}
Suppose $\GG=(\VV,\EE)$ and $\{\eta_e,e\in\EE\}$ satisfy A1, A2, and A3.  There exist constants $C_i$ as follows. 

(i) Let $r\geq 2$ and $t \geq C_{39}\log r$.  Then
\[
  P(\text{there exist $x,y \in B_r(0)\cap \VV$ with } T(x,y) \geq tr) \leq C_{40}e^{-C_{41}t}.
\]

(ii) For all $x,y\in \RR^d$,
\[
  ET(x,y) \leq C_{42}(|y-x|\vee 1).
\]
\end{lemma}

\begin{proof}
To prove (i) we condition on $\VV$.  Given $x,y \in B_r(0)\cap \VV$, by A1(iv) there exists a path $x=x_0,x_1,..,x_m=y$ in $\GG$ with 
\[
  \sum_{j=1}^m |x_j-x_{j-1}| \leq C_{18}|y-x|\leq 2C_{18}r.
\]
Writing $\eta_j$ for $\eta_{\langle x_{j-1}x_j\rangle}$, for $\lambda>0$ we have
\begin{align}\label{bestpath}
  P\Big(T(x,y)\geq t|y-x| \,\Big|\, \VV\Big) &\leq P\left( \sum_{j=1}^m |x_j-x_{j-1}|\eta_j \geq t |y-x| \,\Bigg|\, \VV \right) \notag\\
  &\leq e^{-\lambda t|y-x|} \prod_{j=1}^m \zeta\big(\lambda |x_j-x_{j-1}|\big) \notag\\
  &\leq e^{-\lambda t|y-x|} \zeta\big(C_{18}\lambda |y-x|\big),
\end{align}
so using A2(ii) it follows that for some $c_i$,
\begin{equation}\label{bestpath2}
  P\Big(T(x,y)\geq t|y-x| \,\Big|\, \VV\Big) \leq e^{-I_\eta(t/C_{18})} \leq c_0e^{-c_1t},
\end{equation}
where $\tred{I_\eta(t)} = \sup_{\gamma>0} (\gamma t - \log \zeta(\gamma))$ is the large-deviations rate function of the variables $\eta_e$.
From this, using A1(v),
\begin{align}\label{existxy}
  P(\text{there exist $x,y \in B_r(0)\cap \VV$ with } T(x,y) \geq tr) &\leq \sum_{n=2}^\infty 
    P(|B_r(0)\cap \VV|=n) {n \choose 2} c_1e^{-c_1t/2} \notag\\
  &\leq c_1E|B_r(0)\cap \VV|^2 e^{-c_1t/2} \notag\\
  &\leq c_2r^4 e^{-c_2t/2} \notag\\
  &\leq c_2e^{-c_3t},
\end{align}
proving (i).  

To prove (ii) we apply \eqref{bestpath2} to $\varphi(x)$ and $\varphi(y)$, and use the fact that $|\varphi(x)-x|\leq 1$ for all $x$.
\end{proof}

Building on Lemma \ref{neighbortimes} we obtain the following.

\begin{lemma}\label{connect}
Suppose $\GG=(\VV,\EE)$ and $\{\eta_e,e\in\EE\}$ satisfy A1, A2, and A3.  There exist constants $C_i$ as follows. For all $r\geq 2, u,v\in\RR^d$ and $t>0$ with $\sigma(|u-v|)\geq r$ and $t\sigma(|u-v|)\geq C_{43}r\log r$,
\begin{align}\label{nearby}
  P\Big(&\text{there exist $x\in B_r(u)\cap \VV, y\in B_r(v)\cap \VV$ with } 
    |T(x,y) - ET(u,v)| \geq t\sigma(|u-v|)\Big) \notag\\
  &\qquad \leq C_{44}e^{-C_{45}t}.
\end{align}
\end{lemma}

\begin{proof}
Using Lemma \ref{neighbortimes}(i) and \eqref{expbound} we have
\begin{align}\label{nearby2}
  P\Big(&\text{there exist $x\in B_r(u)\cap \VV, y\in B_r(v)\cap \VV$ with } |T(x,y) - ET(u,v)| \geq t\sigma(|u-v|)\Big) \notag\\
  &\quad \leq P\Big(\text{there exists $x\in B_r(u)\cap \VV$ with } T(x,u) \geq t\sigma(|u-v|)/3\Big) \notag\\
  &\quad \qquad + P\Big(\text{there exists $y\in B_r(v)\cap \VV$ with } T(y,v) \geq t\sigma(|u-v|)/3\Big) \notag\\
  &\quad \qquad + P\Big(|T(u,v) - ET(u,v)| \geq t\sigma(|u-v|)/3\Big) \notag\\
  &\quad \leq 2C_{40}e^{-C_{41}t\sigma(|u-v|)/3r} + C_{24}e^{-C_{25}t/3} \notag\\
  &\quad \leq c_0e^{-c_1t }.
\end{align}
\end{proof}

We write
\[
  \tred{x\to y_1\to\cdots\to y_k\to y}\quad\text{in } \Gamma_{xy}
\]
to denote that the vertices $y_i\in \VV$ appear in the order $y_1,\dots,y_k$ in traversing $\Gamma_{xy}$ from $x$ to $y$. For $a$ preceding $b$ in $\Gamma_{xy}$ we write $\tred{\Gamma_{xy}[a,b]}$ for the segment of $\Gamma_{xy}$ from $a$ to $b$.  (Here we do not require $a,b\in \VV$.)  For $v$ in a geodesic $\Gamma_{xy}$ and $0<s<|v-x|$, let $u$ be
the first vertex in $\Gamma_{xy}\cap \VV$ before $v$ satisfying $\Gamma_{uv} \subset B_s(v)$.  We then call $\Gamma_{uv}$ the \emph{trailing s-segment} of $v$ in $\Gamma_{xy}$. Note that by \eqref{separation}, $s \geq |u-v|\geq s-2$. Define the hyperplanes, slabs, and halfspaces
\[
  \tred{H_s} = \{(x_1,x^*)\in\RR^d: x_1 = s\}, \quad \tred{H_{[r,s]}} = \{(x_1,x^*)\in\RR^d:r\leq x_1\leq s\},
\]
\[
  \tred{H_s^+} = \{(x_1,x^*)\in\RR^d: x_1\geq s\}, \quad \tred{H_s^-} = \{(x_1,x^*)\in\RR^d: x_1\leq s\}.
\]

We turn next to a weaker version of Theorem \ref{ghbound} similar to \eqref{KAnonrand}; we will need it on the way to the proof of Theorem \ref{ghbound}. The proof of \eqref{KAnonrand} in \cite{Al97} does not carry over immediately to the present situation, as it uses translation invariance of the lattice.  But 
we can use radial symmetry here to give a distinctly shorter proof than that of \eqref{KAnonrand} in \cite{Al97}.

Recall $h(r)-\mu r$ is nonnegative by subadditivity of $h$.  For technical convenience we assume $\beta>4$ in A1(ii).

\begin{proposition}\label{hmu}
Suppose $\GG=(\VV,\EE)$ and $\{\eta_e,e\in\EE\}$ satisfy A1, A2, and A3.  
There exists $C_{46}$ such that for all $r\geq 2$,
\begin{equation}\label{hmu2}
  \mu r \leq h(r) \leq \mu r + C_{46}\sigma_r\log r.
\end{equation}
\end{proposition}

\begin{proof}
It is sufficient to prove the bound for all sufficiently large $r$, so we will tacitly assume $r$ is large, as needed.

We consider the geodesic $\Gamma_{0,nre_1}$ for a fixed large $n$.  Let $\beta$ be as in A1(iii).  For $0\leq j\leq n-1$ let $\tred{v_j}$ be the first vertex $v$ in $\Gamma_{0,nre_1}\cap \VV$ with the property that the trailing $(r-2\beta)$-segment of $v$ in $\Gamma_{0,nre_1}$ is contained in $H_{jr}^+$, and let $\tred{u_j}$ 
be the starting point of this trailing segment. From \eqref{separation} it is easy to see that we must have $jr\leq (u_j)_1\leq jr+2$. It follows that $\Gamma_{0,nre_1}[u_j,v_j] \subset B_{r-2\beta}(v_j) \cap H_{jr}^+ \cap H_{(j+1)r-2\beta+2}^-$; we denote this last region by $\tred{W_{v_j}}$.  Note that any two sets $W_{v_j}$ are separated by distance at least $2\beta-2$.

We need to control the entropy of the collection of pairs $\{(u_j,v_j): 0\leq j<n\}$.  To do this, we enlarge this collection in such a way that we can put a natural tree structure on it.  To that end, let $\tred{v_n}$  be the first vertex in $\Gamma_{0,nre_1}\cap \VV$, if one exists, which is at distance at least $r$ from $\cup_{0\leq j<n} W_{v_j}$. Then let $\tred{u_n}\in \VV$ be such that $\Gamma_{0,nre_1}[u_n,v_n]$ is the trailing $(r-2\beta)$-segment of $v_n$ in $\Gamma_{0,nre_1}$, and let $W_{v_n}=B_{r-2\beta}(v_n)$, which contains $\Gamma_{0,nre_1}[u_n,v_n]$.  We repeat this to obtain \tred{$(u_{n+1},v_{n+1}),\dots,(u_N,v_N)$}, stopping when no $v_{N+1}$ exists.  Note this preserves the property of separation by distance at least $2\beta-2$.  Also, each new $v_j$ is within distance $2r-2\beta+2$ of some already-existing $v_i$.

We define the discrete approximations $\tred{\hat u_j} = \psi_{1/\sqrt{d}}(u_j), \tred{\hat v_j} =\psi_{1/\sqrt{d}}(v_j)$, and then let 
\[
  \tred{\hat W_j} = \begin{cases} B_{r-2\beta+1}(\hat v_j) \cap 
    H_{jr-1}^+ \cap H_{(j+1)r-2\beta+3}^- &\text{if } 0\leq j < n,\\ 
  B_{r-2\beta+1}(\hat v_j)  &\text{if } n \leq j\leq N. \end{cases}
\]
which contains $W_{v_j}$; the sets $\hat W_j, 0\leq j\leq N$ are separated from each other by distance at least $2\beta-6$.  We call $\{(\hat u_j,\hat v_j): 0\leq j<n\}$ \emph{primary pairs}, and $\{(\hat u_j,\hat v_j): n\leq j\leq N\}$ \emph{secondary pairs}. Since $r-2\beta \geq |u_j-v_j| \geq r-2\beta-2$, we have 
\begin{equation}\label{gapsize}
  r > r-2\beta+2 \geq |\hat u_j - \hat v_j| \geq r-2\beta-4> r-3\beta.
  \end{equation}
We now make a graph with vertices $\{(\hat u_j,\hat v_j): 0\leq j\leq N\}$ by placing an edge between the ith and jth pairs if $|\hat v_j-\hat v_i|\leq 4r$.  The construction ensures that the resulting graph is connected, and it is easy to see that the disjointness of the sets $\hat W_j$ means the number of neighbors of any pair is bounded by some $c_0$.  We label $(u_0,v_0)$ as the root, and by some arbitrary algorithm, we take a spanning tree of the graph, which we denote $\tred{\mT(\Gamma_{0,nre_1})}$.  For counting purposes, we view two such trees as the same if they have the same pairs $\{(\hat u_j,\hat v_j): 0\leq j\leq N\}$, and the same set of primary pairs. We define \emph{parents} and \emph{offspring} in this rooted tree in the usual way: for a given pair $(\hat u_j,\hat v_j)$, its parent is the first pair after $(\hat u_j,\hat v_j)$ in the unique path from $(\hat u_j,\hat v_j)$ to the root, and its offspring are those pairs having $(\hat u_j,\hat v_j)$ as parent.

The tree $\mT(\Gamma_{0,nre_1})$ determines what we will call an abstract tree, in which all that is specified is the number of offspring of the root, then the number of offspring of each of these offspring, etc. The number of possible abstract trees here with $N+1$ vertices is at most $c_0^N$, and, provided $r$ is large, for each such abstract tree the number of corresponding actual trees $\mT(\Gamma_{0,nre_1})$ is at most $((8r)^d)^{2N+2}$, so the number of trees $\mT(\Gamma_{0,nre_1})$ consisting of $N+1$ pairs is at most $(c_1r)^{2d(N+1)}$.

We now consider all possible trees $\mT_0$, with vertices $\{(\hat u_j,\hat v_j): 0\leq j\leq N\}$, and with $\{(\hat u_j,\hat v_j): 0\leq j<n\}$ primary. Let \tred{$v(\mT_0)$} denote the number of vertices in the tree $\mT_0$.  We have 
\begin{align}\label{treedecomp}
  P\Big(T(0,nre_1) \leq n(\mu r&+1) \Big) = \sum_{N\geq n}\, \sum_{\mT_0: v(\mT_0)=N} 
    P\Big(T(0,nre_1) \leq n(\mu r+1), \mT(\Gamma_{0,nre_1}) = \mT_0\Big) \notag\\
  &\leq \sum_{N\geq n} (c_1r)^{2d(N+1)} \max_{\mT_0: v(\mT_0)=N} 
    P\Big(T(0,nre_1) \leq n(\mu r+1), \mT(\Gamma_{0,nre_1}) = \mT_0\Big)
\end{align}
We proceed by contradiction: we want to show that for some $C_{46}$, if $h(r) > \mu r + C_{46}\sigma_r\log r$ then the right side of \eqref{treedecomp} approaches 0 as $n\to\infty$, which means $\limsup_n T(0,nre_1)/n \geq \mu r+1$ a.s.~, contradicting the definition of $\mu$.  

Thus suppose 
\begin{equation}\label{assump}
  h(r) > \mu r + 2C_{46}\sigma_r\log r,
  \end{equation}
with $C_{46}$ to be specified, and fix $\mT_0$ with vertices $\{(\hat u_j,\hat v_j): 0\leq j\leq N\}$ with $N\geq n$.  For $u,v\in \VV\cap \hat W_j$ let 
\[
  \tred{T_{\hat W_j}(u,v)} = \inf\{T(\Gamma): \Gamma \text{ is a path from $u$ to $v$ in } \GG, \Gamma\subset \hat W_j\},
\]
taking the value $\infty$ when there is no such path. 
$T_{\hat W_j}(u,v)$ is determined by the configuration in 
\[
  \tred{\hat W_j^1} = \{u\in\RR^2: d(u,\hat W_j) < 1\}.
\]
There must exist $u_j,v_j\in \hat W_j\cap \VV$ satisfying $|u_j-\hat u_j|\leq 1, |v_j - \hat v_j|\leq 1, \Gamma_{u_j,v_j}\subset \Gamma_{0,nre_1} \cap \hat W_j$ and 
\[
  \sum_{j=0}^N T_{\hat W_j}(u_j,v_j) \leq T(0,nre_1).
\]
It follows that, letting 
\[
  \tred{\hat T_j} = \inf\{ T_{\hat W_j}(u_j,v_j): u_j,v_j\in \VV, |u_j-\hat u_j|\leq 1, |v_j - \hat v_j|\leq 1 \},
\]
the $\hat T_j$'s are independent (since the sets $\hat W_j^1$ are separated from each other by distance more than $2\beta-8>\beta$) and satisfy
\begin{equation}\label{Tjsum}
  \sum_{j=0}^N \hat T_j \leq T(0,nre_1).
\end{equation}
Each $\hat T_j$ is stochastically larger than
\[
  \tred{T_j} = \inf\{ T(u_j,v_j): u_j,v_j\in \VV, |u_j-\hat u_j|\leq 1, |v_j - \hat v_j|\leq 1 \}
\]

From \eqref{powerlike} and \eqref{gapsize} we have for some $c_2$ that $\sigma(|\hat u_j-\hat v_j|) \leq c_2\sigma_r.$  From \eqref{gapsize}, \eqref{assump},  and subadditivity we also have
\[
  ET(\hat u_j,\hat v_j) \geq h(r) - h(r - |\hat u_j-\hat v_j|) \geq h(r) - c_3 
     \geq \mu r + 2C_{46}\sigma_r\log r - c_3 \geq \mu r + C_{46}\sigma_r\log r.
\]
Hence from Lemma \ref{connect} we obtain that provided $r$ is large, for all $t\geq 1$,
\begin{align}\label{Tjtail}
  P\Big(T_j \leq \mu r + C_{46}\sigma_r\log r - t\sigma_r\Big) &\leq P\Big(T_j \leq ET(\hat u_j,\hat v_j) - t\sigma_r\Big) \notag\\
  &\leq P\Big(T_j \leq ET(\hat u_j,\hat v_j) - c_2^{-1}t\sigma(|\hat u_j-\hat v_j|)\Big) \notag \\
   &\leq C_{44}e^{-C_{45}t/c_2}.
\end{align}
By increasing $C_{44}$ we may make this valid for all $t>0$.  We then have for $\tred{\lambda} = C_{45}/2c_2\sigma_r$ and $\tred{M}=\mu r + C_{46}\sigma_r\log r$,
\begin{align}
  Ee^{-\lambda \hat T_j} &\leq Ee^{-\lambda T_j} \notag \\
  &= \int_0^\infty P(T_j \leq x) \lambda e^{-\lambda x}\ dx \notag\\
  &= e^{-\lambda M} \int_{-\infty}^{M/\sigma_r} P(T_j \leq M-t\sigma_r) \lambda\sigma_r e^{\lambda\sigma_r t}\ dt \notag\\
  &\leq e^{-\lambda M} \left[ \int_{-\infty}^0 \lambda\sigma_r e^{\lambda\sigma_r t}\ dt 
    + \int_0^{M/\sigma_r} C_{44}e^{-C_{45}t/c_2} \lambda\sigma_r e^{\lambda\sigma_r t}\ dt \right] \notag\\
  &\leq (1 + C_{44} ) e^{-\lambda M}.
\end{align}
Recalling \eqref{treedecomp} and \eqref{Tjsum} we then have
\begin{align}\label{onetree}
  P\Big(T(0,nre_1) \leq n(\mu r+1), \mT(\Gamma_{0,nre_1}) = \mT_0\Big) &\leq P\left( \sum_{j=0}^N \hat T_j \leq n(\mu r+1) \right) \notag\\
  &\leq e^{\lambda (\mu r+1)n} (1 + C_{44} )^{N+1} e^{-\lambda M(N+1)},
\end{align}
so by \eqref{treedecomp}, provided $r$ (and hence $\lambda M$) and $C_{46}$ (in \eqref{assump}) are large,
\begin{align}\label{treedecomp2}
  P\Big(T(0,nre_1) \leq n(\mu r+1) \Big) &\leq 
    \sum_{N\geq n} (c_1r)^{2d(N+1)} e^{\lambda (\mu r+1)n} (1 + C_{44} )^{N+1} e^{-\lambda M(N+1)} \notag\\
  &\leq 2(c_1r)^{2d(n+1)} e^{\lambda (\mu r+1)n} (1 + C_{44} )^{n+1} e^{-\lambda M(n+1)} \notag\\
  &\leq e^{-n\log r}.
\end{align}
As we have noted, since this approaches 0 as $n\to\infty$, it contradicts the fact that $T(0,nre_1)/n\to \mu r$ a.s.  Thus \eqref{assump} must be false.
\end{proof}

We need to use a result from \cite{Al20} to the effect that ``geodesics are very straight.'' It is proved there for FPP on a lattice, but the proof goes through essentially unchanged for the present context.  The heuristics are as follows: suppose the geodesic $\Gamma_{0,re_1}$ passes through a vertex $u=(u_1,u^*)$ at distance $s$ from $\Pi_{0,re_1}$; by symmetry we may suppose $u\in H_{r/2}^-$.  The geodesic then travels a corresponding extra distance $g(u) + g(re_1-u) - g(re_1)$.  If the angle between $u$ and $re_1$ is small, this extra distance is of order $|u^*|^2/u_1$, and from \eqref{expbound2}, the cost of this (meaning log of the probability) is of order $|u^*|^2/u_1\sigma(u_1)$.  If instead the angle between $u$ and $re_1$ is not small, the extra distance is of order $|u|$ and the cost is of order $|u|/\sigma(|u|)$.  We can combine these into a single statement by saying the cost for general $u$ should be whichever of these two costs is smaller, at least for $u\in H_{[0,r]}$.

The exact formulation of the straightness result contains extra log factors relative to the preceding heuristic, due to the need to bound the probability for all $u$ simultaneously.  It is as follows, using the constants $C_i,\chi_i$ of \eqref{powerlike}.  Define $\sigma^*(s)$ and $\Phi(s)$ by
\[
    \tred{\Phi(s)} = \frac{s}{C_{23}\tred{\sigma^*(s)}\log(2+s)} 
    = \frac{s^{\chi_2}}{C_{23}} \sup_{t\leq s} \frac{t^{1-\chi_2}}{\sigma_t\log(2+t)}.
\]
Here factoring out a power of $s$  on the right, and the use of the sup, ensure that $\Phi$ is strictly increasing. Note that by \eqref{powerlike} we have
\begin{equation}\label{sigmaineq}
  C_{23}^{-1}\sigma_s \leq \sigma^*(s) \leq \sigma_s.
\end{equation}
Then define 
\begin{equation}\label{Ddef}
  \tred{\Xi(s)} = (s\sigma(s) \log(2+s))^{1/2}, \quad
  \tred{D(u)} = \begin{cases} \min\left( \frac{|u^*|^2}{\Xi(u_1)^2}, \Phi\left(\max\big(|u_1|,|u^*|\big)\right) \right) 
    &\text{if } u_1\geq 0,\\
  \Phi\left(\max\big(|u_1|,|u^*|\big)\right) &\text{if } u_1<0. \end{cases}
\end{equation}
Note that by \eqref{powerlike} and \eqref{sigmaineq}, for large $s$,
\begin{equation}\label{XiPhi}
  \frac{1}{C_{23}\Phi(s)} \leq \left( \frac{\Xi(s)}{s} \right)^2 \leq \frac{1}{\Phi(s)}.
\end{equation}
The ``min'' in the definition of $D$ is in accordance with our heuristic: from \cite{Al20},  we have for large $|u|$ that for $C_{23}$ from \eqref{powerlike},
\begin{equation}\label{minwhich}
  D(u) = \begin{cases} \Phi\left(\max\big(|u_1|,|u^*|\big)\right) &\text{if } |u^*|\geq u_1,\\ 
    \frac{|u^*|^2}{\Xi(u_1)^2} &\text{if } |u^*|\leq C_{23}^{-1/2}u_1. \end{cases}
\end{equation}
Finally, define the symmetric version of $D$:
\begin{equation}\label{Dsym}
  \tred{D_r(u)} = \begin{cases} D(u) &\text{if } u_1\leq \frac r2\\ D(re_1-u) &\text{if } u_1 > \frac r2. \end{cases}
\end{equation}
This makes the right half of the region $\{u:D_r(u)\leq c\}$ the mirror image of the left half; this region is a ``tube'' (narrower near the ends) surrounding the line from 0 to $re_1$ bounded by the shell $\{u:|u^*|=c^{1/2}\Xi(u_1)\}$, augmented by a cylinder of radius $\Phi^{-1}(c)$ and length $2\Phi^{-1}(c)$ around each endpoint, so we will call it a \emph{tube--and--cylinders region}. 

We will also consider tube--and--cylinders regions around general pairs $u,v$ in place of $0,re_1$.  To that end, let $\tred{\Theta_{uv}}:\RR^2\to\RR^2$ be translation by $-u$ followed by some unitary transformation which takes $v-u$ to the positive horizontal axis, so that $\Theta_{uv}(u)=0,\Theta_{uv}(v)=|v-u|e_1$.  (The particular choice of unitary transformation does not matter.)  Then $\Theta_{uv}^{-1}(\{w:D_{|v-u|}(w)\leq c\})$ is a tube--and--cylinders region containing $\Pi_{uv}$.

The proof of the acceptable--random--graphs version of the straightness bound is little changed from the lattice-FPP version in \cite{Al20}; we can readily use Lemma \ref{neighbortimes}(i) to change the result from ``point--to--point'' (say, 0 to $re_1$) to ``ball--to--ball,'' with a sup over $x$ and $y$.  We omit the details.

\begin{proposition}\label{transfluct2}
Suppose $\GG=(\VV,\EE)$ and $\{\eta_e,e\in\EE\}$ satisfy A1, A2, and A3.  There exist constants $C_i$ as follows. For all $r,t>0$,
\begin{equation}\label{transfluct1}
  P\left( \sup_{x\in B_1(0)\cap \VV,\ y\in B_1(re_1)\cap \VV}\ \sup_{u\in\Gamma_{xy}} D_r(u) \geq t \right) 
    \leq C_{47}e^{-C_{48}t\log t}.
\end{equation}
\end{proposition}

We next bound transverse increments of passage times from 0, that is, increments which are (approximately) along the boundary of a ball of large radius $r$, over distances $\ll \Delta_r$. 
The following lattice-FPP result from \cite{Al20} carries over straightforwardly to the present context with the help of Lemma \ref{neighbortimes}(i), and again we omit details.  Recalling the cubes $F_u$ of side $q$ define
\[
  \tred{\hT(u,v)} = \min\{T(y,z): y\in F_u,z\in F_v\}, \quad u,v\in \RR^d.
\]

\begin{proposition}\label{transTincr}
Suppose $\GG=(\VV,\EE)$ and $\{\eta_e,e\in\EE\}$ satisfy A1, A2, and A3.  There exist constants $C_i$ as follows. For all $u,v\in q\ZZ^d$ with 
\begin{equation}\label{uvassump}
  |u|\geq C_{49}, \quad |g(u)-g(v)|\leq C_{50}, \quad\text{and}\quad 3\leq |u-v|\leq C_{51}\Delta(|u|),
\end{equation}
and all $\lambda\geq C_{52}$, we have
\begin{equation}\label{Tchange}
  P\Big(\hT(v,0) - \hT(u,0) \geq \lambda\sigma(\Delta^{-1}(|u-v|))\log |u-v| \Big) \leq C_{53}e^{-C_{54}\lambda\log |u-v|}.
\end{equation}
\end{proposition}

We next prove a seemingly obvious fact:  $h(r)=ET(0,re_1)$ is approximately increasing in $r$.  

\begin{lemma}\label{monotoneE}
Suppose $\GG=(\VV,\EE)$ and $\{\eta_e,e\in\EE\}$ satisfy A1, A2, and A3.  Given $0<\ep<1$ there exist constants $C_i$ as follows. 

(i) Let $\tred{q}>(1+\chi)/(1-\chi)$. Let $r,s>0$ satisfying $\min(r,s) \geq C_{55}\ep^{-q}$.  Then
\begin{equation}\label{monotoneE2}
  h(r+s) \geq h(r) + (1-\ep)h(s).
\end{equation}
Here $C_{55}=C_{55}(q)$.

(ii) For all $r,s\geq 0$,
\begin{equation}\label{monotoneE3}
  h(r+s) \geq h(r) + (1-\ep)h(s) - C_{56}.
\end{equation}
\end{lemma}

\begin{remark}\label{monoh}
Lemma \ref{monotoneE}(i) can be reformulated as follows: given $\gamma>\chi/\xi$ there exist $C_i(\gamma)$ such that 
\[
  h(r+s) \geq h(r) + h(s) - C_{57}s^\gamma \quad\text{for all } r\geq s \geq C_{58}.
\]
Also, from Proposition \ref{hmu} we have for some $C_{59}$
\[
  h(r+s) \geq \mu(r+s) \geq h(r) + h(s) - 2C_{46}\sigma_r\log r  \quad\text{for all } r\geq s \geq 2,
\]
but when $s\leq \sigma_r\log r$ this does not yield \eqref{monotoneE2}.
\end{remark}

\begin{proof}[Proof of Lemma \ref{monotoneE}.]
We prove (i), then obtain (ii) as a straightforward consequence. The idea is to show that $\Gamma_{0,(r+s)e_1}$ must with high probability approach $(r+s)e_1$ approximately horizontally, which forces $T(0,(r+s)e_1) - T(0,re_1)$ to be near $h(s)$ with high probability; a non-horizontal approach would force the sup in \eqref{transfluct1} to be large.  

Suppose \eqref{monotoneE2} holds under the added condition $s\leq r/4$.  Then for $s>r/4$ we can take $n$ with $s/n \leq r/4 < s/(n-1)$ and see that the hypotheses are satisfied with $s/n$ in place of $s$.  (This may require increasing $C_{55}$, but without dependence on $n$.)  Applying \eqref{monotoneE2} $n$ times then yields
\[
  h(r+s) = h\left( r + n\frac sn \right) \geq h(r) + n(1-\ep)h\left( \frac sn \right) \geq h(r) + (1-\ep)h(s).
\]
Therefore it is sufficient to prove the lemma for $s\leq r/4$. It is also sufficient to consider $0<\ep<\ep_0$ for any fixed $\ep_0=\ep_0(q)>0$.

With $c_0$ to be specified, define
\begin{equation}\label{mdef}
  \tred{m} = \inf\left\{u>0: \sigma_u\log u > \frac{\ep\mu s}{16c_0} \right\} \wedge \frac{r}{2}, \quad \tred{t}=\ep\Phi(m)^{1/2},
\end{equation}
\[
  \tred{S_1} = \{w\in\RR^d: r+s-m \leq w_1 \leq r+s-m+2, D_{r+s}(w) \leq t\},
\]
\[
  \tred{S_2} = \{w\in\RR^d: r+s-m \leq w_1 \leq r+s-m+2, |w^*| \leq \ep m\}.
\]
Note that provided $C_{55}$ (and hence $s$ and $m$) is large, which we henceforth tacitly assume, and provided $\ep_0$ is small, we have $\Phi^{-1}(t)<m/2$ and $m\geq 2s$.
Suppose $\tred{w}\in S_1$. We have $|w_1-(r+s)| \geq m-2>\Phi^{-1}(t)$ (meaning $w$ lies to the left of the cylinder around $(r+s)e_1$ in the tube-and-cylinders region $\{u:D_{r+s}(u)\leq t\}$) and so by \eqref{XiPhi},
\begin{equation}\label{w2close}
  \frac{|w^*|}{m} \leq \frac{t^{1/2}\Xi(m)}{m} \leq \frac{t}{\Phi(m)^{1/2}} = \ep.
\end{equation}
Thus $S_1\subset S_2$.  
We let $\tred{\hat w} = (r+s-m,w^*)$, so $|w-\hat w| \leq 2$ for $w\in S_2$. See Figure \ref{Lemma3-6}. Then for such $w$, recalling $m\geq 2s$, we have
\[
  |\hat w - (r+s)e_1| - |\hat w-re_1| = (m^2 + |w^*|^2)^{1/2} 
    - ((m-s)^2 + |w^*|^2)^{1/2} \geq \frac{sm}{(m^2 + |w^*|^2)^{1/2}} \geq (1-\ep^2)s,
\]
so assuming $C_{55}$ is large and $\ep_0$ is small,
\begin{equation}\label{wdist}
  |w-(r+s)e_1| - |w-re_1| \geq (1-\ep^2)s - 2 \geq \left( 1 - \frac{\ep}{4} \right) s.
\end{equation}
Also for $w\in S_2$, $m/2\leq m-s\leq m-2 \leq |w-re_1| \leq 2m$ and hence by \eqref{powerlike}, $\sigma(|w-re_1|) \leq 3C_{23}\sigma_{m-2}$.  It follows that provided we choose $c_0$ large enough in \eqref{mdef},
\[
  C_{46} \sigma(|w-re_1|) \log |w-re_1| \leq 4C_{23}C_{46}\sigma_{m-2} \log(m-2) \leq \frac{\ep}{12}\mu s.
\]
With \eqref{wdist} and Proposition \ref{hmu}, this yields that provided $C_{55}$ is large,
\begin{align}\label{hgap}
  h(|w-(r+s)e_1|) - h(|w-re_1|) &\geq \mu |w-(r+s)e_1| - \mu |w-re_1| - C_{46} \sigma(|w-re_1|) \log |w-re_1| \notag\\
  &\geq \left( 1 - \frac{\ep}{3} \right) \mu s \notag\\
  &\geq \left( 1 - \frac{\ep}{2} \right) h(s).
\end{align}

\begin{figure}
\includegraphics[width=16cm]{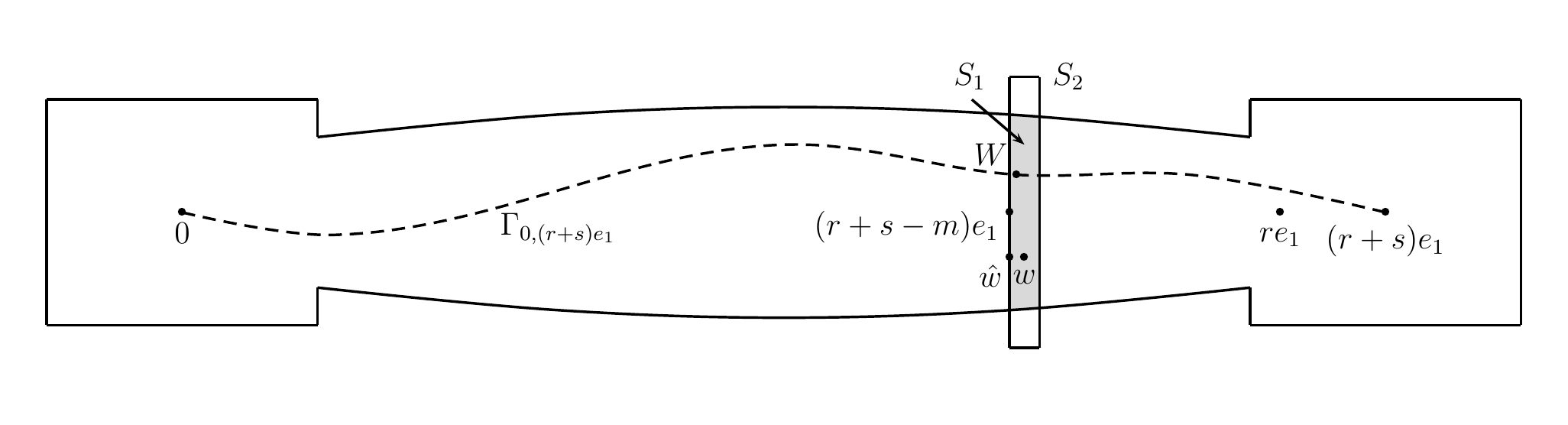}
\caption{ Diagram for the proof of Lemma \ref{monotoneE}. $S_1$ is shaded.}
\label{Lemma3-6}
\end{figure}

We need a lower bound for $t$.  Since $q>\tred{\hat q} = (1+\chi)/(1-\chi)$, we can choose $\tred{b}>0$ small enough so $(q-1)(1-b)^2>(1+b)(\hat q-1)$.
Provided $C_{55}$ is large enough (in the lemma statement, and depending on $b,q$), since $s \geq C_{55}\ep^{-q}$ we have
\begin{equation}\label{mep}
  m \geq (\ep s)^{(1-b)/\chi}, \quad \Phi(m) \geq m^{(1-b)(1-\chi)}, \quad t = \ep \Phi(m)^{1/2} 
    \geq c_1\ep^{1-(1-b)^2(q-1)(1-\chi)/2\chi} \geq c_1\ep^{-b}.
\end{equation}

Observe that for every vertex $w \in \Gamma_{0,(r+s)e_1} \cap \VV$ we have
\begin{equation}\label{0vsw}
  T(0,(r+s)e_1) - T(0,re_1) = T(0,w) + T(w,(r+s)e_1) - T(0,re_1) \geq T(w,(r+s)e_1) - T(w,re_1).
\end{equation}
Let
\[
  \tred{T^*} = \min\{ T(w,(r+s)e_1) - T(w,re_1): w\in S_2\cap V\}
\]
and let $\tred{W}$ be the first vertex in $\Gamma_{0,(r+s)e_1}\cap \VV$ satisfying $r+s-m \leq W_1 \leq r+s-m+2$.
If $W\in S_2$ then by \eqref{0vsw} and subadditivity we have
\[
  T(0,(r+s)e_1) - T(0,re_1) \geq T^*, \qquad |T(0,(r+s)e_1) - T(0,re_1)| \leq T(re_1,(r+s)e_1),
\]
\[
  |T^*| \leq T(re_1,(r+s)e_1).
\]
Hence
\begin{align}\label{Tlower}
  h(r+s) - h(r) &= E[T(0,(r+s)e_1) - T(0,re_1)] \notag\\
  &\geq E(T^*1_{\{W\in S_2\}}) - E(T(re_1,(r+s)e_1)1_{\{W\notin S_2\}}) \notag\\
  &\geq E(T^*) - 2E(T(re_1,(r+s)e_1)1_{\{W\notin S_2\}}).
\end{align}
It follows from \eqref{expbound} that for some $c_2$
\[
  E(T(re_1,(r+s)e_1)^2)^{1/2} \leq c_2s.
\]
Assuming $\ep_0$ is small, by Proposition \ref{transfluct2} and \eqref{mep} we have
\begin{equation}\label{Wloc}
  P(W\notin S_2) \leq P(W\notin S_1) \leq C_{47}e^{-C_{48}t\log t} 
    \leq C_{47}e^{-c_3\ep^{-b}} \leq \left( \frac{\ep\mu}{8c_2} \right)^2.
\end{equation}
Combining these yields
\begin{equation}\label{noS2}
  E(T(re_1,(r+s)e_1)1_{\{W\notin S_2\}}) \leq E(T(re_1,(r+s)e_1)^2)^{1/2} P(W\notin S_2)^{1/2}
    \leq \frac{\ep\mu s}{8}.
\end{equation}

To use \eqref{Tlower} we also need a lower bound for $E(T^*)$.  For $y>0$, using \eqref{hgap},
\begin{align}\label{Tstar}
  P&\left( T^* \leq \left( 1 - \frac{\ep}{2} \right)h(s) - 2y\sigma_m\log m \right) \notag\\
  &\qquad\leq P\Bigg( \text{for some } w\in S_2 \cap \VV,\ T(w,(r+s)e_1) - T(w,re_1) \leq \notag\\
  &\qquad\qquad\qquad\qquad h(|w-(r+s)e_1|) - h(|w-re_1|) - 2y\sigma_m\log m\Bigg) \notag\\
  &\qquad\leq P\Bigg( \text{for some } w\in S_2 \cap \VV,\ T(w,(r+s)e_1) - h(|w-(r+s)e_1|) 
    \leq - y\sigma_m\log m\Bigg) \notag\\
  &\qquad\qquad + P\Bigg( \text{for some } w\in S_2 \cap \VV,\ T(w,re_1) - h(|w-re_1|) \geq y\sigma_m\log m\Bigg).
\end{align}
We consider the first probability on the right in \eqref{Tstar}; the second probability is similar. For $w\in S_2$ we have using \eqref{w2close} that
\[
  m-2 \leq |w - (r+s)e_1| \leq \left( 1 + \ep^2 \right)m,
\]
and then from \eqref{powerlike},
\[
  \frac{\sigma_m}{\sigma(|w - (r+s)e_1|)} \geq \frac{C_{22}}{2}.
\]
From these and Lemma \ref{connect} (see $C_{45}$ there) we obtain that if $c_0$ is large enough, then for $y\geq c_0$,
\begin{align}\label{makesum}
   P&\Bigg( \text{for some } w\in S_2 \cap \VV,\ T(w,(r+s)e_1) - h(|w-(r+s)e_1|) \leq - y\sigma_m\log m\Bigg) \notag\\
   &\leq \sum_{\hw\in S_2\cap d^{-1/2} \ZZ^d} P\Bigg( \text{for some } w\in B_1(\hw) \cap \VV, \notag\\
   &\hskip 3.5cm T(w,(r+s)e_1) - h(|\hw-(r+s)e_1|) \leq h(|w-\hw|) - y\sigma_m\log m\Bigg) \notag\\
   &\leq \sum_{\hw\in S_2\cap d^{-1/2} \ZZ^d} P\Bigg( \text{for some } w\in B_1(\hw) \cap \VV,\ 
     T(w,(r+s)e_1) - h(|\hw-(r+s)e_1|) \leq -\frac y2 \sigma_m\log m\Bigg) \notag\\
   &\leq c_4m^{d-1}\exp\left( -\frac{C_{45}C_{22}y}{4}\log m \right) \notag\\
   &\leq \frac12 e^{-c_5y\log m}.
\end{align}
The same bound holds for the second probability on the right in \eqref{Tstar}. 
With \eqref{Tstar} and the definition of $m$ in \eqref{mdef} this shows that
\[
  ET^* \geq \left( 1 - \frac{\ep}{2} \right)h(s) - 4c_0\sigma_m\log m \geq \left( 1 - \frac{3\ep}{4} \right)h(s).
\]
Combining this with \eqref{Tlower} and \eqref{noS2} yields \eqref{monotoneE2}.

We now prove (ii).  From (i), there exists $c_6$ such that $h(r+s)\geq h(r)+(1-\ep)h(s)$ whenever $r,s\geq c_6$; there then exists $c_7$ such that $h(s)\leq c_7$ whenever $0\leq s\leq c_6$.  We therefore have
\[
  h(r+s) \geq \begin{cases} h(r)+(1-\ep)h(s) &\text{if } r,s \geq c_6 \\ h(r) - h(s) \geq h(r) + (1-\ep)h(s) - 2c_7 &\text{if } s<c_6 \\
    0\geq h(r) +(1-\ep)h(s) - 2c_7 &\text{if } r<c_6 \end{cases}
\]
which proves \eqref{monotoneE3}.
\end{proof}

\section{Proof of Theorem \ref{nofast}---downward deviations}
We use a multiscale argument which is related to chaining.   But first we dispense with simpler cases that only require Lemma \ref{connect} and Proposition \ref{transfluct2}.  The first such case is pairs $x,y$ which are close together.  For technical convenience later, we prove the following also for geodesics with endpoints in the set $G_r^+$, satisfying $G_r^+ \supset G_r(K)$ for $K$ fixed and $r$ large, given by
\begin{equation}\label{Grplus}
  \tred{G_r^+} = [-C_{60}(\log r)^{2/(1-\chi)},r+C_{60}(\log r)^{2/(1-\chi)}]\times C_{60}\Delta_r(\log r) \mkB_{d-1},
\end{equation}
with $C_{60}$ to be specified.

\begin{lemma}\label{shortones}
Suppose $\GG=(\VV,\EE)$ and $\{\eta_e,e\in\EE\}$ satisfy A1, A2, and A3.  There exist constants $C_i$
for all $r\geq 2,K\geq 1,$ and $t>C_{61}K^2$,
\begin{align}\label{Qrshort}
  P&\left( T(x,y) \leq ET(x,y) - t\sigma_r \text{ for some $x,y \in G_r(K)\cup G_r^+$ with $|y-x|
    \leq \frac{C_{62}r}{(\log r)^{1/\chi_1}}$} \right) \notag\\
  &\leq C_{63}e^{-C_{64}t\log r}.
\end{align}
\end{lemma}

\begin{proof}
We may assume $t$ is large. We first discretize: by Lemma \ref{neighbortimes}(ii), for $x\in B_1(\hat x),y\in B_1(\hat y)$, we have $|ET(x,y)-ET(\hat x,\hat y)|\leq 2C_{42}\leq t\sigma_r/2$, so for $c_0>0$,
\begin{align}\label{discrect}
  P&\left( T(x,y) \leq ET(x,y) - t\sigma_r \text{ for some $x,y \in G_r^+$ with $|y-x|
    \leq \frac{c_0r}{(\log r)^{1/\chi_1}}$} \right) \notag\\
  &\leq \sum_{\substack{\hat x,\hat y\in d^{-1/2}\ZZ^d\cap (G_r(K)\cup G_r^+) \\ |\hat x-\hat y|\leq 2c_0r/(\log r)^{1/\chi_1}}}
    P\left( T(x,y) \leq ET(\hat x,\hat y) - \frac{t\sigma_r}{2} \text{ for some $x\in B_1(\hat x),y\in B_1(\hat y)$} \right).
\end{align}
From \eqref{powerlike}, 
\[
  |\hat x-\hat y| \leq \frac{2c_0r}{(\log r)^{1/\chi_1}} \implies \frac{\sigma(r)}{\sigma(|\hat x-\hat y|)} 
    \geq \frac{c_1}{(2c_0)^{1/\chi_1}}\log r,
\]
so if we take $c_0$ small enough then by Lemma \ref{connect},
\begin{align}\label{discrect2}
  &\sum_{\substack{\hat x,\hat y\in d^{-1/2}\ZZ^d\cap (G_r(K)\cup G_r^+) \\ |\hat x-\hat y|\leq c_0r/2(\log r)^{1/\chi_1}}}
    P\left( T(x,y) \leq ET(\hat x,\hat y) - \frac{t\sigma_r}{2} \text{ for some $x\in B_1(\hat x),y\in B_1(\hat y)$} \right) \notag\\
  &\qquad\qquad\leq c_2r(K\Delta_r\log r)^{d-1} \max\left\{ e^{-C_{45}t\sigma_r/2\sigma(|\hat x-\hat y|)}:
     |\hat x-\hat y|\leq c_0r/2(\log r)^{1/\chi_1} \right\} \notag\\
  &\qquad\qquad\leq e^{-C_{45}t\log r},
\end{align}
which with \eqref{discrect} proves \eqref{Qrshort}.
\end{proof}

Lemma \ref{shortones} means we need only consider $x,y\in G_r(K)$ satisfying
\begin{equation}\label{xyfar}
  |y-x| > \frac{C_{62}r}{(\log r)^{1/\chi_1}}.
\end{equation}
Writing $\tred{\alpha_{uv}}$ for the angle between nonzero vectors $u,v\in\RR^d$, this means that for large $r$,
\begin{equation}\label{lowtilt}
  \alpha_{y-x,e_1} \leq \tred{\beta_r} := c_0 \frac{\Delta_r(\log r)^{1/\chi_1}}{r} \leq r^{-(1-\chi)/4}.
\end{equation}

A second simple case is small $r$.  For fixed $r_0$ and $1\leq r\leq r_0$, from Lemma \ref{neighbortimes} for all $x,y\in G_r$ we have $ET(x,y)\leq c_1r$, so for $t\geq c_2r_0$ we have $ET(x,y)-t\sigma_r \leq c_1r-t\sigma_r<0$, and hence the probability in \eqref{Qrunif2} is 0.  Therefore there exist $C_{29},C_{30}$ such that \eqref{Qrunif2} is valid for all $1\leq r\leq r_0$ and $t>0$.

A third simple case is $t\geq c_3\log(Kr)$, with $c_3$ large enough. As in \eqref{discrect} and \eqref{discrect2}, for $C_{45}$ from Lemma \ref{connect} we then have
\begin{align}
  P\left( T(x,y) \leq ET(x,y) - t\sigma_r \text{ for some $x,y \in G_r(K)$} \right)
    &\leq c_4r(K\Delta_r)^{d-1} \max_{u,v\in G_r(K)} e^{-C_{45}t\sigma_r/2\sigma(|u-v|)} \notag\\
  &\leq c_4r(K\Delta_r)^{d-1} e^{-c_5t} \notag\\
  &\leq c_6e^{-c_5t/2}.
\end{align}
It follows that, for $C_{26}$ from Theorem \ref{nofast}, we need only consider $C_{26}K^2 \leq t< c_3\log(Kr)$, and therefore also $K\leq c_7(\log r)^{1/2}$, which means $G_r(2K)\subset G_r^+$. 

A fourth simple case is pairs $x,y$ for which $\Gamma_{xy}$ goes well outside $G_r(K)$, when $r$ is large and $t<c_3\log(Kr)$.  Assuming $r$ is large, $C=2c_3$, from the third simple case, is the choice of interest in the following.

\begin{lemma}\label{bigwander}
Suppose $\GG=(\VV,\EE)$ and $\{\eta_e,e\in\EE\}$ satisfy A1, A2, and A3.  Then given $C>0$ there exist constants $C_i=C_i(C)$ such that for $c_7$ as above, for all $r\geq C_{65}, K\leq c_7(\log r)^{1/2}, C_{66}K^2\leq t\leq C\log(Kr)$ we have 
\[
  P\Big(\Gamma_{xy} \not\subset G_r^+ \text{ for some $x,y\in G_r(K)$ satisfying \eqref{xyfar}} \Big)
    \leq C_{68}e^{-C_{69}t}.
\]
\end{lemma}

\begin{proof}
Note the assumptions guarantee $G_r(K)\subset G_r^+$. 
Let $\tred{\hat x,\hat y}\in d^{-1/2}\ZZ^d\cap G_r(K)$ with $\hat x_1<\hat y_1$, and let $x,y\in G_r$ satisfy \eqref{xyfar}, with $|x-\hat x|<1,|y-\hat y|<1$. 
Recall the transformation $\Theta_{\hat x\hat y}$ which takes $\Pi_{\hat x\hat y}$ to $\Pi_{0,|\hat y-\hat x|e_1}$. 
Suppose $\tred{w}\in\Gamma_{\hat x\hat y}\bs G_r^+$ and let $\tred{\tilde w} = \Theta_{\hat x\hat y}w$; we may take $w$ with $d(w,G_r^+)\leq 2$.  We claim that, for $D_r$ from \eqref{Dsym}, we have $D_r(\tilde w) \geq C_{60}\log r$, with $C_{60}$ from the definition of $G_r^+$. We consider several cases.

{\it Case 1.} $-C_{60}(\log r)^{2/(1-\chi)} \leq w_1 \leq r+C_{60}(\log r)^{2/(1-\chi)}$, that is, $w$ is to the side of $G_r^+$. See Figure \ref{Lemma4-2}. From symmetry we may assume $-C_{60}(\log r)^{2/(1-\chi)} \leq w_1\leq r/2$. Note that in the definition \eqref{Ddef} of $D(u)$, the case $u_1\geq 0$ is the relevant one here if and only if $\tilde w_1 \in [0,|\hat y-\hat x|]$. 
We have
\begin{equation}\label{farPi}
  d(\tilde w,\Pi_{0,|\hat y-\hat x|e_1}) = d(w,\Pi_{\hat x\hat y}) \geq d(w,G_r(K)) \geq \frac{C_{60}}{2}\Delta_r\log r, \qquad |w^*| 
    \geq C_{60}\Delta_r\log r,
\end{equation}
and $\tilde w_1 \in [0,|\hat y-\hat x|] \implies d(\tilde w,\Pi_{0,|\hat y-\hat x|e_1}) = |\tilde w^*|$, so
\[
  \tilde w_1 \in \left[0,\frac{|\hat y-\hat x|}{2}\right] \implies \frac{|\tilde w^*|^2}{\Xi(\tilde w_1)^2}  
    \geq \frac{d(\tilde w,\Pi_{0,|\hat y-\hat x|e_1})^2}{\Xi\left(\frac 12 |\hat y-\hat x|\right)^2}
    \geq \frac{C_{60}^2\Delta_r^2(\log r)^2}{4\Xi(r)^2} \geq \frac{C_{60}^2}{5}\log r.
\]
Also as in \eqref{farPi},
\[
  \max(|\tilde w_1|,|\tilde w^*|) \geq \frac{1}{\sqrt{2}}|\tilde w| = \frac{1}{\sqrt{2}}|w-\hat x| 
    \geq \frac{1}{\sqrt{2}}d(w,\Pi_{\hat x\hat y})
    \geq \frac{C_{60}}{3}\Delta_r\log r
\]
so
\[
  \Phi\left( \max\big(|\tilde w_1|,|\tilde w^*|\big) \right) \geq C_{60}\log r.
\]
Therefore under Case 1,
\begin{equation}\label{minDr}
  -\infty < \tilde w_1 \leq \frac{|\hat y-\hat x|}{2} \implies D_{|\hat y-\hat x|}(\tilde w) \geq C_{60}\log r,
\end{equation}
and symmetrically the same holds for $\tilde w_1>|\hat y-\hat x|/2$.

{\it Case 2.} $w_1 < -C_{60}(\log r)^{2/(1-\chi)}$ with $\tilde w_1>0$, or $w_1 > r + C_{60}(\log r)^{2/(1-\chi)}$ with $\tilde w_1< |\hat y-\hat x|$, that is, $w$ is past the end of $G_r^+$ but $\tilde w$ is to the side of $\Pi_{0,|\hat y-\hat x|e_1}$.  See Figure \ref{Lemma4-2}. The two ends are symmetric so we need only consider $w_1 < -C_{60}(\log r)^{2/(1-\chi)}$ with $\tilde w_1>0$.  Let $\tred{H_{(\hat x)}}$ be the hyperplane through $\hat x$ perpendicular to $\hat y-\hat x$, and let $\tred{\ol w}$ be the orthogonal projection of $w$ into $H_{(\hat x)}$, so $|\ol w - \hat x| = |\tilde w^*|$ and $|w-\ol w| = |\tilde w_1|$.  Since $K\leq c_7(\log r)^{1/2}$, by \eqref{lowtilt} the angle $\alpha_{\hat y-\hat x,e_1}$ is small enough that we have $\ol w_1 < w_1 < 0$ and hence, for $\beta_r$ from \eqref{lowtilt},
\begin{equation}\label{lowtilt2}
  C_{60}(\log r)^{2/(1-\chi)} \leq \hat x_1 - \ol w_1 \leq |\ol w -\hat x|\sin \alpha_{\hat y-\hat x,e_1} \leq |\tilde w^*|\beta_r;
\end{equation}
therefore
\[
  \Phi\left( \max\big(|\tilde w_1|,|\tilde w^*|\big) \right) \geq \Phi\left( \frac{C_{60}(\log r)^{2/(1-\chi)}}{\beta_r} \right) 
    \geq C_{60}\log r.
\]
Further, similarly to \eqref{lowtilt2}, since the small angle $\alpha_{\hat y-\hat x,e_1}$ ensures $\frac 12 \tilde w_1 = \frac 12 |w-\ol w| \leq (w-\ol w)_1$, we have
\begin{equation}\label{smallangle}
  \frac 12 \tilde w_1 \leq (w-\ol w)_1 \leq (\hat x-\ol w)_1 = |\ol w -\hat x|\sin \alpha_{\hat y-\hat x,e_1} \leq |\tilde w^*|\beta_r
\end{equation}
and therefore using \eqref{powerlike},
\[
  \frac{|\tilde w^*|^2}{\Xi(\tilde w_1)^2} \geq \frac{|\tilde w^*|^2}{\Xi(2\beta_r|\tilde w^*|)^2}
    \geq \frac{C_{22}|\tilde w^*|}{\beta_r^{\chi_1}\sigma(|\tilde w^*|)\log(2+|\tilde w^*|)} \geq \frac{1}{\beta_r^{\chi_1}} \geq C_{60}\log r.
\]
This proves the right side of \eqref{minDr} under Case 2.  

\begin{figure}
\includegraphics[width=9cm]{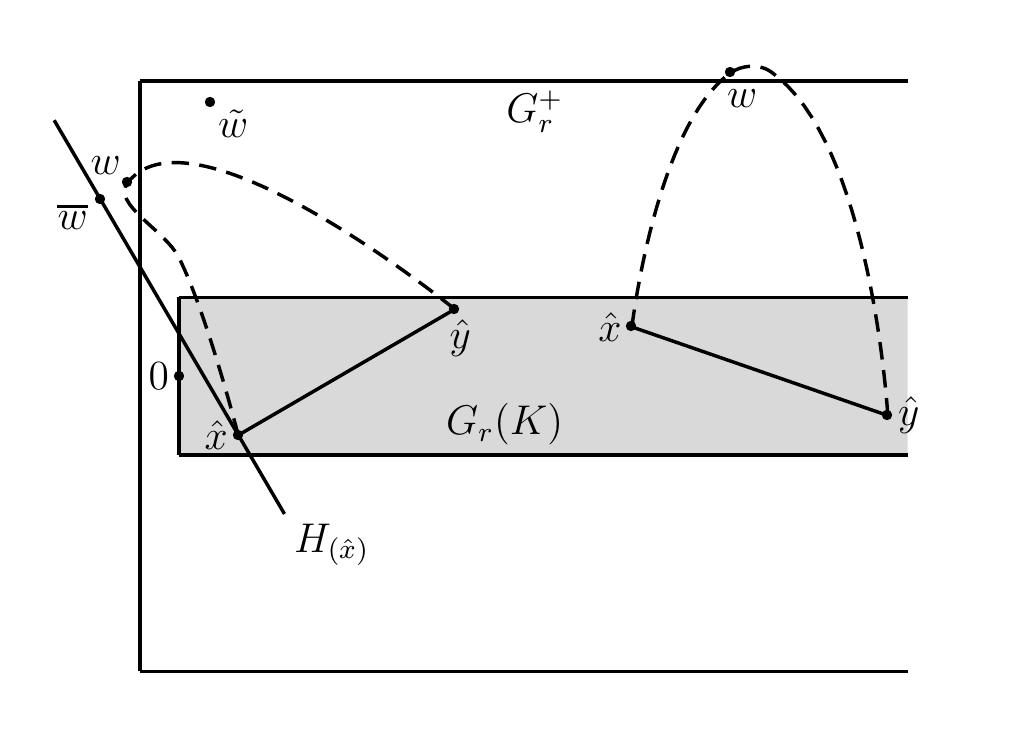}
\caption{ Diagram for Case 1 (right) and Case 2 (left) in the proof of Lemma \ref{bigwander}. On the left, if we translate and rotate the picture so that $\hat x$ becomes 0 and $\hat y - \hat x$ becomes horizontal, then $w$ becomes $\tilde w$. Case 3 is like Case 2 except that $\tilde w$ is to the left of 0.}
\label{Lemma4-2}
\end{figure}

{\it Case 3.} $w_1 < -C_{60}(\log r)^{2/(1-\chi)}$ with $\tilde w_1 \leq 0$, or $w_1 > r + C_{60}(\log r)^{2/(1-\chi)}$ with $\tilde w_1\geq |\hat y-\hat x|$, that is, $w$ is past the end of $G_r^+$ and $\tilde w$ is past the end of $\Pi_{0,|\hat y-\hat x|e_1}$.  The two ends are again symmetric so we need only consider $w_1 < -C_{60}(\log r)^{2/(1-\chi)}$ with $\tilde w_1 \leq 0$.  Then $|\tilde w| = |w-\hat x| \geq \hat x_1-w_1 \geq C_{60}(\log r)^{2/(1-\chi)}$ so
\[
  D_r(\tilde w) = \Phi\left( \max\big(|\tilde w_1|,|\tilde w^*|\big) \right) \geq \Phi\left( \frac{|\tilde w|}{2} \right) 
    \geq C_{60}\log r.
\]

Thus is all cases we have $D_r(\tilde w) \geq C_{60}\log r$,
so by Proposition \ref{transfluct2},
\begin{align*}
  P&\Big(\Gamma_{xy} \not\subset G_r^+ \text{ for some $x,y\in G_r(K)$ satisfying \eqref{xyfar}} \Big) \\
  &\leq \sum_{\hat x,\hat y\in d^{-1/2}\ZZ^d\cap G_r(K)} 
    P\Big(\Gamma_{xy} \not\subset G_r^+ \text{ for some $x,y\in G_r(K)$ satisfying \eqref{xyfar} with } 
    |x-\hat x|<1,|y-\hat y|<1 \Big) \\
  &\leq \sum_{\hat x,\hat y\in d^{-1/2}\ZZ^d\cap G_r(K)} P\left( \sup_{x\in B_1(\hat x)\cap\VV,y\in B_1(\hat y)\cap\VV}\ 
    \sup_{w \in \Gamma_{xy}} D_{|\hat y-\hat x|}(\Theta_{\hat x\hat y}w) \geq C_{60}\log r \right) \\
  &\leq c_0r(K\Delta_r)^{d-1} e^{-C_{47}(\log\log r)\log r} \\
  &\leq c_1e^{-c_2t}.
\end{align*}
\end{proof}

In summary, we may restrict to the following situation (still with $x,y\in G_r(K)$):
\begin{equation}\label{rtxycond}
  r \geq r_0,\quad K\leq c_7(\log r)^{1/2}, \quad c_8 \leq t \leq c_3\log r,\quad |y-x| > \frac{C_{62}r}{(\log r)^{1/\chi_1}},\quad
    \Gamma_{xy}\subset G_r^+,
\end{equation}
with $C_{62}$ from Lemma \ref{shortones}.

\subsection{Step 1. Setting up the coarse-graining.}\label{setup}
For purposes of coarse--graining and multiscale analysis of paths, we build grids inside $G_r^+$ on various scales, using small parameters \tred{$\lambda,\delta,\beta$} satsfying 
\begin{equation}\label{relsize}
  1 \gg \lambda \gg \delta^{\chi_1} \gg \delta^{(1+\chi_1)/2} \gg \beta
\end{equation}
in the sense that the ratio of each term to the one following must be taken sufficiently large, in a manner to be specified. We choose these so $1/\delta$ and $1/\beta$ are integers. There is also a fourth parameter $\tred{\rho}>1$, and we further require
\begin{equation}\label{relsize2}
  \frac{\beta}{\lambda\rho\delta^{(1+\chi_2)/2}} \ll 1,\quad \frac{\rho\beta}{\lambda\delta^{(1-\chi_1)/2}} \ll 1, \quad
    \frac{\beta^2}{\lambda^2\delta} \ll 1,\quad \frac{\beta^{2\chi_1/(1+\chi_1)}}{\lambda\delta^3} \ll 1,\quad
    \rho^2\lambda \gg 1;
\end{equation}
all of \eqref{relsize2} can be satisfied by taking $\beta$ small enough after choosing $\rho,\lambda,\delta$.
For $j\geq 1$, a \emph{jth--scale hyperplane} is one of form $H_{k\delta^jr}, k\in\ZZ$, and the \emph{jth--scale grid} in $G_r^+$ is
\[
  \LL_j = \tred{\LL_j(r)} = \left\{ u\in G_r^+: u_1\in \delta^j r\ZZ, u^* \in \tred{K_0}\beta^j\Delta_r\ZZ^{d-1} \right\},
\]
where $K_0\in[1,2]$ is to be specified.
Note that larger $j$ values correspond to smaller scales, and since $1/\delta$ is an integer, a $j$th--scale hyperplane for some $j$ is also a $k$th--scale hyperplane for $k>j$. A hyperplane is \emph{maximally jth--scale} if it is $j$th--scale but not $(j-1)$th--scale.  

The $j$th--scale grid divides a $j$th--scale hyperplane into cubes which we call \emph{jth--scale blocks}.  For concreteness we take these blocks to be products of left--open--right--closed intervals. Each point $u$ of the hyperplane then lies in a unique such block. 
For a point $u$ in a $j$th--scale hyperplane, the \emph{jth--scale coarse--grain approximation of} $u$ is the point $\tred{V_j(u)}$ which is the unique corner point in the block containing $u$. We abbreviate coarse--grain as CG.  The definition ensures that two points with the same $j$th--scale CG approximation also have the same $k$th--scale CG approximation for all larger scales $k<j$.  

A \emph{transverse step} in the $j$th--scale grid is a step from some $u\in\LL_j$ to some $v\in\LL_j$ satisfying $v_1=u_1, |v^*-u^*| = K_0\beta^j\Delta_r$; a \emph{longitudinal step} is from $u$ to $v$ satisfying $v_1=u_1+\delta^j r, v^*=u^*$.  Define
\[
  \ell_1 = \tred{\ell_1(j)} = \frac{\Delta(\delta^j r)}{\beta^j\Delta(r)}, \qquad \ell_2 = \tred{\ell_2(j)} = \rho^j\ell_1(j),
\]
so from \eqref{powerlike},
\begin{equation}\label{ell12}
  \ell_1 \geq \frac{\delta^{(1+\chi_2)j/2}}{C_{23}\beta^j}, \qquad  \frac{\delta^{(1+\chi_2)j/2}\rho^j}{C_{23}\beta^j}
    \leq \ell_2 \leq \frac{\delta^{(1+\chi_1)j/2}\rho^j}{C_{22}\beta^j}.
\end{equation}
Here $\ell_1(j)$ is chosen so that the typical transverse fluctuation $\Delta(\delta^j r)$ for a geodesic making one longitudinal step is $\ell_1(j)$ transverse steps.  

On short enough length scales, coarse--graining is unnecessary because we can use Proposition \ref{transTincr} and Lemma \ref{shortones}.  More precisely, we will need only consider $j\leq \tred{j_1}=j_1(r)$ where $j_1$ is the least $j$ for which 
\begin{equation}\label{j1}
  \left( \frac{\lambda}{\delta^{\chi_1}} \right)^{j_1} \geq (\log r)^2, \quad\text{so}\quad j_1(r) \asymp \log\log r.
\end{equation}
Provided $r$ is large, this means the spacings $\delta^{j_1}r$ and $K_0\beta^{j_1}\Delta_r$ of the $j_1$th--scale grid are large, and therefore
we can choose $q\in[4,5]$
so that $\delta^{j_1}r$ is an integer multiple of $q$, and then choose $K_0\in[1,2]$ such that $K_0\beta^{j_1}\Delta_r$ is an integer multiple of $q$.  Then (since $1/\delta$ and $1/\beta$ are integers) for all $j\leq j_1$, the $j$th--scale grid in every $j$th--scale hyperplane is contained in $q\ZZ^d$, which we call the \emph{basic grid}. 

We say an interval in $\RR$ has \emph{$k$th--scale length} if its length is between $10\delta^{k+1}r$ and $10\delta^kr$.

Given $x,y\in G_r(K)$ with $x_1<y_1$, satisfying \eqref{rtxycond}, we define a \emph{hyperplane collection} \tred{$\mH_{xy}$}, which depends on the geodesic $\Gamma_{xy}$, constructed inductively as follows.  
All hyperplanes $H_s\in\mH_{xy}$ have $s\in[x_1,y_1]$, and we view the hyperplanes as ordered by their indices. Let $\tred{j_2(x,y)}$ be the least $j$ such that there are at least 4 $j$th--scale hyperplanes between $x$ and $y$.  Subject always to the constraint $s\in[x_1,y_1]$, at scale $j_2$ we put in $\mH_{xy}$ the $j_2$th scale hyperplanes $H_s$ second closest to $x$ and to $y$; we call these \emph{$j_2$--terminal hyperplanes}.  The gap between the $j_2$--terminal hyperplanes is at least $4\delta^{j_2}r$ and at most $5\delta^{j_2-1}r$.
In general, when we have chosen the $j$th--scale hyperplanes in $\mH_{x,y}$ for some $j\geq j_2$, each gap between consecutive ones is called a \emph{$j$th--scale interval}, and the hyperplanes bounding it are the \emph{endpoint hyperplanes} of the interval.  For any interval $I$ we also call $\{H_s:s\in I\}$ an interval; which meaning should be clear from the context.  A $j$th--scale interval is \emph{short} if it has $j$th--scale length, and \emph{long} otherwise.  We then add $(j+1)$th--scale hyperplanes to $\mH_{xy}$, of 3 types.
\begin{itemize}
\item[(i)] The first type consists of two hyperplanes, which are the $(j+1)$th--scale hyperplanes second closest to $x$ and to $y$, which we call \emph{$(j+1)$--terminal hyperplanes}.  A $(j+1)$--terminal hyperplane may also be a $k$th--scale hyperplane on some larger scale $k<j+1$, in which case we call it an \emph{incidental kth--scale hyperplane}.  
\item[(ii)] As a second type, for each non-incidental $j$th--scale hyperplane $H_{k\delta^j r}\in\mH_{xy}$ we put in $\mH_{xy}$ the closest $(j+1)$th--scale hyperplanes on either side of $H_{k\delta^j r}$, that is, $H_{(k\delta^j-\delta^{j+1})r}$ and $H_{(k\delta^j+\delta^{j+1})r}$, which we call \emph{sandwiching hyperplanes}.  
\item[(iii)] The third type is $(j+1)$th--scale \emph{joining hyperplanes}; we place between 1 and 4 of these in each long $j$th--scale interval, depending on the behavior of $\Gamma_{xy}$ in the interval in a manner to be specified below.  Joining hyperplanes are always placed in the ``extremal 10ths'' of the long interval; more precisely, if the $j$th--scale interval has $k$th--scale length then they are placed at distance $\delta^\ell r$ from one of the endpoints for some $k+1\leq\ell\leq j$, with at most 2 such hyperplanes at either end.  
\end{itemize}
We use superscripts $-$ and $+$ for quantities associated with left--end and right--end joining hyperplanes, respectively. We continue adding hyperplanes through all scales from $j_2$ to $j_1$; after adding $j_1$th--scale hyperplanes, $\mH_{xy}$ is complete and we stop.

\begin{figure}
\includegraphics[width=16cm]{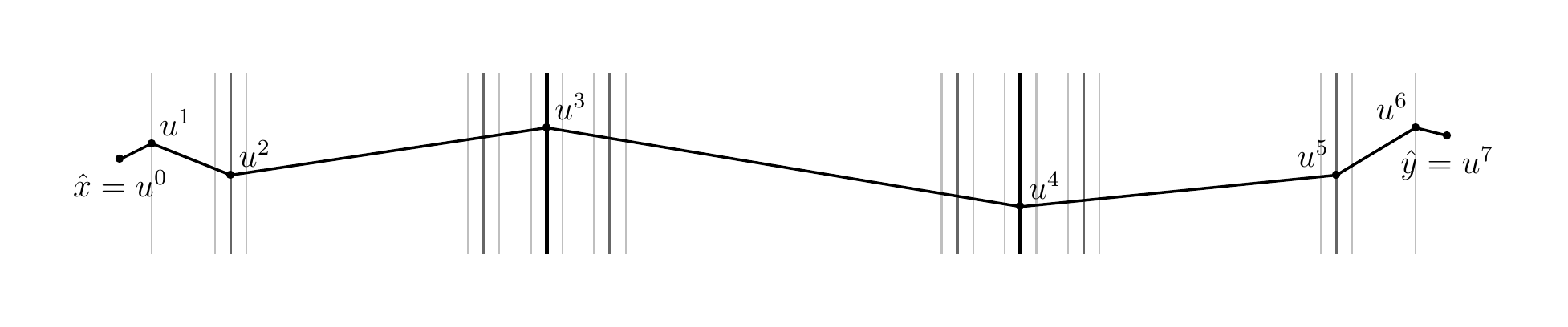}
\caption{ Diagram showing 3 scales of hyperplanes in $\mH_{xy}$: $j$th--scale (black), $(j+1)$th--scale (medium gray), $(j+2)$th--scale (light gray.) Points $u^i, 1\leq i\leq 6,$ are in the terminal hyperplanes; if these are all the scales (i.e.~$j_1=j_2+2$) then the path is an example of a possible final CG path $\Gamma_{xy}^{CG}$, or a path $\Omega_{xy}$ in Section \ref{Th1-2up}. The other hyperplanes are sandwiching ones. Joining hyperplanes are not shown. }
\label{Scales}
\end{figure}

A \emph{terminal jth--scale interval} in $[x_1,y_1]$ is an interval between a terminal $(j+1)$th--scale hyperplane and the terminal $j$th--scale hyperplane closest to it; the length of such an interval is necessarily between $\delta^jr$ and $2\delta^jr$, so it is short. In Figure \ref{Scales}, the interval between the hyperplanes containing $u^2$ and $u^3$ is a terminal $j$th--scale interval.

At most 6 $(j+1)$th--scale hyperplanes are added inside each $j$th--scale interval, and only 1 if the interval is terminal.  Therefore if $\mH_{xy}$ contains $n$ $j$th--scale hyperplanes, then the number of $(j+1)$th--scale hyperplanes is at most $7(n-1)+5$.  It follows that
\begin{equation}\label{hypernum2}
  \Big| \left\{ H_s: H_s \text{ is a $j$th--scale hyperplane in $\mH_{xy}$} \right\} \Big| \leq 7^j-1.
\end{equation}

In keeping with (iii) above, we will designate up to four random values \tred{$\mu_{xy}^{-,1}(I)<\mu_{xy}^{-,2}(I)<\mu_{xy}^{+,2}(I)<\mu_{xy}^{+,1}(I)$} in $I$ for each long $j$th--scale interval $I=[a,b]$, for each $j$, depending on $\Gamma_{xy}$.  These will satisfy
\begin{align*}
  |I| \in &[\delta^kr-4\delta^kr,\delta^kr] \implies \mu_{xy}^{-,1}(I) = a + \delta^{\ell+1} r,\ 
    \mu_{xy}^{-,2}(I) = a + \delta^\ell r \ \ \text{for some } k<\ell\leq j,\\
  &\mu_{xy}^{+,1}(I) = b - \delta^{\ell'+1} r,\ 
    \mu_{xy}^{+,2}(I) = b - \delta^{\ell'} r \ \ \text{for some } k<\ell'\leq j.
\end{align*}
We call the values $\mu_{xy}^{\pm,1}(I)$ \emph{outer joining points}, and $\mu_{xy}^{\pm,2}(I)$ \emph{inner joining points}; the inner ones will represent locations where a certain other path traversing $I$ can be guided to coalesce with $\Gamma_{xy}$, and we call $H_{\mu_{xy}^{\pm,\ep}(I)}$ the (potential) \emph{joining hyperplanes of the interval} $I$.  We say ``potential'' because not all are necessarily actually included in $\mH_{xy}$; which are included, and with what values of $\ell,\ell'$,  depend on rules to be described.

Recall that $\psi_q(u)$ denotes the closest point to $u$ in the basic grid, and $F_y = \psi_q^{-1}(y), y\in q\ZZ^d$. We need only consider the case in which $x,y$ each share a Voronoi cell with a basic grid point, that is,
\begin{equation}\label{startend}
  x = \varphi(\hat x), \quad y = \varphi(\hat y) \quad\text{for some } \tred{\hat x,\hat y} \in q\ZZ^d;
\end{equation}
we readily obtain the general case from this via Lemma \ref{neighbortimes}, since for every $x_0\in\VV$ the point $x=\varphi(\psi_q(x_0))$ satisfies \eqref{startend} and $|x_0-x|\leq q\sqrt{d}$.  Since $q\geq 2$, \eqref{startend} ensures that $\psi_q(x)=\hat x,\psi_q(y)=\hat y$.
For $0\leq s\leq r$ and $\gamma$ a path in $\GG$ from $B_1(0)$ to $B_1(re_1)$, for $v,w$ vertices in $\gamma$, write $\tred{\gamma_{[v,w]}}$ for the segment of $\gamma$ from $v$ to $w$ and let $\tred{u_s(\gamma)}$ denote the \emph{entry point} of $\gamma$ into $H_s^+$, that is, the first vertex of $\gamma$ in $H_s^+$, necessarily next to $H_s$ (in the sense that its Voronoi cell intersects $H_s$.)  Our aim is to approximate a general geodesic $\Gamma_{xy}$ (subject to \eqref{rtxycond}) by a CG one via certain \emph{marked (basic) grid points} which we will designate, lying in hyperplanes $H_s\in\mH_{xy}$. In the geodesic $\Gamma_{xy}$, the first and last marked grid points are $\hat x=\psi_q(x)$ and $\hat y=\psi_q(y)$ (see \eqref{startend}.)
Initially, the ones in between are the discrete approximations $\psi_q(u_s(\gamma))$ of the entry points $u_s(\gamma)$ corresponding to each of the hyperplanes $H_s\in\mH_{xy}$.  Here since $q\geq 4$ and $d(u_s(\gamma),H_s)\leq 2$, we always have $\psi_q(u_s(\gamma))\in H_s$. Later we will remove some of these initial marked grid points, and replace others with their $j$th--scale CG approximations for various $j$.  This means not every $H_s\in\mH_{xy}$ necessarily contains a marked grid point, in all our CG approximations.  When a path $\Gamma$ has a marked grid point in some $H_s$, we denote that marked grid point as \tred{$\mkm_s(\Gamma)$}. For all our CG approximations, the first and last marked grid points are $\psi_q(x)$ and $\psi_q(y)$, and for some $j$ the ones in between each lie in the $j$th--scale grid in some $j$th--scale hyperplane $H_{k\delta^jr}\in\mH_{xy}$. 
If such a CG path has marked grid points $v,w$ in consecutive $j$th--scale hyperplanes of $\mH_{xy}$, we say $\gamma$ makes a \emph{jth--scale transition} from $v$ to $w$.  
Such a transition involves one longitudinal step and some number \tred{$m_i$} of transverse steps in direction $i$ for each $2\leq i\leq d$, going from $v$ to $w$.  
We say a $j$th--scale transition is \emph{normal} if $m_i \in [-\ell_2,\ell_2]$ for all $2\leq i\leq d$, and \emph{sidestepping} otherwise. Even on the smallest scale $j_1$, every $j$th--scale transition with $v,w\in G_r^+$ is nearly in the $e_1$ direction, in that its angle satisfies (similarly to \eqref{lowtilt})
\begin{equation}\label{forward}
  \alpha_{e_1,w-v} \leq \frac{4C_{60}\sqrt{d-1}\Delta_r\log r}{\delta^{j_1}r} \leq r^{-(1-\chi_2)/2},
\end{equation}
provided $r$ is large, with $C_{60}$ from the definition of $G_r^+$.

We refer to a path $\gamma$ from $x$ to $y$, together with its marked grid points, as a \emph{marked path}; the path alone, without the marked grid points, is called the \emph{underlying path}. We may write a marked path as $\psi_q(x)=v^0\to v^1\to \cdots\to v^m\to v^{m+1}=\psi_q(y)$; here the $v^i$ are grid points. The pairs $(v^{i-1},v^i)$ are called \emph{links} of the path. If $\gamma$ is the concatenation of the geodesics $\Gamma_{v^{i-1},v^i}$, we call it a \emph{marked piecewise--geodesic path}; we abbreviate piecewise--geodesic as \tred{PG}.  (Recall that when $v,w\notin \GG$, $\Gamma_{vw}$ denotes $\Gamma_{\varphi(v),\varphi(w)}$.) Unless otherwise specified, when we give a marked path by writing its marked points in this way, we assume the path is the (unique) marked PG path given by those marked points.

Recall that
\[
  \tred{\hT(u,v)} = \min\{T(y,z): y\in F_u,z\in F_v\},
\]
with $F_u,F_v$ being cubes of side $q\in[4,5]$. Suppose we have a marked PG path 
\[
  \Gamma^{CG}: v^0\to v^1\to \cdots\to v^m\to v^{m+1},
\]
contained in $G_r^+$.  We associate four quantities to this path:
\[
  \tred{\Upsilon_{Euc}(\Gamma^{CG})} = \Upsilon_{Euc}(v^0,\dots,v^{m+1}) = \sum_{i=1}^{m+1} |v^i-v^{i-1}|,  
\]
\[
  \tred{\Upsilon_h(\Gamma^{CG})} = \sum_{i=1}^{m+1} h\left( \left| v^i-v^{i-1} \right| \right),\quad
    \tred{\Upsilon_{\hT}(\Gamma^{CG})} = \sum_{i=1}^{m+1} \hT(v^{i-1},v^i),\quad
    \tred{\Psi(\Gamma^{CG})} =  \sum_{i=1}^{m+1} \frac{|(v^i-v^{i-1})^*|^2}{|(v^i-v^{i-1})_1|}.
\]
Using the standard fact that for some $c_9\leq 1$,
\begin{equation}\label{adddist1}
  \frac 38 \frac{|w^*|^2}{|w_1|} \leq |w| - |w_1| \leq \frac{|w^*|^2}{2|w_1|} \quad\text{whenever } \frac{|w^*|}{|w_1|} \leq c_9,
\end{equation}
we see that
provided $r$ is large and all $v^i\in G_r^+$, $\Psi(\Gamma^{CG})$ represents added length in $\Psi(\Gamma^{CG})$ relative to the lower bound $|(v^{m+1}-v^0)_1|$, in that
\begin{equation}\label{adddist}
  \Upsilon_{Euc}(\Gamma^{CG}) 
    \geq \sum_{i=1}^{m+1} |(v^i-v^{i-1})_1| + \frac 38 \Psi(\Gamma^{CG}) = |(v^{m+1}-v^0)_1| + \frac 38 \Psi(\Gamma^{CG}).
\end{equation}
Informally we refer to $\Upsilon_{Euc}(v^0,\dots,v^{m+1}) - |(v^{m+1}-v^0)_1|$ as the \emph{extra length} of the path $\Gamma^{CG}$.
From \eqref{adddist1}, subadditivity of $h$, and Lemma \ref{monotoneE} we obtain, after reducing $c_9$ if necessary,
\begin{equation}\label{addh1}
   \frac \mu 3 \frac{|w^*|^2}{|w_1|} - c_{10} \leq h(|w|) - h(|w_1|) \leq \frac{2\mu}{3} \frac{|w^*|^2}{|w_1|} + c_{10} 
     \quad\text{whenever } \frac{|w^*|}{|w_1|} \leq c_9.
\end{equation}
In our applications of \eqref{addh1}, the last condition will always be satisfied due to \eqref{forward}.
In general, provided $r$ is large and all $v^i$ lie in $G_r^+$, with $(v^i-v^{i-1})_1$ much larger than the width $2C_{60}\Delta_r\log r$ of $G_r^+$, we have 
\begin{equation}\label{addh}
  \Upsilon_h(\Gamma^{CG}) - h((v^{m+1}-v^0)_1) \geq \sum_{i=1}^{m+1} \Big[ h\left( \left| v^i-v^{i-1} \right| \right) 
    - h\left( (v^i-v^{i-1})_1 \right) \Big] \geq \frac{\mu}{3} \Psi(\Gamma^{CG}) - (m+1)c_{10}.
\end{equation}

We can now define the joining points $\mu_{xy}^{\pm,\ep}(I)$ in a long $j$th--scale interval $I=[a,b]$ of $k$th--scale length, where $j_2\leq k<j\leq j_1-1$.  We begin with $\mu_{xy}^{-,\ep}(I), \ep=1,2$, which lie in the left part of $I$. We first define modified values of $t$ which will appear in Lemmas \ref{alltrans} and \ref{sumbits}:
\begin{equation}\label{tstar}
  \tred{t^*(v)} = \frac 13 t + \frac{\delta\mu}{18} \left( \frac{|v^*|}{\Delta_r} \right)^2 \quad\text{and}\quad
    \tred{t^*(v,w)} = \frac 13 t + \frac{\delta\mu}{18} \left[ \left( \frac{|v^*|}{\Delta_r} \right)^2 + \left( \frac{|w^*|}{\Delta_r} \right)^2
    \right].
\end{equation}
We consider a marked PG path with marked points in the $(j+1)$th--scale grid: let 
\[
  \tred{v^0}=\mkm_a(\Gamma_{xy})=V_{j+1}(u_a(\Gamma_{xy})),\quad 
    \tred{v^{fin}}=\mkm_b(\Gamma_{xy})=V_{j+1}(u_b(\Gamma_{xy})),
\]
and for each $k+1\leq\ell\leq j+1$ let 
\[
  \tred{v^\ell} = \mkm_{a+\delta^\ell r}(\Gamma_{xy})=V_{j+1}(u_{a+\delta^\ell r}(\Gamma_{xy})),
    \quad \tred{w^\ell} = \Pi_{v^0v^{fin}} \cap H_{a+\delta^\ell r},
\]
and let
\begin{equation}\label{alkap}
  \tred{\alpha(\ell)} = \frac{|v^\ell-w^\ell |^2}{\delta^\ell r},\quad
  \tred{\kappa(\ell)} = \frac{|v^\ell-w^\ell |^2}{\delta^\ell r\sigma(\delta^\ell r)}.
\end{equation}
We may view $\alpha(\ell)/2$ as an approximation of the ``extra distance at scale $\delta^\ell r$,'' that is, of $|v^\ell-v^0| + |v^{fin} - v^\ell| - |v^{fin}-v^0|$; $\kappa(\ell)/2$ is this extra distance normalized by the fluctuation size.
For $k+1\leq\ell< j+1$ let $\tred{g^{\ell+1}}=\Pi_{v^0v^{\ell}} \cap H_{a+\delta^{\ell+1} r}$, and let
\begin{equation}\label{thdef}
  \tred{\theta(\ell+1)} = \frac{|g^{\ell+1}-w^{\ell+1}|}{|v^{\ell+1}-w^{\ell+1}|} = \delta \frac{|v^{\ell}-w^{\ell}|}{|v^{\ell+1}-w^{\ell+1}|};
\end{equation}
see Figure \ref{JoinPts}. Then
\begin{equation}\label{thetaact}
  \frac{|v^{\ell+1}-g^{\ell+1} |^2}{\delta^{\ell+1} r} = (1-\theta(\ell+1))^2 \alpha(\ell+1)
\end{equation}
so in view of \eqref{forward}, there is an ``extra distance''
\begin{equation}\label{extramain1}
  \Upsilon_{Euc}(v^0,v^{\ell+1},v^{\ell}) - |v^\ell - v^0| \geq \frac 13 \frac{|v^{\ell+1}-g^{\ell+1} |^2}{\delta^{\ell+1} r} 
    = \frac{(1-\theta(\ell+1))^2}{3} \alpha(\ell+1).
\end{equation}
Suppose that for some $\ell\in[k+1,j]$ we have
\begin{equation}\label{kappaord}
  2^{j-\ell}\kappa(j+1) \leq \kappa(\ell+1), \quad 2\kappa(\ell+1) \geq \kappa(\ell),
\end{equation}
as happens if $2^m\kappa(m)$ is maximized at $m=\ell$.
From \eqref{powerlike}, the second inequality ensures that 
\begin{equation}\label{ak}
  \theta(\ell+1)^2 \leq 2\delta \frac{\sigma(\delta^{\ell} r)}{\sigma(\delta^{\ell+1}r)} \leq \frac 14 \quad\text{and}\quad
    2\alpha(\ell+1) \geq \frac{\sigma(\delta^{\ell+1}r)}{\sigma(\delta^\ell r)} \alpha(\ell) \geq C_{23}^{-1}\delta^{\chi_2}\alpha(\ell),
\end{equation}
and the first inequality in \eqref{kappaord} tells us that 
\begin{equation}\label{alphagrow}
  \alpha(\ell+1) \geq 2^{j-\ell}\frac{\sigma(\delta^{\ell+1}r)}{\sigma(\delta^{j+1} r)} \alpha(j+1) 
    \geq C_{22}\left( \frac{2}{\delta^{\chi_1}} \right)^{j-\ell} \alpha(j+1).
\end{equation}
It then follows from \eqref{extramain1}, \eqref{ak}, and \eqref{alphagrow} that the path from $v^0$ to $v^\ell$ is bowed in the sense that
\begin{equation}\label{bowedpath}
  \Upsilon_{Euc}(v^0,v^{\ell+1},v^{\ell}) - |v^\ell - v^0| 
    \geq \frac{C_{22}}{12} \left( \frac{2}{\delta^{\chi_1}} \right)^{j-\ell} \alpha(j+1).
\end{equation}
\begin{figure}
\includegraphics[width=16cm]{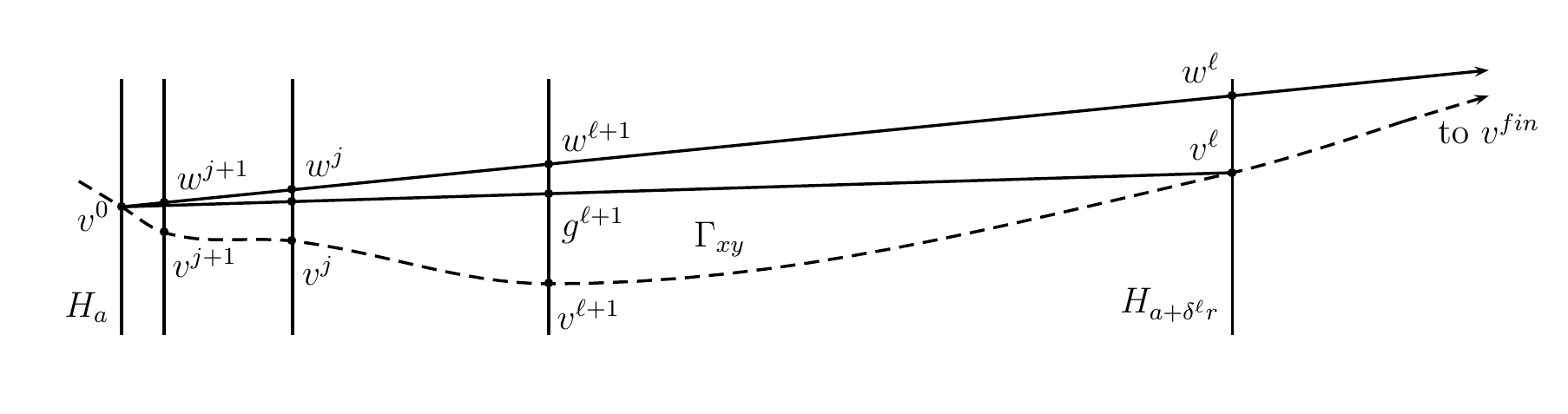}
\caption{The left end of a long $j$th--scale interval in which $L^-(I)=\ell$. Starting from the left, there is an endpoint hyperplane, then one of its sandwiching hyperplanes; the rightmost two are the joining hyperplanes, containing $v^{\ell+1}$ and $v^\ell$. In the bowed case, $v^{\ell+1}$ (at distance $\delta^{\ell+1}r$ from $H_a$) is ``sufficiently far'' from $w^{\ell+1}$, moreso than occurs on other length scales $\delta^ir$.}
\label{JoinPts}
\end{figure}
Motivated by this we let
\begin{equation}\label{Lminus}
  \tred{L^-(I)} = \begin{cases} j &\text{if } \max(\alpha(j+1),\delta^{-1}\alpha(j)) 
    \leq \frac{1}{16\mu}\left( \frac \lambda 7\right)^{j+1}t^*(v^0)\sigma_r, \\ 
    \arg\max_{\ell\in[k+1,j]}2^\ell\kappa(\ell)-1
    &\text{if $\max(\alpha(j+1),\delta^{-1}\alpha(j)) > \frac{1}{16\mu}\left( \frac \lambda 7\right)^{j+1}t^*(v^0)\sigma_r$} \\
    &\text{\hskip.6cm and $\arg\max_{\ell\in[k+1,j]}2^\ell\kappa(\ell) >k+1$},\\
    k+1 &\text{otherwise.} \end{cases}
\end{equation}

We refer to the 3 options in \eqref{Lminus} as the \emph{forward, bowed}, and \emph{totally unbowed} cases, respectively. They may be interpreted as follows.
In the forward case the initial steps $v^0 \to v^1\to v^2$ have little sidestepping, and we will see that this eliminates the need to exploit bowedness; this should be viewed as the ``baseline'' or ``most likely'' case.  Otherwise we look for a scale $\delta^\ell r$, with $k+1\leq \ell\leq j$, on which $\Gamma_{xy}$ is bowed as in \eqref{bowedpath}, by seeking a scale (the arg max) satisfying \eqref{kappaord}.
In the bowed case such a scale exists; see Figure \ref{JoinPts}.  In the totally unbowed case there is no such scale, meaning $2^{\fatdot}\kappa(\cdot)$ is maximized for essentially the full length scale of the interval $I$.  By \eqref{alphagrow} this forces the extra distance $\alpha(k+1)$ to be very large.
We define the inner and outer joining points as
\[
  \tred{\mu_{xy}^{-,2}(I)} = a + \delta^{L^-(I)};\quad
  \tred{\mu_{xy}^{-,1}(I)} = a + \delta^{L^-(I)+1}.
\]

We define \tred{$\mu_{xy}^{+,\ep}(I)$} in a mirror image manner to $\mu_{xy}^{-,\ep}(I)$ in $[a,b]$, going backwards from $b$ to $a$ instead of forward from $a$ to $b$.  That is, we use the points $\tred{\hat v^\ell} = \mkm_{b-\delta^\ell r}(\Gamma_{xy})$ in place of the $v^\ell$'s and $t^*(v^{fin})$ in place of $t^*(v^0)$; otherwise the definition is the same, and the analogs of \eqref{extramain1}, \eqref{ak}, and \eqref{alphagrow} are valid for the analog \tred{$L^+(I)$} of $L^-(I)$.  

We now describe the rules for which of the four $(j+1)$th--scale joining hyperplanes $H_s$ with $s=\mu_{xy}^{\pm,\ep}(I)$, in a long $j$th--scale interval $I$, are included in $\mH_{xy}$.  We note again that the endpoint and sandwiching hyperplanes at both ends of $I$ are always included; in some instances the sandwiching hyperplanes coincide with outer joining hyperplanes, so these criteria never rule out the inclusion of such hyperplanes.
\begin{itemize}
\setlength\itemsep{.3em}
\item[(i)] If both $\pm$ ends of $I$ have the forward case, then we include the inner joining hyperplanes in $\mH_{xy}$; these are at distance $\delta^jr$ from the interval ends.  The outer ones coincide with the sandwiching hyperplanes at distance $\delta^{j+1}r$ from the interval ends, so they are included as well.
\item[(ii)] If both $\pm$ ends have the bowed case, then we include both the inner and outer joining hyperplanes in $\mH_{xy}$.  We note that when either of $L^\pm(I)=j$, the corresponding outer joining hyperplane coincides with the sandwiching hyperplane as in (i), so it is already in $\mH_{xy}$ on that basis.
\item[(iii)] If both $\pm$ ends have the totally unbowed case, then we include only the inner joining hyperplanes.
\item[(iv)] If the two ends have different cases, we determine which end of the interval is \emph{dominant} according to the criterion described next. We then include the joining hyperplane(s) only at the dominant end, 1 or 2 hyperplanes in accordance with (i)---(iii) above.  We call this the \emph{mixed case}.  
\end{itemize}
To determine the dominant end of a long $j$th--scale interval $I=[a,b]$, necessarily having $k$th--scale length for some $k<j$, in the mixed case, we first select those non--sandwiching hyperplanes which are candidates for inclusion in $\mH_{xy}$, in accordance with (i)---(iii) above. (For example, if the path has the bowed case with $L^-(I)\neq j$ at the left end and the forward case at the right, we select the inner and outer joining hyperplanes on the left, and only the inner on the right.)  For these candidate hyperplanes, along with the 4 endpoint and sandwiching hyperplanes in $I$, we consider the corresponding ``tentative'' marked PG path (part of $\Gamma_{xy}$) with a marked point in each of the hyperplanes: $u^0\to\cdots\to u^n$ with $n=5$ or 6, where $u^0 = V_j(\mkm_{(u^0)_1}(\Gamma_{xy}))$ and $u^i = V_{j+1}(\mkm_{(u^i)_1}(\Gamma_{xy})), 1\leq i\leq n,$ are $j$th and $(j+1)$th--scale CG approximations.
The inner joining hyperplanes contain $u^\ell,u^{\ell+1}$, with $\ell=$ 2 or 3; we call this $\ell$ the \emph{central index}; the gap from $u^\ell$ to $u^{\ell+1}$ is the longest in $I$, at least $8\delta^{k+1}r$. Let \tred{$w_\perp^i$} be the orthogonal projection of $u^i$ into the line $\Pi_{u^0u^n}$; see Figure \ref{Dominant}.  We use the fact that the excess length $\tred{\mE(u^0,\dots,u^n)}=\Upsilon_{Euc}(u^0,\dots,u^n) - |u^n-u^0|$ can be approximately split into components associated with the two ends, as follows.  When the central index is $\ell$ we have
\begin{align}\label{splitlength}
   \mE(u^0,\dots,u^n) &= \Big[ \Upsilon_{Euc}(u^0,\dots,u^\ell) - |w_\perp^\ell - u^0| \Big] 
    + \Big[ |u^{\ell+1} - u^\ell| - |w_\perp^{\ell+1} - w_\perp^\ell| \Big] \notag\\
  &\hskip 1.5cm + \Big[ \Upsilon_{Euc}(u^{\ell+1},\dots,u^n) - |w_\perp^{\ell+1} - u^n| \Big].
\end{align}
For the first difference on the right we have from \eqref{adddist1}
\[
  \Upsilon_{Euc}(u^0,\dots,u^\ell) - |w_\perp^\ell - u^0| \geq |u^\ell-u^0| - |w_\perp^\ell - u^0|
    \geq \frac{|u^\ell - w_\perp^\ell|^2}{3\delta^{k+1}r},
\]
and similarly for the third difference, while for the middle one,
\begin{align}\label{middle2}
  |u^{\ell+1} - u^\ell| - |w_\perp^{\ell+1} - w_\perp^\ell| &\leq 
    \frac{(|u^\ell - w_\perp^\ell| + |u^{\ell+1} - w_\perp^{\ell+1}|)^2}{2\cdot 8\delta^{k+1}r}
    \leq \frac 18 \left( \frac{|u^\ell - w_\perp^\ell|^2}{\delta^{k+1}r} + 
    \frac{|u^{\ell+1} - w_\perp^{\ell+1}|^2}{\delta^{k+1}r} \right) \notag\\
\end{align}
so the middle difference in \eqref{splitlength} is only a small fraction of the whole:
\begin{equation}\label{middle}
  |u^{\ell+1} - u^\ell| - |w_\perp^{\ell+1} - w_\perp^\ell| \leq \frac 19 \mE(u^0,\dots,u^n).
\end{equation}
We designate the left end of $I$ as \emph{dominant} if the first of the 3 differences on the right in \eqref{splitlength} is larger than the third difference, and the right end in the reverse case.  As given in (iv) above, we include in $\mH_{xy}$ only the candidate joining hyperplanes from the dominant end; the non-dominant end has its endpoint and sandwiching hyperplanes but no joining ones.

\begin{figure}
\includegraphics[width=16cm]{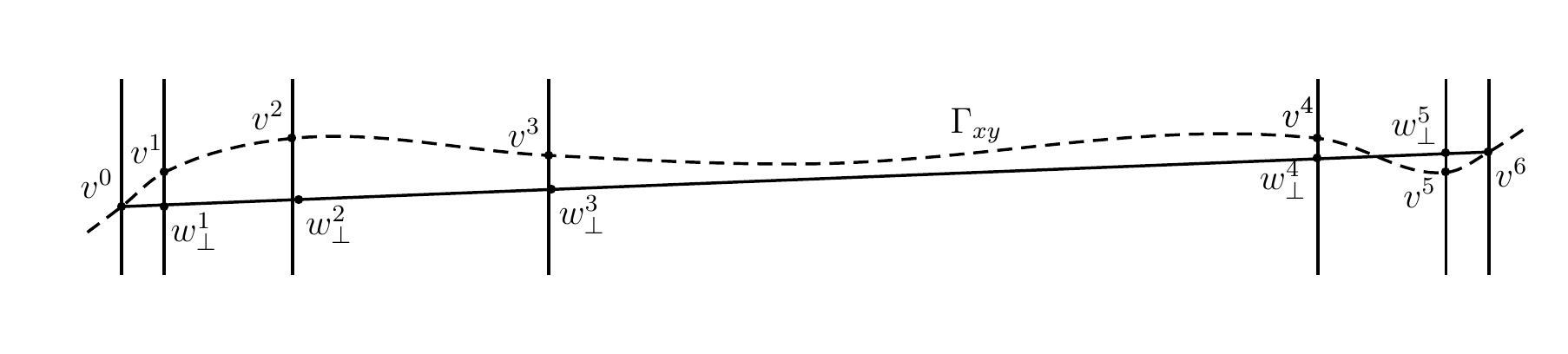}
\caption{ The mixed case with the bowed case at the (dominant) left end of the interval, and the forward case at the right end, showing the candidate hyperplanes (the 3 middle ones) and tentative marked PG path. The central index is 3. Since the right end is not dominant, the outer joining hyperplane there, containing $v^4$, is not included in $\mH_{xy}$. The full length of the interval is between $10\delta^{k+1}r$ and $10\delta^kr$, and all hyperplanes lie in the leftmost or rightmost 1/10 of the interval.} 
\label{Dominant}
\end{figure}

We note that if (for illustration) the left end is dominant, then, after excluding the right-end joining hyperplanes, we are left with the marked PG path $u^0\to\cdots\to u^\ell\to u^{n-1}\to u^n$, for which
the contribution of the right end to the extra length can be bounded: we have 
\begin{equation}\label{nondom}
  |u^n - u^{n-1}| - |u^n - w_\perp^{n-1}| \leq \Upsilon_{Euc}(u^{\ell+1},\dots,u^n) - |w_\perp^{\ell+1} - u^n|
    \leq  \Upsilon_{Euc}(u^0,\dots,u^\ell) - |w_\perp^\ell - u^0| 
\end{equation}
where the last inequality follows from dominance of the left end, so similarly to \eqref{middle},
\begin{equation}\label{nondom2}
  \frac{|u^{n-1} - w_\perp^{n-1}|^2}{3\delta^{j+1}r} \leq |u^n - u^{n-1}| - |u^n - w_\perp^{n-1}| 
    \leq \frac 12 \mE(u^0,\dots,u^\ell,u^{n-1},u^n).
\end{equation}
The same bound with $|u^1 - w_\perp^1|$ in place of $|u^{n-1} - w_\perp^{n-1}|$ holds symmetrically when the right end is dominant.

\begin{remark}\label{novstar}
In the bowed case at the left end of a $j$th--scale interval $I=[a,b]$, with $L^-(I)=\ell$, define
\[
  \tred{z^0}=v^0,\quad \tred{z^m} = \Pi_{v^0v^{\ell}} \cap H_{a+\delta^mr}, \quad \ell\leq m\leq j+1,
\]
and symmetrically at the right end. See Figure \ref{JoinPts}; $g^{\ell+1}$ there is $z^{j+1}$ here.
We have from \eqref{alphagrow} and \eqref{Lminus}
\begin{align}\label{genbowed1}
  \frac{|v^{\ell+1}-w^{\ell+1}|^2}{\delta^{\ell+1}r} \frac{\sigma(\delta^{j+1}r)}{\sigma(\delta^{\ell+1}r)} &\geq
    \max\left( 2^{j-\ell} \frac{|v^{j+1}-w^{j+1}|^2}{\delta^{j+1}r}, 2^{j-\ell-1} 
    \frac{\sigma(\delta^{j+1}r)}{\sigma(\delta^jr)} \frac{|v^j-w^j|^2}{\delta^j r} \right) \notag\\
  &\geq 2^{j-\ell-1} \frac{\delta}{16\mu}\frac{\sigma(\delta^{j+1}r)}{\sigma(\delta^jr)} 
    \left( \frac \lambda 7\right)^{j+1} t^*(v^0)\sigma_r \notag\\
  \text{and}\quad\frac{|v^\ell-w^\ell |^2}{\delta^\ell r} &\leq \frac{2\sigma(\delta^{\ell} r)}{\sigma(\delta^{\ell+1}r)} 
    \frac{|v^{\ell+1}-w^{\ell+1}|^2}{\delta^{\ell+1}r},
\end{align}
and as in \eqref{ak} it follows from these that
\begin{equation}\label{genbowed2}
  \frac{|v^{\ell+1}-z^{\ell+1}|^2}{\delta^{\ell+1}r} \geq 2^{j-\ell}\frac{\delta}{128\mu}\left( \frac \lambda 7\right)^{j+1}
    \frac{\sigma(\delta^{\ell+1}r)}{\sigma(\delta^j r)} t^*(v^0)\sigma_r.
\end{equation}
The advantage of \eqref{genbowed2} is that it depends only on $(v^0,v^1,v^2,v^3)$ and not on $v^{fin}$, whereas in \eqref{genbowed1} the points $w^i$ do depend on $v^{fin}$.  What we have shown is that if there exists $v^{fin}\in G_r^+$ with $(v^{fin}-v^0)_1 \geq \delta^\ell r$ for which \eqref{genbowed1} holds, then \eqref{genbowed2} holds, not involving $v^{fin}$. 
We further have $w^{j+1}-z^{j+1} = \delta^{j-\ell}(w^{\ell+1}-z^{\ell+1})$, while by the first half of \eqref{ak} we have $|v^{\ell+1}-w^{\ell+1}| \leq 2|v^{\ell+1}-z^{\ell+1}|$ and $|w^{\ell+1}-z^{\ell+1}| \leq |v^{\ell+1}-z^{\ell+1}|$, so
\begin{align}\label{firstinc}
  \frac{|v^{j+1}-z^{j+1}|^2}{\delta^{j+1}r} &\leq 2\frac{|v^{j+1}-w^{j+1}|^2}{\delta^{j+1}r}
    + 2\frac{|w^{j+1}-z^{j+1}|^2}{\delta^{j+1}r} \notag\\
  &= 2\frac{|v^{j+1}-w^{j+1}|^2}{\delta^{j+1}r}
    + 2\delta^{2(j-\ell)}  \frac{|v^{\ell+1}-z^{\ell+1}|^2}{\delta^{j+1}r} \notag\\
  &\leq 2^{-(j-\ell-3)} \frac{\sigma(\delta^{j+1}r)}{\sigma(\delta^{\ell+1}r)} \frac{|v^{\ell+1}-z^{\ell+1}|^2}{\delta^{\ell+1}r}
    + 2\delta^{j-\ell}  \frac{|v^{\ell+1}-z^{\ell+1}|^2}{\delta^{\ell+1}r} \notag\\
  &\leq 2^{-(j-\ell-4)} \frac{\sigma(\delta^{j+1}r)}{\sigma(\delta^{\ell+1}r)}  \frac{|v^{\ell+1}-z^{\ell+1}|^2}{\delta^{\ell+1}r}.
\end{align}
Again the left and right expressions in \eqref{firstinc} depend only on $(v^0,v^{j+1},v^{\ell+1},v^\ell)$, not on $v^{fin}$.
\end{remark}

In \eqref{Qrunif2} we may interpret $t\sigma_r$ as a reduction in the time allotted to go from $x$ to $y$, relative to $h(|(y-x)_1|)$.  In place of the reduction $t\sigma_r$ relative to $h(|(y-x)_1|)$, we can consider a modified reduction, call it $R_0$, which is relative to $\Upsilon_h(\Gamma^{CG})$:
\[
  h(|(y-x)_1|) - t\sigma_r = \Upsilon_h(\Gamma^{CG}) - \tred{R_0}.
\]
The modified reduction is larger: using \eqref{addh} we see that 
\[
  R_0\geq t\sigma_r + \frac{\mu}{3}\Psi(\Gamma^{CG}).
\]
We will need to (roughly speaking) allocate pieces of $R_0$ to the various transitions made by $\Gamma^{CG}$ and certain related paths.  Motivated by this, we define the \emph{jth--scale allocation} $A_j^0(v,w)$ of a transition $v\to w$ to be
\begin{equation}\label{Ajdef}
  \tred{A_j^0(v,w)} = \lambda^j\left(\frac{t\sigma_r}{7^j}  + \delta\mu\frac{|(w-v)^*|^2}{|(w-v)_1|} \right).
\end{equation}
In \eqref{Ajdef} the factor $7^j$ is used due to \eqref{hypernum2}.  For a marked PG path $\Gamma^{CG}$ as above, with $m\leq 7^j-1$ marks in the $j$th--scale grid, we have from \eqref{adddist}
\begin{align}\label{basicAsum}
  \sum_{i=1}^{m+1} A_j^0(v^{i-1},v^i) &\leq \lambda^j\left[ t\sigma_r + \delta\mu\Psi\left( \Gamma^{CG} \right) \right] 
    \notag\\
  &\leq  \lambda^j\left[ t\sigma_r  
     + 3\delta\mu\Big( \Upsilon_{Euc}(\Gamma^{CG}) - (v^{m+1}-v^0)_1 \Big) \right].
\end{align}
Also, since $j\leq j_1=O(\log\log r)$, all $A_j^0(v,w)$ are large provided $r$ is large and $t\geq 1$.

\begin{remark}\label{outline}
We now present an outline of the strategy of the proof.  Each geodesic $\Gamma_{xy}$ can be viewed as a marked PG path with marks in the basic grid in each of the hyperplanes of $\mH_{xy}$.  The goal is to gradually coarsen this approximation on successively larger length scales until we obtain a final path \tred{$\Gamma_{xy}^{CG}$}. The number of possible final paths (outside of a collection of ``bad'' paths having negligible probability) is small enough so that a version of \eqref{Qrunif2} can be proved for final paths.

To perform the coarsening we iterate a two-stage process, with the exception that the first iteration has only one stage.  The first iteration is on the $j_1$th scale, the second on the larger $(j_1-1)$th scale, and so on.  For the $j_1$th--scale iteration, we 
perform a set of operations on the original marked PG path (essentially $\Gamma_{xy}$) called \emph{shifting to the $j_1$th--scale grid}, replacing each marked grid point (located in the basic grid) in each hyperplane in $\mH_{xy}$ with a nearby point in the $j_1$th--scale grid.  Each further iteration has two stages. For the $(j_1-1)$th scale (second) iteration, in the first stage we shift those marked points lying in $(j_1-1)$th--scale hyperplanes in $\mH_{xy}$ to the $(j_1-1)$th--scale grid.  In the second stage of the iteration, we remove from the marked PG path those marked points not lying in $(j_1-1)$th--scale hyperplanes, with exceptions for points in terminal hyperplanes. See Figure \ref{Remark4-4}.  In general, for the $j$th--scale iteration ($j\leq j_1$), at the start of the iteration all the marked points in non-terminal hyperplanes are $(j+1)$th--scale grid points in $(j+1)$th--scale hyperplanes; in the first stage we shift the ones in $j$th--scale hyperplanes to the $j$th--scale grid, and in the second stage we remove the ones not in $j$th--scale hyperplanes, again with exceptions in terminal hyperplanes.

\begin{figure}
\includegraphics[width=16cm]{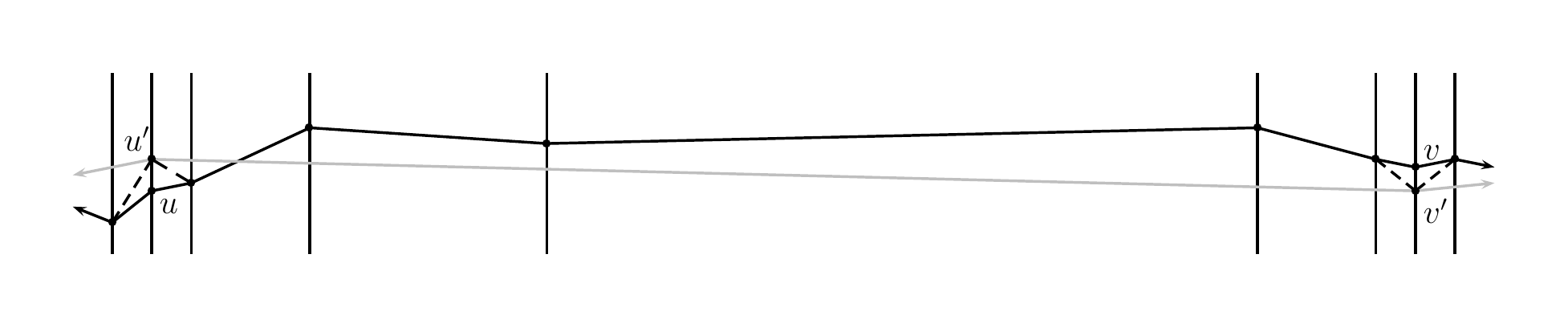}
\caption{ Illustration of the two stages of the $j$th--scale iteration.  The black path is the current one at the start of the iteration. The points $u,u',v,v'$ lie in the endpoint hyperplanes of a long $j$th--scale interval; the adjacent hyperplanes close on either side of these are the sandwiching ones, and the other three are joining hyperplanes.  In the first stage, we shift to the $j$th--scale grid in the endpoint hyperplanes, replacing $u,v$ in the marked PG path with $u',v'$, updating to the dashed path. In the second stage, the marked points between the endpoint hyperplanes are removed, updating to the gray path with marked points $u',v'$. }
\label{Remark4-4}
\end{figure}

The difficulty is that as we alter the marked PG path, the length $\Upsilon_{Euc}(\cdot)$ and corresponding $h$--sum $\Upsilon_h(\cdot) \approx \mu \Upsilon_{Euc}(\cdot)$ change, with marked--point deletions always reducing these sums, and we need to ensure that, with high probability, the corresponding sum of passage times $\Upsilon_{\hT}(\cdot)$ ``tracks'' these changes at least partly, to within a certain allocated error related to the above--mentioned $A_j^0v,w)$.  For shifting to a grid the tracking is not too difficult to achieve, as the allowed error turns out to be larger than the change in $h$--sum being tracked.  But for removal of marked points the tracking requires multiple different strategies, depending on the options in \eqref{Lminus} for the marked grid point locations in the gap between each two successive $j$th--scale hyperplanes in $\mH_{xy}$.  
The particular tracking needed is that, with high probability to within the allocated errors, when marked points are removed from a gap, the decrease in total passage time $\Upsilon_{\hT}(\cdot)$ is at least a positive fraction $\delta$ of the decrease in $h$--sum.  The primary difficulty in achieving this is that if the gap has a large length $L$, then the relevant passage time fluctuation size $\sigma(L)$ may overwhelm both the reduction in $h$--sum and the allocated errors; here the remedy involves the ``joining hyperplanes.'' We also make use of what we call \emph{intermediate paths}, which (in most cases) have total passage time and $h$--sum in between the values that exist before and after the marked--point removal, and are chosen so that they are relatively easy to compare to the pre--removal path.

In \eqref{Qrunif2} the passage time reduction for the full path is $t\sigma_r$, which \emph{a priori} suggests that the total of the allocated errors associated to a path should not exceed this amount.  There is no natural way to work with such a small total error yet achieve bounds uniformly over all $\Gamma_{xy}$.  However, we will see that the tracking enables us to increase the total of the allocated errors by an amount proportional to the extra Euclidean length $\Upsilon_{Euc}(\cdot)$ of the original marked PG path relative to the ``horizontal'' distance $(y-x)_1$, which gives the second term in parentheses in the formula \eqref{Ajdef}.  With this the necessary uniformity can be achieved, both for the tracking and for the fluctuation bounds on passage times of final paths $\Gamma_{xy}^{CG}$.
\end{remark}

\subsection{Step 2. Performing the $j_1$th--scale (first) iteration of coarse--graining.} 
As described in Remark \ref{outline}, for the $j_1$th--scale iteration we 
perform a sequence of operations on marked PG paths called shifting to the jth--scale grid, in the hyperplanes of $\mH_{xy}$.  The $j_1$th--scale iteration is different from those that follow, in that there is no second stage of removing marked points, and no need for the ``tracking'' of Remark \ref{outline}.  In general we will refer to the marked PG path existing before a shift or removal operation as the \emph{current (marked) path}, and the modified one resulting from the operation as the \emph{updated (marked) path}. The allocations $A_j^0(\cdot,\cdot)$ of \eqref{Ajdef} are used only for the $j_1$th--scale iteration; we will define other allocations later.

For the $j_1$th scale, every $H_s\in\mH_{xy}$ is a $j_1$th--scale hyperplane.
Suppose $\mH_{xy} = \{H_{s_i},1\leq i\leq m\}$ with $x_1<\tred{s_1<\cdots<s_m}<y_1$. Let $\tred{x^i}=\mkm_{s_i}(\Gamma_{xy})=\psi_q(u_{s_i}(\Gamma_{xy})),\,1\leq i\leq m$, and define $x^0,p^0,x^{m+1}p^{m+1}$ by $\tred{x^0=p^0}=\hat x=\psi_q(x),\tred{x^{m+1}=p^{m+1}}=\hat y=\psi_q(y)$.
At the start, the current path $\Gamma_{xy}^{j_1,0}$ is $\Gamma_{xy}$ with marks at the grid points $x^i$:
\[
  \tred{\Gamma_{xy}^{j_1,0}}:x^0\to x^1\to\cdots\to x^{m+1}.
\]
Note that $\Gamma_{xy}^{j_1,0}$ is a marked path, but not necessarily a marked PG path, since $\Gamma_{xy}$ need only pass near $\varphi(x^i)$, not necessarily through it.  By near we mean that both $u_{s_i}(\Gamma_{xy})$ and $\varphi(x^i)$ lie in the same cube $F_{x^i}$ of the basic grid.

Recall that the blocks of $\LL_{j_1}$ have side $K_0\beta^{j_1}\Delta_r$. The first shift to the $j_1$th--scale grid happens in $H_{s_1}$, replacing $x^1$ with $\tred{p^1}=V_{j_1}(x^1)$ to produce the updated marked path 
\[
  \tred{\Gamma_{xy}^{j_1,1}}:p^0\to p^1\to x^2\to\cdots\to x^{m+1},
\]
with the underlying path being the concatenation of the geodesics $\Gamma_{xp_1}, \Gamma_{p_1,u_{s_2}(\Gamma_{xy})}, \Gamma_{u_{s_2}(\Gamma_{xy}),y}$.
Next we repeat this in $H_{s_2}$, replacing $x^2$ with $p^2=V_{j_1}(x^2)$.  We continue this way performing shifts to the $j_1$th--scale grid in $H_{s_3},\dots,H_{s_m}$, producing the updated path
\[
  \tred{\Gamma_{xy}^{j_1,m}}: p^0\to p^1\to \cdots\to p^m\to p^{m+1},
\]
with the underlying path now (in view of \eqref{startend}) being the marked PG path given by these points.  

Let us analyze the effect of these shifts on $\Upsilon_h(\Gamma^{CG})$.  Consider the $i$th shift, replacing $x^i$ with $p^i$.  From \eqref{lowtilt} and basic geometry we have
\begin{equation}\label{sideeffect}
  |p^i-p^{i-1}| \geq |x^i-p^{i-1}| - c_{11}\frac{|(p^i-x^i)^*| |(p^i-p^{i-1})^*| + |(p^i-x^i)^*|^2}{(p^i-p^{i-1})_1}.
\end{equation}
Consider first the ``sidestepping'' case: $|(p^i-p^{i-1})^*| \geq \ell_2(j_1)\beta^{j_1}\Delta_r$. Since $|(p^i-x^i)^*| \leq \sqrt{d-1}K_0\beta^{j_1}\Delta_r$, we have from \eqref{relsize2}, \eqref{ell12}, and \eqref{sideeffect} that provided $r$ is large, for $1\leq i\leq m$,
\begin{align}\label{adddist4}
  |p^i-p^{i-1}|  &\geq |x^i-p^{i-1}| - \frac{2c_{11}K_0\sqrt{d-1}}{\ell_2(j_1)}\, \frac{|(p^i-p^{i-1})^*|^2}{(p^i-p^{i-1})_1} \notag\\
  &\geq |x^i-p^{i-1}| - 2c_{11}K_0\sqrt{d-1} \left( \frac{\beta}{\rho\delta^{(1+\chi_2)/2}} \right)^{j_1}
    \frac{|(p^i-p^{i-1})^*|^2}{(p^i-p^{i-1})_1} \notag\\
  &\geq |x^i-p^{i-1}| - \frac{1}{32\mu} A_{j_1}^0(p^{i-1},p^i).
\end{align}
Now consider the ``normal'' case: $|(p^i-p^{i-1})^*| < \ell_2(j_1)\beta^{j_1}\Delta_r$. From \eqref{relsize2}, \eqref{ell12}, and \eqref{sideeffect} we have
\begin{align}\label{adddist2}
  |p^i-p^{i-1}|  &\geq |x^i-p^{i-1}| - 2c_{11}K_0\ell_2(j_1)\sqrt{d-1}\frac{(\beta^{j_1}\Delta_r)^2}{\delta^{j_1}r} \notag\\
  &\geq |x^i-p^{i-1}| - c_{12}\left( \frac{\rho\beta}{\delta^{(1-\chi_1)/2}} \right)^{j_1} \sigma_r \notag\\
  &\geq |x^i-p^{i-1}| - \frac{1}{32\mu} A_{j_1}^0(p^{i-1},p^i).
\end{align}
We can interchange the roles of $p^i$ and $x^i$ and/or replace $p^{i-1}$ with $x^{i+1}$, so it follows from \eqref{adddist4} and \eqref{adddist2} that
\begin{equation}\label{absadd}
  \Big| |p^i-p^{i-1}| - |x^i-p^{i-1}| \Big| \leq  \frac{1}{32\mu} \Big( A_{j_1}^0(p^{i-1},p^i)
    + A_{j_1}^0(p^{i-1},x^i) \Big)
\end{equation}
and
\begin{equation}\label{absadd2}
  \Big| |p^i-x^{i+1}| - |x^i-x^{i+1}| \Big| \leq  \frac{1}{32\mu} \Big( A_{j_1}^0(p^i,x^{i+1}),
    + A_{j_1}^0(x^i,x^{i+1}) \Big)
\end{equation}
and then also
\begin{equation}\label{absaddh}
  \Big| h(|p^i-p^{i-1}|) - h(|x^i-p^{i-1}|) \Big| \leq  \frac{1}{16} \Big( A_{j_1}^0(p^{i-1},p^i)
    + A_{j_1}^0(p^{i-1},x^i) \Big)
\end{equation}
and 
\begin{equation}\label{absaddh2}
  \Big| h(|p^i-x^{i+1}|) - h(|x^i-x^{i+1}|) \Big| \leq  \frac{1}{16} \Big( A_{j_1}^0(p^i,x^{i+1})
    + A_{j_1}^0(x^i,x^{i+1}) \Big).
\end{equation}

We claim that
\begin{equation}\label{lastev}
  A_{j_1}^0(p^{i-1},p^i) \leq \left( 1 + \frac{1}{4}\lambda^{j_1} \right) A_{j_1}^0(x^{i-1},x^i).
\end{equation}
It is enough to show
\begin{equation}\label{pvsw}
  |(p^i-p^{i-1})^*|^2 - |(x^i-x^{i-1})^*|^2 \leq \frac{1}{4\delta\mu}\lambda^{j_1}\frac{t\sigma_r}{7^{j_1}}(x^i-x^{i-1})_1
    + \frac14 \lambda^{j_1}|(x^i-x^{i-1})^*|^2.
\end{equation}
To that end, we have
\[
  |(p^i-p^{i-1})^*| \leq |(x^i-x^{i-1})^*| + |(x^i-p^i)^*| + |(p^{i-1}-x^{i-1})^*| \leq  |(x^i-x^{i-1})^*| 
    + 2K_0\sqrt{d-1}\beta^{j_1}\Delta_r,
\]
so
\begin{equation}\label{sqdif}
  |(p^i-p^{i-1})^*|^2 - |(x^i-x^{i-1})^*|^2 \leq 4K_0\sqrt{d-1}\beta^{j_1}\Delta_r|(x^i-x^{i-1})^*| 
    + 4K_0^2(d-1)\beta^{2j_1}r\sigma_r.
\end{equation}
From \eqref{relsize2}, the last term satisfies
\begin{equation}\label{longlower}
  4K_0^2(d-1)\beta^{2j_1}r\sigma_r \leq \frac18 \left( \frac{\lambda\delta}{7}\right)^{j_1} tr\sigma_r
    \leq \frac18 \left( \frac \lambda 7\right)^{j_1} t\sigma_r(x^i-x^{i-1})_1.
\end{equation}
If the $i$th transition has very small sidestep, that is,
\begin{equation}\label{opt1}
  |(x^i-x^{i-1})^*| \leq 32K_0\sqrt{d-1}\left( \frac{\beta}{\lambda} \right)^{j_1}\Delta_r,
\end{equation}
then provided $r$ is large, since $(x^i-x^{i-1})_1 \geq \delta^{j_1}r$,
using \eqref{relsize2} the first term on the right in \eqref{sqdif} satisfies
\begin{align}\label{smalltrans}
  4K_0\sqrt{d-1}\beta^{j_1}\Delta_r&|(x^i-x^{i-1})^*| \leq 
    128K_0^2(d-1)\left( \frac{\beta^2}{\lambda} \right)^{j_1}r\sigma_r \notag\\
  &\leq 128K_0^2(d-1) \left( \frac{\beta^2}{\lambda\delta} \right)^{j_1}\sigma_r(x^i-x^{i-1})_1
    \leq \frac{1}{8\delta\mu} \left( \frac \lambda 7\right)^{j_1} t\sigma_r(x^i-x^{i-1})_1.
\end{align}
If instead the $i$th transition has larger sidestep, meaning
\begin{equation}\label{opt2}
  |(x^i-x^{i-1})^*| > 32K_0\sqrt{d-1}\left( \frac{\beta}{\lambda} \right)^{j_1}\Delta_r,
\end{equation}
then the first term on the right in \eqref{sqdif} satisfies
\begin{align}\label{bigtrans}
  4K_0\sqrt{d-1}\beta^{j_1}\Delta_r|(x^i-x^{i-1})^*| &\leq \frac18 \lambda^{j_1}|(x^i-x^{i-1})^*|^2.
\end{align}
Together, \eqref{sqdif}---\eqref{bigtrans} prove \eqref{pvsw}, and thus also \eqref{lastev}.  This same proof shows that
\begin{equation}\label{near1}
  A_{j_1}^0(p^{i-1},p^i), A_{j_1}^0(p^{i-1},x^i), A_{j_1}^0(x^{i-1},x^i) \text{ are all within a factor of } 
    1 + \frac{1}{4}\lambda^{j_1} \leq \frac 43,
\end{equation}
and similarly for $A_{j_1}^0(x^i,x^{i+1}),A_{j_1}^0(p^i,x^{i+1}),A_{j_1}^0(p^i,p^{i+1})$.

From \eqref{basicAsum}, \eqref{absadd}, \eqref{absadd2}, and \eqref{near1} we bound the total change in length from all shifts to the $j_1$th--scale grid:
\begin{equation}\label{Philoss3}
  \left| \Upsilon_{Euc}(\Gamma_{xy}^{j_1,m}) - \Upsilon_{Euc}(\Gamma_{xy}^{j_1,0}) \right| \leq 
    \frac{1}{6\mu} \sum_{i=1}^{m+1} A_{j_1}^0(p^{i-1},p^i) \leq 
    \frac{ \lambda^{j_1}}{6\mu} 
    \left[ t\sigma_r  + 3\delta\mu\Big( \Upsilon_{Euc}(\Gamma_{xy}^{j_1,m}) - (\hat y - \hat x)_1 \Big) \right]
\end{equation}
and similarly, using \eqref{absaddh}--\eqref{absaddh2} instead of \eqref{absadd}--\eqref{absadd2},
\begin{equation}\label{Philoss3h}
  \left| \Upsilon_h(\Gamma_{xy}^{j_1,m}) - \Upsilon_h(\Gamma_{xy}^{j_1,0}) \right| \leq 
    \frac{ \lambda^{j_1}}{3} \left[ t\sigma_r  
    + 3\delta\mu\Big( \Upsilon_{Euc}(\Gamma_{xy}^{j_1,m}) - (\hat y - \hat x)_1 \Big) \right].
\end{equation}
In view of \eqref{near1}, the derivation of \eqref{Philoss3} and \eqref{Philoss3h} is also valid if we replace $A_{j_1}^0(p^{i-1},p^i)$ with $A_{j_1}^0(x^{i-1},x^i)$, which gives the alternate bound
\begin{equation}\label{Philoss3ha}
  \left| \Upsilon_h(\Gamma_{xy}^{j_1,m}) - \Upsilon_h(\Gamma_{xy}^{j_1,0}) \right| \leq 
    \frac{ \lambda^{j_1}}{3} \left[ t\sigma_r  
    + 3\delta\mu\Big( \Upsilon_{Euc}(\Gamma_{xy}^{j_1,0}) - (\hat y - \hat x)_1 \Big) \right].
\end{equation}

For basic grid points $u,v$ define
\[
  \tred{M(u)} = \max\{T(y,z): y,z\in F_u\}.
\]
We have
\begin{equation}\label{dize}
  T(x,y) \geq \sum_{i=1}^{m+1} \hT(x^{i-1},x^i) = \Upsilon_{\hT}(\Gamma_{xy}^{j_1,0})
\end{equation}
and a form of approximate subadditivity holds: for basic grid points $u,v,w$,
\begin{equation}\label{nearsub}
  \hT(u,v) \leq \hT(u,w) + \hT(w,v) + M(w).
\end{equation}
From Lemma \ref{neighbortimes} we have for sufficiently large $s$ that the event
\[
  \tred{J^{(0)}(s)}: M(u) \geq s\log r \text{ for some } u\in q\ZZ^d\cap G_r^+
\]
satisfies
\begin{equation}\label{Mmax}
  P\left( J^{(0)}(s) \right) \leq c_{13} |G_r^+| e^{-c_{14}s\log r} \leq c_{13}e^{-c_{15}s\log r}.
\end{equation}
Before proceeding we stress that $m,p^i$, and $x^i$ should always be viewed at functions of $(x,y,\omega)$.  From \eqref{basicAsum} we have
\begin{align}\label{Asum}
  \sum_{i=1}^{m+1} A_{j_1}^0(x^{i-1},x^i) \leq \lambda^{j_1}\Big( t\sigma_r 
    + 3\delta\mu\left[ \Upsilon_{Euc}(\Gamma_{xy}^{j_1,0}) - (\hat y - \hat x)_1 \right] \Big).
\end{align}
Therefore 
\begin{align}\label{mainsplit}
  \Big\{ &T(x,y) \leq h((\hat y - \hat x)_1) - t\sigma_r \Big\} \notag\\
  &\subset \Bigg\{ \Upsilon_{\hT}(\Gamma_{xy}^{j_1,0}) - h((\hat y - \hat x)_1) \leq -\left( 1 - 2\lambda^{j_1} \right)t\sigma_r 
    + 4\delta\mu\lambda^{j_1}\left[ \Upsilon_{Euc}(\Gamma_{xy}^{j_1,m}) - (\hat y - \hat x)_1 \right] \notag\\
  & \hskip 8cm - \sum_{i=1}^{m+1} A_{j_1}^0(x^{i-1},x^i) \Bigg\} \notag\\
  &\subset \Bigg\{ \Upsilon_{\hT}(\Gamma_{xy}^{j_1,m}) - h((\hat y - \hat x)_1) \leq -\left( 1 - 2\lambda^{j_1} \right)t\sigma_r 
    + 4\delta\mu\lambda^{j_1}\left[ \Upsilon_{Euc}(\Gamma_{xy}^{j_1,m}) - (\hat y - \hat x)_1 \right] \Bigg\} \notag\\
  &\qquad\quad \bigcup \left\{ \Upsilon_{\hT}(\Gamma_{xy}^{j_1,m}) - \Upsilon_{\hT}(\Gamma_{xy}^{j_1,0}) > 
    \sum_{i=1}^{m+1} A_{j_1}^0(x^{i-1},x^i) \right\}.
\end{align}
The key inclusion here is the second one, as it takes us from an event involving the passage time of the original path $\Gamma_{xy}^{j_1,0}$ to an event involving the $j_1$th--scale CG approximation $\Gamma_{xy}^{j_1,m}$, up to the ``tracking error event'' given in the last line.  The name is only partly suitable here---bounding the last probability in \eqref{mainsplit} is related to the tracking of Remark \ref{outline}, in that we are ensuring that changing the path from $\Gamma_{xy}^{j_1,0}$ to $\Gamma_{xy}^{j_1,m}$ doesn't change $\Upsilon_{\hT}(\cdot)$ too much, but the change in $h$-sum here is too small to require being tracked.

\subsection{Step 3. Bounding the tracking--error event for the $j_1$th--scale--iteration.}  Consider next the tracking--error event
\[
  \tred{J_{xy}^{(1)}}:\ \Upsilon_{\hT}(\Gamma_{xy}^{j_1,m}) - \Upsilon_{\hT}(\Gamma_{xy}^{j_1,0}) >
    \sum_{i=1}^{m+1} A_{j_1}^0(x^{i-1},x^i)
\]
from the right side of \eqref{mainsplit}. Recalling Remark \ref{outline}, this reflects the failure of the passage time to track well when the path changes from $\Gamma_{xy}^{j_1,0}$ to its $j_1$th--scale CG approximation $\Gamma_{xy}^{j_1,m}$.
We have
\begin{align}\label{splitup}
  \Upsilon_{\hT}(\Gamma_{xy}^{j_1,m}) - \Upsilon_{\hT}(\Gamma_{xy}^{j_1,0}) 
    &\leq \sum_{i=1}^m \left[ \left( \hT(p^{i-1},p^i) - \hT(p^{i-1},x^i) \right) 
    + \left( \hT(p^i,x^{i+1}) - \hT(x^i,x^{i+1}) \right) \right]
\end{align}
and therefore using \eqref{near1},
\begin{align}\label{Jsub}
  J_{xy}^{(1)} &\subset \bigcup_{i=1}^m \left\{ \hT(p^{i-1},p^i) - \hT(p^{i-1},x^i) \geq 
    \frac 38 A_{j_1}^0(p^{i-1},x^i) \right\} \notag\\
  &\qquad \cup \bigcup_{i=1}^m \left\{ \hT(p^i,x^{i+1}) - \hT(x^i,x^{i+1}) \geq \frac 38 A_{j_1}^0(x^i,x^{i+1}) \right\}.
\end{align}

Let us consider any one of the events in the first union on the right in \eqref{Jsub}; we prepare to apply Proposition \ref{transTincr}. We assume $|p^i-p^{i-1}| \geq |x^i-p^{i-1}|$ as the opposite case is similar.  Define $\ep$ by
\[
  (1-\tred{\ep})|p^i-p^{i-1}| = |x^i-p^{i-1}|
\]
and let
\[
  \tred{\ol p^i} = p^{i-1} + (1-\ep)(p^i-p^{i-1}),\qquad \tred{\tilde p^i} = \psi_q(\varphi(\ol p^i)).
\]
We observe that in \eqref{adddist4} and \eqref{adddist2}, one can replace $1/32\mu$ with any given positive constant, if $r$ is large enough.  Consequently we have
\begin{equation}\label{gapsize2}
  |\ol p^i - p^{i-1}| = |x^i-p^{i-1}|, \quad |p^i-\ol p^i| = \ep |p^i-p^{i-1}| = |p^i-p^{i-1}| - |x^i-p^{i-1}| 
    \leq \frac{1}{32\mu} A_{j_1}^0(x^{i-1},p^i).
\end{equation}
We split $\hT(p^{i-1},p^i)$ into two corresponding increments, using \eqref{nearsub}:
\begin{align}\label{mainorcurv}
  &\left\{ \hT(p^{i-1},p^i) - \hT(p^{i-1},x^i) \geq \frac 38 A_{j_1}^0(p^{i-1},x^i) \right\} \notag\\
  &\quad \subset \left\{ \hT(p^{i-1},\tilde p^i) - \hT(p^{i-1},x^i) \geq \frac{3}{16} A_{j_1}^0(p^{i-1},x^i) \right\}
    \cup \left\{ \hT(\tilde p^i,p^i) + M(\tilde p^i) \geq \frac{3}{16} A_{j_1}^0(p^{i-1},x^i) \right\}
\end{align}
and define the corresponding unions 
\[
  \tred{J_{xy}^{(1a)}} = \bigcup_{i=1}^m \left\{ \hT(p^{i-1},\tilde p^i) - \hT(p^{i-1},x^i) \geq 
    \frac{3}{16} A_{j_1}^0(p^{i-1},x^i) \right\},
\]
\[
  \tred{J_{xy}^{(1b)}} = \bigcup_{i=1}^m \left\{ \hT(\tilde p^i,p^i) + M(\tilde p^i) \geq \frac{3}{16} A_{j_1}^0(p^{i-1},x^i) \right\},
\]
so that \eqref{mainorcurv} says the first union in \eqref{Jsub} is contained in $J_{xy}^{(1a)} \cup J_{xy}^{(1b)}$.
Define the set of $(x,y)$ corresponding to \eqref{rtxycond}
\[
  \tred{X_r} = \left\{ (x,y)\in q\ZZ^d\times q\ZZ^d: x,y\in G_r(K), |y-x| > \frac{C_{62}r}{(\log r)^{1/\chi_1}}, \eqref{startend}
    \text{ holds} \right\},
\]
with $C_{62}$ from \eqref{rtxycond}, and define the events
\[
  \tred{J^{(1a)}} = \cup_{(x,y)\in X_r} J_{xy}^{(1a)}, \quad \tred{J^{(1b)}} = \cup_{(x,y)\in X_r} J_{xy}^{(1b)},\quad 
  \tred{J^{(1c)}} = \Big\{ \Gamma_{xy} \not\subset G_r^+ \text{ for some } (x,y)\in X_r \Big\}.
\]
For configurations $\omega\notin J^{(1c)}$, all $p^{i-1},p^i,x^i,\tilde p^i$ lie in $G_r^+$, so the number of possible tuples $(p^{i-1},p^i,x^i,\tilde p^i)$ arising from some $(x,y)\in X_r$ is at most $c_{16}|G_r^+|^4$.
Define \tred{$s,\alpha$} by
\[
  \Delta_s = 2K_0\sqrt{d-1}\beta^{j_1}\Delta_r, \quad \alpha
    = \frac{3\lambda^{j_1}t\sigma_r}{16\sigma_s\log(2K_0\sqrt{d-1}\beta^{j_1}\Delta_r)},
\]
so that $|\tilde p^i-x^i| \leq |\ol p^i-x^i|+|\tilde p^i-\ol p^i| \leq 2|p^i-x^i| +q\sqrt{d-1} \leq \Delta_s$. By \eqref{powerlike},
\[
  \frac{r\sigma_r}{s\sigma_s} = \frac{\Delta_r^2}{\Delta_s^2} \geq \frac{c_{17}}{\beta^{2j_1}} \quad\text{and hence}\quad
    \frac{\sigma_r}{\sigma_s} \geq \frac{c_{18}}{\beta^{2\chi_1j_1/(1+\chi_1)}}.
\]
From this, \eqref{relsize2}, and \eqref{j1},
\begin{equation}\label{alphalow}
  \alpha \geq c_{19} \left( \frac{\lambda}{\beta^{2\chi_1/(1+\chi_1)}} \right)^{j_1} \frac{t}{\log r}
    \geq c_{20}t(\log r)^2.
\end{equation}
Now
\[
  \frac{3}{16} A_{j_1}^0(p^{i-1},x^i) \geq \frac{3\lambda^{j_1}t\sigma_r}{16} = 
    \alpha\sigma_s\log(2K_0\sqrt{d-1}\beta^{j_1}\Delta_r) = \alpha\sigma_s\log\Delta_s
\]
so for $\omega \in J_{xy}^{(1a)}$ we have for some $i$ that
\[
  \hT(p^{i-1},\tilde p^i) - \hT(p^{i-1},x^i) \geq \alpha\sigma_s\log \Delta_s \geq \alpha\sigma(s\log\Delta_s)
    \geq\alpha\sigma\Big( \Delta^{-1}(|\tilde p^i - x^i|) \log |\tilde p^i - x^i| \Big).
\]
In view of \eqref{gapsize2} and \eqref{alphalow} we then get from Proposition \ref{transTincr} and Lemma \ref{bigwander} that
\begin{equation}\label{J1a}
  P\left( J^{(1a)} \cup J^{(1c)} \right) \leq P\left( J^{(1a)} \bs J^{(1c)} \right) + P\left( J^{(1c)} \right) 
    \leq c_{21}|G_r^+|^4 e^{-C_{54}\alpha} + C_{68}e^{-C_{69}t} \leq e^{-c_{22}t}.
\end{equation}

Next, recalling \eqref{gapsize2}, we observe that $|\tilde p^i - \ol p^i| \leq q\sqrt{d-1}$ and provided $r$ is large,
\[
  |u-v|\leq \frac{1}{32\mu} A_{j_1}^0(x^{i-1},p^i) + q\sqrt{d-1} \implies h(|u-v|) \leq \frac{1}{16}A_{j_1}^0(x^{i-1},p^i),
\]
so using \eqref{gapsize2} and the bound on $t$ in \eqref{rtxycond} we have from Lemmas \ref{neighbortimes} and \ref{connect} that
\begin{align}\label{J1b}
  P\left( J^{(1b)}\bs J^{(1c)} \right) &\leq P\bigg( \text{for some $u,v\in q\ZZ^d\cap G_r^+$ and 
    $A\geq \left( \frac\lambda 7\right)^{j_1}t\sigma_r$
    we have $h(|u-v|) \leq \frac{A}{16}$} \notag\\
  &\hskip 2cm \text{and } \hT(u,v) \geq \frac A8 \bigg) \notag\\
  &\qquad + P\bigg( \text{for some $u\in q\ZZ^d\cap G_r^+$ we have } M(u) 
    \geq \frac{1}{16} \left( \frac\lambda 7\right)^{j_1}t\sigma_r 
    \bigg) \notag\\
  &\leq c_{23}|G_r^+|^2 \exp\left( -c_{24}\left[ \left( \frac\lambda 7\right)^{j_1}t\sigma_r \right]^{1-\chi_2} \right)
    + c_{25}|G_r^+|\exp\left( -c_{26}\left( \frac\lambda 7\right)^{j_1}t\sigma_r \right) \notag\\
  &\leq c_{27}e^{-c_{28}t}.
\end{align}

From \eqref{mainorcurv}, \eqref{J1a}, and \eqref{J1b} the first union in \eqref{Jsub} combined over $(x,y)$, together with $J^{(1c)}$, has probability bounded as
\[
  P\left( J^{(1a)}\cup J^{(1b)}\cup J^{(1c)} \right) \leq e^{-c_{22}t} + c_{27}e^{-c_{28}t},
\]
and the same bound applies to the second union in \eqref{Jsub}, while from \eqref{rtxycond} and \eqref{Mmax}, for large $c_{29}$ we have
\[
  P\left( J^{(0)}(c_{29}) \right) \leq c_{30}e^{-c_{31}t}.
\]
Hence from \eqref{mainsplit} we obtain
\begin{align}\label{firstit}
  P&\Big( T(x,y) \leq h((y-x)_1) - t\sigma_r \text{ for some } (x,y)\in X_r \Big) \notag\\
  &\quad \leq P\Bigg( \Upsilon_{\hT}(\Gamma_{xy}^{j_1,m}) - h((y-x)_1) 
    \leq -\left( 1-2\lambda^{j_1} \right) t\sigma_r + 4\delta\mu\lambda^{j_1}\Big( \Upsilon_{Euc}(\Gamma_{xy}^{j_1,m}) -
    (y-x)_1 \Big) \notag\\
  &\qquad\qquad \text{ for some } (x,y)\in X_r;\ \omega\notin J^{(0)}\cup J^{(1c)} \Bigg) + c_{32}e^{-c_{33}t}. 
\end{align}
This completes the $j_1$th--scale (first) iteration.  

It should be noted that, due to \eqref{mainsplit}, the terms with coefficients 2 and 4 in \eqref{firstit} represent a reduction taken from the original bound $t\sigma_r$ in \eqref{Qrunif2}, used to bound errors created in the $j_1$th--scale iteration.

\subsection{Step 4. Further iterations of coarse--graining: preparation.}  
For the $(j_1-1)$th--scale and later iterations of coarse--graining we use allocations $A_j^1(\Gamma_{xy}^{\cur},u^i)$ associated not to a particular transition, but to the shifting of the marked grid point $u^i$ in some $(j+1)$th--scale current marked PG path $\Gamma_{xy}^{\cur}$ to the $j$th--scale grid.  Specifically, in such a shift the marked grid point $u^i$ in the $(j+1)$th--scale grid is replaced by the $j$th--scale CG approximation $V_j(u^i)$, and the other marked grid points are left unchanged.  Consider an initial marked PG path when an iteration of shifts to the $j$th--scale grid begins:
\[
  \Gamma^0: u^0\to u^1\to\cdots\to u^{n+1},
\]
with $n\leq 7^{j+1}-1$.
Suppose that for some $\tred{I}\subset \{1,\dots,n\}$ not containing two consecutive integers, the marked grid points $(u^i,i\in I)$ are the ones shifted, one at a time, in the iteration, updating the path from $\tred{\Gamma^0}$ via a sequence of intermediate paths $\tred{\Gamma^1,\dots,\Gamma^{|I|-1}}$ to a final $\tred{\Gamma^{|I|}}$.
Note that since $I$ does not contain two consecutive integers, the shifts of different $u^i$'s do not ``interact,'' as shifting $u^i$ only affects the path between $\varphi(u^{i-1})$ and $\varphi(u^{i+1})$ and only affects $u^i$ among the marks; we will refer to this aspect as \emph{noninteraction of shifts}.  

Recall $t^*(\cdot,\cdot)$ from \eqref{tstar}. The allocation we use most widely, in both stages of each iteration, is $A_j^1(\cdot,\cdot)$ which appears in the next lemma. $A_j^2(\cdot,\cdot)$ is a variant of $A_j^1(\cdot,\cdot)$ 

\begin{lemma}\label{alltrans}
Let $j,K\geq 1$. Consider a marked PG path
\[
  \Gamma: u^0\to u^1\to\cdots\to u^{n+1}
\]
in $G_r^+$ with $u^0,u^{n+1}\in G_r(K),\,n\leq 7^j-1$, and $(u^0)_1<\cdots<(u^{n+1})_1$, and let $I\subset \{1,\dots,n\}$, not containing two consecutive integers.  Define
\begin{align*}
  \tred{A_j^1(\Gamma,i)} &= \lambda^j \left[ \frac{1}{7^j} t^*(u^{i-1},u^i) \sigma_r + \frac{\delta\mu}{9} 
    \left( \frac{|(u^i-u^{i-1})^*|^2}{|(u^i-u^{i-1})_1|} 
    + \frac{|(u^{i+1}-u^i)^*|^2}{|(u^{i+1}-u^i)_1|} \right) \right],
\end{align*}
and for $v,w\in \LL_j$ with $v_1<w_1$, let
\begin{equation}\label{Aj2}
  \tred{A_j^2(v,w)} = \frac{1}{4}\lambda^j \left[ \frac{1}{7^j} t^*(v,w) \sigma_r 
    + \frac{\delta\mu}{9} \frac{|(w-v)^*|^2}{|(w-v)_1|} \right].
\end{equation}
Then provided $t/K^2$ is sufficiently large,
\begin{equation}\label{Atotal}
  \sum_{i\in I} A_j^1(\Gamma,i) \leq \frac 23 \lambda^j\left[ t\sigma_r 
   + \delta\mu\Big( \Upsilon_{Euc}\left( \Gamma \right) - (u^{n+1}-u^0)_1 \Big)
    \right] 
\end{equation}
and 
\begin{equation}\label{barAtotal}
  \sum_{i=1}^{n+1} A_j^2(u^{i-1},u^i) \leq \frac{1}{4}\lambda^j \Big[ t\sigma_r 
    + \delta\mu\Big(\Upsilon_{Euc}\left( \Gamma \right) - (u^{n+1}-u^0)_1) \Big) \Big].
\end{equation}
\end{lemma}

The bound \eqref{Atotal} would be trivial if we used $t$ instead of $t^*(u^{i-1},u^i)$ the definition of $A_j^1(\Gamma,i)$. The point of Lemma \ref{alltrans} is that when the transitions affected by the shifting start or end at points $u^i$ which are at distance a large multiple of $\Delta_r$ from $G_r(K)$, it increases the usable value of $t$ by a multiple of $(|(u^i)^*|/\Delta_r)^2$, manifested in the definition of $t^*(u^{i-1},u^i)$.  Here by ``usable'' we mean ``not so big that \eqref{Atotal} or \eqref{barAtotal} fails.''

\begin{proof}[Proof of Lemma \ref{alltrans}.]

If $u^k\notin G_r(2K)$ for some $k$ then $d(u^k,\Pi_{u^0u^{n+1}}) \geq |(u^k)^*|/2$ so we have from \eqref{adddist1}
\begin{equation}\label{sizecurv}
  |u^{n+1}-u^0| + \frac{|(u^k)^*|^2}{3r} \leq |u^0-u^k| + |u^k-u^{n+1}| \leq \sum_{i=1}^{n+1} |u^i-u^{i-1}| 
    \leq (u^{n+1}-u^0)_1 + \sum_{i=1}^{n+1} \frac{|(u^i-u^{i-1})^*|^2}{2|(u^i-u^{i-1})_1|},
\end{equation}
so
\begin{equation}\label{sizecurv1}
  \frac{\delta\mu|(u^k)^*|^2}{18r} \leq \sum_{i=1}^{n+1} \frac{\delta\mu}{12} \frac{|(u^i-u^{i-1})^*|^2}{|(u^i-u^{i-1})_1|},
\end{equation}
while if $u^k\in G_r(2K)$, since $t/K^2$ is large,
\begin{equation}\label{sizecurv2}
  \frac{\delta\mu|(u^k)^*|^2}{18r} \leq \frac{2\delta\mu K^2\sigma_r}{9} \leq \frac{\delta t\sigma_r}{3},
\end{equation}
so using \eqref{hypernum2} and \eqref{adddist1},
\begin{equation}\label{starbound}
  \sum_{k=0}^n \frac{\delta\mu|(u^k)^*|^2}{18r7^j} \leq  \frac{\delta t\sigma_r}{3} + \sum_{i=1}^{n+1}  
    \frac{\delta\mu}{12} \frac{|(u^i-u^{i-1})^*|^2}{|(u^i-u^{i-1})_1|} \leq \frac 13\Big[ \delta t\sigma_r 
    + \delta\mu\Big( \Upsilon_{Euc}\left( \Gamma \right) - (u^{n+1}-u^0)_1 \Big) \Big].
\end{equation}
From \eqref{adddist1} we also get
\[
  \frac 13 \sum_{i\in I} \left( \frac{|(u^i-u^{i-1})^*|^2}{|(u^i-u^{i-1})_1|} 
    + \frac{|(u^{i+1}-u^i)^*|^2}{|(u^{i+1}-u^i)_1|} \right) \leq \Upsilon_{Euc}\left( \Gamma \right) - (u^{n+1}-u^0)_1
\]
which with \eqref{starbound} yields \eqref{Atotal}. The proof of \eqref{barAtotal} is similar, the only (inconsequential) difference being that when we sum the terms $t^*(u^{i-1},u^i)$ over all $i\leq n$, most terms $\delta\mu|(u^k)^*|^2/18r7^j$ get counted twice, for $k=i$ and $k=i+1$, whereas each was counted at most once when we summed over $i\in I$ in \eqref{Atotal}.
\end{proof}

\begin{lemma}\label{sumbits}
There exist constants $C_i$ such that, provided \eqref{relsize2} holds and $t$ is sufficiently large, we have for all $j\leq j_1$
\begin{align}\label{alllinks}
  P\Bigg( \big| \hT(&v,w) - h(|w-v|) \big| \geq A_j^2(v,w) \text{ for some $v,w\in G_r^+\cap\LL_j$ with } 
    \delta^{j+1}r \leq (w-v)_1 \leq 10\delta^{j-1}r \Bigg) \notag\\
  &\leq C_{70} \exp\left( -C_{71}\left( \frac{\lambda}{7\delta^{\chi_1}} \right)^j t \right). 
\end{align}
\end{lemma}

\begin{proof}
Let
\[
   \tred{\nu_0} = \nu_0(r) = \min\{k: (2^\nu t)^{1/2} \geq C_{60}\log r\}, 
 \]
\[
     \tred{\ell_0}=\ell_0(r,j+1) = \min\{\ell: 2^\ell \ell_2(j+1)K_0\beta^j\geq 2C_{60}\log r\}
 \]
for $C_{60}$ from the definition \eqref{Grplus} of $G_r^+$.  We decompose the pairs $(v,w)$ appearing in \eqref{alllinks} into subclasses according to the size of $|v^*|$ and the degree of sidestepping $|(w-v)^*|$, as follows. Fix $1\leq j\leq j_1$. Let
\begin{align*}
  \tred{R_{r,j}^{\nu,\ell}} &= \Big\{ (v,w)\in (G_r^+\cap\LL_{j+1})^2: 
    \delta^jr \leq (w-v)_1 \leq 10\delta^{j-1}r, (2^{\nu-1}t)^{1/2}\Delta_r < |v^*| \leq (2^\nu t)^{1/2}\Delta_r,\\
  &\hskip 1cm 2^{\ell-1}\ell_2(j+1)K_0\beta^j\Delta_r < |(w-v)^*| 
    \leq 2^\ell \ell_2(j+1)K_0\beta^j\Delta_r \Big\},\\
  &\hskip 1.5cm \text{for }1\leq k\leq \nu_0, 1\leq\ell\leq\ell_0(r,j+1),
\end{align*}
\begin{align*}
  \tred{R_{r,j}^{\nu,0}} &= \Big\{ (v,w)\in (G_r^+\cap\LL_{j+1})^2: 
    \delta^jr \leq (w-v)_1 \leq 10\delta^{j-1}r, (2^{\nu-1}t)^{1/2}\Delta_r < |v^*| \leq (2^\nu t)^{1/2}\Delta_r,\\
  &\hskip 1.5cm |(w-v)^*| 
    \leq \ell_2(j+1)K_0\beta^j\Delta_r \Big\},\quad \text{for } 1\leq k\leq \nu_0,
\end{align*}
\begin{align*}
  \tred{R_{r,j}^{0,\ell}} &= \Big\{ (v,w)\in (G_r^+\cap\LL_{j+1})^2: 
    \delta^jr \leq (w-v)_1 \leq 10\delta^{j-1}r, |v^*| \leq t^{1/2}\Delta_r,\\
  &\hskip 1cm 2^{\ell-1}\ell_2(j+1)K_0\beta^j\Delta_r < |(w-v)^*|
    \leq 2^\ell \ell_2(j+1)K_0\beta^j\Delta_r \Big\},\quad 1\leq\ell\leq\ell_0(r,j+1),
\end{align*}
\begin{align*}
  \tred{R_{r,j}^{0,0}} &= \Big\{ (v,w)\in (G_r^+\cap\LL_{j+1})^2: 
    \delta^jr \leq (w-v)_1 \leq 10\delta^{j-1}r, |v^*| \leq t^{1/2}\Delta_r, \notag\\
  &\hskip 1.5cm |(w-v)^*|
    \leq \ell_2(j+1)K_0\beta^j\Delta_r \Big\}.
\end{align*}
First, for $1\leq \nu\leq \nu_0, 1\leq\ell\leq\ell_0(r,j+1)$ we have 
\begin{align*}
  \frac{A_j^2(v,w)}{\sigma(10\delta^{j-1}r)} &\geq \frac{1}{4}\lambda^j 
    \left[ \frac{1}{7^j} \frac{\delta\mu}{36}2^\nu t\frac{\sigma_r}{\sigma(10\delta^{j-1}r)} 
    + \frac{ \delta\mu(2^\ell \ell_2(j+1)K_0\beta^j\Delta_r)^2 }{360\delta^{j-1}r\sigma(10\delta^{j-1}r)} \right]\\
  &\geq \lambda^j 
    \left[ \frac{c_0}{7^j} 2^\nu t\frac{1}{\delta^{\chi_1j}} + c_12^{2\ell}\rho^{2j} \right], \quad
    (v,w)\in R_{r,j}^{\nu,\ell}
\end{align*}
and from \eqref{ell12}
\[
   \left| R_{r,j}^{\nu,\ell} \right| \leq \frac{c_2}{\delta^j} \left( \frac{(2^\nu t)^{1/2}}{\beta^j} \right)^{d-1} 
     \left( 2^\ell \ell_2(j+1) \right)^{d-1} \leq c_3 (2^\nu t)^{(d-1)/2} 2^{(d-1)\ell}
     \frac{1}{\delta^j} \left( \frac{\rho\delta^{(1+\chi_1)/2}}{\beta^2} \right)^{(d-1)j}.
\]
Hence and Lemma \ref{connect}, provided $t$ is large,
\begin{align}\label{bound11}
  \sum_{\nu=1}^{\nu_0} &\sum_{\ell=1}^{\ell_0} P\Big( \big| \hT(v,w) - h(|w-v|) \big| 
    \geq A_j^2(v,w) \text{ for some } (v,w) \in R_{r,j}^{\nu,\ell} \Big) \notag\\
  &\leq c_3\frac{1}{\delta^j} \left( \frac{\rho\delta^{(1+\chi_1)/2}}{\beta^2} \right)^{(d-1)j} \notag\\
  &\hskip 1cm \cdot \sum_{\nu=1}^{\nu_0} (2^\nu t)^{(d-1)/2} \sum_{\ell=1}^{\ell_0} 2^{(d-1)\ell} C_{44} 
    \exp\left( -C_{45}\lambda^j 
    \left[ \frac{c_0}{7^j} 2^\nu t\frac{1}{\delta^{\chi_1j}} + c_12^{2\ell}\rho^{2j} \right] \right) \notag\\
  &\leq c_4\frac{1}{\delta^j} \left( \frac{\rho\delta^{(1+\chi_1)/2}}{\beta^2} \right)^{(d-1)j}
    \sum_{\nu=1}^{\nu_0} (2^\nu t)^{(d-1)/2} 
    \exp\left( -c_5\left( \frac{\lambda}{7\delta^{\chi_1}} \right)^j 2^\nu t \right) \notag\\
  &\leq c_4\frac{1}{\delta^j} \left( \frac{\rho\delta^{(1+\chi_1)/2}}{\beta^2} \right)^{(d-1)j}
    \exp\left( -c_5\left( \frac{\lambda}{7\delta^{\chi_1}} \right)^j t \right).
\end{align}
Second, for $1\leq \nu\leq \nu_0, \ell=0$ we can drop the second term in brackets in \eqref{Aj2}:
\begin{align*}
  \frac{A_j^2(v,w)}{\sigma(10\delta^{j-1}r)} &\geq \frac{1}{4}\left( \frac\lambda 7 \right)^j 
    \frac{\delta\mu}{36}2^\nu t\frac{\sigma_r}{\sigma(10\delta^{j-1}r)} \geq c_6
    \left( \frac{\lambda}{7\delta^{\chi_1}} \right)^j 2^\nu t, \quad (v,w)\in R_{r,j}^{\nu,0}
\end{align*}
and from \eqref{ell12}
\[
   \left| R_{r,j}^{\nu,0} \right| \leq \frac{c_7}{\delta^j} \left( \frac{(2^\nu t)^{1/2}\ell_2(j+1)}{\beta^j} \right)^{d-1} 
     \leq c_8 (2^\nu t)^{(d-1)/2} \frac{1}{\delta^j} \left( \frac{\rho\delta^{(1+\chi_1)/2}}{\beta^2} \right)^{(d-1)j}.
\]
Hence again using Lemma \ref{connect}, provided $t$ is large,
\begin{align}\label{bound10}
  \sum_{\nu=1}^{\nu_0} &P\Big( \big| \hT(v,w) - h(|w-v|) \big| 
    \geq A_j^2(v,w) \text{ for some } (v,w) \in R_{r,j}^{\nu,0} \Big) \notag\\
  &\leq c_8\frac{1}{\delta^j} \left( \frac{\rho\delta^{(1+\chi_1)/2}}{\beta^2} \right)^{(d-1)j}
    \sum_{\nu=1}^{\nu_0} (2^\nu t)^{(d-1)/2} 
    \exp\left( -c_9\left( \frac{\lambda}{7\delta^{\chi_1}} \right)^j 2^\nu t \right) \notag\\
  &\leq c_8\frac{1}{\delta^j} \left( \frac{\rho\delta^{(1+\chi_1)/2}}{\beta^2} \right)^{(d-1)j}
    \exp\left( -c_9\left( \frac{\lambda}{7\delta^{\chi_1}} \right)^j t \right).
\end{align}
Third, for $\nu=0,1\leq\ell\leq\ell_0(r,j+1)$ we have
\begin{align*}
  \frac{A_j^2(v,w)}{\sigma(10\delta^{j-1}r)} &\geq \frac{1}{4}\lambda^j 
    \left[ \frac{1}{7^j} \frac t3 \frac{\sigma_r}{\sigma(10\delta^{j-1}r)} 
    +  \frac{ \delta\mu(2^\ell \ell_2(j+1)K_0\beta^j\Delta_r)^2 }{360\delta^{j-1}r\sigma(10\delta^{j-1}r)} \right] \\
  &\geq c_{10}\left( \frac{\lambda}{7\delta^{\chi_1}} \right)^j t 
    + c_{11}2^{2\ell}\left( \lambda\rho^2 \right)^j,
      \quad (v,w)\in R_{r,j}^{0,\ell}
\end{align*}
and from \eqref{ell12}
\[
  \left| R_{r,j}^{0,\ell} \right| \leq \frac{c_{12}}{\delta^j} \left( \frac{t^{1/2}}{\beta^j} \right)^{d-1} 
     \left( 2^\ell \ell_2(j+1) \right)^{d-1} \leq c_{12} t^{(d-1)/2} 2^{(d-1)\ell}
     \frac{1}{\delta^j} \left( \frac{\rho\delta^{(1+\chi_1)/2}}{\beta^2} \right)^{(d-1)j}.
\]
so similarly to \eqref{bound10}
\begin{align}\label{bound01}
  \sum_{\ell=1}^{\ell_0} &P\Big( \big| \hT(v,w) - h(|w-v|) \big| 
    \geq A_j^2(v,w) \text{ for some } (v,w) \in R_{r,j}^{0,\ell} \Big) \notag\\
  &\leq c_{13}t^{(d-1)/2} \frac{1}{\delta^j} \left( \frac{\rho\delta^{(1+\chi_1)/2}}{\beta^2} \right)^{(d-1)j}
    \sum_{\ell=1}^{\ell_0} 2^{(d-1)\ell} \exp\left( -c_{14}\left[ \left( \frac{\lambda}{7\delta^{\chi_1}} \right)^j t 
    + 2^{2\ell}(\lambda\rho^2)^j \right] \right) \notag\\
  &\leq c_{13}\frac{1}{\delta^j} \left( \frac{\rho\delta^{(1+\chi_1)/2}}{\beta^2} \right)^{(d-1)j}
    \exp\left( -c_{14}\left( \frac{\lambda}{7\delta^{\chi_1}} \right)^j t \right).
\end{align}
Finally, for $k=\ell=0$ we can again drop the second term in brackets in \eqref{Aj2}, and we have analogously
\begin{align*}
  \frac{A_j^2(v,w)}{\sigma(10\delta^{j-1}r)} 
    &\geq \frac{1}{4}\left( \frac\lambda 7 \right)^j \frac t3 \frac{\sigma_r}{\sigma(10\delta^{j-1}r)} 
  \geq c_{15}\left( \frac{\lambda}{7\delta^{\chi_1}} \right)^j t, \quad (v,w)\in R_{r,j}^{0,0},
\end{align*}
and 
\[
  \left| R_{r,j}^{0,0} \right| \leq \frac{c_{16}}{\delta^j} \left( \frac{t^{1/2}\ell_2(j+1)}{\beta^j} \right)^{d-1} 
     \leq c_{17} t^{(d-1)/2} \frac{1}{\delta^j} \left( \frac{\rho\delta^{(1+\chi_1)/2}}{\beta^2} \right)^{(d-1)j}.
\]
so
\begin{align}\label{bound00}
  P\Big( &\big| \hT(v,w) - h(|w-v|) \big| 
    \geq A_j^2(v,w) \text{ for some } (v,w) \in R_{r,j}^{0,0} \Big) \notag\\
  &\leq c_{18}t^{(d-1)/2} \frac{1}{\delta^j} \left( \frac{\rho\delta^{(1+\chi_1)/2}}{\beta^2} \right)^{(d-1)j}
    \exp\left( -c_{19}\left( \frac{\lambda}{7\delta^{\chi_1}} \right)^j t \right) \notag\\
  &\leq c_{20}\frac{1}{\delta^j} \left( \frac{\rho\delta^{(1+\chi_1)/2}}{\beta^2} \right)^{(d-1)j}
    \exp\left( -c_{21}\left( \frac{\lambda}{7\delta^{\chi_1}} \right)^j t \right).
\end{align}
Since $t$ is large, \eqref{alllinks} follows from \eqref{relsize2} and \eqref{bound11}---\eqref{bound00}.
\end{proof}

For $j_2\leq j<j_1$ and $j_2<\ell\leq j-1$, a $(j,\ell)$\emph{th--scale joining 4--path} is a marked PG path $\Gamma: v^0\to v^1\to v^2\to v^3$ with $v^i\in\LL_{j+1} \cap G_r^+$, such that for some $j$th--scale hyperplane $H_a$, 
\begin{equation}\label{vloc}
  v^0 \in H_a,\quad v^1 \in H_{a+\delta^{j+1}r},\quad v^2 \in H_{a+\delta^{\ell+1}r}, \quad v^3 \in H_{a+\delta^\ell r}.
\end{equation}
In a $j$th--scale interval with left endpoint $a$, these four hyperplanes represent the hyperplane at $a$, a sandwiching $(j+1)$th--scale hyperplane, and two possible joining hyperplanes.  The pairs $(v^0,v^1),(v^1,v^2)$, $(v^2,v^3)$ are the \emph{links} of the joining 4--path.  Let \tred{$z^i$} be the point where $\Pi_{v^0v^3}$ intersects the hyperplane $H_{\fatdot}$ containing $v^i$. The \emph{intermediate path corresponding to} $\Gamma$ is $\tred{\Gamma^{int}}: z^0\to z^1\to z^2\to z^3$; note these points are collinear.  Recalling Remark \ref{novstar}, we say the $(j,\ell)$th--scale joining 4--path is \emph{internally bowed} if 
\begin{equation}\label{intbowed}
  \frac{|v^2-z^2|^2}{\delta^{\ell+1}r} \geq 2^{j-\ell}\frac{\delta}{128\mu}\left( \frac \lambda 7\right)^{j+1}
    \frac{\sigma(\delta^{\ell+1}r)}{\sigma(\delta^j r)} t^*(v^0)\sigma_r
\end{equation}
(compare to \eqref{genbowed2}) and
\begin{align}
  \frac{|v^1-z^1|^2}{\delta^{j+1}r} 
    \leq 2^{-(j-\ell-4)} \frac{\sigma(\delta^{j+1}r)}{\sigma(\delta^{\ell+1}r)}  \frac{|v^2-z^2|^2}{\delta^{\ell+1}r}
\end{align}
(compare to \eqref{firstinc}) or equivalently
\begin{align}\label{firstinc2}
  \left( \frac{|v^1-z^1|}{\Delta(\delta^{j+1}r)} \right)^2
    \leq 2^{-(j-\ell-4)} \left( \frac{|v^2-z^2|}{\Delta(\delta^{\ell+1}r)} \right)^2.
\end{align}

We define special allocations to deal with internally bowed paths. 
For each $(j,\ell)$th--scale joining 4--path $\Gamma: v^0\to v^1\to v^2\to v^3$ let
\[
  \tred{A_j^3(\Gamma)} = \frac{\delta\mu}{324}\left( \frac{|v^2-z^2|^2}{\delta^{\ell+1}r}
    + \lambda^j\frac{|(v^3-v^0)^*|^2}{\delta^\ell r} \right). 
\]
Note that by \eqref{vloc}, $\ell$ here is a function of $\Gamma$.
We will need an analog of Lemma \ref{sumbits} which enables us to deal with joining 4-paths as single units.  Normally for $\ell$th--scale links (length of order $\delta^\ell r$), where the fluctuations of the passage time $\hT(v^{i-1},v^i)$ are of order $\sigma(\delta^\ell r)$, we need $\ell$th--scale allocations (proportional to $\lambda^\ell$) to get a good bound from Lemma \ref{connect}, but in the next lemma we are able to use $j$th--scale allocations even for much longer $\ell$th--scale lengths, by taking advantage of bowedness.  

\begin{lemma}\label{sumbits2}
There exist constants $C_i$ as follows.  Let $j_2\leq j<j_1$ and $j_2<\ell\leq j-1$.  

(i) For every 
$(j,\ell)$th--scale joining 4--path $\Gamma: v^0\to v^1\to v^2\to v^3$ in $G_r^+$, and $C_{56}$ from \eqref{monotoneE3},
\begin{align}\label{joinalloc}
  6A_{j+1}^3(\Gamma) \leq \sum_{i=1}^3 A_{j+1}^2(v^{i-1},v^i) 
    + \frac{\delta\mu}{18} \Big[ \Upsilon_{Euc}(\Gamma) - |v^3-v^0| \Big] - 6C_{56}.
\end{align}

(ii)
\begin{align}\label{joinbound}
  P&\Big( \text{there exists an internally bowed $(j,\ell)$th--scale joining 4--path 
    $\Gamma: v^0\to v^1\to v^2\to v^3$} \notag\\
  &\hskip 1cm \text{in $G_r^+$, and $1\leq i\leq 3$, for which $\big| \hT(v^{i-1},v^i) - h(|v^i-v^{i-1}|) \big| \geq 
     A_{j+1}^3(\Gamma) \Big)$} \notag\\
  &\leq C_{72}\left( \frac{2}{\delta^{\chi_2+1}} \right)^{j-\ell} 
    \left( \frac{\delta}{\beta^2} \right)^j \exp\left( - C_{73} \left( \frac{\lambda}{\delta^{\chi_1}} \right)^{j+1} 
    2^{j-\ell}t \right).
\end{align}

(iii)
\begin{align}\label{joinbound2}
  P&\Bigg( \text{there exists an internally bowed $(j,\ell)$th--scale joining 4--path $\Gamma$ in $G_r^+$, with} \notag\\
  &\hskip 1cm \text{corresponding intermediate path $\Gamma^{int}: u^0\to u^1\to u^2\to u^3$,} \notag\\
  &\hskip 1cm \text{and $1\leq i\leq 3$, for which $\big| \hT(u^{i-1},u^i) - h(|u^i-u^{i-1}|) \big| 
    \geq A_{j+1}^3(\Gamma^{int}) \Bigg)$} \notag\\
  &\leq C_{72}\left( \frac{2}{\delta^{\chi_2+1}} \right)^{j-\ell} 
    \left( \frac{\delta}{\beta^2} \right)^j \exp\left( - C_{73} \left( \frac{\lambda}{\delta^{\chi_1}} \right)^{j+1} 
    2^{j-\ell}t \right).
\end{align}
\end{lemma}

\begin{proof}
{\it (i).} Using \eqref{adddist},
\begin{align}
  6A_{j+1}^3(\Gamma) &\leq \frac{\delta\mu}{54}\left( 3\Big[ \Upsilon_{Euc}(v^0,v^2,v^3)
    - |v^3-v^0| \Big] + 3\lambda^j\Big[ |v^3-v^0| - (v^3-v^0)_1 \Big] \right) \notag\\
  &\leq \frac{\delta\mu}{18} \left( \Big[ \Upsilon_{Euc}(\Gamma)
    - |v^3-v^0| \Big] + \frac{1}{2}\lambda^j\Psi(\Gamma) \right) \notag\\
  &\leq \sum_{i=1}^3 A_{j+1}^2(v^{i-1},v^i) - \frac14\left( \frac \lambda 7 \right)^j t\sigma_r 
    + \frac{\delta\mu}{18} \Big[ \Psi_{Euc}(\Gamma) - |v^3-v^0| \Big] \notag\\
  &\leq \sum_{i=1}^3 A_{j+1}^2(v^{i-1},v^i) - 6C_{56}
    + \frac{\delta\mu}{18} \Big[ \Psi_{Euc}(\Gamma) - |v^3-v^0| \Big],
\end{align}
where the last inequality follows from $j\leq j_1 = O(\log log r)$.

{\it (ii) and (iii).} The proof of (iii) is a slightly simplified version of the proof of (ii), so we only prove (ii). We proceed as in Lemma \ref{sumbits}.  We decompose the set of joining 4--paths according to the sizes of $|(v^0)^*|$ and $|(v^3-v^0)^*|$, and the degree of bowedness, measured by the left side of \eqref{intbowed}.  From Remark \ref{novstar}, every internally bowed $(j,\ell)$th--scale joining 4--path $\Gamma: v^0\to v^1\to v^2\to v^3$ in $G_r^+$ satisfies
\begin{equation}\label{novstar1}
  \frac{|v^2-z^2|^2}{\delta^{\ell+1}r} \geq \frac{\delta}{128\mu}\left( \frac \lambda 7\right)^{j+1} t^*(v^0)\sigma_r 
    \frac{\sigma(\delta^{\ell+1}r)}{\sigma(\delta^j r)}.
\end{equation}
Thus define for $\nu,m_2,m_3\geq 1$
\begin{align}\label{Rdef}
  &\tred{R_{r,j,\ell}(\nu,m_2,m_3)} \notag\\
  &= \Bigg\{ \Gamma: v^0\to v^1\to v^2\to v^3\Big|\, \Gamma 
    \text{ is an internally bowed $(j,\ell)$th--scale joining 4--path in $G_r^+$ such that} \notag\\
  &\qquad (2^{\nu-1}t)^{1/2}\Delta_r \leq |(v^0)^*| \leq (2^\nu  t)^{1/2}\Delta_r, \notag\\
  &\qquad 2^{m_2-1}2^{j-\ell}\frac{\delta}{128\mu}\left( \frac \lambda 7\right)^{j+1} t^*(v^0)\sigma_r 
    \frac{\sigma(\delta^{\ell+1}r)}{\sigma(\delta^j r)} < \frac{|v^2-z^2|^2}{\delta^{\ell+1}r} 
    \leq 2^{m_2}2^{j-\ell}\frac{\delta}{128\mu}\left( \frac \lambda 7\right)^{j+1} 
    t^*(v^0)\sigma_r\frac{\sigma(\delta^{\ell+1}r)}{\sigma(\delta^j r)},  \notag\\
  &\qquad 2^{m_3-1}2^{j-\ell}t^*(v^0)\sigma_r \frac{\sigma(\delta^{\ell}r)}{\sigma(\delta^{j+1}r)} < 
    \frac{|(v^3-v^0)^*|^2}{\delta^\ell r} 
    \leq 2^{m_3}2^{j-\ell}t^*(v^0)\sigma_r \frac{\sigma(\delta^{\ell}r)}{\sigma(\delta^{j+1}r)} \Bigg\}.
\end{align}
$\tred{R_{r,j,\ell}(0,m_2,m_3)}$ is defined similarly with the first condition in \eqref{Rdef} replaced by $|(v^0)^*| \leq t^{1/2}\Delta_r$.  $\tred{R_{r,j,\ell}(\nu,m_2,0)}$ is defined similarly with the last condition in \eqref{Rdef} replaced by 
\[
  \frac{|(v^3-v^0)^*|^2}{\delta^\ell r} 
    \leq 2^{j-\ell}t^*(v^0)\sigma_r \frac{\sigma(\delta^\ell r)}{\sigma(\delta^{j+1}r)}.
\]
Every internally bowed joining 4--path in $G_r^+$ is in one of these classes, since by \eqref{novstar1} we don't need classes with $m_2=0$.

Suppose $\Gamma: v^0\to v^1\to v^2\to v^3$ is in $R_{r,j,\ell}(\nu,m_2,m_3)$. We have
\begin{align}\label{link1}
  \frac{ A_{j+1}^3(\Gamma)}{\sigma(\delta^{\ell}r)} &\geq 
    \frac{\delta\mu}{324}2^{j-\ell}
    \frac{t^*(v^0)\sigma_r}{\sigma(\delta^j r)} \left( \frac{\delta}{128\mu}\left( \frac \lambda 7\right)^{j+1}  2^{m_2-1}  
    \frac{\sigma(\delta^{\ell+1}r)}{\sigma(\delta^{\ell}r)} + \lambda^{j+1} 2^{m_3-1} \right) \notag\\
  &\geq c_0 \left( \frac{\lambda}{7\delta^{\chi_1}} \right)^{j+1} 2^{j-\ell}(2^{m_2} + 2^{m_3}) 2^\nu  t.
\end{align}
Hence for each $1\leq i\leq 3$, by Lemma \ref{connect},
\begin{align}\label{linkbound}
  P\Big( \big| \hT(v^{i-1},v^i) &- h(|v^i-v^{i-1}|) \big| \geq A_{j+1}^3(\Gamma) \Big)
    \leq C_{44} \exp\left( - C_{45} \frac{ A_{j+1}^3(\Gamma)}{2\sigma(\delta^{\ell}r)} \right) \notag\\
  &\leq C_{44} \exp\left( - c_1 \left( \frac{\lambda}{7\delta^{\chi_1}} \right)^{j+1} 
    2^{j-\ell}(2^{m_2} + 2^{m_3}) 2^\nu  t \right).
\end{align}
Regarding the size of $R_{r,j,\ell}(\nu,m_2,m_3)$, the number of possible $v^0$ (necessarily in $\LL_{j+1}$) in a given $j$th--scale hyperplane is at most
\[
  2\left( \frac{2(2^\nu t)^{1/2}}{K_0\beta^{j+1}} \right)^{d-1}.
\]
The upper bound for $|(v^0)^*|$ in \eqref{Rdef} gives $t^*(v^0)\leq c_2 2^\nu t$, which with the
last upper bound in \eqref{Rdef} yields 
\[
  \left( \frac{|(v^3-v^0)^*|}{\Delta_r} \right)^2 \leq c_32^{m_3} \left( \frac{2}{\delta^{\chi_2}} \right)^{j-\ell} \delta^\ell 2^\nu t,
\]
so for a given $v^0$ the number of possible $v^3$ is at most
\[
  2\left( c_4 2^{m_3} \left( \frac{2}{\delta^{\chi_2}} \right)^{j-\ell} \frac{1}{\beta^{2(j+1)}} \delta^\ell 2^\nu t \right)^{(d-1)/2}
    = 2\left( \frac{c_4}{\beta^2} 2^{m_3} \left( \frac{2}{\delta^{\chi_2+1}} \right)^{j-\ell} 
    \left( \frac{\delta}{\beta^2} \right)^j 2^\nu t \right)^{(d-1)/2}.
\]
The upper bound for $|v^2-z^2|^2/\delta^{\ell+1}r$ in \eqref{Rdef} implies
\[
   \left( \frac{|(v^2-z^2|}{\Delta_r} \right)^2 \leq c_5 2^{m_2} \left( \frac{2}{\delta^{\chi_2}} \right)^{j-\ell} 
     \delta^\ell 2^\nu t
\]
and for given $v^0,v^3$, the point $z^2$ is determined, and then the number of possible $v^2$ is at most
\[
  2\left( \frac{c_6}{\beta^2} 2^{m_2} \left( \frac{2}{\delta^{\chi_2+1}} \right)^{j-\ell} 
    \left( \frac{\delta}{\beta^2} \right)^j 2^\nu t \right)^{(d-1)/2}.
\]
Combining these, we see that 
\[
  |R_{r,j,\ell}(\nu,m_2,m_3)| \leq \frac{c_7}{\delta^j} \left( \frac{(2^\nu t)^{1/2}}{\beta^{j+1}} \right)^{d-1} 
    \left( 2^{m_2/2} 2^{m_3/2} \left( \frac{2}{\delta^{\chi_2+1}} \right)^{j-\ell} 
    \left( \frac{\delta}{\beta^2} \right)^j 2^\nu t \right)^{d-1}.
\]
Multiplying this by \eqref{linkbound} and summing over $\nu,m_2,m_3$ gives \eqref{joinbound}.
\end{proof}

\subsection{Step 5. First stage of the $(j_1-1)$th--scale (second) iteration of coarse--graining: shifting to the $(j_1-1)$th--scale grid.}  
The current marked PG path at the start of the $(j_1-1)$th--scale iteration step is $\Gamma_{xy}^{j_1,m}$. We rename it now as 
\[
  \tred{\Gamma_{xy}^{j_1-1,0}}: p^0\to p^1\to \cdots\to p^m\to p^{m+1}.
\]
We shift certain points to the $(j_1-1)$th grid; the procedure is somewhat different from the $j_1$th--scale iteration step, as we are starting from a $j_1$th--scale marked PG path $\Gamma_{xy}^{j_1-1,0}$.  Fix $(x,y)\in X_r$ and, recalling $\mH_{xy} = \{H_{s_i},1\leq i\leq m\}$, let 
\begin{align*}
  \tred{I_{xy}} &= \Big\{ i\in\{1,\dots,m\}: H_{s_i} \text{ is a non-incidental $(j_1-1)$th--scale hyperplane in $\mH_{xy}$} \Big\}, 
\end{align*}
then relabel $\{s_i:i\in I_{xy}\cup\{1,m\}\}$ as \tred{$a_1<\cdots<a_n$}, so $H_{a_1}=H_{s_1}$ and $H_{a_n}=H_{s_m}$ are the terminal $j_1$th--scale hyperplanes, and define indices \tred{$\gamma(N)$} by $a_N=s_{\gamma(N)}$, so the $p^{\gamma(N)}, 2\leq N\leq n-1,$ are the marked points in non-incidental $(j_1-1)$th--scale hyperplanes. Observe that every $(j_1-1)$th--scale interval, including terminal ones, has at least one $j_1$th--scale hyperplane in its interior; hence $\gamma(N+1)\geq \gamma(N)+2$ for all $1\leq N<n$.

For each non-incidental $(j_1-1)$th--scale hyperplane in $\mH_{xy}$ let $\tred{b^N}=V_{j_1-1}(p^{\gamma(N)})$.  Also let $\tred{b^0}=p^0=\hat x,\tred{b^1}=p^1,\tred{b^n}=p^m$, and $\tred{b^{n+1}}=p^{m+1}=\hat y$. We shift to the $(j_1-1)$th scale grid in each non-incidental $(j_1-1)$th--scale hyperplane $H_{s_i},i\in I_{xy}$, replacing $p^{\gamma(N)}$ with $b^N$ for $2\leq N\leq n-1$, to create the updated path which we denote \tred{$\Gamma_{xy}^{j_1-1,1}$}.
Letting $\tred{\tp^i} = \mkm_{s_i}(\Gamma_{xy}^{j_1-1,1})$ (so $\tp^i=p^i$ if $i\notin I_{xy},\, \tp^i = b^{\eta^{-1}(i)}$ if $i\in I_{xy}$) , we may equivalently write this as
\[
  \Gamma_{xy}^{j_1-1,1}: \tp^0\to \tp^1\to \tp^2\to \cdots\to \tp^{m+1}.
\]
The black path from $u'$ to $v'$ in Figure \ref{Remark4-4} illustrates a segment of such a path, with $u'=b^{N-1}=\tp^{\gamma(N-1)},v'=b^N=\tp^{\gamma(N)}$ for some $N$ there; the vertices between $u'$ and $v'$ are points $p^i=\tp^i$ with $\gamma(N-1)<i<\gamma(N)$.
In the second stage of the iteration we will remove marked points in non-terminal $j_1$th--scale hyperplanes to create the path
\[
  \tred{\Gamma_{xy}^{j_1-1,2}}: b^0 \to b^1\to \cdots \to b^n\to b^{n+1}.
\]

By Lemma \ref{alltrans} we have
\begin{equation}\label{Atotal2}
  \sum_{i\in I_{xy}} A_{j_1}^1(\Gamma_{xy}^{j_1-1,0},i) \leq \frac{2}{3}\lambda^{j_1}\Big[ t\sigma_r 
   + \delta\mu\Big( \Upsilon_{Euc}\left( \Gamma_{xy}^{j_1-1,0} \right) - (\hat y - \hat x)_1 \Big)
    \Big], 
\end{equation}
while similarly to \eqref{Philoss3} and \eqref{Philoss3ha}, 
\begin{equation}\label{upgap}
  \Big| \Upsilon_{Euc}\left( \Gamma_{xy}^{j_1-1,1} \right) -\Upsilon_{Euc}\left( \Gamma_{xy}^{j_1-1,0} \right) \Big|
    \leq \frac{ \lambda^{j_1}}{2\mu} \left[ \frac{t\sigma_r}{3}
    + \delta\mu\Big( \Upsilon_{Euc}(\Gamma_{xy}^{j_1-1,0}) - (\hat y - \hat x)_1 \Big) \right]
\end{equation}
and
\begin{equation}\label{upgaph}
  \Big| \Upsilon_h\left( \Gamma_{xy}^{j_1-1,1} \right) -\Upsilon_h\left( \Gamma_{xy}^{j_1-1,0} \right) \Big|
    \leq \lambda^{j_1} \left[ \frac{t\sigma_r}{3}
    + \delta\mu\Big( \Upsilon_{Euc}(\Gamma_{xy}^{j_1-1,0}) - (\hat y - \hat x)_1 \Big) \right].
\end{equation}
(The only notable change from the derivation of \eqref{Philoss3} and \eqref{Philoss3ha} is that now, since we are shifting to the $(j_1-1)$th scale, we have the bound
\[
  |(p^i-\tp^i)^*| \leq \sqrt{d-1}K_0\beta^{j_1-1}\Delta_r,
\]
larger only by a constant $\beta^{-1}$ compared to the analogous bound in \eqref{adddist4}--\eqref{adddist2}.)
Then \eqref{Atotal2} and \eqref{upgaph} give
\begin{equation}\label{Atotal3}
  \sum_{i\in I_{xy}} A_{j_1}^1(\Gamma_{xy}^{j_1-1,0},i) + \Upsilon_h\left( \Gamma_{xy}^{j_1-1,1} \right) 
    -\Upsilon_h\left( \Gamma_{xy}^{j_1-1,0} \right) \leq \lambda^{j_1}\left[ t\sigma_r 
   + \frac 53 \delta\mu\Big( \Upsilon_{Euc}\left( \Gamma_{xy}^{j_1-1,0} \right) - (\hat y - \hat x)_1 \Big) \right].
\end{equation}
From \eqref{upgap} and \eqref{Atotal3}, analogously to \eqref{mainsplit}, we have for the event on the right in \eqref{firstit} that
\begin{align}\label{mainsplit2}
  &\bigg\{ \Upsilon_{\hT}\left( \Gamma_{xy}^{j_1-1,0} \right) - h((y-x)_1) 
    \leq -\left( 1-2\lambda^{j_1} \right) t\sigma_r + 4\delta\mu\lambda^{j_1}\Big( 
    \Upsilon_{Euc}\left( \Gamma_{xy}^{j_1-1,0} \right) - (\hat y - \hat x)_1 \Big) \bigg\} \notag\\
  &\subset \bigg\{ \Upsilon_{\hT}\left( \Gamma_{xy}^{j_1-1,1} \right) - h((y-x)_1) \leq -\left( 1-4\lambda^{j_1} \right) t\sigma_r 
    + 6\delta\mu\lambda^{j_1}\Big( \Upsilon_{Euc}\left( \Gamma_{xy}^{j_1-1,0} \right) - (\hat y - \hat x)_1 \Big)\notag\\
  &\qquad\qquad 
    + \delta\mu\Big( \Upsilon_{Euc}\left( \Gamma_{xy}^{j_1-1,1} \right) 
    -\Upsilon_{Euc}\left( \Gamma_{xy}^{j_1-1,0} \right) \Big) \bigg\} \notag\\
  &\qquad \bigcup \left\{  \Upsilon_{\hT}\left( \Gamma_{xy}^{j_1-1,0} \right) - \Upsilon_{\hT}\left( \Gamma_{xy}^{j_1-1,1} \right) 
    \leq \Upsilon_h\left( \Gamma_{xy}^{j_1-1,0} \right) 
    -\Upsilon_h\left( \Gamma_{xy}^{j_1-1,1} \right) -\sum_{i\in I_{xy}} A_{j_1}^1(\Gamma_{xy}^{j_1-1,0},i) \right\}.
\end{align}

\begin{remark}\label{tracking}
In view of Remark \ref{outline}, in the context of shifting to the grid, we can think of ``tracking'' as corresponding showing that a probability of form
\begin{align}\label{shiftnotr}
  P\Bigg( \Upsilon_{\hT}\left( \Gamma_{xy}^{j_1-1,0} \right) - &\Upsilon_{\hT}\left( \Gamma_{xy}^{j_1-1,1} \right) 
    \leq \delta\Big[ \Upsilon_h\left( \Gamma_{xy}^{j_1-1,0} \right) - \Upsilon_h\left( \Gamma_{xy}^{j_1-1,1} \right) \Big]
    - \text{ (allocations)} \notag\\
  &\text{ for some } (x,y)\in X_r \Bigg),
\end{align}
is small, or similarly with $\Upsilon_h$ replaced by its approximation $\mu\Upsilon_{Euc}$.  In \eqref{mainsplit2} the particular allocations chosen are $A_{j_1}^1(\Gamma_{xy}^{j_1-1,0},i)$.  The event in \eqref{shiftnotr} says that, as the shifting--to--the--grid process changes the $h$--length of the path, the change in $\hT$--sum fails to track even a small fraction $\delta$ of the change in $h$--length, to within the error given by the allocations.  Further, in \eqref{firstit} and on the left in \eqref{mainsplit2}, 
\begin{equation}\label{accum}
  2\lambda^{j_1} t\sigma_r + 4\delta\mu\lambda^{j_1}\Big( 
    \Upsilon_{Euc}\left( \Gamma_{xy}^{j_1-1,0} \right) - (\hat y - \hat x)_1 \Big)
\end{equation}
represents roughly the accumulated error allocations used in the completed $j_1$th--scale iteration.  Thus 
what \eqref{upgap} and \eqref{upgaph} say is that for the current shift to the grid, the change in $h$--length (or in its approximation $\mu\Upsilon_{Euc}$) is negligible in the sense of being small relative to the accumulated allocations, making tracking only a minor issue here.  This negligible--ness will not be valid for the marked--point--removal stage of the iteration, however.  

It should also be noted that from the first to second lines in \eqref{mainsplit2}, the coefficients 2, 4 become 4, 6. The difference, if it is negative, represents a reduction of the original value 

Observe that failure to track is a one--sided phenomenon---we only care about failure of the $\hT$--sum to track decreases in $h$--length created by shifting to the grid, not increases. In the point--removal stage, the $h$--length always decreases, and we care that the $\hT$--sum decreases at least proportionally.
\end{remark}

Using Lemma \ref{sumbits} we have for the last event in \eqref{mainsplit2}
\begin{align}\label{badtrans}
  P&\Bigg(  \Upsilon_{\hT}\left( \Gamma_{xy}^{j_1-1,0} \right) - \Upsilon_{\hT}\left( \Gamma_{xy}^{j_1-1,1} \right) 
    \leq \Upsilon_h\left( \Gamma_{xy}^{j_1-1,0} \right) 
    -\Upsilon_h\left( \Gamma_{xy}^{j_1-1,1} \right) - \sum_{i\in I_{xy}} A_{j_1}^1(\Gamma_{xy}^{j_1-1,0},i) \notag\\
  &\hskip 1.5cm \text{ for some } (x,y)\in X_r;\ \omega\notin J^{(0)}\cup J^{(1c)} \Bigg) \notag\\
  &\leq P\Bigg( \max\Big( \big|\hT(\tp^{i-1},\tp^i) - h(|\tp^i-\tp^{i-1}|)\big|, \big|\hT(p^{i-1},p^i) - h(|p^i-p^{i-1}|)\big| \Big) 
    \geq \frac 14 A_{j_1}^1(\Gamma_{xy}^{j_1-1,0},i)  \notag\\
  &\hskip 1.5cm \text{ for some $1\leq i\leq m+1$ and $(x,y)\in X_r$ with $\{i-1,i\}\cap I_{xy}\neq\emptyset$};\, 
    \omega\notin J^{(0)}\cup J^{(1c)} \Bigg) \notag\\
  &\leq P\Bigg( \big| \hT(v,w) - h(|w-v|) \big| \geq  A_{j_1}^2(v,w) \notag\\
  &\hskip 1.5cm \text{ for some $j_1$th--scale transition $v\to w$ with } v,w\in G_r^+\cap\LL_{j_1} \Bigg) \notag\\
  &\leq  C_{70} \exp\left( -C_{71}\left( \frac{\lambda}{7\delta^{\chi_1}} \right)^{j_1} t \right),
\end{align}
establishing tracking as desired.  Combining this with \eqref{firstit} and \eqref{mainsplit2} yields
\begin{align}\label{secondit}
  &P\Big( T(x,y) \leq h((y-x)_1) - t\sigma_r \text{ for some } (x,y)\in X_r \Big) \notag\\
  &\leq P\Bigg( \Upsilon_{\hT}(\Gamma_{xy}^{j_1-1,1}) - h((y-x)_1) \leq -\left( 1-4\lambda^{j_1} \right) t\sigma_r
    + 6\delta\mu\lambda^{j_1}\Big(  \Upsilon_{Euc}\left( \Gamma_{xy}^{j_1-1,0} \right) - (\hat y - \hat x)_1 \Big)  \notag\\
  &\hskip 2cm + \delta\mu\Big( \Upsilon_{Euc}\left( \Gamma_{xy}^{j_1-1,1} \right) 
    -\Upsilon_{Euc}\left( \Gamma_{xy}^{j_1-1,0} \right) \Big) \text{ for some } 
    (x,y)\in X_r;\ \omega\notin J^{(0)}(c_{29})\cup J^{(1c)} \Bigg) \notag\\
  &\hskip 1.2cm + c_{32}e^{-c_{33}t} + C_{70} \exp\left( -C_{71}\left( \frac{\lambda}{7\delta^{\chi_1}} \right)^{j_1} t \right).
\end{align}

Analogously to the comment after \eqref{firstit}, the increase of the coefficients 2, 4 in \eqref{firstit} to be 4, 6 in \eqref{secondit}, together with the introduction of the additional term with coefficient $\delta\mu$, represent a further reduction taken from the original bound $t\sigma_r$ in \eqref{Qrunif2}, allocated to bound errors in the first stage of the $(j_1-1)$th--scale iteration, just completed.   

\subsection{Step 6. Second stage of the $(j_1-1)$th--scale (second) coarse--graining iteration: removing marked points.} For the next update of the current marked PG path $\Gamma_{xy}^{j_1-1,1}$, we remove all the marked points in non-terminal $j_1$th--scale hyperplanes to create the updated marked PG path
\[
  \tred{\Gamma_{xy}^{j_1-1,2}: b^0 \to b^1\to \cdots \to b^n\to b^{n+1}}.
\]
Here we recall that $b^0 = \hat x, b^{n+1}=\hat y,\,b^1$ and $b^n$ lie in terminal $j_1$th--scale hyperplanes, and $b^2,\dots,b^{n-1}$ lie in $(j_1-1)$th--scale hyperplanes.  As with $p^i,\tp^i$, etc., $n$ and $b^i$ should be viewed as functions of $x,y,\omega$. 

In this second stage, the removal  of marked points always reduces the Euclidean length of the path (see Figure \ref{Dominant}), meaning $\Upsilon_{Euc}( \Gamma_{xy}^{j_1-1,1}) -\Upsilon_{Euc}( \Gamma_{xy}^{j_1-1,2} )\geq 0$, and on average the reduction in passage time, $\Upsilon_{\hT}(\Gamma_{xy}^{j_1-1,1}) - \Upsilon_{\hT}(\Gamma_{xy}^{j_1-1,2})$, should be about $\mu$ times the reduction in length.  Here a tracking failure means the actual reduction in passage time is at most $\delta\mu$ times the reduction in length, to within an allocated error; see the last probability in \eqref{secondit2}.

From \eqref{upgap} we have
\[
  \Upsilon_{Euc}\left( \Gamma_{xy}^{j_1-1,2} \right) - (\hat y - \hat x)_1
    \leq \Upsilon_{Euc}\left( \Gamma_{xy}^{j_1-1,1} \right) - (\hat y - \hat x)_1 
    \leq \frac 12 t\sigma_r + \frac 32 \Big( \Upsilon_{Euc}\left( \Gamma_{xy}^{j_1-1,0} \right) - (\hat y - \hat x)_1 \Big)
\]
and therefore, assuming $\delta$ is small,
\[
  \frac 13 \Big[ 2t\sigma_r 
    + \delta\mu\Big( \Upsilon_{Euc}\left( \Gamma_{xy}^{j_1-1,1} \right) - (\hat y - \hat x)_1 
    +  \Upsilon_{Euc}\left( \Gamma_{xy}^{j_1-1,2} \right) - (\hat y - \hat x)_1 \Big) \Big]
    \leq t\sigma_r + \delta\mu \Big(  \Upsilon_{Euc}\left( \Gamma_{xy}^{j_1-1,0} \right) - (\hat y - \hat x)_1 \Big).
\]
From this and \eqref{secondit} we obtain the setup to establish tracking:
\begin{align}\label{secondit2}
  P&\Big( T(x,y) \leq h((y-x)_1) - t\sigma_r \text{ for some } (x,y)\in X_r \Big) \notag\\
  &\leq P\bigg( \Upsilon_{\hT}(\Gamma_{xy}^{j_1-1,1}) - h((y-x)_1) \leq -\left( 1-4\lambda^{j_1} \right) t\sigma_r
    + 6\delta\mu\lambda^{j_1}\Big(  \Upsilon_{Euc}\left( \Gamma_{xy}^{j_1-1,0} \right) - (\hat y - \hat x)_1 \Big)  \notag\\
  &\hskip 2cm + \delta\mu\Big( \Upsilon_{Euc}\left( \Gamma_{xy}^{j_1-1,1} \right) 
    -\Upsilon_{Euc}\left( \Gamma_{xy}^{j_1-1,0} \right) \Big) \text{ for some } 
    (x,y)\in X_r;\ \omega\notin J^{(0)}(c_{29})\cup J^{(1c)} \bigg) \notag\\
  &\hskip 1.2cm + c_{32}e^{-c_{33}t} 
    + C_{70} \exp\left( -C_{71}\left( \frac{\lambda}{7\delta^{\chi_1}} \right)^{j_1} t \right) \notag\\
  &\leq P\bigg( \Upsilon_{\hT}(\Gamma_{xy}^{j_1-1,2}) - h((y-x)_1) \leq -\left( 1-5\lambda^{j_1} \right) t\sigma_r
    + 7\delta\mu\lambda^{j_1}\Big(  \Upsilon_{Euc}\left( \Gamma_{xy}^{j_1-1,0} \right) - (\hat y - \hat x)_1 \Big)  \notag\\
  &\hskip 2cm + \delta\mu\Big( \Upsilon_{Euc}\left( \Gamma_{xy}^{j_1-1,2} \right) 
    -\Upsilon_{Euc}\left( \Gamma_{xy}^{j_1-1,0} \right) \Big)
    \text{ for some } (x,y)\in X_r;\ \omega\notin J^{(0)}(c_{29})\cup J^{(1c)} \bigg) \notag\\
  &\hskip .8 cm + P\bigg( \Upsilon_{\hT}(\Gamma_{xy}^{j_1-1,1}) - \Upsilon_{\hT}(\Gamma_{xy}^{j_1-1,2}) \notag\\
  &\hskip 2 cm \leq - \frac{1}{3}\lambda^{j_1} \Big[ 2t\sigma_r 
    + \delta\mu\Big( \Upsilon_{Euc}\left( \Gamma_{xy}^{j_1-1,1} \right) - (\hat y - \hat x)_1 
    +  \Upsilon_{Euc}\left( \Gamma_{xy}^{j_1-1,2} \right) - (\hat y - \hat x)_1 \Big) \Big] \notag\\
  &\hskip 2cm + \delta\mu\Big( \Upsilon_{Euc}\left( \Gamma_{xy}^{j_1-1,1} \right) 
    -\Upsilon_{Euc}\left( \Gamma_{xy}^{j_1-1,2} \right) \Big)
    \text{ for some } (x,y)\in X_r;\ \omega\notin J^{(0)}(c_{29})\cup J^{(1c)} \bigg) \notag\\
  &\hskip 1.2cm + c_{32}e^{-c_{33}t} 
    + C_{70} \exp\left( -C_{71}\left( \frac{\lambda}{7\delta^{\chi_1}} \right)^{j_1} t \right).
\end{align} 

Let us consider the contribution to the difference of sums $\Upsilon_{Euc}\left( \Gamma_{xy}^{j_1-1,1} \right) -\Upsilon_{Euc}\left( \Gamma_{xy}^{j_1-1,2} \right)$, appearing in the last probability in \eqref{secondit2}, from a single $(j_1-1)$th--scale interval $\tred{I_N}=[b^N,b^{N+1}] = [\tp^{\gamma(N)},\tp^{\gamma(N+1)}]$.  Removing the marked points from the hyperplanes in the interior of $I_N$ changes the marked PG path from the \emph{full path}
\[
  \tred{\Gamma^{N,full}}:\ \tp^{\gamma(N)}\to \tp^{\gamma(N)+1}\to \tp^{\gamma(N)+2}\to \tp^{\gamma(N+1)} 
\]
to the \emph{direct path} $\tp^{\gamma(N)}\to \tp^{\gamma(N+1)}$ (that is, $b^N\to b^{N+1}$.)  We then have
\begin{equation}\label{intsplit}
  \Upsilon_{\hT}(\Gamma_{xy}^{j_1-1,1}) - \Upsilon_{\hT}(\Gamma_{xy}^{j_1-1,2}) = \sum_{N=1}^{n+1}
    \left[ \Upsilon_{\hT}\left( \Gamma^{N,full} \right) - \hT( \tp^{\gamma(N)},\tp^{\gamma(N+1)} ) \right].
\end{equation}
Given $n\geq 1$ and $\Gamma: u_1\to\dots\to u_n$ with all $u_i$ in some grid $\LL_j$, we define
\[
    S_j(u_0,\dots,u_n) = \tred{S_j(\Gamma)} = \sum_{i=1}^n A_j^2(u_{i-1},u_i).
\]
From Lemma \ref{alltrans} we have
\begin{align}\label{Sjbound}
  \sum_{N=0}^n S_{j_1}\left( \Gamma^{N,full} \right) &= \sum_{i=1}^{m+1} A_{j_1}^2(\tp^{i-1},\tp^i) 
    \leq \frac{1}{3}\lambda^j \Big[ t\sigma_r 
    + \delta\mu\Big(\Upsilon_{Euc}\left( \Gamma_{xy}^{j_1-1,1} \right) - (\hat y - \hat x)_1) \Big) \Big],
\end{align}
so the last probability in \eqref{secondit2} is bounded above by
\begin{align}\label{trackprob}
  P\Bigg( &\Upsilon_{\hT}(\Gamma_{xy}^{j_1-1,1}) - \Upsilon_{\hT}(\Gamma_{xy}^{j_1-1,2}) 
    \leq \delta\mu\Big( \Upsilon_{Euc}\left( \Gamma_{xy}^{j_1-1,1} \right) 
    -\Upsilon_{Euc}\left( \Gamma_{xy}^{j_1-1,2} \right) \Big) \notag\\
  &- \sum_{N=0}^n S_{j_1}\left( \Gamma^{N,full} \right) - \frac{1}{3}\lambda^{j_1} \Big[ t\sigma_r 
    + \delta\mu\Big( \Upsilon_{Euc}\left( \Gamma_{xy}^{j_1-1,2} \right) - (\hat y - \hat x)_1 \Big) \Big] \notag\\
  &\text{ for some } (x,y)\in X_r;\ \omega\notin J^{(0)}(c_{29})\cup J^{(1c)} \Bigg).
\end{align}
This is the tracking--failure event (see Remark \ref{outline}) for the marked--point--removal stage of the iteration, and our main task is to bound its probability. The last of the 3 terms on the right inside the probability can be viewed as part of the allocation of allowed errors. As noted in Remark \ref{tracking}, the quantity
\[
  4\lambda^{j_1} t\sigma_r
    + 6\delta\mu\lambda^{j_1}\Big(  \Upsilon_{Euc}\left( \Gamma_{xy}^{j_1-1,0} \right) - (\hat y - \hat x)_1 \Big)
\]
in the second probabilty in \eqref{secondit2} represents the accumulated error allocations used in the 1.5 iterations completed so far; the allocations for the present stage increase the 4 and 6 to 5 and 7 in the third probability in \eqref{secondit2}.  Our ability to bound the tracking--failure event is what allows us to replace the quantity
\[
  \delta\mu\Big( \Upsilon_{Euc}\left( \Gamma_{xy}^{j_1-1,1} \right) 
    -\Upsilon_{Euc}\left( \Gamma_{xy}^{j_1-1,0} \right) \Big) \quad\text{with}\quad
    \delta\mu\Big( \Upsilon_{Euc}\left( \Gamma_{xy}^{j_1-1,2} \right) 
    -\Upsilon_{Euc}\left( \Gamma_{xy}^{j_1-1,0} \right) \Big)
\]
in that second probability in \eqref{secondit2} to obtain the third probability.  We need each iteration to involve similar such replacement, so that when the iterations are complete this term becomes
\[
  \delta\mu\Big( \Upsilon_{Euc}\left( \Gamma_{xy}^{CG} \right) 
    -\Upsilon_{Euc}\left( \Gamma_{xy}^{j_1-1,0} \right) \Big)
\]
which is typically negative and can in part cancel the accumulated error allocations (see \eqref{cancel}.)

Let $\tred{\hw^i}$ be the point $\Pi_{\tp^{\gamma(N)},\tp^{\gamma(N+1)}}\cap H_{(\tp^i)_1}$; note $\hw^i$ does not necessarily lie in any $j$th--scale grid. Let $\tred{w_\perp^i}$ be the orthogonal projection of $\tp^i$ into $\Pi_{\tp^{\gamma(N)},\tp^{\gamma(N+1)}}$, noting that by \eqref{forward}, $|w_\perp^i-\hw^i|$ is much smaller than $|\tp^i-\hw^i|$.  (Note the indexing of marked points differs here from that used in Step 1 in defining $L^-(I)$. Our $\hw^i$ here has index $i$ matching that of the point $\tp^i$ in the hyperplane, whereas $w^\ell$ in Step 1 has index corresponding to the distance $\delta^\ell r$ from the left end of the interval.)

We continue considering the contribution to the difference of sums $\Upsilon_{Euc}\left( \Gamma_{xy}^{j_1-1,1} \right) -\Upsilon_{Euc}\left( \Gamma_{xy}^{j_1-1,2} \right)$ in \eqref{trackprob} from a single $(j_1-1)$th--scale interval $\tred{I_N}=[b^N,b^{N+1}] = [\tp^{\gamma(N)},\tp^{\gamma(N+1)}]$.  We split into cases according to the type of interval $I_N$.

\emph{Case 1.} $I_N$ is a short non--terminal $(j_1-1)$th--scale interval. Here $\mH_{xy}$ includes no joining hyperplanes in the interval, so it includes exactly two maximally $j_1$th--scale (sandwiching) hyperplanes there, at distance $\delta^{j_1}r$ from each end; see Figure \ref{Case1}.  We introduce the \emph{intermediate path} 
\[
  \tred{\Gamma^{N,int}}:\ w_\perp^{\gamma(N)}\to w_\perp^{\gamma(N)+1}\to 
    w_\perp^{\gamma(N)+2}\to w_\perp^{\gamma(N+1)}
\]
which has the same endpoints and satisfies $\Upsilon_{\hT}(\Gamma^{N,int}) \geq \hT(\tp^{\gamma(N)},\tp^{\gamma(N+1)})$.  To bound \eqref{trackprob} we use the expression on the right in \eqref{intsplit}.  Applying Lemma \ref{monotoneE} with $\ep=1-\delta$ yields
\[
  h(|\tp^i - \tp^{i-1}|) - h(|w_\perp^i - w_\perp^{i-1}|) 
    \geq \delta\mu\left( |\tp^i - \tp^{i-1}|-|w_\perp^i - w_\perp^{i-1}| \right) -C_{56}
\]
and therefore
\begin{equation}\label{posfrac}
  \Upsilon_h\left( \Gamma^{N,full} \right) - \Upsilon_h\left( \Gamma^{N,int} \right) \geq
    \delta\mu\left( \Upsilon_{Euc}\left( \Gamma^{N,full} \right) - | \tp^{\gamma(N+1)} - \tp^{\gamma(N)} | \right) - 3C_{56}.
\end{equation}
\begin{figure}
\includegraphics[width=16cm]{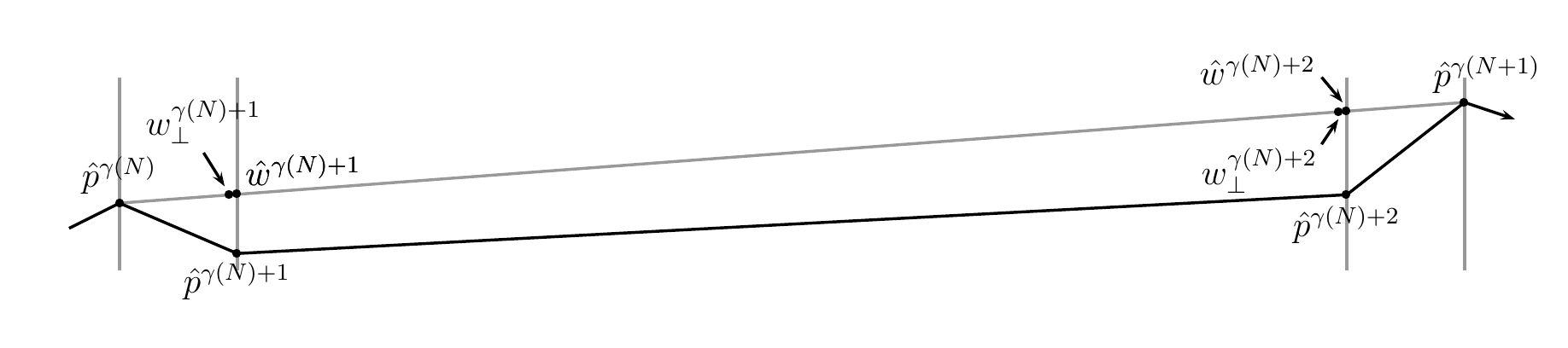}
\caption{ Diagram for Case 1 showing $(j_1-1)$th--scale endpoint hyperplanes and $j_1$th--scale sandwiching hyperplanes in a short $(j_1-1)$th--scale interval. The black path is $\Gamma^{N,full}$ and the gray is $\Gamma^{N,int}$.}
\label{Case1}
\end{figure}
This is the essential property of the intermediate path: when we look at the bowedness of the full path relative to the intermediate path (represented by the left side of \eqref{posfrac}), and relative to the direct path (right side of \eqref{posfrac}, without $\delta$), the first is at least $\delta$ fraction of the second, to within a constant. By \eqref{nearsub}, for $\omega \notin J^{(0)}(c_{29})$ the intermediate path also satisfies
\begin{equation}\label{addM}
  \hT( \tp^{\gamma(N)},\tp^{\gamma(N+1)} ) \leq \Upsilon_{\hT}\left( \Gamma^{N,int} \right) 
    + (\gamma(N+1) - \gamma(N) -1)c_{29}\log r.
\end{equation}
Similarly to \eqref{Sjbound}, since in Case 1 $\Upsilon_{Euc}\left( \Gamma^{N,int} \right) = | \tp^{\gamma(N+1)} - \tp^{\gamma(N)} |= |b^{N+1}-b^N|$, we have
\begin{align}\label{Sjbound2}
  \sum_{N=0}^n S_{j_1}\left( \Gamma^{N,int} \right) &= \sum_{i=1}^{m+1} A_{j_1}^2(w_\perp^{i-1},w_\perp^i) 
    \leq \frac{1}{4}\lambda^{j_1} \Big[ t\sigma_r 
    + \delta\mu\Big(\Upsilon_{Euc}\left( \Gamma_{xy}^{j_1-1,2} \right) - (\hat y - \hat x)_1) \Big) \Big].
\end{align}
From \eqref{intsplit}---\eqref{Sjbound2} it follows that the contribution to the tracking--failure probability \eqref{trackprob} from short non--terminal intervals is bounded by
\begin{align}
  &P\Bigg( \hT( \tp^{\gamma(N)},\tp^{\gamma(N+1)} ) - \Upsilon_{\hT}\left( \Gamma^{N,full} \right) \geq  
    S_{j_1}\left( \Gamma^{N,full} \right) + S_{j_1}\left( \Gamma^{N,int} \right) + \frac{1}{12}\lambda^{j_1}t\sigma_r  \notag\\
  &\hskip 1.5cm + \delta\mu\Big( | \tp^{\gamma(N+1)} - \tp^{\gamma(N)} | - \Upsilon_{Euc}\left( \Gamma^{N,full} \right) \Big)
    \text{ for some $N$ with $I_N$ short non--terminal}\notag\\
  &\hskip 1.5cm \text{$(j_1-1)$th--scale, and some } (x,y)\in X_r;\ \omega\notin J^{(0)}(c_{29})\cup J^{(1c)} \Bigg) \notag
\end{align}
\begin{align}
  &\leq P\Bigg( \Upsilon_{\hT}\left( \Gamma^{N,int} \right) - \Upsilon_{\hT}\left( \Gamma^{N,full} \right) \geq  
    S_{j_1}\left( \Gamma^{N,full} \right) + S_{j_1}\left( \Gamma^{N,int} \right) + \frac{1}{12}\lambda^{j_1}t\sigma_r
    - 2c_{29}\log r \notag\\
  &\hskip 1.5cm +  \Upsilon_h\left( \Gamma^{N,int} \right) 
    - \Upsilon_h\left( \Gamma^{N,full} \right) - 2C_{56} 
    \text{ for some $N$ with $I_N$ short non--terminal}\notag\\
  &\hskip 1.5cm \text{$(j_1-1)$th--scale, and some } (x,y)\in X_r;\ \omega\notin J^{(0)}(c_{29})\cup J^{(1c)} \Bigg) \notag\\
  &\leq P\Bigg( \left| \Upsilon_{\hT}\left( \Gamma^{N,full} \right) - \Upsilon_h\left( \Gamma^{N,full} \right) \right| 
    \geq S_{j_1}\left( \Gamma^{N,full} \right)
    \text{ or } \Big| \Upsilon_{\hT}\left( \Gamma^{N,int} \right) - \Upsilon_h\left( \Gamma^{N,int} \right) \Big| \notag\\
  &\hskip 1.5cm  \geq S_{j_1}\left( \Gamma^{N,int} \right) \text{ for some $N$ with $I_N$ short non--terminal}\notag\\
  &\hskip 1.5cm \text{$(j_1-1)$th--scale, and some } (x,y)\in X_r;\ \omega\notin J^{(0)}(c_{29})\cup J^{(1c)} \Bigg) \notag\\
  &\leq P\Bigg( \left| \hT(\tp^{i-1},\tp^i) - h(|\tp^i - \tp^{i-1}|) \right| \geq A_{j_1}^2(\tp^{i-1},\tp^i) \text{ or } 
    \left| \hT(w_\perp^{i-1},w_\perp^i) - h(|w_\perp^i - w_\perp^{i-1}|) \right| \notag\\
  &\hskip 1.5cm \geq A_{j_1}^2(w_\perp^{i-1},w_\perp^i) \text{ for some } \gamma(N)+1\leq i \leq \gamma(N+1), 
    \text{ for some $N$}\notag\\
  &\hskip 1.5cm \text{ with $I_N$ short non--terminal $(j_1-1)$th--scale, and some } 
    (x,y)\in X_r;\ \omega\notin J^{(0)}(c_{29})\cup J^{(1c)} \Bigg).\label{segments}
\end{align}
The terms $C_{56}$ and $c_{29}\log r$ are negligible in \eqref{segments} relative to $\lambda^{j_1}t\sigma_r$, because the latter is of the order of a power of $r$, due to $j_1=O(\log\log r)$.
All of the increments $\hT(u,v)$ in the last event have $\delta^{j_1}r\leq (v-u)_1\leq\delta^{j_1-1}r$. We can therefore bound the last probability similarly to Lemma \ref{sumbits}, with the main difference being that in place of pairs $(u,v)$ in the definitions of $R_{r,j}^{*,*}$ we need to consider 4--tuples $(u,v,w,z)$ corresponding to values of $(\tp^{\gamma(N)},\tp^{\gamma(N)+1},\tp^{\gamma(N)+2}, \tp^{\gamma(N+1)}) \in (\LL_{j_1}\cap G_r^+)^4$.  This means the exponents $d-1$ in the bounds on $|R_{r,j}^{*,*}|$ become $3(d-1)$.  The values $w_\perp^i$ are determined by $(\tp^{\gamma(N)},\tp^{\gamma(N)+1},\tp^{\gamma(N)+2}, \tp^{\gamma(N+1)})$ so their presence does not increase the necessary size of $R_{r,j}^{*,*}$ in the lemma.  As with the sets $R_{r,j}^{\nu,\ell}$ in the lemma proof, we can decompose the possible 4--tuples $(\tp^{\gamma(N)},\tp^{\gamma(N)+1},\tp^{\gamma(N)+2}, \tp^{\gamma(N+1)})$ according to the size of $|(\tp^{\gamma(N)+1}-\tp^{\gamma(N)})^*|,\,|\tp^{\gamma(N)+1} - w_\perp^{\gamma(N)+1}|$, and $|\tp^{\gamma(N)+2} - w_\perp^{\gamma(N)+2}|$ and sum over the possible size ranges.
Otherwise the proof remains the same, and we get that the last probability in \eqref{segments} is bounded by
\begin{equation}\label{segments2}
   c_{34} \exp\left( -c_{35}\left( \frac{\lambda}{7\delta^{\chi_1}} \right)^{j_1} t \right).
\end{equation}

\emph{Case 2.} $I_N$ is a terminal $(j_1-1)$th--scale interval (meaning either $[\tp^{\gamma(1)},\tp^{\gamma(2)}]$ or $[\tp^{\gamma(n)-1},\tp^{\gamma(n)}]$). 
The proof is similar to Case 1, except that the interval includes only one maximally $j_1$th--scale (sandwiching) hyperplane between the two terminal hyperplanes that are at the ends of the interval, so the full path in the interval has form $\tp^{\gamma(N)}\to \tp^{\gamma(N)+1}\to \tp^{\gamma(N+1)}$.  We obtain for the terminal--interval contribution to the tracking--failure probability \eqref{trackprob} the bound
\begin{align}\label{terminal}
  &P\Bigg( \hT( \tp^{\gamma(N)},\tp^{\gamma(N+1)} ) - \Upsilon_{\hT}\left( \Gamma^{N,full} \right) \geq  
    S_{j_1}\left( \Gamma^{N,full} \right) + S_{j_1}\left( \Gamma^{N,int} \right) \notag\\
  &\hskip 1.5cm + \delta\mu\Big( | \tp^{\gamma(N+1)} - \tp^{\gamma(N)} | - \Upsilon_{Euc}\left( \Gamma^{N,full} \right) \Big)
    \text{ for some $N$ with $I_N$ terminal}\notag\\
  &\hskip 1.5cm \text{$(j_1-1)$th--scale, and some } (x,y)\in X_r;\ \omega\notin J^{(0)}(c_{29})\cup J^{(1c)} \Bigg) \notag\\
  &\leq c_{34} \exp\left( -c_{35}\left( \frac{\lambda}{7\delta^{\chi_1}} \right)^{j_1} t \right).
\end{align}

\emph{Case 3.} $I_N$ is a long non--terminal $(j_1-1)$th--scale interval.  Such an interval, and thus also the middle increment $\tp^{\gamma(N+1)-1} - \tp^{\gamma(N)+1}$ of the 3 comprising $\Gamma^{N,full}$ in Case 1,
may be much longer than $\delta^{j_1-1}r$.  Therefore the quantity $A_{j_1}^2(\tp^{\gamma(N)+1},\tp^{\gamma(N+1)-1})$ used on the right side of \eqref{segments} for that increment is no longer large enough to give a useful bound on the probability.  To avoid this problem we will sometimes use a different intermediate path in $I_N$ which coincides with the full path between the inner joining hyperplanes, so differs from the full path only near the ends of $I_N$, while preserving the property \eqref{posfrac} (this preservation being the purpose of our choice of $L^\pm(I)$.)  Equation \eqref{posfrac} represents what we may informally call \emph{deterministic tracking}, a nonrandom analog of the tracking of Remark \ref{outline} which facilitates our desired (random) form of tracking.

Fix a long non--terminal $(j_1-1)$th--scale interval $I_N$, and suppose it has $k$th--scale length for some $k<j_1-1$. 
The hyperplanes of $\mH_{xy}$ in $I_N$ are the $(j_1-1)$th--scale ones at each endpoint, two sandwiching $j_1$th--scale hyperplanes at distance $\delta^{j_1}r$ from each end, and between 2 and 4 $j_1$th--scale joining hyperplanes between these.  These hyperplanes are at the joining points $(\tp^{\gamma(N)})_1 + \delta^{L^-(I)+1}r, (\tp^{\gamma(N)})_1 + \delta^{L^-(I)}r, (\tp^{\gamma(N+1)})_1-\delta^{L^+(I)}r$, and $(\tp^{\gamma(N+1)})_1-\delta^{L^+(I)+1}r$, with two exceptions.  First, 
if $L^-(I_N) = j_1-1$ then the first of these 4 joining hyperplanes coincides with the left--end sandwiching one, and similarly for $L^+(I_N)$, which reduces the number of $j_1$th--scale joining hyperplanes to fewer than 4, as discussed in criterion (ii) after \eqref{Lminus}.  Second, in the totally unbowed case (third option in \eqref{Lminus}) at either end of $I_N$, there is no outer joining hyperplane at that end, as in criterion (iii).
We define $4\leq \tred{V(N)}\leq 6$ to be the number of $j_1$th--scale hyperplanes in the interior of $I_N$.  

{\it Case 3a.} The bowed case (second option in \eqref{Lminus}) for both $L^\pm(I)$, with $V(N)=6$. Since $V(N)=6$ we must have $L^-(I_N) <j_1-1$ and $L^+(I_N) <j_1-1$.  Here the full path is
\[
  \Gamma^{N,full}:\ \tp^{\gamma(N)}\to \tp^{\gamma(N)+1}\to \cdots\to \tp^{\gamma(N)+6}\to \tp^{\gamma(N+1)} 
\]
and the direct path again is $\tp^{\gamma(N)}\to \tp^{\gamma(N+1)}$.  Define
\[
  \tred{\hz^i} = \begin{cases} \text{the point } \Pi_{\tp^{\gamma(N)},\tp^{\gamma(N)+3}} \cap H_{(\tp^i)_1} 
    &\text{if } i=\gamma(N)+1,\gamma(N)+2,\\
    \text{the point } \Pi_{\tp^{\gamma(N)+4},\tp^{\gamma(N+1)}} \cap H_{(\tp^i)_1} &\text{if } i=\gamma(N)+5,\gamma(N)+6,\\
    \tp^i &\text{if } i = \gamma(N),\gamma(N)+3,\gamma(N)+4,\gamma(N+1);
    \end{cases}
\]
see Figure \ref{Case3a}. This time the intermediate path is defined as
\[
 \tred{\Gamma^{N,int}}: \hz^{\gamma(N)}\to \hz^{\gamma(N)+1}\to \cdots\to \hz^{\gamma(N)+6}\to \hz^{\gamma(N+1)},
\]
which coincides with $\Gamma^{N,full}$ between the inner joining hyperplanes, which is ``most'' of $I_N$. We denote the parts of the paths outside the inner joining hyperplanes by
\[
  \tred{\Gamma^{N,full,-}}:\ \tp^{\gamma(N)}\to \tp^{\gamma(N)+1}\to \tp^{\gamma(N)+2}\to \tp^{\gamma(N)+3},
\]
\[
  \tred{\Gamma^{N,full,+}}:\ \tp^{\gamma(N)+4}\to \tp^{\gamma(N)+5}\to \tp^{\gamma(N)+6}\to \tp^{\gamma(N+1)},
\]
with $\Gamma^{N,int,\pm}$ defined similarly with $\hz^i$ in place of $\tp^i$,
so that
\begin{align}
  \Upsilon_{\hT}\left( \Gamma^{N,full} \right) - \Upsilon_{\hT}\left( \Gamma^{N,int} \right)
    &= \Big[ \Upsilon_{\hT}\left( \Gamma^{N,full,-} \right) 
    - \Upsilon_{\hT}\left(  \Gamma^{N,int,-} \right) \Big]
    + \Big[ \Upsilon_{\hT}\left( \Gamma^{N,full,+} \right) 
    - \Upsilon_{\hT}\left(  \Gamma^{N,int,+} \right) \Big]. 
\end{align}
For the corresponding quantities for means, similarly to \eqref{posfrac}, but using $\ep=1/2$ in Lemma \ref{monotoneE}, we have
\begin{align}\label{endsonly}
  \Upsilon_h&\left( \Gamma^{N,full} \right) - \Upsilon_h\left( \Gamma^{N,int} \right) \notag\\
  &= \Big[ \Upsilon_h\left( \Gamma^{N,full,-} \right) - \Upsilon_h\left(  \Gamma^{N,int,-} \right) \Big] + 
    \Big[ \Upsilon_h\left( \Gamma^{N,full,+} \right) - \Upsilon_h\left(  \Gamma^{N,int,+} \right) \Big] \notag\\
  &\geq \frac \mu 2\left[ \Upsilon_{Euc}\left( \Gamma^{N,full,-} \right) - | \tp^{\gamma(N)+3} - \tp^{\gamma(N)} | \right]
    + \frac\mu 2 \left[ \Upsilon_{Euc}\left( \Gamma^{N,full,+} \right) - | \tp^{\gamma(N+1)} - \tp^{\gamma(N)+4} | \right]
    - 6C_{56} \notag\\
  &= \frac\mu 2 \Big[\Upsilon_{Euc}\left( \Gamma^{N,full} \right) - \Upsilon_{Euc}\left( \Gamma^{N,int} \right) \Big] -6C_{56}.
\end{align}

\begin{figure}
\includegraphics[width=16cm]{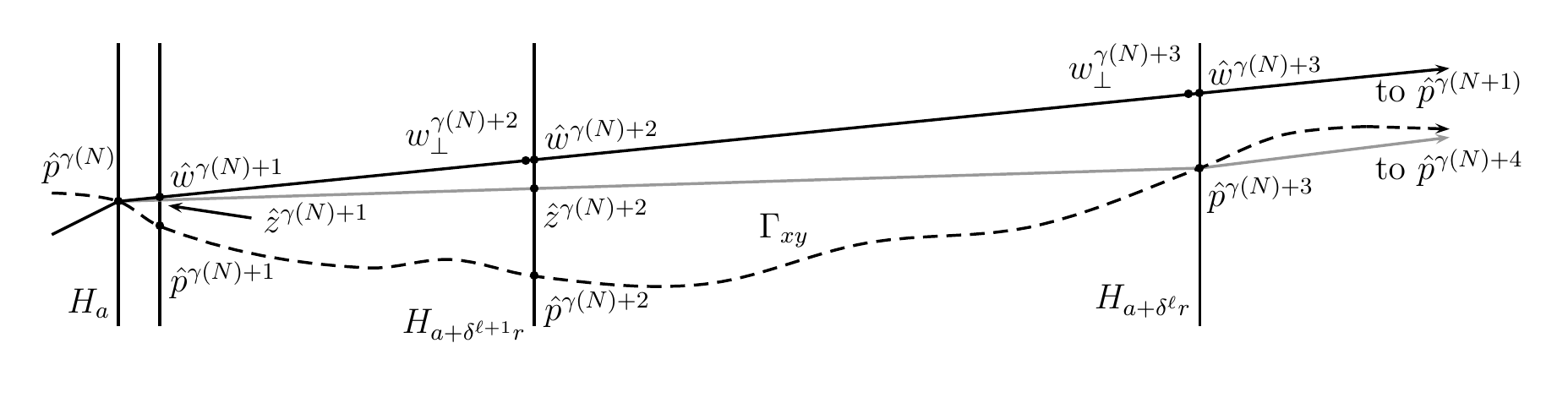}
\caption{ The left end of a long $(j_1-1)$th--scale interval $I_N=[a,b]$ with the bowed case, with $L^-(I_n)=\ell$, showing the direct path (black) and intermediate path (gray.) The full path has marked points at the $\hat p^{\gamma(N)+i}$. }
\label{Case3a}
\end{figure}
   
We claim that deterministic tracking holds in the sense that
\begin{align}\label{partbowed}
  \Upsilon_{Euc}\left( \Gamma^{N,full} \right) - \Upsilon_{Euc}\left( \Gamma^{N,int} \right) 
    \geq 3\delta \left[ \Upsilon_{Euc}\left( \Gamma^{N,full} \right) - | \tp^{\gamma(N+1)} - \tp^{\gamma(N)} | \right].
\end{align}
This and \eqref{endsonly} give the full analog of \eqref{posfrac} for Case 3a. To prove the claim, suppose $L^-(I_N)=\ell^-$ for some $k+1 \leq \ell^- < j_1-1$.  Recall $\alpha(\cdot)$ and $\theta(\cdot)$ from \eqref{alkap} and \eqref{thdef}, noting that in our present notation, $v,w,z$ there have become $\tp,\hw,\hz$, with different superscript labeling.
Using \eqref{extramain1}---\eqref{ak} with $j=j_1-1$,
\[
  \alpha(\ell^-+1) = \frac{|\tp^{\gamma(N)+2} - \hw^{\gamma(N)+2}|^2}{\delta^{\ell^-+1}r},\quad 
  \theta(\ell^-+1) = \frac{|\hz^{\gamma(N)+2} - \hw^{\gamma(N)+2}|}{|\tp^{\gamma(N)+2} - \hw^{\gamma(N)+2}|} \leq \frac 12
\]
and hence
\begin{align}\label{bowedend}
  \Upsilon_{Euc}\left( \Gamma^{N,full,-} \right) - | \tp^{\gamma(N)+3} - \tp^{\gamma(N)} | 
    &\geq \Upsilon_{Euc}\left( \tp^{\gamma(N)},\tp^{\gamma(N)+2},\tp^{\gamma(N)+3} \right) - | \tp^{\gamma(N)+3} - \tp^{\gamma(N)} | \notag\\
  &\geq \frac{|\tp^{\gamma(N)+2} - \hz^{\gamma(N)+2}|^2}{3\delta^{\ell^-+1}r} \notag\\
  &\geq \frac{\alpha(\ell^-+1)}{12}. 
\end{align}
Here the second inequality uses the fact that $\tp^{\gamma(N)+2} - \hz^{\gamma(N)+2}$ is nearly perpendicular to $\Pi_{\tp^{\gamma(N)},\tp^{\gamma(N)+3}}$, by \eqref{forward}.
In the other direction, recalling $I_N$ has $k$th--scale length (so $(\tp^{\gamma(N)+4} - \tp^{\gamma(N)+3})_1 \geq 8\delta^{k+1}r$) and 
supposing $L^+(I_N)=\ell^+$, since $k+1\leq\ell^+$ and $k+1\leq\ell^-$ we have similarly to \eqref{middle2} and \eqref{middle}
\begin{align}\label{lowbow}
  \Upsilon_{Euc}\left( \Gamma^{N,int} \right) - | \tp^{\gamma(N+1)} - \tp^{\gamma(N)} | 
    &\leq \frac{|\tp^{\gamma(N)+3} - w_\perp^{\gamma(N)+3}|^2}
    {2\delta^{\ell^-} r} + \frac{|\tp^{\gamma(N)+4} - w_\perp^{\gamma(N)+4}|^2}{2\delta^{\ell^+} r} \notag\\
  &\qquad + \frac{|(\tp^{\gamma(N)+3} - w_\perp^{\gamma(N)+3}) - (\tp^{\gamma(N)+4} - w_\perp^{\gamma(N)+4})|^2}
    {16\delta^{k+1}r} \notag\\
  &\leq \frac{|\tp^{\gamma(N)+3} - \hw^{\gamma(N)+3}|^2}{\delta^{\ell^-} r} + \frac{|\tp^{\gamma(N)+4} - \hw^{\gamma(N)+4}|^2}{\delta^{\ell^+} r}.
\end{align}
Now the first ratio on the right is $\alpha(\ell^-)$, so it follows from \eqref{ak} and \eqref{bowedend} that
\[
   \frac{|\tp^{\gamma(N)+3} - \hw^{\gamma(N)+3}|^2}{\delta^{\ell^-} r} \leq 24C_{23}\delta^{-\chi_2}
     \Big[ \Upsilon_{Euc}\left( \Gamma^{N,full,-} \right) - | \tp^{\gamma(N)+3} - \tp^{\gamma(N)} | \Big],
\]
and similarly
\[
   \frac{|\tp^{\gamma(N)+4} - \hw^{\gamma(N)+4}|^2}{\delta^{\ell^-} r} \leq 24C_{23}\delta^{-\chi_2}
     \Big[ \Upsilon_{Euc}\left( \Gamma^{N,full,+} \right) - | \tp^{\gamma(N+1)} - \tp^{\gamma(N)+4} | \Big].
\]
Hence from \eqref{lowbow} and the second equality in \eqref{adddist1},
\begin{align}
  \Upsilon_{Euc}\left( \Gamma^{N,int} \right) - | \tp^{\gamma(N+1)} - \tp^{\gamma(N)} | &\leq 24C_{23}\delta^{-\chi_2}
    \Big[ \Upsilon_{Euc}\left( \Gamma^{N,full} \right) - \Upsilon_{Euc}\left( \Gamma^{N,int} \right) \Big]
\end{align}
which implies
\begin{align}\label{bowed3}
  \Upsilon_{Euc}\left( \Gamma^{N,full} \right) - | \tp^{\gamma(N+1)} - \tp^{\gamma(N)} | &\leq (1+24C_{23}\delta^{-\chi_2})
    \Big[ \Upsilon_{Euc}\left( \Gamma^{N,full} \right) - \Upsilon_{Euc}\left( \Gamma^{N,int} \right) \Big],
\end{align}
proving the claim \eqref{partbowed} when we take $\delta$ small.  

From Lemma \ref{sumbits2}(i) we have
\begin{equation}\label{AvsS}
  6A_{j_1}^3\left(\Gamma^{N,full,-}\right) \leq S_{j_1}\left( \Gamma^{N,full,-} \right) 
    + \frac{\delta\mu}{18} \Big[ \Psi_{Euc}(\Gamma^{N,full,-}) - |\tp^{\gamma(N)+3}-\tp^{\gamma(N)}| \Big] - 6C_{56},
\end{equation}
and the equivalent ``mirror image'' of Lemma \ref{sumbits2} covers $\Gamma^{N,full,+}$ symmetrically, incorporating a symmetric definition of joining 4--path.

The analog of \eqref{addM} remains valid, and we now have the ingredients \eqref{endsonly}, \eqref{partbowed} and \eqref{AvsS} for the analog of \eqref{segments}, bounding the bowed--case contribution to the tracking--failure probability \eqref{trackprob}---see \eqref{calcs} below:
\begin{align}
  &P\Bigg( \Upsilon_{\hT}\left( \Gamma^{N,full} \right) - \hT( \tp^{\gamma(N)},\tp^{\gamma(N+1)} ) \leq  
    -S_{j_1}\left( \Gamma^{N,full,-} \right) - S_{j_1}\left( \Gamma^{N,full,+} \right) 
    - \frac{1}{3}\lambda^{j_1}t\sigma_r \notag\\
  &\hskip 0.6cm 
    + \delta\mu\Big( \Upsilon_{Euc}\left( \Gamma^{N,full} \right) - | \tp^{\gamma(N+1)} - \tp^{\gamma(N)} | \Big)
    \text{ for some $N$ with $V(N)=6$}\notag\\
  &\hskip 0.6cm \text{and $I_N$ long non--terminal $(j_1-1)$th--scale, in the bowed case of \eqref{Lminus}} \notag\\
  &\hskip .6cm \text{for both of $L^\pm(I_N)$, for some } (x,y)\in X_r;\ \omega\notin J^{(0)}(c_{29})\cup J^{(1c)} \Bigg) \notag\\
  &\leq P\Bigg(  \Big[ \Upsilon_{\hT}\left( \Gamma^{N,full,-} \right) - \Upsilon_{\hT}\left( \Gamma^{N,int,-} \right) \Big] +
    \Big[ \Upsilon_{\hT}\left( \Gamma^{N,full,+} \right) - \Upsilon_{\hT}\left( \Gamma^{N,int,+} \right) \Big] \notag\\
  &\hskip 1 cm \leq  -6A_{j_1}^3(\Gamma^{N,full,-}) - 6A_{j_1}^3(\Gamma^{N,full,+}) 
    + \Upsilon_h\left( \Gamma^{N,full} \right) - \Upsilon_h\left( \Gamma^{N,int} \right) +6c_{29}\log r + 18C_{56} \notag\\
  &\hskip 1cm 
    \text{$- \frac{1}{3}\lambda^{j_1}t\sigma_r$ for some $N$ with $V(N)=6$ and $I_N$ long non--terminal $(j_1-1)$th--scale, 
    in the}\notag\\
  &\hskip 1cm \text{bowed case of \eqref{Lminus} for both of $L^\pm(I_N)$, and for some } 
    (x,y)\in X_r;\ \omega\notin J^{(0)}(c_{29})\cup J^{(1c)} \Bigg) \notag\\
  &\leq \sum_{\ell^-=k+1}^{j_1-2}\,\sum_{\ell^+=k+1}^{j_1-2}
    P\Bigg( \left| \hT(\tp^{i-1},\tp^i) - h(|\tp^i - \tp^{i-1}|) \right| \geq A_{j_1}^3(\Gamma^{N,full,\pm})  \text{ or } \notag\\
  &\hskip 1cm \left| \hT(\hz^{i-1},\hz^i) - h(|\hz^i - \hz^{i-1}|) \right| \geq A_{j_1}^3(\Gamma^{N,full,\pm}) \text{ for some } 
    \gamma(N)+1\leq i \leq \gamma(N+1) \notag\\
  &\hskip 1cm \text{with $i\neq\gamma(N)+4$, for some $N$ with $V(N)=6$ and $I_N$ long non--terminal}  \notag\\
  &\hskip 1cm \text{$(j_1-1)$th--scale with $L^-(I_N)=\ell^-$ and $L^+(I_N)=\ell^+$, 
    in the bowed case of  } \notag\\
  &\hskip 1cm \text{\eqref{Lminus} for both of $L^\pm(I_N)$, and some } 
    (x,y)\in X_r;\ \omega\notin J^{(0)}(c_{29})\cup J^{(1c)} \Bigg). \label{segments6}
\end{align}
Here the first inequality uses that 
\[
  \hT( \tp^{\gamma(N)},\tp^{\gamma(N+1)} ) \leq \Upsilon_{\hT}\left( \Gamma^{N,int,-} \right) 
    + \Upsilon_{\hT}\left( \Gamma^{N,int,+} \right) + 6c_{29}\log r,
\]
and that from \eqref{AvsS}, \eqref{partbowed}, and then \eqref{endsonly},
\begin{align}\label{calcs}
  -S_{j_1}&\left( \Gamma^{N,full,-} \right) -S_{j_1}\left( \Gamma^{N,full,+} \right) 
    +\delta\mu\Big( \Upsilon_{Euc}\left( \Gamma^{N,full} \right) - | \tp^{\gamma(N+1)} - \tp^{\gamma(N)} | \Big) \notag\\
  &\leq -6A_{j_1}^3(\Gamma^{N,full,-}) - 6A_{j_1}^3(\Gamma^{N,full,+})
    +\frac{10}{9} \delta\mu\Big( \Upsilon_{Euc}\left( \Gamma^{N,full} \right) - | \tp^{\gamma(N+1)} - \tp^{\gamma(N)} | \Big)
    + 12C_{56}\notag\\
  &\leq -6A_{j_1}^3(\Gamma^{N,full,-}) - 6A_{j_1}^3(\Gamma^{N,full,+})
    +  \frac{10}{27}\mu \Big[ \Upsilon_{Euc}\left( \Gamma^{N,full} \right) 
    - \Upsilon_{Euc}\left( \Gamma^{N,int} \right) \Big] + 12C_{56}\notag\\
  &\leq -6A_{j_1}^3(\Gamma^{N,full,-}) - 6A_{j_1}^3(\Gamma^{N,full,+})
    + \Upsilon_h\left( \Gamma^{N,full} \right) - \Upsilon_h\left( \Gamma^{N,int} \right) + 18C_{56}.
\end{align}
For the $\pm$ in the last event in \eqref{segments6}, the $-$ applies to $\gamma(N)+1\leq i\leq\gamma(N)+3$ and the $+$ applies to $\gamma(N)+5\leq i\leq\gamma(N+1)$.  An application of Lemma \ref{sumbits2}(ii) and (iii) with $j=j_1-1$, followed by summing over $\ell^\pm$, bounds the last (tracking--failure) probability in \eqref{segments6} by
\begin{align}\label{case3a}
  &\left[ \sum_{\ell=1}^{j_1-2} C_{72}\left( \frac{2}{\delta^{\chi_2+1}} \right)^{j_1-1-\ell} 
    \left( \frac{\lambda\delta}{\beta^2} \right)^{j_1-1} \exp\left( - C_{73} \left( \frac{\lambda}{\delta^{\chi_1}} \right)^{j_1-1}
    2^{j_1-1-\ell}t \right) \right]^2 \notag\\
  &\qquad \leq c_{36} \left( \frac{\lambda\delta}{\beta^2} \right)^{2(j_1-1)} \exp\left( - C_{73}
    \left( \frac{\lambda}{\delta^{\chi_1}} \right)^{j_1-1} t \right).
\end{align}

{\it Case 3b.} The bowed case (second option in \eqref{Lminus}) for both $L^\pm(I)$, with $V(N)<6$.  This means we have at least one of $L^\pm(I_N)=j_1-1$.  If for example $L^-(I_N)=j_1-1$, then what were in Case 3a the two points $\tp^{\gamma(N)+1},\tp^{\gamma(N)+2}$ are now the same point, so effectively there is no separate point $\tp^{\gamma(N)+1}$. Instead we have only the equivalent of $\tp^{\gamma(N)},\tp^{\gamma(N)+2},\tp^{\gamma(N)+3}$, so joining 4--paths become joining 3--paths.  This has no significant effect on the arguments, including Lemma \ref{sumbits2}, other than some simplifications, and the bound in \eqref{case3a} still applies for the corresponding contribution to the tracking--failure probability \eqref{trackprob}:
\begin{align}\label{case3b}
  &P\Bigg( \hT( \tp^{\gamma(N)},\tp^{\gamma(N+1)} ) - \Upsilon_{\hT}\left( \Gamma^{N,full} \right) \geq  
    S_{j_1}\left( \Gamma^{N,full,-} \right) + S_{j_1}\left( \Gamma^{N,full,+} \right) \notag\\
  &\hskip 0.6cm 
    + \delta\mu\Big( | \tp^{\gamma(N+1)} - \tp^{\gamma(N)} | - \Upsilon_{Euc}\left( \Gamma^{N,full} \right) \Big)
    \text{ for some $N$ with $V(N)<6$}\notag\\
  &\hskip 0.6cm \text{and $I_N$ long non--terminal $(j_1-1)$th--scale, in the bowed case of \eqref{Lminus}} \notag\\
  &\hskip .6cm \text{for both of $L^\pm(I_N)$, and for some } (x,y)\in X_r;\ \omega\notin J^{(0)}(c_{29})\cup J^{(1c)} \Bigg) \notag\\
  &\leq c_{36} \left( \frac{\lambda\delta}{\beta^2} \right)^{2(j_1-1)} \exp\left( - C_{73}
    \left( \frac{\lambda}{\delta^{\chi_1}} \right)^{j_1-1} t \right).
\end{align}

{\it Case 3c.} The forward case (first option in \eqref{Lminus}) for both $L^\pm(I_N)$.  Here we have $V(N)=4$ as there are only inner joining points, at distance $\delta^{j_1}r$ from each end of $I_N$.  As before the direct path is $\tp^{\gamma(N)}\to \tp^{\gamma(N+1)}$, and the full path is 
\[
  \tred{\Gamma^{N,full}}:\ \tp^{\gamma(N)}\to \tp^{\gamma(N)+1}\to \cdots\to \tp^{\gamma(N)+4}\to \tp^{\gamma(N+1)};
\]
by \eqref{nearsub}, for $\omega\notin J^{(0)}(c_{29})$ we have
\begin{equation}\label{nearsub2}
  \Upsilon_{\hT}\left( \Gamma^{N,full} \right) - \hT( \tp^{\gamma(N)},\tp^{\gamma(N+1)} ) \geq -4c_{29}\log r.
\end{equation}
We claim that in the forward case we have
\begin{equation}\label{notrack}
  S_{j_1}\left( \Gamma^{N,full} \right) \geq
    \delta\mu\Big( \Upsilon_{Euc}\left( \Gamma^{N,full} \right) - | \tp^{\gamma(N+1)} - \tp^{\gamma(N)} | \Big)
    + 4c_{29}\log r.
\end{equation}
By \eqref{nearsub2}, this means that in the forward case there is no tracking failure:
\begin{align}\label{nofail}
  &P\Bigg( \Upsilon_{\hT}\left( \Gamma^{N,full} \right) - \hT( \tp^{\gamma(N)},\tp^{\gamma(N+1)} ) \leq  
    -S_{j_1}\left( \Gamma^{N,full} \right) \notag\\
  &\hskip 0.6cm + \delta\mu\Big( \Upsilon_{Euc}\left( \Gamma^{N,full} \right) - | \tp^{\gamma(N+1)} - \tp^{\gamma(N)} | \Big)
    \text{ for some $N$ with}\notag\\
  &\hskip 0.6cm \text{$I_N$ long non--terminal $(j_1-1)$th--scale, in the forward case of \eqref{Lminus}} \notag\\
  &\hskip .6cm \text{for both of $L^\pm(I_N)$, and for some } (x,y)\in X_r;\ 
    \omega\notin J^{(0)}(c_{29})\cup J^{(1c)} \Bigg) = 0.
\end{align}
To prove \eqref{notrack}, we first observe that from the definition \eqref{Aj2},
\begin{align}\label{Sjlower}
  S_{j_1}\left( \Gamma^{N,full} \right) \geq \frac{1}{4}
    \left( \frac \lambda 7 \right)^{j_1} \left( t^*(\tp^{\gamma(N)}) + t^*(\tp^{\gamma(N+1)} \right) \sigma_r.
\end{align}
From \eqref{adddist1} and \eqref{Lminus}, since $k<j_1-1$,
\begin{align}\label{pieces}
  \Upsilon_{Euc}&\left( \Gamma^{N,full} \right) - | \tp^{\gamma(N+1)} - \tp^{\gamma(N)} | \notag\\
  &\leq \frac 12 \sum_{i=\gamma(N)+1}^{\gamma(N+1)} \frac{ |(\tp^i - w_\perp^i) - (\tp^{i-1} - w_\perp^{i-1})|^2 }
    { |w_\perp^i - w_\perp^{i-1}| } \notag\\
  &\leq (1-2\delta) \Bigg[ \frac{|\tp^{\gamma(N)+1} - w_\perp^{\gamma(N)+1}|^2}{\delta^{j_1}r}
    + \frac{( |\tp^{\gamma(N)+1} - w_\perp^{\gamma(N)+1}| + |\tp^{\gamma(N)+2} - w_\perp^{\gamma(N)+2}| )^2}
    {(1-\delta)\delta^{j_1-1}r} \notag\\
  &\hskip 2cm + \frac{( |\tp^{\gamma(N)+2} - w_\perp^{\gamma(N)+2}| + |\tp^{\gamma(N)+3} - w_\perp^{\gamma(N)+3}| )^2}
    {(1-2\delta)10\delta^{k+1}r} \notag\\
  &\hskip 2cm + \frac{( |\tp^{\gamma(N)+3} - w_\perp^{\gamma(N)+3}| + |\tp^{\gamma(N)+4} - w_\perp^{\gamma(N)+4}| )^2}
    {(1-\delta)\delta^{j_1-1}r} + \frac{|\tp^{\gamma(N)+4} - w_\perp^{\gamma(N)+4}|^2}{\delta^{j_1}r} \Bigg] \notag\\
  &\leq \frac{|\tp^{\gamma(N)+1} - w_\perp^{\gamma(N)+1}|^2}{\delta^{j_1}r}
    + \frac{|\tp^{\gamma(N)+4} - w_\perp^{\gamma(N)+4}|^2}{\delta^{j_1}r}
    + \frac{6}{\delta^{j_1-1}r} \max_{1\leq i\leq 4} |\tp^{\gamma(N)+i} - w_\perp^{\gamma(N)+i}|^2 \notag\\
  &\leq \frac{1+6\delta}{16\mu} \left( \frac \lambda 7\right)^{j_1}
    \left( t^*(\tp^{\gamma(N)}) + t^*(\tp^{\gamma(N+1)} \right)\sigma_r.
\end{align}
Since \eqref{j1} ensures the right side of \eqref{Sjlower} is much larger than $\log r$, the claim \eqref{notrack} follows from \eqref{Sjlower} and \eqref{pieces}.

{\it Case 3d}.  The totally unbowed case (third option in \eqref{Lminus}) for both $L^\pm(I_N)$.  Here we have $V(N)=4$ as there are the sandwiching hyperplanes at distance $\delta^{j_1}r$ from each end of $I_N$, and inner joining hyperplanes at distance $\delta^{k+1}r$ from each end, with $k+1\leq j_1-1$ such that $I_N$ is of $k$th--scale length; there are no outer joining hyperplanes.  See Figure \ref{Case3d}. From \eqref{alphagrow} (valid now for $k+1$ in place of $\ell+1$) and \eqref{Lminus}, this means that
\begin{align}\label{totun}
  \frac{|\tp^{\gamma(N)+2} - w_\perp^{\gamma(N)+2}|^2}{\delta^{k+1}r} 
    &\geq \frac 12 \max\Bigg( 2^{j_1-k-1} \frac{\sigma(\delta^{k+1}r)}{\sigma(\delta^{j_1}r)} 
    \frac{|\tp^{\gamma(N)+1} - w_\perp^{\gamma(N)+1}|^2}{\delta^{j_1}r}, \notag\\
  &\hskip 1.8cm 2^{j_1-k-2} \frac{\sigma(\delta^{k+1}r)}{\sigma(\delta^{j_1-1}r)}
    \frac{\delta}{16\mu}\left( \frac \lambda 7\right)^{j_1}t^*(\tp^{\gamma(N)})\sigma_r \Bigg).
\end{align}
Here the second term in the max comes from the fact that to come under the third (totally unbowed) option in \eqref{Lminus}, we must have $2^{k+1}\kappa(k+1) \geq 2^{j_1-1}\kappa(j_1-1)$.
We note that the definition of the totally unbowed case gives \eqref{totun} for $\hw^i$ in place of $w_\perp^i$, but by \eqref{forward} this only changes each of the 4 terms in \eqref{totun} by a factor $1+o(1)$ as $r\to\infty$, so we have accounted for this via the factor 1/2 in front of the max.  The bound \eqref{totun} is for the left end of $I_N$; a symmetric bound is valid for the right end.
As before, the direct path is $\tp^{\gamma(N)}\to \tp^{\gamma(N+1)}$, and the full path is 
\[
  \tred{\Gamma^{N,full}}:\ \tp^{\gamma(N)}\to \tp^{\gamma(N)+1}\to \cdots\to \tp^{\gamma(N)+4}\to \tp^{\gamma(N+1)},
\]
but this time the intermediate path is 
\[
  \tred{\Gamma^{N,int}}:\ w_\perp^{\gamma(N)}\to w_\perp^{\gamma(N)+1}\to \cdots\to 
    w_\perp^{\gamma(N)+4}\to w_\perp^{\gamma(N+1)}
\]
These points are collinear so $\Upsilon_{Euc}(\Gamma^{N,int}) = |\tp^{\gamma(N+1)} - \tp^{\gamma(N)}|$.

\begin{figure}
\includegraphics[width=17cm]{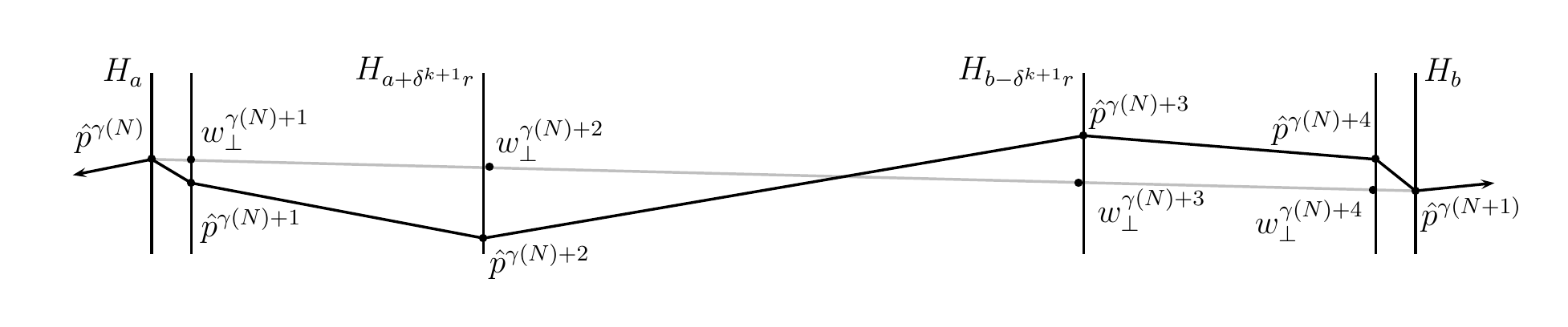}
\caption{ A long $(j_1-1)$th--scale interval $I_N=[a,b]$ of $k$th--scale length with the totally unbowed case, with $L^-(I_n)=k+1$, showing the full path (black) and intermediate path (gray.) The full path has marked points at the $\hat p^{\gamma(N)+i}$. }
\label{Case3d}
\end{figure}

We use a special allocation for the totally unbowed case, the same for all 10 links of $\Gamma^{N,full}$ and $\Gamma^{N,int}$, defined (when $I_N$ is of $k$th--scale length) as
\[
  \tred{A_j^4(\Gamma^{N,full})} = \frac{\delta\mu}{540} 
    \left( \frac{|\tp^{\gamma(N)+2} - w_\perp^{\gamma(N)+2}|^2}{\delta^{k+1}r}
    + \frac{|\tp^{\gamma(N)+3} - w_\perp^{\gamma(N)+3}|^2}{\delta^{k+1}r} + \lambda^j 
    \frac{|(\tp^{\gamma(N+1)} - \tp^{\gamma(N)})^*|^2}{3\delta^k r} \right).
\]
Now $\Upsilon_{Euc}\left( \Gamma^{N,int} \right) = |\tp^{\gamma(N+1)} - \tp^{\gamma(N)}|$ so by Lemma \ref{monotoneE},
\[
  \Upsilon_h\left( \Gamma^{N,full} \right) - \Upsilon_h\left( \Gamma^{N,int} \right)
    \geq \frac \mu 2 \Big[ \Upsilon_{Euc}\left( \Gamma^{N,full} \right)
    - |\tp^{\gamma(N+1)} - \tp^{\gamma(N)}| \Big] - 4C_{56}.
\]
while from\eqref{adddist1},
\begin{equation}\label{Sjlower2}
  S_{j_1}(\Gamma^{N,full}) \geq \lambda^{j_1}\frac{\delta\mu}{18} \Big[ \Upsilon_{Euc}\left( \Gamma^{N,full} \right)
    - (\tp^{\gamma(N+1)} - \tp^{\gamma(N)})_1 \Big] + \frac{1}{3}\left(\frac \lambda 7 \right)^{j_1}t\sigma_r .
\end{equation}
From these we get deterministic tracking:
\begin{align}\label{dtrac}
  10&A_{j_1}^4(\Gamma^{N,full}) \notag\\
  &\leq \frac{\delta\mu}{54} \Bigg( 3\Big[ \Upsilon_{Euc}\left( \Gamma^{N,full} \right)
    - |\tp^{\gamma(N+1)} - \tp^{\gamma(N)}| \Big] \notag\\
  &\hskip 2cm + 3\lambda^{j_1}\Big[ |\tp^{\gamma(N+1)} - \tp^{\gamma(N)}|
    - (\tp^{\gamma(N+1)} - \tp^{\gamma(N)})_1 \Big] \Bigg) \notag\\
  &\leq \left( \frac 12 - \delta \right)\mu \Big[ \Upsilon_{Euc}\left( \Gamma^{N,full} \right)
    - |\tp^{\gamma(N+1)} - \tp^{\gamma(N)}| \Big] \notag\\
  &\hskip 2cm + \lambda^{j_1}\frac{\delta\mu}{18} \Big[ \Upsilon_{Euc}\left( \Gamma^{N,full} \right)
    - (\tp^{\gamma(N+1)} - \tp^{\gamma(N)})_1 \Big] \notag\\
  &\leq \Upsilon_h\left( \Gamma^{N,full} \right) - \Upsilon_h\left( \Gamma^{N,int} \right) 
    - \delta\mu \Big[ \Upsilon_{Euc}\left( \Gamma^{N,full} \right) - |\tp^{\gamma(N+1)} - \tp^{\gamma(N)}| \Big]  \notag\\
  &\hskip 2cm + S_{j_1}(\Gamma^{N,full}) - \frac{1}{6}\left(\frac \lambda 7 \right)^{j_1}t\sigma_r.
\end{align}
We now can establish the analog of \eqref{segments} and \eqref{segments6} for the totally--unbowed--case contribution to the tracking--failure probability \eqref{trackprob}, using \eqref{nearsub} and \eqref{dtrac}:
\begin{align}
  &P\Bigg( \Upsilon_{\hT}\left( \Gamma^{N,full} \right) - \hT( \tp^{\gamma(N)},\tp^{\gamma(N+1)} ) \leq  
    - S_{j_1}\left( \Gamma^{N,full} \right) + \delta\mu\Big( \Upsilon_{Euc}\left( \Gamma^{N,full} \right)  
    - | \tp^{\gamma(N+1)} - \tp^{\gamma(N)} | \Big) \notag\\
  &\hskip 0.6cm \text{ for some $N$ with $I_N$ long non--terminal $(j_1-1)$th--scale, in the totally unbowed case of}\notag\\
  &\hskip 0.6cm \text{\eqref{Lminus} for both of $L^\pm(I_N)$, 
    for some } (x,y)\in X_r;\ \omega\notin J^{(0)}(c_{29})\cup J^{(1c)} \Bigg) \notag\\
  &\leq P\Bigg( \Upsilon_{\hT}\left( \Gamma^{N,full} \right) - \Upsilon_{\hT}\left( \Gamma^{N,int} \right) \leq  
    \Upsilon_h\left( \Gamma^{N,full} \right) - \Upsilon_h\left( \Gamma^{N,int} \right) 
    - 10A_{j_1}^4(\Gamma^{N,full})  - \frac{1}{6}\left(\frac \lambda 7 \right)^{j_1}t\sigma_r \notag\\
  &\hskip 1.2cm 
     + 4c_{29}\log r \text{ for some $N$ with $I_N$ long non--terminal $(j_1-1)$th--scale, in the totally unbowed}\notag\\
  &\hskip 1.2cm \text{case of \eqref{Lminus} for both of $L^\pm(I_N)$, for some } 
    (x,y)\in X_r;\ \omega\notin J^{(0)}(c_{29})\cup J^{(1c)} \Bigg) \notag
\end{align}
\begin{align}
  &\leq P\Bigg( \Big| \hT( \tp^{\gamma(N)+i-1},\tp^{\gamma(N)+i} ) - h\big(|\tp^{\gamma(N)+i} - \tp^{\gamma(N)+i-1}|\big) \Big|  \geq  
    A_{j_1}^4(\Gamma^{N,full}) \text{ or} \notag\\
  &\hskip 1.2cm \Big| \hT( w_\perp^{\gamma(N)+i-1},w_\perp^{\gamma(N)+i} ) - h\big(|w_\perp^{\gamma(N)+i} - w_\perp^{\gamma(N)+i-1}|\big) 
    \Big|  \geq  A_{j_1}^4(\Gamma^{N,full}) \text{ for some $1\leq i\leq 5$}\notag\\
  &\hskip 1.2cm \text{and some $N$ with $I_N$ long non--terminal $(j_1-1)$th--scale, in the totally unbowed case} \notag\\
  &\hskip 1.2cm \text{of \eqref{Lminus} for both of $L^\pm(I_N)$, for some } 
    (x,y)\in X_r;\ \omega\notin J^{(0)}(c_{29})\cup J^{(1c)} \Bigg).  \label{segmentsun}
\end{align}
Here we have again used $j_1=O(\log \log r)$, from \eqref{j1}, to ensure $(\lambda/7)^{j_1}t\sigma_r/6\geq 4c_{29}\log r$.
With the second term of the max in \eqref{totun} in mind, suppose that for some $N$ and some $\nu,m_0,m^*\geq 1$ we have
\begin{align}\label{nusize}
  (2^{\nu-1}t)^{1/2}\Delta_r < |(\tp^{\gamma(N)})^*| \leq (2^\nu t)^{1/2}\Delta_r,
\end{align}
\begin{align}\label{m0size}
  2^{m_0-1} 2^{j_1-k-2} \frac{\sigma(\delta^{k+1}r)}{\sigma(\delta^{j_1-1}r)}
    \frac{\delta}{16\mu}\left( \frac \lambda 7\right)^{j_1}t^*(\tp^{\gamma(N)})\sigma_r
    &\leq \frac{|\tp^{\gamma(N)+2} - w_\perp^{\gamma(N)+2}|^2}{\delta^{k+1}r}
    + \frac{|\tp^{\gamma(N)+3} - w_\perp^{\gamma(N)+3}|^2}{\delta^{k+1}r} \notag\\
  &\leq 2^{m_0} 2^{j_1-k-2} \frac{\sigma(\delta^{k+1}r)}{\sigma(\delta^{j_1-1}r)}
    \frac{\delta}{16\mu}\left( \frac \lambda 7\right)^{j_1}t^*(\tp^{\gamma(N)})\sigma_r,
\end{align}
and 
\begin{align}\label{mstar}
  2^{m^*-1} 2^{j_1-k-2} \frac{\sigma(\delta^{k+1}r)}{\sigma(\delta^{j_1-1}r)}
    \frac{\delta}{16\mu} \left( \frac \lambda 7\right)^{j_1} t^*(\tp^{\gamma(N)})\sigma_r 
    &\leq  \lambda^{j_1-1} \frac{|(\tp^{\gamma(N+1)} - \tp^{\gamma(N)})^*|^2}{\delta^k r} \notag\\
  &\leq 2^{m^*} 2^{j_1-k-2} \frac{\sigma(\delta^{k+1}r)}
    {\sigma(\delta^{j_1-1}r)}
    \frac{\delta}{16\mu}\left( \frac \lambda 7\right)^{j_1}t^*(\tp^{\gamma(N)})\sigma_r
\end{align}
Then using \eqref{powerlike},
\begin{align}\label{expcost}
  \frac{ A_{j_1}^4(\Gamma^{N,full}) }{\sigma(\delta^kr)} &\geq c_{37} (2^{m_0} + 2^{m^*}) 2^{j_1-k} 
    \left( \frac{\lambda}{7\delta^{\chi_1}} \right)^{j_1} 2^\nu t,
\end{align}
and by Lemma \ref{connect}, for every possible value $(u^1,\dots,u^6)$ of $(\tp^{\gamma(N)},\dots,\tp^{\gamma(N+1)})$ and every link $(u^{i-1},u^i)$, we have 
\begin{equation}\label{eachlink}
  P\Bigg( \Big| \hT(u^i-u^{i-1}) - h\big(|u^{i-1}-u^i|\big) \Big|  \geq  A_{j_1}^4(u^1,\dots,u^6) \Bigg)
    \leq C_{44}\exp\left( - C_{45} \frac{ A_{j_1}^4(u^1,\dots,u^6) }{\sigma(\delta^kr)} \right).
\end{equation}
It is important here that the lower bound \eqref{expcost} not depend on the length scale $k$ of $I_N$, except through the factor $2^{j_1-k}$ which is always at least 1.
Similarly to the entropy bounds in Lemmas \ref{sumbits} and \ref{sumbits2}, and in Case 3c, we see that the number of possible choices of $\Gamma^{N,full}$ satisfying \eqref{nusize}---\eqref{mstar} is at most
\begin{equation}\label{entbd}
  \frac{c_{38}}{\delta^{2j_1}}\cdot
  \left( \frac{2^\nu t}{\beta^{2j_1}} \right)^{(d-1)/2}\cdot \left( 2^{m^*}2^\nu \left( \frac{2}{\delta^{1+\chi_2}} \right)^{j_1-k} 
    \left( \frac{\delta}{7\beta^2} \right)^{j_1} \right)^{(d-1)/2}
    \cdot \Bigg( 2^{m_0}2^\nu \left( \frac{2}{\delta^{1+\chi_2}} \right)^{j_1-k} 
    \left( \frac{\lambda\delta}{7\beta^2} \right)^{j_1}
    t \Bigg)^{2(d-1)}.
\end{equation}
Here the dots separate bounds for the number of choices (up to a constant) of endpoint hyperplanes $((\tp^{\gamma(N)})_1,(\tp^{\gamma(N+1)})_1)$ and then of $\tp^{\gamma(N)},\tp^{\gamma(N+1)}$, and finally of $(\tp^{\gamma(N)+1},\dots,\tp^{\gamma(N)+4})$. For $|(\tp^{\gamma(N)})^*| \leq t^{1/2}\Delta_r$, \eqref{expcost} and \eqref{entbd} remain valid for $m_0,m^*\geq 1$ with $2^\nu$ replaced by 1. For $\tp^{\gamma(N)},\tp^{\gamma(N+1)}$ with
\[
   \lambda^{j_1-1} \frac{|(\tp^{\gamma(N+1)} - \tp^{\gamma(N)})^*|^2}{\delta^k r} \notag\\
    \leq 2^{j_1-k-2} \frac{\sigma(\delta^{k+1}r)}{\sigma(\delta^{j_1-1}r)} 
    \frac{\delta}{16\mu}\left( \frac \lambda 7\right)^{j_1}t^*(\tp^{\gamma(N)})\sigma_r
\]
(i.e.~too small to satisfy \eqref{mstar} for any $m^*\geq 1$) 
they remain valid with $2^{m^*}$ replaced by 1. (Note that from the definition \eqref{Lminus} of the totally unbowed case, $m_0\geq 1$ covers all cases.) Hence as in previous cases, inserting the bound \eqref{expcost} into \eqref{eachlink}, multiplying by the entropy factor \eqref{entbd}, and summing over $\nu\geq 0, m^*\geq 0, m_0\geq 1$, and $k\leq j_1-2$ we get that the right side of \eqref{segmentsun} is bounded by 
\begin{equation}\label{case3d}
  c_{39}\left( \frac{2\lambda}{ 7^5\beta^{12}\delta^{5\chi_2} } \right)^{(d-1)j_1/2}
    \exp\left( - c_{40}  \left( \frac{\lambda}{7\delta^{\chi_1}} \right)^{j_1}t \right).
\end{equation}
We have thus bounded the totally--unbowed--case contribution to the tracking--failure probability \eqref{trackprob}.

{\it Case 3e.}  Mixed cases, which we subdivide as mixed forward case, mixed bowed case, and mixed totally unbowed case, according to the condition at the dominant end (as defined after \eqref{middle}) of $I_N$.  In previous cases we have assumed that the same option in \eqref{Lminus} occurs at both ends of the interval $I_N$, but this is strictly for clarity of exposition.  As explained in Step 1, in mixed cases we have joining hyperplanes only at the dominant end of $I_N$.  The situation is symmetric so let us assume the left end is dominant.  In mixed cases the full path in $I_N$ is always $\tp^{\gamma(N)}\to \tp^{\gamma(N)+1}\to\cdots\to\tp^{\gamma(N+1)}$, with $\gamma(N+1)-\gamma(N)=4$ or 5.  Suppose $I_N$ is a $(j_1-1)$th--scale interval of $k$th--scale length.  The intermediate paths are analogous to those in cases 3a---3d, as follows.
If the dominant left end has the forward case then there is one joining hyperplane in $I_N$, containing $\tp^{\gamma(N)+2}$, at the left end at distance $\delta^{j_1+1}r$ from $\tp^{\gamma(N)}$, and the intermediate path is
\[
  \tred{\Gamma^{N,int}}:\ \tp^{\gamma(N)}\to w_\perp^{\gamma(N)+1}\to\tp^{\gamma(N)+2}\to 
    \tp^{\gamma(N)+3}\to \tp^{\gamma(N+1)}.
\]
If the (dominant) left end has the totally unbowed case then there is one (inner) joining hyperplane in $I_N$, containing $\tp^{\gamma(N)+2}$, at the left end at distance $\delta^{k+1}r$ from $\tp^{\gamma(N)}$, and the intermediate path follows the line from $\tp^{\gamma(N)}$ to $\tp^{\gamma(N)+1}$:
\[
  \tred{\Gamma^{N,int}}:\ \tp^{\gamma(N)}\to w_\perp^{\gamma(N)+1}\to\cdots\to 
    w_\perp^{\gamma(N)+3}\to \tp^{\gamma(N+1)}.
\]
If the left end has the bowed case with $L^-(N)=\ell<j_1-1$ then there are two joining hyperplanes in $I_N$, containing $\tp^{\gamma(N)+2}$ and $\tp^{\gamma(N)+3}$, at the left end at distances $\delta^{\ell+1}r$ and $\delta^\ell r$ from $\tp^{\gamma(N)}$, and the intermediate path is
\[
  \tred{\Gamma^{N,int}}:\ \tp^{\gamma(N)}\to \hz^{\gamma(N)+1}\to\hz^{\gamma(N)+2}\to \tp^{\gamma(N)+3}\to 
    \tp^{\gamma(N)+4} \to \tp^{\gamma(N+1)};
\]
see Figure \ref{Case3a}. 
If the left end has the bowed case with $L^-(N)=j_1-1$ then there is one (inner) joining hyperplane in $I_N$, containing $\tp^{\gamma(N)+2}$, at the left end at distance $\delta^{j_1-1}r$ from $\tp^{\gamma(N)}$, and the intermediate path is
\[
  \tred{\Gamma^{N,int}}:\ \tp^{\gamma(N)}\to \hz^{\gamma(N)+1}\to\tp^{\gamma(N)+2}\to 
    \tp^{\gamma(N)+3} \to \tp^{\gamma(N+1)},
\]
similar to Figure \ref{Case3a} but with the middle two hyperplanes coinciding.
In each case, we define the subpaths $\tred{\Gamma^{N,full,-}}$ and $\tred{\Gamma^{N,int,-}}$ between the left end of the interval and the left inner joining hyperplane.
In all cases the arguments are essentially the same as in the analogous non-mixed case, with the addition that for the final transition ending at $\tp^{\gamma(N+1)}$, one uses \eqref{nondom2} to bound $|\tp^{\gamma(N+1)-1} - w_\perp^{\gamma(N+1)-1}|$.  We will not reiterate the arguments here, but simply state that the result is again the analog of \eqref{segments} and \eqref{segments6}, bounding the mixed--case contribution to the tracking--failure probability \eqref{trackprob}:
\begin{align} \label{segmentsmix}
  &P\Bigg( \hT( \tp^{\gamma(N)},\tp^{\gamma(N+1)} ) - \Upsilon_{\hT}\left( \Gamma^{N,full} \right) \geq  
    S_{j_1}\left( \Gamma^{N,full} \right) + \delta\mu\Big( | \tp^{\gamma(N+1)} - \tp^{\gamma(N)} | 
    - \Upsilon_{Euc}\left( \Gamma^{N,full} \right) \Big) \notag\\
  &\hskip 0.6cm + \frac{1}{3}\lambda^{j_1} \Big[ t\sigma_r 
    + \delta\mu\Big( \Upsilon_{Euc}\left( \Gamma_{xy}^{j_1-1,2} \right) - (\hat y - \hat x)_1 \Big) \Big]
    \text{ for some $N$ with $I_N$ long non--terminal} \notag\\
  &\hskip 0.6cm \text{$(j_1-1)$th--scale, in the mixed case of \eqref{Lminus} for some } 
    (x,y)\in X_r;\ \omega\notin J^{(0)}(c_{29})\cup J^{(1c)} \Bigg) \notag\\
  &\leq P\Bigg( \Big| \hT(u,v) - h(|v-u|) \Big| \geq A \text{ for some link $(u,v)$ of $\Gamma^{N,full,-}$ or $\Gamma^{N,int,-}$
    (or $\Gamma^{N,full},\Gamma^{N,int}$} \notag\\
  &\hskip 1.3cm \text{in the mixed totally unbowed case) for some $N$ with $I_N$ long non--terminal}
    \notag\\
  &\hskip 1.3cm \text{$(j_1-1)$th--scale, in the mixed case of \eqref{Lminus} for some}
    (x,y)\in X_r;\ \omega\notin J^{(0)}(c_{29})\cup J^{(1c)} \Bigg),
\end{align}
where
\[
  \tred{A} = \begin{cases} \infty 
    &\text{in the mixed forward case for $I_N$},\\
  A_{j_1}^3(\Gamma^{N,full,-}) &\text{in the mixed bowed case for $I_N$}, \\
  A_{j_1}^4(\Gamma^{N,full}) &\text{in the mixed totally unbowed case for $I_N$}. \end{cases}
\]
As in Cases 3a--3d, the right side of \eqref{segmentsmix}, and thus the mixed--case contribution to the tracking--failure event \eqref{trackprob}, is bounded above by 
\begin{equation}\label{case3e}
  f(\lambda,\delta,\beta)^{j_1}\exp\left( -c_{41} \left( \frac{\lambda}{7\delta^{\chi_1}} \right)^{j_1}t \right)
\end{equation}
for some $\tred{f(\lambda,\delta,\beta)}$.

We have now bounded all contributions to \eqref{trackprob}, by the sum of \eqref{segments2}, \eqref{terminal}, \eqref{case3a}, \eqref{case3b}, \eqref{nofail}, \eqref{case3d}, and \eqref{case3e}.  From this and \eqref{secondit2} we have the completion of the $(j_1-1)$th--scale iteration:
\begin{align}\label{secondit3}
    P&\Big( T(x,y) \leq h((y-x)_1) - t\sigma_r \text{ for some } (x,y)\in X_r \Big) \notag\\
  &\leq P\Bigg( \Upsilon_{\hT}\left( \Gamma_{xy}^{j_1-1,2} \right) - h((y-x)_1) \leq -\left( 1-5\lambda^{j_1} \right) t\sigma_r
    + 7\delta\mu\lambda^{j_1}\Big(  \Upsilon_{Euc}\left( \Gamma_{xy}^{j_1-1,0} \right) - (\hat y - \hat x)_1 \Big)  \notag\\
  &\hskip 2cm + \delta\mu\Big( \Upsilon_{Euc}\left( \Gamma_{xy}^{j_1-1,2} \right) 
    -\Upsilon_{Euc}\left( \Gamma_{xy}^{j_1-1,0} \right) \Big)
    \text{ for some } (x,y)\in X_r;\ \omega\notin J^{(0)}(c_{29})\cup J^{(1c)} \Bigg) \notag\\
  &\hskip 1.2cm + c_{32}e^{-c_{33}t} 
    + c_{42} \exp\left( -c_{43}\left( \frac{\lambda}{7\delta^{\chi_1}} \right)^{j_1} t \right).
\end{align} 

As in the comments after \eqref{firstit} and \eqref{secondit}, the increase of the coefficients 4, 6 in \eqref{firstit} to be 5, 7 in \eqref{secondit}, together with an increment to the ``length--change'' term with coefficient $\delta\mu$, represent a further reduction taken from the original bound $t\sigma_r$ in \eqref{Qrunif2}, and allocated to bound errors in the second stage of the $(j_1-1)$th--scale iteration, just completed. In this second stage, however, the increment of the length--change term, due to replacing $\Gamma_{xy}^{j_1-1,1}$ with $\Gamma_{xy}^{j_1-1,2}$, is negative (removing points reduces path length) and may compensate at least in part for the increases from 4, 6 to 5, 7.  This compensation, made possible by control of tracking errors, is what allows the total allocation through all iterations to potentially exceed the entire original bound $t\sigma_r$---see \eqref{cancel}.

\subsection{Step 7. Further iterations of coarse--graining.}  The further iterations proceed quite similarly to the $(j_1-1)$th--scale iteration; for the most part, to do the $j$th--scale iteration we simply replace $j_1-1,j_1$ throughout by $j,j+1$. We will sketch the $(j_1-2)$th--scale (third) iteration to make the pattern clear.

For the first stage of the $(j_1-2)$th--scale iteration, the current marked PG path at the start is $\Gamma_{xy}^{j_1-1,2}: b^0 \to b^1\to \cdots \to b^{n+1}$, which we rename $\tred{\Gamma_{xy}^{j_1-2,0}}$.  
Letting 
\[
  \tred{I_{xy}^2} = \{i\in \{3,\dots,n-2\}: b^i \text{ lies in a non-incidental $(j_1-2)$th--scale hyperplane in $\mH_{xy}$}\}
\]
we shift each $b^i,\,i\in I_{xy}^2$, to the $(j_1-2)$th--scale grid, creating the updated marked PG path
\[
  \tred{\Gamma_{xy}^{j_1-2,1}: \tb^0 \to \tb^1 \to\cdots\to \tb^n\to \tb^{n+1}}.
\]
Removing those $b^i$ which lie in non-terminal $(j_1-1)$th--scale hyperplanes then creates the marked PG path
\[
  \tred{\Gamma_{xy}^{j_1-2,2}: \zeta^0 \to \zeta^1 \to\cdots\to \zeta^{n_2}\to \zeta^{n_2+1}},
\]
in which (recalling \eqref{startend}) $\zeta^0=\hat x,\,\zeta^{n_2+1}=\hat y$, $\zeta^1$ and $\zeta^{n_2}$ lie in terminal $j_1$th--scale hyperplanes, $\zeta^2,\zeta^{n_2-1}$ lie in terminal $(j_1-1)$th--scale hyperplanes, and $\zeta^3,\dots,\zeta^{n_2-2}$ lie in $(j_1-2)$th--scale hyperplanes.
Noninteraction of shifts again applies (see Step 4), since if $b^i$ and $b^k$ both lie in $(j_1-2)$th--scale hyperplanes then there are sandwiching hyperplanes, at least, in between, ensuring $|i-k|\geq 2$.  For $N\leq n_2$ we write $\tred{\Gamma^{2,N,full}}$ for the portion of $\Gamma_{xy}^{j_1-2,1}$ in the interval (denoted $\tred{I_N^2}$) from $(\zeta^N)_1$ to $(\zeta^{N+1})_1$, and define $\tred{\gamma(2,N)}$ by $\zeta^N=\tb^{\gamma(2,N)}, N\leq n_2+1$, analogous to $\gamma(N)$ used in the previous iteration.

Similarly to \eqref{Atotal2}---\eqref{upgaph} we have
\begin{equation}\label{AtotalG}
  \sum_{i\in I_{xy}} A_{j_1-1}^1(\Gamma_{xy}^{j_1-2,0},i) \leq \frac{2}{3}\lambda^{j_1-1}\Big[ t\sigma_r 
   + \delta\mu\Big( \Upsilon_{Euc}\left( \Gamma_{xy}^{j_1-2,0} \right) - (\hat y - \hat x)_1 \Big)
    \Big], 
\end{equation}
\begin{equation}\label{upgapG}
  \delta\mu\Big| \Upsilon_{Euc}\left( \Gamma_{xy}^{j_1-2,1} \right) -\Upsilon_{Euc}\left( \Gamma_{xy}^{j_1-2,0} \right) \Big|
    \leq \frac{\delta \lambda^{j_1-1}}{2} \left[ \frac{t\sigma_r}{3}
    + \delta\mu\Big( \Upsilon_{Euc}(\Gamma_{xy}^{j_1-2,0}) - (\hat y - \hat x)_1 \Big) \right],
\end{equation}
and
\begin{equation}\label{upgaphG}
  \Big| \Upsilon_h\left( \Gamma_{xy}^{j_1-2,1} \right) -\Upsilon_h\left( \Gamma_{xy}^{j_1-2,0} \right) \Big|
    \leq  \lambda^{j_1-1} \left[ \frac{t\sigma_r}{3}
    + \delta\mu\Big( \Upsilon_{Euc}(\Gamma_{xy}^{j_1-2,0}) - (\hat y - \hat x)_1 \Big) \right].
\end{equation}
These lead to a tracking bound like \eqref{badtrans} for the first stage (shifting to the $(j_1-2)$th--scale grid) using Lemma \ref{sumbits}:
\begin{align}\label{badtransG}
  P&\Bigg( \Upsilon_{\hT}(\Gamma_{xy}^{j_1-2,0}) -  \Upsilon_{\hT}(\Gamma_{xy}^{j_1-2,1})
    < - \sum_{i\in I_{xy}^2} A_{j_1-1}^1(\Gamma_{xy}^{j_1-2,0},i) + \Upsilon_h\left( \Gamma_{xy}^{j_1-2,0} \right) 
    -\Upsilon_h\left( \Gamma_{xy}^{j_1-2,1} \right) \notag\\
  &\hskip 1.5cm \text{ for some } (x,y)\in X_r;\ \omega\notin J^{(0)}(c_{29})\cup J^{(1c)} \Bigg) \notag\\
  &\leq P\Bigg( \max\Big( \big|\hT(\hb^{i-1},\hb^i) - h(|\hb^i-\hb^{i-1}|)\big|, \big|\hT(b^{i-1},b^i) - h(|b^i-b^{i-1}|)\big| \Big) 
    > \frac 14 A_{j_1-1}^1(\Gamma_{xy}^{j_1-2,0},i)  \notag\\
  &\hskip 1.5cm \text{ for some $1\leq i\leq n+1$ and $(x,y)\in X_r$ with $\{i-1,i\}\cap I_{xy}^2\neq\emptyset$};\, 
    \omega\notin J^{(0)}(c_{29})\cup J^{(1c)} \Bigg) \notag\\
  &\leq P\Bigg( \big| \hT(v,w) - h(|w-v|) \big| \geq  A_{j_1-1}^2(v,w) \notag\\
  &\hskip 1.5cm \text{ for some $(j_1-1)$th--scale transition $v\to w$ with } v,w\in G_r^+\cap\LL_{j_1-1} \Bigg) \notag\\
  &\leq  C_{70} \exp\left( -C_{71}\left( \frac{\lambda}{7\delta^{\chi_1}} \right)^{j_1-1} t \right).
\end{align}
As with \eqref{firstit} and \eqref{secondit}, with \eqref{secondit3} this yields
\begin{align}\label{seconditG}
  &P\Big( T(x,y) \leq h((y-x)_1) - t\sigma_r \text{ for some } (x,y)\in X_r \Big) \notag\\
  &\leq P\Bigg( \Upsilon_{\hT}(\Gamma_{xy}^{j_1-2,1}) - h((y-x)_1) \leq -\left( 1-2\lambda^{j_1} \right) t\sigma_r
    + 4\delta\mu\lambda^{j_1}\Big(  \Upsilon_{Euc}\left( \Gamma_{xy}^{j_1-1,0} \right) - (\hat y - \hat x)_1 \Big)  \notag\\
  &\hskip 2cm + 3 \lambda^{j_1}\Big( t\sigma_r 
    + \delta\mu\Big[  \Upsilon_{Euc}\left( \Gamma_{xy}^{j_1-1,0} \right) - (\hat y - \hat x)_1 \Big] \Big) \notag\\
  &\hskip 2cm + 2 \lambda^{j_1-1}\Big( t\sigma_r 
    + \delta\mu\Big[  \Upsilon_{Euc}\left( \Gamma_{xy}^{j_1-2,0} \right) - (\hat y - \hat x)_1 \Big] \Big) \notag\\
  &\hskip 2cm + \delta\mu\Big( \Upsilon_{Euc}\left( \Gamma_{xy}^{j_1-2,1} \right) 
    -\Upsilon_{Euc}\left( \Gamma_{xy}^{j_1-1,0} \right) \Big)
    \text{ for some } (x,y)\in X_r;\ \omega\notin J^{(0)}(c_{29})\cup J^{(1c)} \Bigg) \notag\\
  &\hskip 1.2cm + c_{32}e^{-c_{33}t} + c_{42} \exp\left( -c_{43}\left( \frac{\lambda}{7\delta^{\chi_1}} \right)^{j_1} t \right)
    + C_{70} \exp\left( -C_{71}\left( \frac{\lambda}{7\delta^{\chi_1}} \right)^{j_1-1} t \right).
\end{align}

\begin{remark}\label{totalloc}
The second probability here contains three quantities of the form 
\[
  \lambda^{j+1}\Big( t\sigma_r 
    + \delta\mu\Big[  \Upsilon_{Euc}\left( \Gamma_{xy}^{j,0} \right) - (\hat y - \hat x)_1 \Big] \Big),
\]
with certain additional integer coefficients: 2 and 4 in the first quantity, 3 in the second quantity, and 2 in the third.
As noted in Remark \ref{tracking} and after \eqref{secondit}, \eqref{trackprob}, and \eqref{secondit3}, these represent the accumulated error allocations from the $j_1$th, $(j_1-1)$th, and (first--stage) $(j_1-2)$th--scale iterations, respectively. We have split these out here for clarity; in \eqref{secondit2} and \eqref{secondit3} the first two of the three quantities are combined, giving the terms with coefficients 5 and 7.  In general after the first stage of the $j$th--scale iteration (which is the source of the third quantity), the second quantity would represent allocations from all stages $j_1-1$ through $j+1$.  With this in mind we define the accumulated--allocations upper bounds
\begin{align}\label{Aj1acc}
  \tred{\mA_j^1(\Gamma_{xy})} = &2 \lambda^{j_1}\Big( t\sigma_r 
    + 2\delta\mu\Big[  \Upsilon_{Euc}\left( \Gamma_{xy}^{j_1-1,0} \right) - (\hat y - \hat x)_1 \Big] \Big) \notag\\
  &+ \sum_{k=j+1}^{j_1-1} 3 \lambda^{k+1}\Big( t\sigma_r 
    + \delta\mu\Big[  \Upsilon_{Euc}\left( \Gamma_{xy}^{k,0} \right) - (\hat y - \hat x)_1 \Big] \Big) \notag\\
  &+ 2 \lambda^{j+1}\Big( t\sigma_r 
    + \delta\mu\Big[  \Upsilon_{Euc}\left( \Gamma_{xy}^{j,0} \right) - (\hat y - \hat x)_1 \Big] \Big)
\end{align}
and
\begin{align}\label{Aj2acc}
  \tred{\mA_j^2(\Gamma_{xy})} = &2 \lambda^{j_1}\Big( t\sigma_r 
    + 2\delta\mu\Big[  \Upsilon_{Euc}\left( \Gamma_{xy}^{j_1-1,0} \right) - (\hat y - \hat x)_1 \Big] \Big) \notag\\
  &+ \sum_{k=j}^{j_1-1} 3 \lambda^{k+1}\Big( t\sigma_r 
    + \delta\mu\Big[  \Upsilon_{Euc}\left( \Gamma_{xy}^{k,0} \right) - (\hat y - \hat x)_1 \Big] \Big),
\end{align}
with $\mA_j^1(\Gamma_{xy})$ a valid upper bound after the first stage of the $j$th--scale iteration, and $\mA_j^2(\Gamma_{xy})$ valid after the second stage.  Comparing \eqref{Aj1acc} to \eqref{secondit}, in this case we have $j=j_1-1$, the sum in \eqref{Aj1acc} has no terms, and the other two terms on the right in \eqref{secondit} merge into one. Comparing \eqref{Aj2acc} to \eqref{secondit3}, in this instance also $j=j_1-1$, and the sum in \eqref{Aj2acc} has one term which is merged with the other term in \eqref{secondit3}.  Equation \eqref{Aj2acc} may also be compared to \eqref{firstit}, where $j=j_1$ and the sum in \eqref{Aj2acc} has no terms.
We can rewrite the second probability in \eqref{seconditG} as
\begin{align}\label{seconditG2}
  P&\Bigg( \Upsilon_{\hT}\left( \Gamma_{xy}^{j_1-2,1} \right) - h((y-x)_1) \leq -t\sigma_r + \mA_{j_1-2}^1(\Gamma_{xy}) \notag\\
  &\hskip 1cm + \delta\mu\Big( \Upsilon_{Euc}\left( \Gamma_{xy}^{j_1-2,1} \right) 
    -\Upsilon_{Euc}\left( \Gamma_{xy}^{j_1-1,0} \right) \Big)
    \text{ for some } (x,y)\in X_r;\ \omega\notin J^{(0)}(c_{29})\cup J^{(1c)} \Bigg).
\end{align}
\end{remark}

Moving on to the second stage of the $(j_1-2)$th--scale iteration, from \eqref{seconditG}, using the form \eqref{seconditG2}, we have similarly to \eqref{secondit2} and \eqref{trackprob}
\begin{align}\label{basictrack}
  P&\Big( T(x,y) \leq h((y-x)_1) - t\sigma_r \text{ for some } (x,y)\in X_r \Big) \notag\\
  &\leq P\Bigg( \Upsilon_{\hT}\left( \Gamma_{xy}^{j_1-2,2} \right) - h((y-x)_1) \leq -t\sigma_r 
    + \mA_{j_1-2}^2(\Gamma_{xy}) \notag\\
  &\hskip 1cm + \delta\mu\Big( \Upsilon_{Euc}\left( \Gamma_{xy}^{j_1-2,2} \right) 
    -\Upsilon_{Euc}\left( \Gamma_{xy}^{j_1-1,0} \right) \Big)
    \text{ for some } (x,y)\in X_r;\ \omega\notin J^{(0)}(c_{29})\cup J^{(1c)} \Bigg) \notag\\
  &\qquad + P\Bigg( \Upsilon_{\hT}\left( \Gamma_{xy}^{j_1-2,1} \right) - \Upsilon_{\hT}\left( \Gamma_{xy}^{j_1-2,2} \right)
    \leq -\sum_{N=0}^{n_2} S_{j_1-1}(\Gamma^{2,N,full}) \notag\\
  &\qquad\hskip 1cm -\frac{1}{3} \lambda^{j_1-1}\Big( t\sigma_r 
    + \delta\mu\Big[  \Upsilon_{Euc}\left( \Gamma_{xy}^{j_1-2,2} \right) - (\hat y - \hat x)_1 \Big] \Big) \notag\\
  &\qquad\hskip 1cm + \delta\mu\Big( \Upsilon_{Euc}\left( \Gamma_{xy}^{j_1-2,1} \right) 
    -\Upsilon_{Euc}\left( \Gamma_{xy}^{j_1-2,2} \right) \Big)
    \text{ for some } (x,y)\in X_r;\ \omega\notin J^{(0)}(c_{29})\cup J^{(1c)} \Bigg) \notag\\
  &\hskip 1cm + c_{32}e^{-c_{33}t} + c_{42} \exp\left( -c_{43}\left( \frac{\lambda}{7\delta^{\chi_1}} \right)^{j_1} t \right)
    + C_{70} \exp\left( -C_{71}\left( \frac{\lambda}{7\delta^{\chi_1}} \right)^{j_1-1} t \right),
\end{align}
the last probability being the tracking--failure probability.  We now divide into Cases 1--5 as in the $(j_1-1)$th--scale iteration.  In each case (excluding 3c) there is an intermediate path $\Gamma^{2,N,int}$ in each $(j_1-2)$th--scale interval $I_N^2$, which is slower than the direct path (up to a multiple of log $r$), satisfying
\[
  \Upsilon_{\hT}(\Gamma^{2,N,int}) \geq \hT(\hb^{\gamma(2,N)},\hb^{\gamma(2,N+1)}) 
    - \big(\gamma(2,N+1)-\gamma(2,N)-1\big)c_{29}\log r
\]
assuming $\omega\notin J^{(0)}(c_{29})$, and hence
\[
    \sum_{N=0}^{n_2} \Upsilon_{\hT}(\Gamma^{2,N,int}) \geq \Upsilon_{\hT}\left( \Gamma_{xy}^{j_1-2,2} \right)
      - (n-n_2)c_{29}\log r,
\]
and which satisfies deterministic tracking in a form like \eqref{posfrac}, sometimes with extra terms in the form of error allocations, as in \eqref{dtrac}; this deterministic tracking says that the intermediate path is sufficiently shorter relative to the full path, to within the error given by the allocations.  This enables us to bound the last probability in \eqref{basictrack} by a tracking--failure probability of the form
\begin{align}\label{basictf}
  &P\Bigg( \Upsilon_{\hT}\left( \Gamma^{2,N,full} \right) - \Upsilon_{\hT}\left( \Gamma^{2,N,int} \right) \leq  
    \Upsilon_h\left( \Gamma^{2,N,full} \right) - \Upsilon_h\left( \Gamma^{2,N,int} \right) - \text{ (allocations)}\notag\\
  &\hskip 1.2cm 
    \text{ for some $N$ for some } (x,y)\in X_r;\ \omega\notin J^{(0)}(c_{29})\cup J^{(1c)} \Bigg),
\end{align}
differently for each of the 5 cases; see for example \eqref{segmentsun} for the totally--unbowed case.  This is then bounded by summing probabilities of form
\[
  P\Big( \Big| \hT(v,w) - h(|w-v|) \Big| \geq A)
\]
over all possible links $(v,w)$ of paths $\Gamma^{2,N,full}, \Gamma^{2,N,int}$ for all $N$, using Lemmas \ref{sumbits} and \ref{sumbits2}.  The end result of the $(j_1-2)$th--scale iteration is that
\begin{align}\label{basictrack2}
  P&\Big( T(x,y) \leq h((y-x)_1) - t\sigma_r \text{ for some } (x,y)\in X_r \Big) \notag\\
  &\leq P\Bigg( \Upsilon_{\hT}\left( \Gamma_{xy}^{j_1-2,2} \right) - h((y-x)_1) \leq -t\sigma_r 
    + \mA_{j_1-2}^2(\Gamma_{xy}) \notag\\
  &\hskip 1cm + \delta\mu\Big( \Upsilon_{Euc}\left( \Gamma_{xy}^{j_1-2,2} \right) 
    -\Upsilon_{Euc}\left( \Gamma_{xy}^{j_1-1,0} \right) \Big)
    \text{ for some } (x,y)\in X_r;\ \omega\notin J^{(0)}(c_{29})\cup J^{(1c)} \Bigg) \notag\\
  &\hskip 1cm + c_{32}e^{-c_{33}t} + \sum_{j=j_1-2}^{j_1-1} 
    c_{42} \exp\left( -c_{43}\left( \frac{\lambda}{7\delta^{\chi_1}} \right)^{j+1} t \right).
\end{align}

After all iterations are completed through the $j_2$th scale, this becomes
\begin{align}\label{finalit}
  P&\Big( T(x,y) \leq h((y-x)_1) - t\sigma_r \text{ for some } (x,y)\in X_r \Big) \notag\\
  &\leq P\Bigg( \Upsilon_{\hT}\left( \Gamma_{xy}^{j_2,2} \right) - h((y-x)_1) \leq -t\sigma_r 
    + \mA_{j_2}^2(\Gamma_{xy}) + \delta\mu\Big( \Upsilon_{Euc}\left( \Gamma_{xy}^{j_2,2} \right) 
    -\Upsilon_{Euc}\left( \Gamma_{xy}^{j_1-1,0} \right) \Big) \notag\\
  &\hskip 1.6cm \text{ for some } (x,y)\in X_r;\ \omega\notin J^{(0)}(c_{29})\cup J^{(1c)} \Bigg) \notag\\
  &\hskip .6cm + c_{32}e^{-c_{33}t} + \sum_{j=j_2}^{j_1-1} 
    c_{42} \exp\left( -c_{43}\left( \frac{\lambda}{7\delta^{\chi_1}} \right)^{j+1} t \right).
\end{align}
We claim that 
\begin{equation}\label{cancel}
  \mA_{j_2}^2(\Gamma_{xy}) + \delta\mu\Big( \Upsilon_{Euc}\left( \Gamma_{xy}^{j_2,2} \right) 
    -\Upsilon_{Euc}\left( \Gamma_{xy}^{j_1-1,0} \right) \Big) \leq \frac 12 t\sigma_r 
    + \delta\mu\Big( \Upsilon_{Euc}\left( \Gamma_{xy}^{j_2,2} \right) - (\hat y - \hat x)_1 \Big).
\end{equation}
This says that the (typically negative) second term on the left, representing part of the cumulative effects of tracking, is enough to adequately cancel the (positive) cumulative allocations $\mA_{j_2}^2(\Gamma_{xy})$; see the comments after \eqref{trackprob} and \eqref{secondit3}. To prove \eqref{cancel} we make the subclaim
\begin{equation}\label{cancel2}
  \tred{D}:= \max_{j_2\leq j\leq j_1-1} \Upsilon_{Euc}\left( \Gamma_{xy}^{j,0} \right) - (\hat y - \hat x)_1
    \leq 2\Big(  \Upsilon_{Euc}\left( \Gamma_{xy}^{j_1-1,0} \right) - (\hat y - \hat x)_1 \Big) + t\sigma_r.
\end{equation}
Assuming \eqref{cancel2} and assuming $\lambda$ satisfies $7\lambda/(1-\lambda)<1/4$ we have from the definition \eqref{Aj2acc}
\begin{align}
  \mA_{j_2}^2(\Gamma_{xy}) &\leq \frac{7\lambda}{1-\lambda} \Bigg( t\sigma_r 
    + \delta\mu\bigg[  \max_{j_2\leq j\leq j_1-1} \Upsilon_{Euc}\left( \Gamma_{xy}^{j,0} \right) - (\hat y - \hat x)_1 \bigg] \Bigg)
    \notag\\
  &\leq \frac{1 + \delta\mu}{4}t\sigma_r 
    + \delta\mu\Big(  \Upsilon_{Euc}\left( \Gamma_{xy}^{j_1-1,0} \right) - (\hat y - \hat x)_1 \Big)
\end{align}
and \eqref{cancel} follows. To prove \eqref{cancel2} we use \eqref{upgap} (which generalizes to all $j$ in place of $j_1-1$) and the fact that, since removing points reduces length, we have $\Upsilon_{Euc}(\Gamma_{xy}^{j,2}) \leq \Upsilon_{Euc}(\Gamma_{xy}^{j,1})$ for all $j$. This yields that for all $j\geq 1$,
\begin{align}
  \Upsilon_{Euc}\left( \Gamma_{xy}^{j,0} \right) &\leq \Upsilon_{Euc}\left( \Gamma_{xy}^{j_1-1,0} \right)
    + \sum_{k=j+1}^{j_1-1} \Big| \Upsilon_{Euc}\left( \Gamma_{xy}^{k,1} \right)
    - \Upsilon_{Euc}\left( \Gamma_{xy}^{k,0} \right) \Big| \notag\\
  &\leq \Upsilon_{Euc}\left( \Gamma_{xy}^{j_1-1,0} \right)
    + \sum_{k=j+1}^{j_1-1} \frac{\lambda^{k+1}}{2\mu}\Big[ \frac{t\sigma_r}{3} + \delta\mu D \Big]
\end{align}
and therefore
\begin{align}
  D \leq \Upsilon_{Euc}\left( \Gamma_{xy}^{j_1-1,0} \right) - (\hat y - \hat x)_1 
     + \frac{\lambda^3}{6\mu(1-\lambda)} t\sigma_r + \frac{\delta\lambda^3}{2(1-\lambda)} D,
\end{align}
from which \eqref{cancel2} follows, and hence also \eqref{cancel}. The right side of \eqref{finalit} is therefore bounded above by
\begin{align}\label{finalit2}
  P&\Bigg( \Upsilon_{\hT}\left( \Gamma_{xy}^{j_2,2} \right) - h((y-x)_1) \leq -\frac 12 t\sigma_r 
    + \delta\mu\Big( \Upsilon_{Euc}\left( \Gamma_{xy}^{j_2,2} \right) - (\hat y - \hat x)_1 \Big) \notag\\
  &\hskip 1cm    \text{ for some } (x,y)\in X_r;\ \omega\notin J^{(0)}(c_{29})\cup J^{(1c)} \Bigg) 
    + c_{44} \exp\left( -c_{45}\left( \frac{\lambda}{7\delta^{\chi_1}} \right)^{j_2+1} t \right).
\end{align}

\subsection{Step 8. Final marked CG paths.}
In $ \Gamma_{xy}^{j_2,2}$, only terminal hyperplanes contain marked points (excluding $\hat x,\hat y$), one at each end of $[\hat x_1,\hat y_1]$ for each scale $j_2\leq j\leq j_1$.  The marked point in the $j$--terminal hyperplane is in the grid $\LL_j$ for all $j$.   
As noted in Step 1, the gap between the $j_2$--terminal hyperplanes is at least $4\delta^{j_2}r$ and at most $5\delta^{j_2-1}r$.  Now $\Gamma_{xy}^{j_2,2}$ is our final CG path, so we rename it 
\[
  \tred{\Gamma_{xy}^{CG}}:\, \hat x = u^0\to\cdots\to u^{R+1} = \hat y,
\]
where $\tred{R} = 2(j_1-j_2+1)$.  
We also define a projected path with collinear marked points
\[
  \tred{\Gamma_{xy}^{CG,pr}}:\, v^0\to\cdots\to v^{R+1},
\]
where $\tred{v^i} = (u^i)_1e_1$ is the projection onto the $e_1$ axis. 
For $j_1-1\leq j\leq j_2$, the links $(u^{j_1-j},u^{j_1-j+1})$ and $(u^{R-j_1+j},u^{R-j_1+j+1})$ each have one end in a $j$--terminal hyperplane and the other in a $(j+1)$--terminal hyperplane; these will be called the \emph{jth--scale links} of $ \Gamma_{xy}^{CG}$.
The links $(u^0,u^1)$ and $(u^R,u^{R+1})$ are called \emph{final links}, and the link $(u^{R/2},u^{R/2+1})$ between $j_2$--terminal hyperplanes is called a \emph{macroscopic link}. See Figure \ref{Scales} in Section \ref{setup}, where $(u^3,u^4)$ is the macroscopic link, which (by definition of $j_2$) always covers at least 1/5 the length of $[\hat x_1,\hat y_1]$.

We have $\Upsilon_{Euc}\left( \Gamma_{xy}^{CG,pr} \right) = (\hat y - \hat x)_1$ so by Lemma \ref{monotoneE}, since $R\leq 2j_1$,
\[
  \Upsilon_h\left( \Gamma_{xy}^{CG} \right) - \Upsilon_h\left( \Gamma_{xy}^{CG,pr} \right)
    \geq \left( \frac 12 + \delta\right)\mu \Big( \Upsilon_{Euc}\left( \Gamma_{xy}^{CG} \right) 
    - (\hat y - \hat x)_1 \Big)-2j_1C_{56}.
\]
Therefore, using \eqref{nearsub}, the probability in \eqref{finalit2} is bounded above by
\begin{align}\label{finalit3}
  P&\Bigg( \Upsilon_{\hT}\left( \Gamma_{xy}^{CG} \right) - \Upsilon_h\left( \Gamma_{xy}^{CG} \right) 
    \leq -\frac 13 t\sigma_r 
    + \delta\mu\Big( \Upsilon_{Euc}\left( \Gamma_{xy}^{CG} \right) - (\hat y - \hat x)_1 \Big) 
    - \Big[ \Upsilon_h\left( \Gamma_{xy}^{CG} \right) - \Upsilon_h\left( \Gamma_{xy}^{CG,pr} \right) \Big] \notag\\
  &\hskip 1cm \text{ for some } (x,y)\in X_r;\ \omega\notin J^{(0)}(c_{29})\cup J^{(1c)} \Bigg) \notag\\
  &\leq P\Bigg( \Upsilon_{\hT}\left( \Gamma_{xy}^{CG} \right) - \Upsilon_h\left( \Gamma_{xy}^{CG} \right) 
    \leq -\frac 13 t\sigma_r 
    - \frac \mu 2 \Big( \Upsilon_{Euc}\left( \Gamma_{xy}^{CG} \right) - (\hat y - \hat x)_1 \Big) \notag\\
  &\hskip 1.5cm \text{ for some } (x,y)\in X_r;\ \omega\notin J^{(0)}(c_{29})\cup J^{(1c)} \Bigg).
\end{align}
To bound this we use allocations 
\[
  \tred{A_j^{CG}(v,w)} = \frac{1}{4}\lambda^j \left[ t^*(v,w) \sigma_r
    + \frac{\delta\mu}{18} \frac{|(w-v)^*|^2}{|(w-v)_1|} \right]
\]
for $j$th--scale links, $A_{j_1}^{CG}(v,w)$ for final links, and $A_{j_2}^{CG}(v,w)$ for macroscopic links. From \eqref{sizecurv1} and \eqref{sizecurv2} we have
\[
  \frac{\delta\mu}{18}\max_{k\leq R+1} \frac{|(u^k)^*|^2}{r\sigma_r} \leq \frac{\delta t}{3} + 
     \frac{\delta\mu}{12\sigma_r} \sum_{i=1}^{R+1} \frac{|(u^{i-1}-u^i)^*|^2}{|(u^{i-1}-u^i)_1|}
\]
Therefore using also \eqref{adddist} the total of these for all links in a path $\Gamma_{xy}^{CG}$ satisfies
\begin{align}
  A_{j_2}^{CG}&(u^{R/2},u^{R/2+1}) + A_{j_1}^{CG}(u^0,u^1) + A_{j_1}^{CG}(u^R,u^{R+1}) \notag\\
  &\qquad + \sum_{j=j_2}^{j_1-1} 
    \Big[ A_{j+1}^{CG}(u^{j_1-j},u^{j_1-j+1}) + A_{j+1}^{CG}(u^{R-j_1+j},u^{R-j_1+j+1}) \Big] \notag\\
  &\leq 4 \sum_{j=j_2-1}^{j_1-1} \frac{\lambda^{j+1}}{4} \left[ 
    \left( \frac t3 + 2\frac{\delta\mu}{18}\max_{k\leq R+1} \frac{|(u^k)^*|^2}{r\sigma_r}\right)\sigma_r 
    +  \frac{\delta\mu}{18} \sum_{i=1}^{R+1} \frac{|(u^{i-1}-u^i)^*|^2}{|(u^{i-1}-u^i)_1|} \right]\notag\\
  &\leq \sum_{j=j_2}^{j_1} \frac 13 \lambda^j \left( 1 + 2\delta \right)t\sigma_r
    + \sum_{j=j_2}^{j_1} \lambda^j \left( \frac{\delta\mu}{3} + \frac{\delta\mu}{6} \right) 
    \sum_{i=1}^{R+1} \frac{|(u^{i-1}-u^i)^*|^2}{|(u^{i-1}-u^i)_1|} \notag\\
  &\leq \frac{\lambda}{1-\lambda} t\sigma_r + \frac{3\delta\mu\lambda}{2(1-\lambda)} 
    \Big( \Upsilon_{Euc}\left( \Gamma_{xy}^{CG} \right) - (\hat y - \hat x)_1 \Big) \notag\\
  &\leq \frac 13 t\sigma_r 
    + \frac \mu 2 \Big( \Upsilon_{Euc}\left( \Gamma_{xy}^{CG} \right) - (\hat y - \hat x)_1 \Big).
\end{align}
Therefore the last probability in \eqref{finalit3} is bounded above by 
\begin{align}\label{final}
  &\sum_{j=1}^{j_1} P\bigg( \Big| \hT(u^{i-1},u^i) - h(|u^{i-1}-u^i|) \Big| \geq A_{j+1}^{CG}(u^{i-1},u^i)
    \text{ for some $j$th--scale link} \notag\\
  &\hskip 2cm (u^{i-1},u^i) \text{ of $\Gamma_{xy}^{CG}$ and some } (x,y)\in X_r;\ \omega\notin J^{(0)}(c_{29})\cup J^{(1c)} \bigg)
    \notag\\
  &\qquad + P\bigg( \Big| \hT(u^{i-1},u^i) - h(|u^{i-1}-u^i|) \Big| \geq A_{j_2}^{CG}(u^{i-1},u^i)
    \text{ for some macroscopic link} \notag\\
  &\hskip 2cm (u^{i-1},u^i) \text{ of $\Gamma_{xy}^{CG}$ and some } (x,y)\in X_r;\ \omega\notin J^{(0)}(c_{29})\cup J^{(1c)} \bigg)
    \notag\\
  &\qquad + P\bigg( \Big| \hT(u^{i-1},u^i) - h(|u^{i-1}-u^i|) \Big| \geq A_{j_1}^{CG}(u^{i-1},u^i)
    \text{ for some final link} \notag\\
  &\hskip 2cm (u^{i-1},u^i) \text{ of $\Gamma_{xy}^{CG}$ and some } (x,y)\in X_r;\ \omega\notin J^{(0)}(c_{29})\cup J^{(1c)} \bigg).
\end{align}
Considering the probability for final links $(v,w)$, the number of possible such links in $G_r^+$ is at most $c_{46}r^{2d}$ and for each such link we have from Lemma \ref{connect} and \eqref{j1}
\begin{align}
  P\Big( \Big| \hT(v,w) &- h(|w-v|) \Big| \geq A_{j_1}^{CG}(v,w) \Big) \leq 
    C_{44}\exp\left( - C_{45}\frac{A_{j_1}^{CG}(v,w)}{\sigma(2\delta^{j_1}r)} \right) \notag\\
  &\leq c_{47}\exp\left( -\frac{c_{48}}{j_1}\left( \frac{\lambda}{\delta^{\chi_1}} \right)^{j_1} t \right)
    \leq c_{47}\exp\left( -c_{49}t\log r \right).
\end{align}
The probabilities for $j$th--scale and macroscopic links can be bounded using minor modifications of Lemma \ref{sumbits}, showing that \eqref{final} (and hence also the probability on the right in \eqref{finalit}) is bounded above by
\begin{align}
  \sum_{j=1}^{j_1} &c_{50}\exp\left( -c_{51}\left( \frac{\lambda}{7\delta^{\chi_1}} \right)^j t \right)
    + c_{50}\exp\left( -c_{51} \frac{\lambda}{7\delta^{\chi_1}} t \right) 
    + c_{46}c_{47} r^{2d} \exp\left( -c_{49}t\log r \right) \notag\\
  &\leq c_{52}e^{-c_{53}t}.
\end{align}
With \eqref{finalit} this completes the proof of \eqref{Qrunif2}.  As noted after Theorem \eqref{nofast}, the downward--deviations part of \eqref{Qrunif1} is a consequence, so the proof of the downward--deviations part of Theorem \eqref{nofast} is complete.

\section{Proof of Theorem \ref{tversethm} for fixed $(x,y)$.}\label{yxfixed}
We move on to step (2) of the strategy in Remark \ref{strategy}.
As with the proof of Theorem \ref{nofast} (downward deviations), there are simpler cases which can be proved from Lemma \ref{connect} and do not require Theorem \ref{nofast}.  These we can handle uniformly over $(x,y)$; after dealing with these, in Lemma \ref{xyfixed} we will handle the main case for the moment only for fixed $(x,y)$.  

Let 
\begin{equation}\label{r0}
  \tred{r_0} = \frac{c_0r}{(\log r)^{1/(1+\chi_1)}},
\end{equation}
with $c_0$ to be specified;
we will consider separately short geodesics (meaning $|\hat y - \hat x| \leq r_0$) and longer ones.
When $\Gamma_{xy} \not\subset G_{r,s} \text{ for some } x,y\in \VV\cap G_r(K)$ with $|(y-x)^*|\leq (y-x)_1$, taking $\tred{u}$ to be the first vertex in $\Gamma_{xy}$ outside $G_{r,s}$, we see that one of the following cases must hold.\\

\noindent {\bf Transverse wandering cases (see Figure \ref{Theorem1-5}):}
\begin{itemize}
\item[(i)] ({\it non-transverse wandering}) there exist $\tred{\hat x,\hat y,\hat u}\in q\ZZ^d$ with $\hat x,\hat y \in G_r(K), \hat u\notin G_{r,s}$ with $d(\hat u,G_{r,s})\leq 5$, and $\tred{x,y,u}\in\VV$ with 
\[
  \max(|x-\hat x|,|y-\hat y|,|u-\hat u|) \leq 5\sqrt{d-1}, \quad u\in\Gamma_{xy}, \quad \hat u_1 \notin [\hat x_1,\hat y_1],
\]
\item[(ii)] ({\it wandering by short geodesics}) there exist $\hat x,\hat y,\hat u\in q\ZZ^d$ with $\hat x,\hat y \in G_r(K), \hat u\notin G_{r,s}$ with $d(\hat u,G_{r,s})\leq 5$, and $x,y,u\in\VV$ with
\[
  \max(|x-\hat x|,|y-\hat y|,|u-\hat u|) \leq 5\sqrt{d-1}, \quad u\in\Gamma_{xy}, \quad \hat u_1 \in [\hat x_1,\hat y_1],
    \quad |\hat y - \hat x| \leq r_0,
\]
\item[(iii)] ({\it large wandering}) $s> c_1(\log r)^{1/2}$ and there exist $\hat x,\hat y,\hat u\in q\ZZ^d$ with $\hat x,\hat y \in G_r(K), \hat u\notin G_{r,s}$ with $d(\hat u,G_{r,s})\leq 5$, and $x,y,u\in\VV$ with
\begin{equation}\label{twiii}
  \max(|x-\hat x|,|y-\hat y|,|u-\hat u|) \leq 5\sqrt{d-1}, \quad 
    |y-x| > r_0,\quad u\in\Gamma_{xy}, \quad u_1\in [x_1,y_1],
\end{equation}
\item[(iv)] ({\it moderate wandering}) $s\leq c_1(\log r)^{1/2}$ and there exist $x,y,u\in\VV$ with 
\begin{equation}\label{twiv}
  x,y \in G_r(K),\quad 
    u \notin G_{r,s},\quad d(u,G_{r,s})\leq 2\quad |y-x| > r_0,\quad u\in\Gamma_{xy}, \quad u_1\in [x_1,y_1],
\end{equation}
\end{itemize}
where $c_1$ is to be specified. In all cases (i)--(iv) we assume $|(y-x)^*|\leq (y-x)_1$. \\

\begin{figure}
\includegraphics[width=17cm]{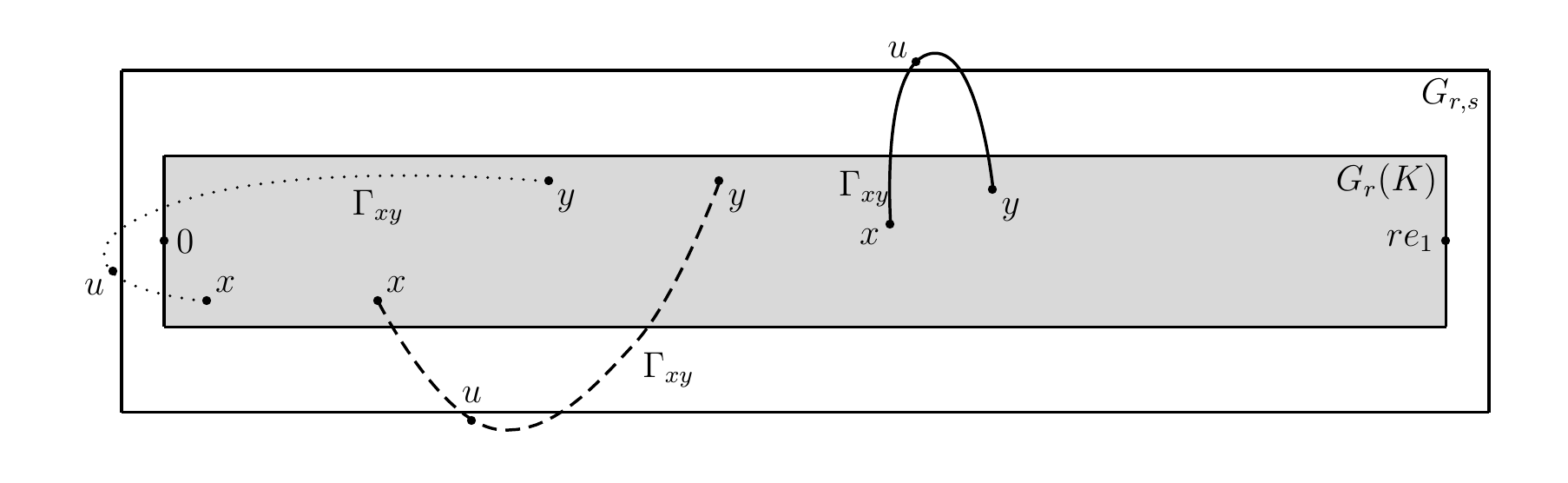}
\caption{ The inner box $G_r(K)$ has height $2K\Delta_r$, and the outer box $G_{r,s}$ has height $2s\Delta_r$.  The dotted line shows $\Gamma_{xy}$ in transverse wandering case (i), the solid line shows case (ii), and the dashed line shows cases (iii) and (iv), for large and small $s$, respectively. }
\label{Theorem1-5}
\end{figure}

Cases (i)---(iii) are the ones which do not require Theorem \ref{nofast} and can be proved from Lemma \ref{connect}. This is because the exponent obtained from that lemma is at least of order $\log r$, meaning it dominates the entropy from the number of possible choices of $\hat x,\hat y,\hat u$.

It follows from routine geometry that under (i), for $C_{34}$ from \eqref{Irs} we have
\[
  |\hat u -\hat x| +|\hat y - \hat u| - |\hat y - \hat x| \geq d(G_r(K),G_{r,s}^c) \geq
    \begin{cases} s^2\sigma_r\log r &\text{if } s \leq (C_{34}\log r)^{1/2},\\
    s^2\sigma_r &\text{if } (C_{34}\log r)^{1/2} < s \leq r/\Delta_r,\\
    s\Delta_r &\text{if } s > r/\Delta_r \end{cases}
\]
so also, using Lemma \ref{monotoneE},
\begin{equation}\label{hextra}
  h(|\hat u -\hat x|) + h(|\hat y - \hat u|) - h(|\hat y - \hat x|) \geq \begin{cases} \frac \mu 2 s^2\sigma_r\log r 
    &\text{if } s \leq (C_{34}\log r)^{1/2}, \\
    \frac \mu 2 s^2\sigma_r &\text{if } (C_{34}\log r)^{1/2} < s \leq r/\Delta_r,\\
    \frac \mu 2 s\Delta_r &\text{if } s > r/\Delta_r, \end{cases}
\end{equation}
while 
\[
  \Big| \hT(\hat x,\hat u) + \hT(\hat u,\hat y) - \hT(\hat x,\hat y) \Big| \leq M(\hat x)+M(\hat y)+M(\hat u).
\]
Therefore one of the following 6 quantities
\[
  \Big| \hT(\hat x,\hat u) - h(|\hat u -\hat x|) \Big|,\quad \Big| \hT(\hat u,\hat y) - h(|\hat y -\hat u|) \Big|,\quad
    \Big| \hT(\hat x,\hat y) - h(|\hat y -\hat x|) \Big|,\quad M(\hat x),\ \ M(\hat y),\ \ M(\hat u)
\]
exceeds 1/6 of the right side of \eqref{hextra}.  There are at most $c_2s^d r^{2d}$ possible values of $(\hat x,\hat y,\hat u)$ under (i), so from Lemmas \ref{neighbortimes} and  \ref{connect}, provided $C_{34}$ is sufficiently large we have
\begin{align}\label{casei}
  P\Big( (i) \text{ holds} \Big) &\leq \begin{cases}  c_2s^d r^{2d}\cdot 
    C_{44}\exp\left( - c_3 \frac{s^2\sigma_r\log r}{\sigma(2r)} \right) \leq c_4e^{-c_5s^2} 
    &\text{if } s \leq (C_{34}\log r)^{1/2}\\
    c_2s^d r^{2d}\cdot C_{44}\exp\left( - c_3 \frac{s^2\sigma_r}{\sigma(2r)} \right) \leq c_4e^{-c_5s^2} 
    &\text{if } (C_{34}\log r)^{1/2} < s \leq r/\Delta_r, \\
    c_2s^d r^{2d}\cdot C_{44}\exp\left( - c_3 \frac{s\Delta_r}{\sigma(s\Delta_r)} \right) 
    \leq c_4e^{-c_5s\Delta_r/\sigma(s\Delta_r)} &\text{if } s> r/\Delta_r. \end{cases}
\end{align}

Under (ii) we have for some $c_6$
\[
  |\hat u -\hat x| +|\hat y - \hat u| - |\hat y - \hat x| \geq 
    \begin{cases} c_6 s^2\Delta_r^2/r_0 &\text{if } s \leq r_0/\Delta_r,\\
    c_6 s\Delta_r &\text{if } s > r_0/\Delta_r. \end{cases}
\]
Further, for $s\leq r/\Delta_r$ we have $1/s\Delta_r\sigma(s\Delta_r) \leq c_7/\Delta_r^2$ or equivalently $s\Delta_r/\sigma(s\Delta_r) \geq c_7s^2$.
Hence similarly to (i), provided we take $c_0$ large in \eqref{r0},
\begin{align}\label{caseii}
  P\Big( (ii) \text{ holds} \Big) &\leq \begin{cases}  c_1s^d r^{2d}\cdot 
    C_{44}\exp\left( - c_8 \frac{ s^2r\sigma_r }{r_0\sigma(r_0)} \right) \leq c_9e^{-c_9s^2\log r} 
    &\text{if } s \leq r_0/\Delta_r \\
    c_1s^d r^{2d}\cdot C_{44}\exp\left( - c_8\frac{s\Delta_r}{\sigma(s\Delta_r)} \right) \leq c_9e^{-c_9s^2} 
    &\text{if } r_0/\Delta_r < s \leq r/\Delta_r,  \\
    c_1s^d r^{2d}\cdot C_{44}\exp\left( - c_8\frac{s\Delta_r}{\sigma(s\Delta_r)} \right) \leq c_9e^{-c_9s\Delta_r/\sigma(s\Delta_r)} 
    &\text{if } s > r/\Delta_r.  \end{cases}
\end{align}

Under (iii) we have
\[
  |\hat u -\hat x| +|\hat y - \hat u| - |\hat y - \hat x| \geq 
    \begin{cases} c_{11} s^2\sigma_r &\text{if } c_0(\log r)^{1/2} \leq s \leq r/\Delta_r,\\
    c_{11} s\Delta_r &\text{if } s > r/\Delta_r. \end{cases}
\]
Hence as in (i) and (ii) we get
\begin{align}\label{caseiii}
  P\Big( (iii) \text{ holds} \Big) &\leq \begin{cases}  c_1s^d r^{2d}\cdot 
    C_{44}\exp\left( - c_{12} \frac{ s^2\sigma_r }{\sigma(2r)} \right) \leq c_{13}e^{-c_{14}s^2} 
    &\text{if } c_0(\log r)^{1/2} \leq s \leq r/\Delta_r \\
    c_1s^d r^{2d}\cdot C_{44}\exp\left( - c_{12}\frac{s\Delta_r}{\sigma(s\Delta_r)} \right) 
     \leq c_{13}e^{-c_{14}s\Delta_r/\sigma(s\Delta_r)} 
    &\text{if } s > r/\Delta_r.  \end{cases}
\end{align}

To deal with (iv) and complete the proof for fixed $(x,y)$, we have the following.  We note the difference from Lemma \ref{bigwander}, which covered many $x,y$ simultaneously but considered only larger transverse wandering, of order $\Delta_r\log r$ or more.

\begin{lemma}\label{xyfixed}
Suppose $\GG=(\VV,\EE)$ and $\{\eta_e,e\in\EE\}$ satisfy A1, A2, and A3.  There exist $C_i$ such that for all $K\geq C_{74}$ and all $x,y\in G_r(K)$ with $|(y-x)^*|\leq (y-x)_1$,
\begin{align}
  P\Big( \Gamma_{xy} \not\subset G_{r,s} \Big) 
    \leq C_{75}e^{-C_{76}s^2} \quad\text{for all } C_{77}K \leq s \leq C_{78}(\log r)^{1/2}.
\end{align}
\end{lemma}

\begin{proof}
Fix $\tred{x,y}\in G_r(K)$ and let $\tred{r_1} = (y-x)_1$. We consider first the case of $r_1<r/2$. Let $\tred{m}\geq 0$ satisfy $mr_1 \leq x_1<y_1\leq (m+2)r_1$, and for $C_{22}$ from \eqref{powerlike} let 
\[
  \tred{s_1} = 2C_{22}s\left( \frac{r}{2r_1} \right)^{(1+\chi_1)/2},\quad 
    \tred{G_{r,2s}^{(m)}} = [mr_1,(m+2)r_1] \times 2s\Delta_r\mkB_{d-1},
\]
the latter being a slice of $G_{r,2s}$. Define $\tred{s_2}$ by $2s\Delta_r=s_2\Delta_{2r_1}$.
Then $s_2\geq s$, and from \eqref{powerlike} we have $s_2\geq s_1$, and $G_{r,2s}^{(m)}$ is a translate of $G_{2r_1}(s_2)$.

In view of \eqref{casei} and \eqref{caseii}, we need only consider $(x,y,u)$ as in \eqref{twiv}; in particular $r_1\geq r_0/2$.
Here we have for some $c_0$
\[
  |u-x| +|y - u| - |y - x| \geq  c_0 s_2^2\sigma(2r_1) \quad\text{and hence}\quad 
    h(|u-x|) + h(|y-u|) - h(|y-x|) \geq \frac{c_0\mu}{2}s_2^2\sigma(2r_1),
\]
so as in transverse wandering cases (i)--(iii), since $u\in\Gamma_{xy}$ one of the quantities
\[
  h(|u-x|) - T(x,u),\quad h(|y-u|) - T(u,y),\quad T(x,y) - h(|y-x|)
\]
must exceed $c_0\mu s_2^2\sigma(2r_1)/6$. It follows from \eqref{expbound2} and the downward--deviations part of Theorem \ref{nofast} (with $\ep,r,K,t$ there taken as $1/2, 2r_1,Ks_2/2s,c_0\mu s_2^2/6$ here) that for $\tred{x_0} = x - mr_1e_1,\tred{y_0} =y-mr_1e_1$ we have
\begin{align}
  P&\left( \Gamma_{xy} \text{ contains a vertex $u$ as in \eqref{twiv}} \right) \notag\\
  &\leq P\bigg( \max\Big( h(|u-x|) - T(x,u),h(|y-u|) - T(u,y) \Big) \geq \frac{c_0\mu}{6}s_2^2\sigma(2r_1) 
    \text{ for some $u\in G_{r,2s}^{(m)}$} \bigg) \notag\\
  &\quad + P\bigg( T(x,y) - h(|y-x|) \geq \frac{c_0\mu}{6}s_2^2\sigma(2r_1) \bigg) \notag\\
  &\leq P\bigg( \max\Big( h(|\tilde u-x_0|) - T(x_0,\tilde u),h(|y_0-\tilde u|) - T(\tilde u,y_0) \Big) 
    \geq \frac{c_0\mu}{6}s_2^2\sigma(2r_1) 
    \text{ for some $\tilde u\in G_{2r_1}(s_2)$} \bigg) \notag\\
  &\quad + P\bigg( T(x,y) - h(|y-x|) \geq \frac{c_0\mu}{6}s_2^2\sigma(2r_1) \bigg) \notag\\
  &\leq C_{27}e^{-c_1 s_2^2} \notag\\
  &\leq C_{27}e^{-c_1 s^2}.
\end{align}
\end{proof}

This completes the proof of Theorem \ref{tversethm} for fixed $(x,y)$, and thus of (2) in Remark \ref{strategy}.

\section{Proof of Theorem \ref{ghbound}.}
We move to (3) of Remark \ref{strategy}.
We begin with an extension of a consequence of Theorem \ref{nofast}, removing the requirement that $y$ lie in the same tube $G_r(K)$ as $x$, to create a bound on the probability of a fast ``hyperplane--block to hyperplane'' passage time. Define
\[
  \tred{\Lambda_r} = [0,2]\times [-\Delta_r,\Delta_r]^{d-1}.
\]

\begin{lemma}\label{tubes}
Suppose $\GG=(\VV,\EE)$ and $\{\eta_e,e\in\EE\}$ satisfy A1, A2, and A3.  There exist constants $C_i$ such that for all $r\geq C_{79}$ and $t\geq C_{80}$,
\begin{equation}\label{tubes1}
  P\Big( T(x,y) \leq h(r) - t\sigma_r \text{ for some $x\in \Lambda_r\cap \VV$ and $y\in H_r^+$} \Big)
    \leq C_{81}e^{-C_{82}t}.
\end{equation}
\end{lemma}

\begin{proof}
It is enough to consider $y\in H_{[r,r+2]}$ and $t\leq h(r)/\sigma_r$. We first deal with large $|y^*|$.  With $c_0$ to be specified we have
\begin{align}\label{angles}
  P\Bigg( &T(x,y) \leq h(r) - t\sigma_r \text{ for some $x\in \Lambda_r\cap \VV$ and $y\in H_{[r,r+2]}$ 
    with $|y^*|> c_0(\log r)^{1/2} \Delta_r$} \Bigg) \notag\\
  &\leq \sum_{k=1}^\infty P\Bigg( T(x,y) \leq h(r) - t\sigma_r \text{ for some $x\in \Lambda_r\cap \VV$ and $y\in H_{[r,r+2]}$}
    \notag\\ 
  &\hskip 2cm \text{with $2^{k-1}c_0(\log r)^{1/2}\Delta_r < |y^*| \leq 2^kc_0(\log r)^{1/2} \Delta_r$} \Bigg) \notag\\
  &\leq \sum_{k=1}^\infty P\Bigg( \hT(\hat x,\hat y) \leq h(r) - t\sigma_r 
    \text{ for some $\hat x\in \Lambda_r\cap q\ZZ^d$ and $\hat y\in H_r\cap q\ZZ^d$}
    \notag\\ 
  &\hskip 2cm \text{with $2^{k-1}c_0(\log r)^{1/2}\Delta_r < |\hat y^*| \leq 2^kc_0(\log r)^{1/2} \Delta_r$} \Bigg).
\end{align}
For $c_9$ from \eqref{adddist1}, let $\tred{k_0}=\max\{k: 2^kc_0(\log r)^{1/2} \Delta_r \leq c_9r\}$.
For $k\leq k_0$, and for $\hat x,\hat y$ of the last event, we have from \eqref{adddist1} and Lemma \ref{monotoneE}
\begin{equation}\label{yxvsr}
  |\hat y-\hat x| \geq r + \frac{|\hat y^*|^2}{3r} \geq r + \frac{2^{2k}c_0^2\sigma_r\log r}{12}
    \quad\text{so}\quad h(|\hat y-\hat x|) \geq h(r) + \frac{\mu c_0^2}{24}2^{2k}\sigma_r\log r,
\end{equation}
and the number of allowed $\hat x,\hat y$ is at most $c_1 (2^k(\log r)^{1/2})^{d-1} \Delta_r^{2(d-1)}\leq 2^{kd}r^{2d}$. Therefore by Lemma \ref{connect}, provided $c_0$ is large, the sum up to $k_0$ on the right side of \eqref{angles} is bounded above by
\begin{align}\label{largey}
  \sum_{k=1}^{k_0} P&\Bigg( \hT(\hat x,\hat y) \leq h(|\hat y-\hat x|) - t\sigma_r - \frac{\mu c_0^2}{24}2^{2k}\sigma_r\log r
    \text{ for some $\hat x\in \Lambda_r\cap q\ZZ^d$ and $\hat y\in H_r\cap q\ZZ^d$}
    \notag\\ 
  &\hskip .5cm \text{with $2^{k-1}c_0(\log r)^{1/2}\Delta_r < |\hat y^*| \leq 2^kc_0(\log r)^{1/2} \Delta_r$} \Bigg) \notag\\
  &\leq \sum_{k=1}^{k_0} 2^{kd}r^{2d} C_{44} \exp\left( -C_{45}\frac{\sigma_r}{\sigma(2r)}
    \left( t+2^{2k}\frac{\mu c_0^2}{24}\log r \right) \right) \notag\\
  &\leq c_2 e^{-c_3t}.
\end{align}
For $k>k_0$, in place of \eqref{yxvsr} we have 
\[
   |\hat y-\hat x| \geq r+c_4 |\hat y^*| \geq r+c_4 2^{k-1}c_0(\log r)^{1/2}\Delta_r
     \quad\text{so}\quad h(|\hat y-\hat x|) \geq h(r) + \frac{c_4c_0\mu}{4} 2^k(\log r)^{1/2}\Delta_r,
\]
and we obtain similarly that the sum from $k_0+1$ to $\infty$ on the right side of \eqref{angles} is bounded above by
\begin{align}\label{verylgy}
  \sum_{k=k_0+1}^\infty 
    P&\Bigg( \hT(\hat x,\hat y) \leq h(|\hat y-\hat x|) - t\sigma_r - \frac{c_4c_0\mu}{4} 2^k(\log r)^{1/2}\Delta_r
    \text{ for some $\hat x\in \Lambda_r\cap q\ZZ^d$ and $\hat y\in H_r\cap q\ZZ^d$}
    \notag\\ 
  &\hskip .5cm \text{with $2^{k-1}c_0(\log r)^{1/2}\Delta_r < |\hat y^*| \leq 2^kc_0(\log r)^{1/2} \Delta_r$} \Bigg) \notag\\
  &\leq \sum_{k=k_0+1}^\infty 2^{kd}r^{2d} C_{44} \exp\left( -\frac{C_{45}}{\sigma(r+2^kc_0(\log r)^{1/2}\Delta_r)}
    \left( t\sigma_r+\frac{c_4c_0\mu}{4} 2^k(\log r)^{1/2}\Delta_r \right) \right) \notag\\
  &\leq c_5 \exp\left( -\frac{c_6(t\sigma_r+r)}{\sigma(2r)} \right) \notag\\
  &\leq c_2 e^{-c_3t}.
\end{align}
The last inequality uses $t\leq h(r)/\sigma_r$.

We now deal with the remaining case $|y^*|\leq c_0(\log r)^{1/2} \Delta_r$.
For $C_{26}$ from Theorem \ref{nofast}, we take $\tred{K}=\sqrt{t/C_{26}(d-1)}$ to be specified and subdivide $H_{[r,r+2]}$ into blocks
\[
  \tred{Y_{\bdm}} = [r,r+2] \times \prod_{i=1}^{d-1} [(m_i-1)K\Delta_r,m_iK\Delta_r], \quad \bdm\in\ZZ^{d-1}.
\]
Let $\tred{z_{\bdm}}$ denote the center point of $Y_{\bdm}$ and let
\[
  \tred{\mkM} = \Big\{\bdm\in\ZZ^{d-1}: 0\notin Y_{\bdm},\,
    \exists y\in Y_{\bdm} \text{ with } |y^*| \leq c_0(\log r)^{1/2} \Delta_r \Big\}.  
\]
Given $\bdm\in\mkM$, there exists a cylinder $\mC_{\bdm}$, with axis containing $\Pi_{0z_{\bdm}}$, radius $2\sqrt{d-1}K\Delta_r$, and length $2r$, which contains $\Lambda_r$ and $Y_{\bdm}$. See Figure \ref{Lemma6-1}. Further, $d_\infty(Y_{\bdm},[r,r+2]\times\{0\})$ is a positive integer multiple of $K\Delta_r$, and for $x\in\Lambda_r$ and $y\in Y_{\bdm}$ satisfying $d_\infty(Y_{\bdm},[r,r+2]\times\{0\}) = jK\Delta_r$ for some $j$, we have using \eqref{adddist1} and Lemma \ref{monotoneE}(i) that
\[
  |y-x| \geq r + \frac{(jK\Delta_r)^2}{3r} = r + \frac{(jK)^2}{3}\sigma_r \quad\text{and hence}\quad h(|y-x|) \geq h(r) + 
    \frac{\mu(jK)^2}{6}\sigma_r.
\]
\begin{figure}
\includegraphics[width=15cm]{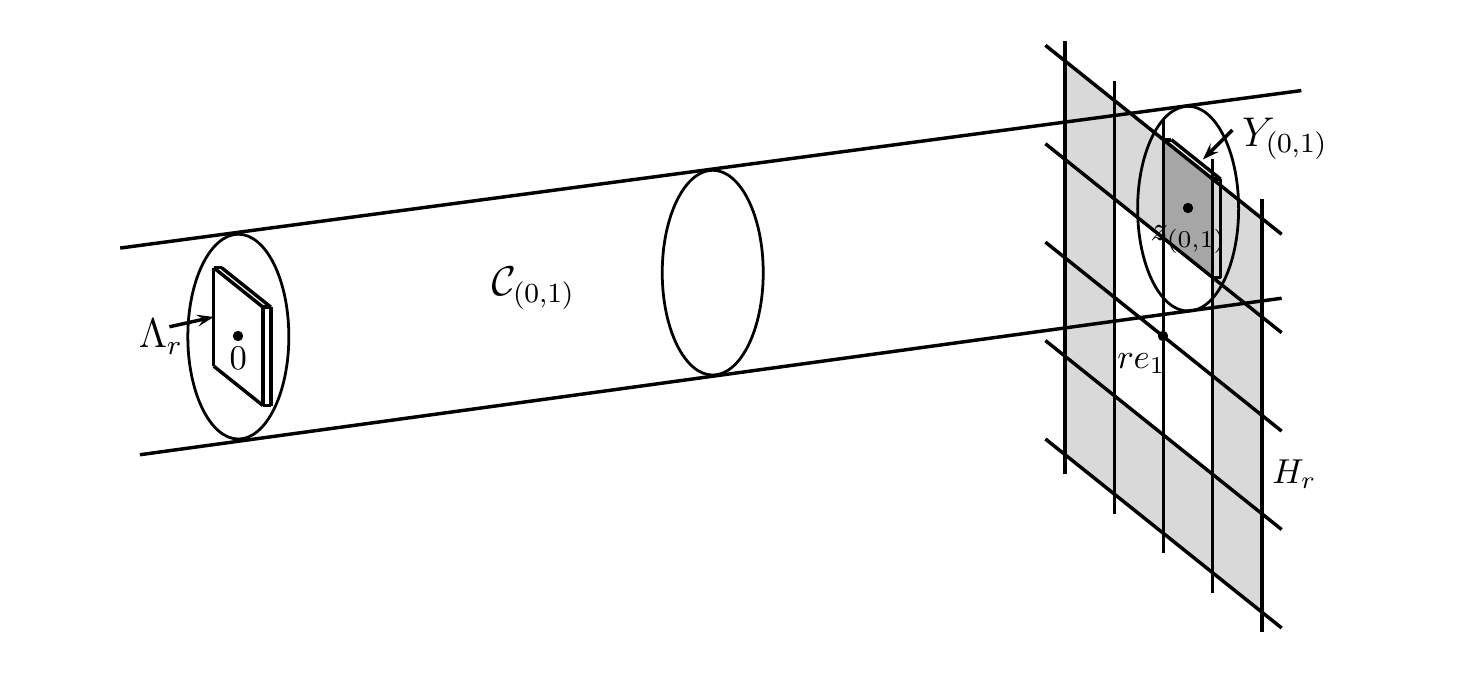}
\caption{ The block $\Lambda_r$ in $H_0$ (``fattened'' slightly to thickness 2), the similar block $Y_{\bdm}$ of $H_r$, and the cylinder $\mC_{\bdm}$ containing both, for $\bdm=(0,1)$. The gray blocks in $H_r$ are those with $j=1$. }
\label{Lemma6-1}
\end{figure}
The definition of $K$ says $t=C_{26}(d-1)K^2$, so
by rotational invariance we can apply Theorem \ref{nofast} (for downward deviations) to the cylinder $\mC_{\bdm}$ and obtain
\begin{align}
  P&\Big( T(x,y) \leq h(r) - t\sigma_r \text{ for some $x\in \Lambda_r\cap \VV$ and $y\in Y_{\bdm}$} \Big) \notag\\
  &\leq P\left( T(x,y) \leq h(|y-x|) - \left( t + \frac{\mu(jK)^2}{6} \right) \sigma_r 
    \text{ for some $x\in \Lambda_r\cap \VV$ and $y\in Y_{\bdm}$} \right) \notag\\
  &\leq C_{27}\exp\left( -C_{28}\left( t + \frac{\mu(jK)^2}{6} \right) \right).
\end{align}
Summing over $j\leq \tred{j_0} = \max\{j: jK\Delta_r \leq c_0(\log r)^{1/2} \Delta_r\}$ gives 
\begin{align}\label{medy}
  P&\Big( T(x,y) \leq h(r) - t\sigma_r \text{ for some $x\in \Lambda_r\cap \VV, y\in H_{[r,r+2]}$ with 
    $|y^*|\leq c_0(\log r)^{1/2} \Delta_r$} \Big) \notag\\
  &\leq P\Big( T(x,y) \leq h(r) - t\sigma_r \text{ for some $x\in \Lambda_r\cap \VV, y\in Y_{\bdm}$, 
    and $\bdm\in\mkM$}\notag\\
  &\hskip 2cm \text{with $d_\infty(Y_{\bdm},[r,r+2]\times\{0\}) \leq j_0K\Delta_r$} \Big) \notag\\
  &\leq \sum_{j=0}^{j_0} P\bigg( T(x,y) \leq h(r) - t\sigma_r 
    \text{ for some $x\in \Lambda_r\cap \VV, y\in Y_{\bdm}$, and $\bdm\in\mkM$} \notag\\
  &\hskip 2cm \text{with $d_\infty(Y_{\bdm},[r,r+2]\times\{0\}) = jK\Delta_r$} \bigg) \notag\\
  &\leq \sum_{j=0}^{j_0} (2j+2)^{d-1} C_{27}\exp\left( -C_{28}\left( t + \frac{\mu(jK)^2}{6} \right) \right)\notag\\
  &\leq c_7e^{-c_8t}.
\end{align}
With \eqref{angles}, \eqref{largey}, and \eqref{verylgy} this completes the proof of Lemma \ref{tubes}.
\end{proof}

Define the ``small box''
\[
  \tred{\mG_r} = [0,r]\times [-\Delta_r,\Delta_r]^{d-1}.
\]
We now tile a region of $\RR^d$ with translates of  $\mG_r$.  Fix \tred{$k,K$} to be specified and let $j=\tred{j(r,k,K)}$ be the least $j$ for which $(2j+1)\Delta_r \geq K\Delta(kr)$.  Define the ``big box''
\[
  \tred{\mG^+} = [0,kr] \times [-(2j+1)\Delta_r, (2j+1)\Delta_r]^{d-1},
\]
so the width of $\mG^+$ is near $2K\Delta(kr)$.  
For $\bdm\in\ZZ^d$ let \tred{$\mG_{r,\bdm},\Lambda_{r,\bdm}$} be $\mG_r,\Lambda_r$ translated by $(rm_1,2\Delta_r \bdm^*)$, and let $\tred{\mM} = \{\bdm\in\ZZ^d: \mG_{r,\bdm}\subset \mG^+ \}$; see Figure \ref{Thm1-4}. Then the number of small boxes comprising $\mG^+$ is 
\[
  |\mM| = k(2j+1)^{d-1} \leq 2k\left( \frac{K\Delta(kr)}{\Delta_r} \right)^{d-1} \leq c_0K^{d-1} k^{1+(d-1)\chi_2},
\]
and we have $\mG^+\supset G_{kr,K}\cap H_{[0,kr]}$.
Let \tred{$\mE_{\bdm}(t)$} be the event in \eqref{tubes1} translated to $\mG_{r,\bdm}$, that is, we replace $\Lambda_r$ with $\Lambda_{r,\bdm}$ and $H_r^+$ with $H_{(m_1+1)r}^+$. Fix \tred{$C$} to be specified; when $\omega\in \mE_{\bdm}(C\log k)$ we say $\Lambda_{r,\bdm}$ is a \emph{fast source}. Let $\tred{\mL(r)} = (h(r)-r\mu)/\sigma_r$, so 
\begin{equation}\label{Lbound}
  0\leq \mL(r)\leq C_{46}\log r
\end{equation}
by Proposition \ref{hmu}.  

\begin{figure}
\includegraphics[width=17cm]{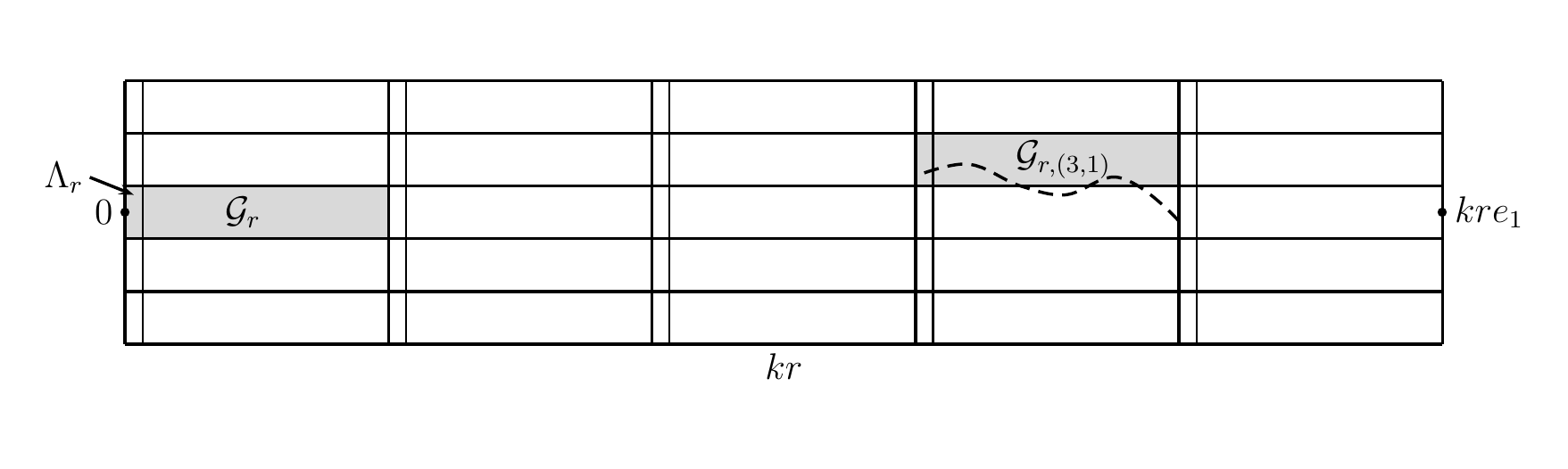}
\caption{ The big box $\mG^+$, of size $kr\times (4j+2)\Delta_r\approx kr\times K\Delta(kr)$, here with $k=5,j=2$. Each small box $\mG_{r,\bdm}$ has size $r\times 2\Delta_r$. $\Lambda_{r,(3,1)}$ is a fast source if some path like the dashed one is ``sufficiently fast'' relative to $h(r)$.}
\label{Thm1-4}
\end{figure}

We now consider the passage time between the ends of the big box.  If $\Gamma_{xy} \subset G_{kr,K}$ then $\Gamma_{0,kre_1}$ must intersect at least $k$ of the regions $\Lambda_{r,\bdm}$, one for each $0\leq m_1<k$.  If there are no fast sources, this means $T(0,kre_1) > kr\mu + k\mL(r)\sigma_r - C\sigma_rk\log k$.  Provided $K$ is large, it then follows from Lemmas \ref{xyfixed} and \ref{tubes} that 
\begin{align}\label{allslow}
  P\Big( T(0,kre_1) \leq kr\mu + k\mL(r)\sigma_r - C\sigma_rk\log k \Big)
    &\leq P\Big( \Gamma_{0,kre_1}\not\subset G_{kr,K} \Big) + P\Big( \cup_{\bdm\in\mM} \mE_{\bdm}(C\log k) \Big) \notag\\
  &\leq C_{75} e^{-C_{76}K^2} + C_{81} |\mM| e^{-C_{82}C\log k}\notag\\
  &\leq C_{75} e^{-C_{76}K^2} + c_1K^{d-1} k^{1+(d-1)\chi_2} e^{-C_{82}C\log k}.
\end{align}
By taking $K$ large and then $C$ large, we can make the right side of \eqref{allslow} less than 1/2 for all $k\geq 3$.  On the other hand, from \eqref{expbound2} for large $c>1$ we have
\[
  P\Big( T(0,kre_1) \leq kr\mu + \mL(kr)\sigma(kr) + c\sigma(kr) \Big) \geq 1-C_{24}e^{-C_{25}c} > \frac 12.
\]
From this and \eqref{allslow} it follows that
\begin{equation}\label{allr}
  (\mL(kr) + c)\sigma(kr) > k(\mL(r) - C\log k)\sigma_r \quad\text{for all $r$ large and } k\geq 3.
\end{equation}
Fix $k\geq 3$ large enough so $k^{1-\chi_2}\geq 3C_{23}c$, with $C_{23}$ from \eqref{powerlike}, and, to get a contradiction, suppose there exists $r$ arbitrarily large with $\mL(r)\geq 2C\log k$. By \eqref{powerlike} and \eqref{allr} we then have
\[
  \mL(kr) \geq \frac{\sigma_r}{\sigma(kr)} \frac{k\mL(r)}{2} - c \geq \frac{k^{1-\chi_2}}{3C_{23}}\mL(r) \geq c\mL(r).
\]
Then iteration and \eqref{Lbound} give that for all $n\geq 1$, $C_{46}(n\log k+\log r) \geq \mL(k^nr)\geq c^n\mL(r)$. Since this is false for large $n$, we must have $\mL(r)< 2C\log k$ for all large $r$, so $\mL(\cdot)$ is bounded, which completes the proof of Theorem \ref{ghbound}, and thus (3) of Remark \ref{strategy}.

\section{Proof of Theorem \ref{nofast}---upward deviations.}\label{Th1-2up}
We move on to (4) of Remark \ref{strategy}. The proof is similar to the LPP proof for $d=2$ in (\cite{BSS16} Proposition 10.1). 
As with downward deviations, we need only consider $x,y$ as in \eqref{startend}. Let $\mH_{xy}$ be as in the downward--deviations proof, and 
\[
  \tred{\mH_{xy}^{ter}} = \{\text{all terminal hyperplanes in } \mH_{xy}\},
\]
so $\mH_{xy}^{ter} = \{H_{s_1},\dots,H_{s_n}\}$ for some \tred{$n,s_1,\dots,s_n$} depending on $(x,y)$.  Note that $\mH_{xy}^{ter}$ does not depend on $\Gamma_{xy}$; the $j$--terminal hyperplane at each end of $[\hat x_1,\hat y_1]$ is just the second--closest $j$th--scale hyperplane to the endpoint.  We form a marked PG path
\[
  \tred{\Omega_{xy}: \hat x = u^0\to\cdots\to u^{n+1} = \hat y}
\]
by taking $u^i$ as the closest point to $\tred{w^i}=\Pi_{\hat x\hat y} \cap H_{s_i}$ in the $j$th--scale grid, when $H_{s_i}$ is a $j$--terminal hyperplane; this means $|u^i-w^i| \leq \sqrt{d-1}\beta^j\Delta_r$ and $n\leq 2j_1$ (from \eqref{j1}.)  Thus the CG approximation gets coarser as we move away from the endpoints of the interval. See Figure \ref{Scales}.

We proceed generally as in Step 8 of the proof for downward deviations.  Let $\tred{w_\perp^i}$ be the orthogonal projection of $u^i$ into $\Pi_{\hat x\hat y}$. We use the terminology $j$th--scale link, final link, and macroscopic link from there.  For $j\leq j_1$, for a $j$th--scale link $(u^{i-1},u^i)$ we have
\begin{align}\label{uw}
  |u^i-u^{i-1}| \leq |w_\perp^i - w_\perp^{i-1}| + \frac{(|u^{i-1}-w_\perp^{i-1}| + |u^i- w_\perp^i|)^2}{\delta^jr}
    \leq |w_\perp^i - w_\perp^{i-1}| + \frac{2(d-1)\beta^{2j}\Delta_r^2}{\delta^jr}.
\end{align}
Further, the angle \tred{$\alpha$} between $e_1$ and $\hat y - \hat x$ satisfies $\tan\alpha \leq 2K\Delta_r/\ep r$ and therefore
\begin{equation}\label{perpcl}
  |w^i - w_\perp^i| \leq c_0 \frac{K\Delta_r}{\ep r}|u^i-w^i| \leq c_1 \frac{K\beta^j}{\ep}\sigma_r,
\end{equation}
and similarly for $i-1$ in place of $i$.  We have also $\delta^jr/2 \leq (u^i-u^{i-1})_1 \leq 2\delta^jr$ so from \eqref{adddist1} and the bound on $\alpha$ we also get 
\[
  |w^i-w^{i-1}| - (w^i-w^{i-1})_1 = \frac{(w^i-w^{i-1})_1}{(\hat y - \hat x)_1}\Big( |\hat y - \hat x| - (\hat y - \hat x)_1 \Big)
    \leq \frac{2\delta^j}{\ep} \frac{2K^2}{\ep}\sigma_r.
\]
Combining this with \eqref{relsize}, \eqref{uw}, and \eqref{perpcl} we see that
\[
  |u^i-u^{i-1}| \leq (u^i-u^{i-1})_1 + \left( \frac{2(d-1)\beta^{2j}}{\delta^j} + \frac{2c_1K\beta^j}{\ep} 
    + \frac{4\delta^jK^2}{\ep^2} \right)\sigma_r
    \leq (u^i-u^{i-1})_1 + \frac{8\delta^jK^2}{\ep^2}\sigma_r
\]
and therefore by Theorem \ref{ghbound}, using \eqref{relsize} and the assumption $t\geq C_{26}K^2$,
\begin{align}
  h(|u^i-u^{i-1}|) &\leq h((u^i-u^{i-1})_1) + \frac{16\mu\delta^jK^2}{\ep^2}\sigma_r
    \leq \mu (u^i-u^{i-1})_1 + 2C_{33}\sigma(2\delta^jr) + \frac{16\mu\delta^jK^2}{\ep^2}\sigma_r \notag\\
  &\leq \mu (u^i-u^{i-1})_1 + \frac{c_2\delta^{j\chi_1}K^2}{\ep^2}\sigma_r
    \leq  \mu (u^i-u^{i-1})_1 + \frac{\lambda^j}{2} t\sigma_r.
\end{align}
It follows using Lemma \ref{connect} that 
\begin{align}\label{jth}
  P\Big( \hT(u^{i-1},u^i) - \mu (u^i-u^{i-1})_1 \geq \lambda^j t\sigma_r \Big)
    &\leq P\left( \hT(u^{i-1},u^i) - h(|u^i-u^{i-1}|) \geq \frac{\lambda^j}{2} t\sigma_r \right) \notag\\
  &\leq C_{44}\exp\left( -C_{45}\frac{\lambda^jt\sigma_r}{\sigma(2\delta^jr)} \right) \notag\\
  &\leq C_{44}\exp\left( -c_3\left( \frac{\lambda}{\delta^{\chi_1}} \right)^j t \right).
\end{align}
The same is valid for a final link, with $j=j_1$.  For macroscopic links it is valid with $\sigma(2\delta^jr)$ replaced by $\sigma_r$, giving
\begin{equation}\label{mac}
  P\Big( \hT(u^{i-1},u^i) - \mu (u^i-u^{i-1})_1 \geq \lambda t\sigma_r \Big)
    \leq C_{44}e^{-c_4\lambda t}.
\end{equation}

The numbers of possible links inside $G_r(K)$ are as follows:
\begin{itemize}
\item[(i)] $j$th--scale links: at most $c_5K^{d-1}/\delta^{2j}\beta^{2j(d-1)}$,
\item[(ii)] final links: at most $c_6K^{d-1}r^d$,
\item[(iii)] macroscopic links: at most $c_7K^{d-1}/\delta^2\beta^{2(d-1)}$.
\end{itemize}
Since $t\geq C_{26}K^2$, provided $C_{26}$ is taken large enough it follows from these bounds and \eqref{relsize2}, \eqref{jth}, \eqref{mac} that
\begin{align}\label{somelink}
  P&\Big( \text{for some $\hat x,\hat y \in G_r(K)$ and $j\leq j_1$, there is a $j$th--scale link $(v,w)$ in $\Omega_{xy}$} \notag\\   
    &\qquad \text{satisfying } \hT(v,w) - \mu(w-v)_1 \geq \lambda^jt\sigma_r \Big)
    \leq c_8\exp\left( -c_{10}\left( \frac{\lambda}{\delta^{\chi_1}} \right)^j t \right), \notag\\
  P&\Big( \text{for some $\hat x,\hat y \in G_r(K)$ there is a final link $(v,w)$ in $\Omega_{xy}$ } \notag\\   
  &\qquad \text{satisfying } \hT(v,w) - \mu(w-v)_1 \geq \lambda^{j_1}t\sigma_r \Big)
    \leq c_8\exp\left( -c_{10}\left( \frac{\lambda}{\delta^{\chi_1}} \right)^{j_1} t \right), \notag\\
  P&\Big( \text{for some $\hat x,\hat y \in G_r(K)$ the macroscopic link $(v,w)$ in $\Omega_{xy}$ } \notag\\   
  &\qquad \text{satisfies } \hT(v,w) - \mu(w-v)_1 \geq \lambda t\sigma_r \Big)
    \leq c_8e^{-c_{10}\lambda t}.
\end{align}
If $\omega$ is not in any of the events in \eqref{somelink}, and $\omega\notin J^{(0)}(c_{29})$, then for all $x,y$ as in \eqref{Qrunif1},
\begin{align*}
  T(x,y) = T(\hat x,\hat y) &\leq \sum_{i=1}^{n+1} \hT(u^{i-1},u^i) + c_{29}n\log r \notag\\
  &\leq \mu (y-x)_1 + c_{29}j_1\log r + (\lambda + \lambda^{j_1})t\sigma_r
    + \sum_{j=1}^{j_1} \lambda^j t\sigma_r \notag\\
  &\leq h(|y-x|) + \frac{2\lambda}{1-\lambda}t\sigma_r.
\end{align*}
Taking $\lambda<1/3$ it follows that
\[
  P\Big( T(x,y) - ET(x,y) \geq t\sigma_r \text{ for some $x,y \in G_r(K)$ with } |y-x|\geq \ep r \Big) 
    \leq c_{10}e^{-c_{11}t},
\]
which completes the proof of Theorem \ref{nofast} for downward deviations.

\section{Proof of Theorem \ref{tversethm} and Lemma \ref{sublin}.}\label{last}
We finish with (5) of Remark \ref{strategy} by proving Theorem \ref{tversethm}.  In the proof in section \ref{yxfixed} for fixed $x,y$, transverse wandering cases (i)--(iii) were dealt with uniformly over $x,y$, so we need only consider the case there:

(iv) $s\leq c_0(\log r)^{1/2}$ and there exist $x,y,u\in\VV$ with 
\begin{equation}
  x,y \in G_r(K),\quad 
    u \notin G_{r,s},\quad d(u,G_{r,s})\leq 2\quad |y-x| > r_0,\quad u\in\Gamma_{xy}, \quad u_1\in [x_1,y_1].
\end{equation}
See the dashed line in Figure \ref{Theorem1-5}.
As in section \ref{yxfixed}, this means that, for $C_{56}$ from Lemma \ref{monotoneE},
\[
  |u-x| + |y-u| - |y-x| \geq s^2\sigma_r, \quad\text{so}\quad h(|u-x|) + h(|y-u|) - h(|y-x|) \geq \frac{\mu s^2}{2} \sigma_r - C_{56}
  \]
while
\[
  T(x,u) + T(u,y) - T(x,y) = 0
\]
so by Theorem \ref{nofast} and Remark \ref{smalltubes},
\begin{align}
  P\Big( (iv) \text{ holds} \Big) &\leq P\left( \text{there exist $v,w \in G_r(2s)$ with } \big| T(v,w) - h(|w-v|) \big| \geq 
    \frac{\mu s^2}{7} \sigma_r \right) \notag\\
  &\leq C_{36} e^{-c_1s^2}.
\end{align}
Together with cases (i)--(iii) in section \ref{yxfixed}, this completes the proof of Theorem \ref{tversethm}.

\begin{proof}[Proof of Lemma \ref{sublin}]
Fix $M$ (large) and $\ep>0$, and define
\[
  \tred{f(r)} = \log \rho(e^r) - \chi r, \quad \tred{\beta_k} = \sup \{f(r): 2^{k-1}M<r\leq 2^kM\},
\]
so $f(r)=o(r)$ and hence $\beta_k=o(2^k)$.  Letting
\[
  \tred{a_k} = \frac{(2^{k-1}+2^k)M}{2}, 
\]
we define $\tilde f$ 
as follows.  
First let $\tilde f \equiv \beta_1$ on $(0,a_1]$.  Then for those $k\geq 1$ with $\beta_k\geq \beta_{k+1}$, define $\tilde f$ on $(a_k,a_{k+1}]$ by
\[
  \tred{\tilde f(t)} = \begin{cases} \beta_k &\text{if } t\in (a_k,2^kM] \\ \beta_{k+1} &\text{if } t=a_{k+1} \\
    \text{linear on } [2^kM,a_{k+1}]. \end{cases}
\]
For $k\geq 1$ with $\beta_k< \beta_{k+1}$, define $\tilde f$ on $(a_k,a_{k+1}]$ by
\[
  \tred{\tilde f(t)} = \begin{cases} \beta_{k+1} &\text{if } t\in [2^kM,a_{k+1}] \\ \beta_k &\text{if } t=a_k \\
    \text{linear on } [a_k,2^kM]. \end{cases}
\]
Since $\beta_k=o(2^k)$ it is easily seen that if we take $M$ sufficiently large, then the slope of the piecewise--linear function $\tilde f$ is never more than $\ep$, and the slope at $r$ approaches 0 as $r\to\infty$.  Defining $\trho(r)$ by
\[
  \log\tred{\trho(e^r)} = \tilde f(r) + \chi r
\]
it follows that for all $s\geq r>1$,
\[
  \frac{\log\trho(s)-\log\trho(r)}{\log s - \log r} = \frac{\tilde f(\log s) - \tilde f(\log r)}{\log s - \log r} + \chi
    \in [\chi-\ep,\chi+\ep],
\]
and the sublinearly powerlike property follows.  
\end{proof}

\vskip 2mm

\end{document}